\documentclass[a4paper, 11 pt]{article}
\usepackage{a4wide, amsmath, amssymb, mathtools, yfonts}
\usepackage{comment}
\usepackage{tikz}
\usepackage{tikz-cd}
\usepackage[all]{xy}
\usepackage[utf8]{inputenc}
\usepackage{amsthm, mathrsfs}
\usepackage[english]{babel}
\usepackage{hyperref}
\usepackage{authblk}
\usepackage[OT2,T1]{fontenc}
\DeclareSymbolFont{cyrletters}{OT2}{wncyr}{m}{n}
\DeclareMathSymbol{\Sha}{\mathalpha}{cyrletters}{"58}

\numberwithin{equation}{section}
\newtheorem{lemma}{Lemma}[section]
\newtheorem*{claim}{Claim}
\newtheorem*{lemma*}{Lemma}
\newtheorem{theorem}[lemma]{Theorem}
\newtheorem{proposition}[lemma]{Proposition}
\newtheorem{corollary}[lemma]{Corollary}

\newtheorem{conjecture}{Conjecture}
\newtheorem{remark}[lemma]{Remark}

\theoremstyle{definition}
\newtheorem{mydef}[lemma]{Definition}

\newcommand{\Z}{\mathbb{Z}}
\newcommand{\Q}{\mathbb{Q}}
\newcommand{\C}{\mathbb{C}}

\newcommand{\R}{\mathbb{R}}
\newcommand{\FF}{\mathbb{F}}

\newcommand{\Hom}{\mathrm{Hom}}
\newcommand{\Frob}{\textup{Frob}}

\newcommand\Gal{\mathrm{Gal}}

\newcommand\Sym{\mathrm{Sym}}

\newcommand{\res}{\textup{res}}
\newcommand{\inv}{\textup{inv}}

\title{\vspace{-\baselineskip}\sffamily\bfseries Chowla's non-vanishing conjecture over $\FF_q(T)$}
\author[1]{Peter Koymans\thanks{Mathematisch Instituut, Universiteit Utrecht, Postbus 80.010, 3508 TA Utrecht, The Netherlands, p.h.koymans@uu.nl}}
\author[2]{Carlo Pagano\thanks{Department of Mathematics and Statistics, Concordia University, Quebec H3G 1M8, Canada, carlein90@gmail.com; Google DeepMind, carlopagano@google.com}}
\author[3]{Mark Shusterman\thanks{The Dr. A. Edward Friedmann Career Development Chair in Mathematics at the Faculty of Mathematics and Computer Science, Weizmann Institute of Science, 234 Herzl Street, Rehovot 76100, Israel}}
\affil[1]{Utrecht University}
\affil[2]{Concordia University and Google DeepMind}
\affil[3]{Weizmann Institute of Science}

\begin{document}
\maketitle

\begin{abstract}
Let $q \equiv 1 \bmod 8$ be a prime power. We prove that $L\left(\tfrac{1}{2}, \chi\right) \neq 0$ for 100\% of the imaginary quadratic Dirichlet characters of $\FF_q(T)$.
\end{abstract}

\section{Introduction}
Chowla has conjectured in \cite{Cho65} that for all (primitive) quadratic Dirichlet characters $\chi$, the value $L\left(\frac{1}{2}, \chi \right)$ is non-zero. This conjecture, numerically supported by \cite{Rum93}, has been studied in \cite{OS99} by \"{O}zl\"{u}k and Snyder. Assuming GRH, they show non-vanishing for $93.75\%$ of these quadratic characters. Under the same assumptions, an improvement to more than $94.27\%$ has been obtained by Katz and Sarnak. The difficulty of obtaining further numerical improvements is discussed in \cite[Appendix A]{ILS00}.

These works approach the question of vanishing via a statistic of the zeros of $L(s, \chi)$ called $1$-level density. This statistic involves a sum of our quadratic characters over primes; a very natural randomness assumption (going beyond GRH) on the values $\{\chi(p)\}_{\chi,p}$, implies $100\%$ non-vanishing. 

Unconditional results have been obtained by Jutila in \cite{Jut81}, using estimates for low moments of $L \left(\frac{1}{2}, \chi \right)$. In the pioneering work \cite{Sou00}, Soundararajan succeeded in estimating mollified versions of low moments, which allowed him to unconditionally deduce that $L \left(\frac{1}{2}, \chi_{8d} \right) \neq 0$ for at least $87.5\%$ of the odd squarefree integers $d$.

For points $s \neq \frac{1}{2}$ on the critical line $\mathrm{Re}(s) = \frac{1}{2}$, it is sometimes possible to find quadratic $\chi$ with $L(s, \chi) = 0$, yet all the results above admit variants (with possibly different percentages) valid at such $s$. Moreover, $100\%$ non-vanishing is expected to hold for any fixed $s$.   

In this work we consider the function field analog of the above non-vanishing problem. A variant of the randomness assumption above suggests the following $100\%$ non-vanishing folklore conjecture.

\begin{conjecture}
Let $q$ be an odd prime power, and let $s$ be a point on the critical line. For a nonnegative integer $n$ define
\[
\delta_q(s;n) = 
\frac{\#\{ D \in \mathbb{F}_q[T] : \deg(D) = n, \ D \ \textup{is squarefree}, \ L \left( s, \chi_D \right) \neq 0  \}}{\#\{ D \in \mathbb{F}_q[T] : \deg(D) = n, \ D \ \textup{is squarefree}  \}}
\]
where $\chi_D(f)$ is the Jacobi symbol $\left( \frac{D}{f} \right)$, and 
\[
L(s, \chi_D) = \sum_{\substack{f \in \mathbb{F}_q[T] \\ f \ \textup{is monic}}} \chi_D(f)|f|^{-s}, \quad |f| = |\mathbb{F}_q[T]/(f)| = q^{\deg(f)}.
\]
Then
\[
\lim_{n \to \infty} \delta_q(s; n) = 1.
\]
\end{conjecture} 

As shown by Li in \cite{Li18}, the central values of $L$-functions of (primitive) quadratic Dirichlet characters over $\mathbb F_q[T]$ with $q$ odd can vanish infinitely often, so in this setting Chowla's non-vanishing conjecture can only be stated (and approached) statistically.

We establish the following. 

\begin{theorem}
\label{tChowla}
There are absolute constants $C, c > 0$ such that the following holds. Let $q \equiv 1 \bmod 8$ be a prime power and let $n \in \Z_{\geq 1}$ be odd. Then we have
$$
\# \left\{D \in \FF_q[T] : \deg(D) = n, \ D \textup{ is squarefree}, \ L\left(\tfrac{1}{2}, \chi_D \right) = 0\right\} \leq \frac{C q^{n + 1}}{n^c}.
$$
\end{theorem}

As a consequence of Theorem \ref{tChowla}, we obtain:

\begin{corollary}
Let $q \equiv 1 \bmod 8$ be a prime power. For 100\% of the imaginary quadratic Dirichlet characters $\chi$ of $\FF_q(T)$, ordered by their genus, we have $L\left(\tfrac{1}{2}, \chi\right) \neq 0$. In other words, we have that
\[
\lim_{k \to \infty} \delta_q\left(\frac{1}{2}; 2k + 1\right) = 1.
\]
\end{corollary}

\subsection{Comparison with the literature}
The `large finite field' version of the conjecture, stating that as $q \to \infty$ we have
\[
\delta_q(s; n) = 1 + O_n \left(q^{-\frac{1}{2}} \right)
\]
for any fixed $n$, is a consequence of the work \cite{Cha97} by Chavdarov, building on a big monodromy theorem of Yu. The role of the monodromy group of a family of $L$-functions in this and other problems (such as moments and $1$-level density mentioned above) was put on a firm ground by Katz and Sarnak in \cite{KS99}. Their work provides a broad conjectural picture, unifying number fields and function fields, where results in the large finite field limit can sometimes be obtained using a monodromy computation combined with Deligne's equidistribution theorem.

One major attribute of the situation over function fields is that GRH, for $L$-functions of characters, is unconditionally known, due to Weil. This allowed Bui and Florea in \cite{BF18} to unconditionally obtain the aforementioned non-vanishing percentage of Katz--Sarnak, building on (and refining) the earlier work \cite{Rud10} on $1$-level density by Rudnick. There has been significant recent work on developing new techniques for quadratic families \cite{AK, Florea, GZ, LOP}, extending non-vanishing results to higher-order characters \cite{CDD, DDDS, DFL, DFL2}, and studying the related problem of vanishing for twists of $L$-functions of elliptic curves \cite{CDLL}.

Another perspective on Weil's result, and on vanishing in general, is given, for any prime $\ell \nmid q$ and positive integer $k$, by Grothendieck's congruence between $q^{-s(\deg(D) - 1)}L(-s, \chi_D)$ and the characteristic polynomial of $\mathrm{Frob}_q$ acting on the $\ell^k$-torsion of the Jacobian of (the smooth completion $C_D$ of) the curve $y^2 = D(T)$ over $\overline {\mathbb{F}_q}$. In the previously discussed approaches, non-vanishing of $L(s, \chi_D)$ is closely tied with showing that $|L(s, \chi_D)|$ is not too small in the archimedean sense, whereas in this perspective, it is the $\ell$-adic nature of $L(s, \chi_D)$ that one pursues.

The above perspective, for odd primes $\ell$ and $k = 1$, is taken in the work \cite{ELS19} of Ellenberg, Li, and Shusterman. Building on geometric (and topological) results of Ellenberg, Venkatesh, and Westerland from \cite{EVW16} and of Lipnowski and Tsimerman from \cite{LT19}, it is proven in \cite{ELS19} that
\[
\liminf_{n \to \infty} \delta_q(s; n) = 1 + O \left( q^{-1/276} \right), \quad q \to \infty.
\]
The starting point of the proof of Theorem \ref{tChowla} is to examine the situation for $\ell = 2$ and $k$ arbitrary.

\subsection{Method of proof}
Our first step is to construct for each character $\chi$ a decreasing sequence of $\mathbb{F}_2$-vector spaces along with pairings\footnote{For each $\chi$ there will be two such sequences that we also index somewhat differently in the later sections. We shall omit both details in this introduction.}
$$
\text{Art}_{m, \chi}: V_m(\chi)^\circ \times W_m(\chi)^\circ \to \mathbb{F}_2,
$$
defined for all $m \in \Z_{\geq 1}$. The main two features of this sequence of spaces and pairings are:
\begin{enumerate}
    \item We have
    $$
    \dim_{\mathbb{F}_2} \, V_m(\chi)^\circ = 0 \ \text{for} \ m \gg 0 \Longrightarrow L\left( \tfrac{1}{2}, \chi \right) \neq 0.
    $$
    \item The left and right kernels of $\text{Art}_{m, \chi}$ are $V_{m + 1}(\chi)^\circ$ and $W_{m + 1}(\chi)^\circ$.
\end{enumerate}
The spaces $V_m(\chi)^\circ$ and $W_m(\chi)^\circ$ and the pairings $\text{Art}_{m, \chi}$ are best understood as generalized Selmer spaces equipped with Cassels--Tate pairings. One could build these pairings using Flach's theory \cite{Flach} of Bloch--Kato Selmer groups \cite{BK}, but we opted for a direct construction translating $L\left(\tfrac{1}{2}, \chi \right)=0$ via the Weil conjectures as the presence of an eigenspace of $\sqrt{q}$ on $\text{Jac}(C_\chi)[2^{\infty}]$: in turn this translates into the existence of a tower of Galois extensions of $\mathbb{F}_q(T)(\chi)$ with prescribed Galois group and ramification behavior. The characters in 
$$
V_\infty(\chi)^\circ := \bigcap_{m \geq 1} V_m(\chi)^\circ
$$
are precisely the $\mathbb{F}_2$-characters at the bottom of such towers. 

One can show that the spaces $V_1(\chi)^\circ$ and $W_1(\chi)^\circ$ naturally arise as left and right kernel of a R\'edei-like matrix of mutual Legendre symbols of primes ramifying in $\mathbb{F}_q(T)(\chi)$. In particular, using techniques dating back to Heath-Brown \cite{HB, HB2} and Fouvry--Kl\"uners \cite{FK}, one can show that $\dim_{\mathbb{F}_2} \, V_1(\chi)^\circ$ has a distribution. This allows us to view the sequence $\dim_{\mathbb{F}_2} \, V_m(\chi)^\circ$ as a non-increasing Markov chain on the nonnegative integers and thus reduces our task to showing that the chain decreases at each step with positive probability.
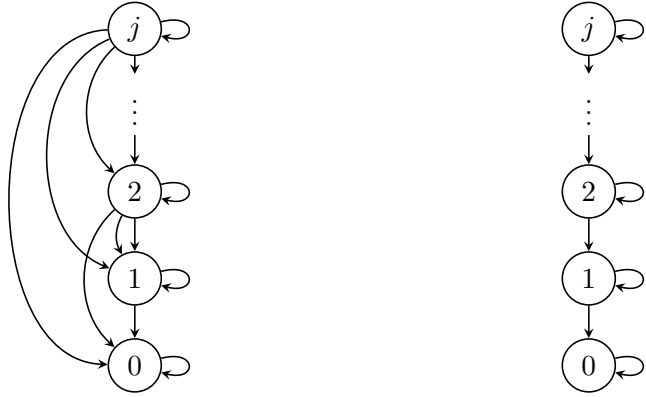
\begin{figure}
\centering
\begin{tikzpicture}[
  >=stealth,
  semithick,
  state/.style={circle,draw,minimum size=7mm,inner sep=0pt}
]

  \node[state] (zero)  at (-3,0)    {$0$};
  \node[state] (one)   at (-3,1.15) {$1$};
  \node[state] (two)   at (-3,2.30) {$2$};
  \node        (dots)  at (-3,3.45) {$\vdots$};
  \node[state] (j)     at (-3,4.45) {$j$};

  \node[state] (zero1) at (3,0)    {$0$};
  \node[state] (one1)  at (3,1.15) {$1$};
  \node[state] (two1)  at (3,2.30) {$2$};
  \node        (dots1) at (3,3.45) {$\vdots$};
  \node[state] (j1)    at (3,4.45) {$j$};

  \path[->]
    (zero)  edge[loop right,min distance=5mm] (zero)
    (one)   edge[loop right,min distance=5mm] (one)
            edge (zero)
    (two)   edge[loop right,min distance=5mm] (two)
            edge[bend right=28] (one)
            edge[bend right=48] (zero)
            edge (one)
    (j)     edge[loop right,min distance=5mm] (j)
            edge[bend right=48] (two)
            edge[bend right=68] (one)
            edge[bend right=88] (zero);

  \draw[->] (j.south) -- (dots.north);
  \draw[->] (dots.south) -- (two.north);

  \path[->]
    (zero1) edge[loop right,min distance=5mm] (zero1)
    (one1)  edge[loop right,min distance=5mm] (one1)
            edge (zero1)
    (two1)  edge[loop right,min distance=5mm] (two1)
            edge (one1)
    (j1)    edge[loop right,min distance=5mm] (j1);

  \draw[->] (j1.south) -- (dots1.north);
  \draw[->] (dots1.south) -- (two1.north);

  \pgfresetboundingbox
  \path (-3.4,-0.4) rectangle (3.4,4.8);

\end{tikzpicture}
\caption{Two Markov chains}
\label{fMarkov}
\end{figure}
This calls for a strategy determining the initial distribution of the Markov chain and all of the transition probabilities, yielding a full distribution of all of the dimensions $\dim_{\mathbb{F}_2} \, V_m(\chi)^\circ$, as depicted on the left side of Figure \ref{fMarkov}.

On the surface level, this strategy is strongly reminiscent of Smith's approach to Goldfeld's conjecture \cite{Smi1, Smi2, Smi3} and one could hope to adapt his methods to this sequence of Selmer spaces. However, we emphasize that the method faces a fundamental roadblock whenever even the first pairing $\text{Art}_{1, \chi}$ is symmetric. This is for instance the case for the distribution of $2 \cdot \text{Cl}(K)[2^{\infty}]$ when $K$ is a random quadratic extension of $\Q(i)$, which is one of the most fundamental open cases left out by Smith's methods. As we shall now explain, in our case \emph{all} of the pairings $\text{Art}_{m, \chi}$ are symmetric, putting us outside of the direct scope of \cite{Smi1, Smi2, Smi3}. 

Indeed, the number $\sqrt{q}$ is a self-dual Weil-number, namely it is a fixed point under Poincar\'e duality
$$
\alpha \mapsto \frac{q}{\alpha}.
$$
Such a duality reflects algebraically into the fact that our Selmer groups come from a \emph{fractional} Tate-twist along with a perfect Galois equivariant pairing $\Z_2(1/2)^{\otimes 2} \to \Z_2(1)$. This translates into a fundamental piece of extra structure, which is an isomorphism
$$
\text{Sym}(\chi): V_1(\chi)^\circ \to W_1(\chi)^\circ
$$
such that the induced pairing $\text{Art}_{m, \chi} \circ (\text{id}, \text{Sym})$ is a \emph{symmetric pairing} for each $m \in \Z_{\geq 1}$. A framework to study how such self-duality propagates at the level of Cassels--Tate pairings has been developed in great generality in Morgan--Smith \cite{MS}.

The fundamental obstruction in using the techniques from Smith \cite{Smi1, Smi2, Smi3} comes down to handling the \emph{diagonal pairings}
$$
\text{Art}_{m,\chi}(a, \text{Sym}(\chi)(a)).
$$
In fact, for the purpose of showing that $V_\infty(\chi)^\circ$ vanishes almost always, it suffices to control only such pairings. More precisely, it suffices to show that the first moment of $|V_1(\chi)^\circ|$ is bounded and prove that, from each positive state, the chain descends down by $1$ with probability at least $\frac{1}{2}$. This is the strategy we execute (starting at $j := \dim_{\FF_2} V_1(\chi)^\circ$) as depicted on the right of Figure \ref{fMarkov}.

The only work that has previously addressed the distribution of Selmer groups in the presence of such diagonal pairings is the work of the first two authors on the negative Pell equation \cite{KP}. However, not only did their techniques exploit the presence of an additional systematic class specific to that family, but one can also show that the technique developed there can \emph{never work} for a sequence of pairings that is symmetric at all levels (such as $\mathrm{Art}_{m, \chi}$). In the context of \cite{KP}, only the first pairing turns out to be symmetric. 

This calls for developing a new method to handle diagonal pairings extending Smith's techniques \cite{Smi1, Smi2, Smi3}: this is one of the main accomplishments of the present work, which we now briefly overview.

A crucial ingredient in Smith's work to prove equidistribution for the pairings $\text{Art}_{m,\chi}(a,b)$ is to compare $2^m$ fields selected in the ``cube''
$$
C := \{\pi_1(1), \pi_1(2)\} \times \ldots \times \{\pi_m(1), \pi_m(2)\} \times \{d\},
$$
where $a, b \mid d$. It turns out that under suitable conditions on the lower-dimensional faces of this cube, one can express $\sum_{x \in C} \text{Art}_{m, x}(a, b)$ as a first Artin pairing over the field 
$$
F(C) := \mathbb{F}_q(T)(\{\sqrt{\pi_1(1)\pi_1(2)}, \ldots, \sqrt{\pi_m(1)\pi_m(2)}\}),
$$
which is essentially a Legendre symbol
$$
\left(\frac{\overline{\alpha_a}}{\alpha_b} \right)_{F(C),2},
$$
where $\alpha_a, \alpha_b$ are products of primes of $F(C)$ lying above the primes of $a$ and $b$ respectively, and $\overline{\alpha}$ denotes a conjugate element under any choice of an automorphism of $\Gal(F(C)/\mathbb{F}_q(T))$. If $a$ and $b$ are sufficiently disjoint, one can hope to apply large sieve inequalities to show oscillation of this symbol. This is for example the approach taken in \cite{Smi1, Smi2}. Instead when $a = b$, one gets a so-called \emph{involution spin} and the work of the first two authors \cite[Theorem 5.5]{KP} shows that 
$$
\left(\frac{\overline{\alpha_a}}{\alpha_a} \right)_{F(C), 2 } =1.
$$
Our new approach is to work one dimension lower and instead sum $2^{m - 1}$ Artin pairings in this case. It turns out that for cubes of the shape
$$
C' := \{\pi_1(1), \pi_1(2)\} \times \ldots \times \{\pi_{m - 1}(1), \pi_{m - 1}(2)\} \times \{d\},
$$
one can express $\sum_{x \in C'} \text{Art}_{m, x}(a, a)$ as a second Artin pairing over $F(C')$. It is well-known that such second Artin pairings are related to R\'edei symbols. In fact, this sum is essentially equal to $\left[\overline{\alpha_a}, \alpha_a, \alpha_a\right]_{F(C')}$ (see Theorem \ref{theorem: reflection principle} for a precise statement), which after applying R\'edei's reciprocity \cite{KS} can be rewritten as 
$$
\left( \frac{\overline{\alpha_a}}{\alpha_a} \right)_{F(C'), 4},
$$
i.e.~a \emph{quartic spin symbol}. Crucially, such a symbol satisfies a twisted multiplicative relation, namely quartic reciprocity (recall that $q \equiv 1 \bmod 8$) yields
$$
\left( \frac{\overline{\alpha_{a_1a_2}}}{\alpha_{a_1a_2}} \right)_{F(C'), 4} \approx \left( \frac{\overline{\alpha_{a_1}}}{\alpha_{a_1}} \right)_{F(C'), 4} \left(\frac{\overline{\alpha_{a_2}}}{\alpha_{a_2}}\right)_{F(C'), 4} \left( \frac{\overline{\alpha_{a_2}}}{\alpha_{a_1}} \right)_{F(C'), 2}.
$$
One of our insights is to leverage the fact that our $a$ will typically consist of many prime factors, and exploit once more large sieve inequalities to establish oscillation of the quadratic cross term in the equation above. Indeed, a fundamental idea underpinning our approach is that the symbol $(\overline{\alpha_{a_i}}/\alpha_{a_i})_{F(C'), 4}$ is very hard to control, but fortunately it can be absorbed as a coefficient in a large sieve type result for $(\overline{\alpha_{a_2}}/\alpha_{a_1})_{F(C'), 2}$.
 
In order to make this approach work, it is vital to have the additional flexibility to bring in extra prime divisors of $a$ after summing the Artin pairings over the cube. In order to achieve this, a key new analytic ingredient in our approach is the use of repeated Cauchy--Schwarz. This idea was previously exploited in \cite{KS} for the second Artin pairing (in the context of $3$-Selmer groups of the elliptic curve $x^3 + y^3 = n$), and we are deeply grateful to Alexander Smith for explaining to us how to extend this to higher-dimensional cubes. An additional benefit of this technique is that it provides much stronger savings (i.e.~small log power savings) than previously possible.

\subsection{Layout of the paper}
Section \ref{Section:preliminaries} introduces the relevant Selmer groups, Artin pairings and translates their eventual vanishing into $L\left(\frac{1}{2},\chi \right) \neq 0$. Section \ref{Section: first Artin pairing} expresses the first Artin pairing as a matrix of Legendre symbols and also establishes symmetry of all of the higher Artin pairings. 

Section \ref{section: expansion maps} transfers to function fields the basics on expansion maps and in Section \ref{Section: phi hat} we introduce a new category of expansion maps (called $\hat{\phi}$-expansion maps). The main result is a certain linear independence result needed to prove that the final symbol coming from the Artin pairing does not get canceled by the conditions coming from the lower-dimensional faces of the cube; since we work one dimension lower than usual, this is substantially more subtle than in the previous works in the literature.

Section \ref{Section: redei} establishes the basics on R\'edei symbols over general function fields adapting \cite{KS}. Section \ref{Section: large sieve} covers squarefree polynomials in grids (i.e.~convenient product spaces of sets of primes) and gives the large sieve inequality used at the end of the proof. 

Section \ref{Section: reflection principles} establishes our new reflection principles, chiefly the one involving the quartic spin symbol. Section \ref{sCS} executes the Cauchy--Schwarz approach and concludes the proof of our main result. 

\subsection{Human-AI collaboration}
All of the new ideas of this paper are due to the authors and were conceived and developed starting from June $2019$ until April $2026$. The execution of these ideas are also due to the authors with the following exceptions, where artificial intelligence was used for minor technical points:
\begin{itemize} 
\item parts of the proof of Theorem \ref{tGrid} were originally executed by \emph{GPT 5.5} in instant mode and then rewritten by the authors,
\item some parts of the proof of Lemma \ref{lFuncSeqEq} were originally executed by \emph{GPT 5.6 Sol pro} and then rewritten by the authors,
\item the elegant trick employed in the proof of Theorem \ref{tReduction} to expand a characteristic function via an interpolating polynomial to prove the reduction to sets $Y(R, \mathbf{T}, \mathbf{v})$ with fixed possible type sequences is due to \emph{GPT 5.6 Sol pro}. 
\end{itemize}
\noindent We have extensively used \emph{GPT 5.6 Sol pro} and \emph{Fable} to proofread this work. 

\subsection*{Acknowledgements}
The authors have discussed parts of this paper at several institutes and semester programs, and are grateful for the hospitality of these institutions. These include: the Hausdorff  Trimester Program ``Definability, decidability, and computability'', CRM special semester ``Universal statistics in number theory'', the Max Planck Institute, and the Institute for Theoretical Studies (ETH). The authors are also grateful to the organizers of the following conferences: Bonn (``Enumerative Arithmetic and the Cohen-Lenstra heuristics''), Oberwolfach (``Arithmetic Statistics for Algebraic Objects''), Providence (``Number Field Counting in the LMFDB'' at ICERM), Paderborn (``Arithmetic Statistics'') and Lausanne (``Fouvry 73'').

We would also like to thank Marco d'Addezio, Henri Darmon, Jordan Ellenberg, Adam Morgan and Will Sawin for valuable discussions. We are especially grateful to Alexander Smith for explaining to us the higher-dimensional version of the repeated Cauchy--Schwarz technique. 

PK gratefully acknowledges the support of the Dutch Research Council (NWO) through the Veni grant ``New methods in arithmetic statistics'' and the ERC Starting Grant ``CEDE''. MS is co-funded by the European Union (ERC, Function Fields, 101161909).
\section{Preliminaries} \label{Section:preliminaries}
\subsection{Notation}
\label{ssNot}
We begin by fixing some important conventions that will be used throughout.

In the rest of this paper $q$ will denote an odd prime power. All our finite separable extensions of $\mathbb{F}_q(T)$ will be chosen inside a fixed separable closure $\mathbb{F}_q(T)^{\text{sep}}$ of $\mathbb{F}_q(T)$. A finite extension of $\mathbb{F}_q(T)$ chosen inside this separable closure will be called a \emph{function field}. We use the following notation
\begin{itemize}
\item for $n \in \Z_{\geq 0}$, we denote by $[n]$ the set of integers between $1$ and $n$;

\item we define $G_K := \Gal(\mathbb{F}_q(T)^{\text{sep}}/K)$ for $K$ a finite separable extension of $\mathbb{F}_q(T)$;

\item we define $\Omega_{\mathbb{F}_q(T)}$ to be the set of equivalence classes of non-trivial absolute values on $\mathbb{F}_q(T)$; these correspond to monic irreducible polynomials in $\mathbb{F}_q[T]$ plus the place at $\infty$;

\item we write $\Omega_{\mathbb{F}_q(T)}^{\text{fin}}$ for the subset of $\Omega_{\mathbb{F}_q(T)}$ corresponding to the monic irreducible polynomials $\pi(T)$ in $\mathbb{F}_q[T]$, and we denote by $\infty$ the unique remaining place of $\Omega_{\mathbb{F}_q(T)}$;

\item for every place $v \in \Omega_{\mathbb{F}_q(T)}$, we fix once and for all a separable closure $\mathbb{F}_q(T)_v^{\text{sep}}$ of $\mathbb{F}_q(T)_v$, which comes with a unique extension of the absolute value $| \cdot |_v$;

\item for a function field $K$, we write $\Omega_K$ for the set of places of $K$;

\item for every field $L$ of which we have fixed a separable closure $L^{\text{sep}}$, we will denote by $G_L$ the absolute Galois group $\Gal(L^{\text{sep}}/L)$;

\item for every place $v \in \Omega_{\mathbb{F}_q(T)}$, we fix once and for all an embedding $i_v \colon \mathbb{F}_q(T)^{\text{sep}} \to \mathbb{F}_q(T)_v^{\text{sep}}$ and we denote the corresponding map on Galois groups by $i_v^\ast \colon G_{\mathbb{F}_q(T)_v} \to G_{\mathbb{F}_q(T)}$, which is an injection in view of Krasner's lemma. Its image is the decomposition subgroup attached to the chosen prolongation of $v$ to $\mathbb F_q(T)^{\mathrm{sep}}$;

\item we fix once and for all a separable closure $\mathbb{F}_q^{\text{sep}}$ of $\mathbb{F}_q$, living inside $\mathbb{F}_q(T)^{\text{sep}}$. Hence we have $\mathbb{F}_q^{\text{sep}}(T) \subseteq \mathbb{F}_q(T)^{\text{sep}}$;

\item we denote by $\Frob_q$ the topological generator of $G_{\mathbb{F}_q}$, given by the automorphism of $\mathbb{F}_q^{\text{sep}}$ sending each $\alpha$ to $\alpha^q$;

\item for each $v$ in $\Omega_{\mathbb{F}_q(T)}$, we put $q_v := |i_v(\mathbb{F}_q^{\text{sep}}) \cap \mathbb{F}_q(T)_v|$. This is the size of the residue field at $v$, and is independent of $i_v$. We denote by $\deg(v)$ the unique positive integer such that $q^{\deg(v)} = q_v$ and we call it the degree of $v$;

\item for a profinite group $G$, we denote by $G(2)$ the maximal pro-$2$ quotient of $G$, and we denote by $\text{proj} \colon G \twoheadrightarrow G(2)$ the natural continuous surjective group homomorphism;

\item finite groups are always viewed as discrete topological groups; 

\item for a function field $K$ and a finite abelian $2$-group $M$ (viewed as a discrete $G_K$-module with trivial action), we denote by $\Gamma_M(K)$ the group
$$
\text{Hom}_{\text{top.gr.}}(G_K, M) = \text{Hom}_{\text{top.gr.}}(G_K(2), M).
$$
For each $\chi$ in $\Gamma_{\mathbb{F}_2}(K)$, we put 
$$
K(\chi) := (K^{\text{sep}})^{\text{ker}(\chi)}.
$$
More generally, for $A$ a finite subset of $\Gamma_{\mathbb{F}_2}(K)$, we denote by $K(A)$ the compositum of all the fields $K(\chi)$, as $\chi$ runs in $A$. We put $\text{Ram}(A) := \{v \in \Omega_K : v \text{ ramifies in } K(A)/K\}$ and $\text{Ram}(\chi) := \text{Ram}(\{\chi\})$;

\item for a function field $K$, a discrete topological space $X$ and a continuous function $f\colon G_K \to X$, one can show that there exists an open normal subgroup $N$ of $G_K$ such that the map $f$ factors through $G_K/N$. The largest such $N$ is called the group of definition of $f$, denoted by $N(f)$, and the field
$$
L(f) := (K^{\text{sep}})^{N(f)}
$$
is called the field of definition of $f$. In case $f$ is a quadratic character, we will use both notations $L(f)$ and $K(f)$ interchangeably depending on the circumstances;

\item for a function field $K$ and an element $a$ in $K^\ast$, we denote by $\chi_a$ the element of $\Gamma_{\mathbb{F}_2}(K)$ coming from adjoining a square root of $a$ to $K$. For $A$ a finite subset of $K^\ast$, we define the continuous $1$-cochain
$$
\chi_A:=\prod_{a \in A}\chi_a,
$$
where the product is performed in $\mathbb{F}_2$;

\item we denote by $\mathbb{F}_{q^n}$ the unique subfield of degree $n$ over $\mathbb{F}_q$ living inside $\mathbb{F}_q^{\text{sep}}$; 

\item we make once and for all a choice $(\epsilon_{\mathbb{F}_{q^{n}}})_{n \geq 1}$ of a norm compatible set, where $\epsilon_{\mathbb{F}_{q^n}}$ is a generator of the quotient group $\mathbb{F}_{q^{n}}^\ast/\mathbb{F}_{q^{n}}^{\ast 2}$;

\item for a profinite group $G$ and a discrete $G$-module $A$, the group $H^i(G, A)$ denotes the usual continuous Galois cohomology as in \cite[Chapter 1]{NSW}; it is the $i$th derived functor of taking $G$-invariants;

\item we denote by $\Gamma_{\mathbb{F}_2}^{\text{im}}(\mathbb{F}_q(T))$ the subset of those characters $\chi$ in $\Gamma_{\mathbb{F}_2}(\mathbb{F}_q(T))$ for which $\infty$ is in $\text{Ram}(\chi)$. We call these characters imaginary; 

\item for each $v$ in $\Omega_{\mathbb{F}_q(T)}^{\text{fin}}$, we denote by $\chi_v$ the element of $\Gamma_{\mathbb{F}_2}(\mathbb{F}_q(T))$ corresponding to the quadratic extension of $\mathbb{F}_q(T)$ obtained by adjoining a square root of $\pi_v(T)$, where $\pi_v(T)$ is the unique monic irreducible polynomial in $\mathbb{F}_q[T]$ for which $v(\pi_v(T)) = 1$. For $v=\infty$ instead, we put $\chi_{\infty}:=\chi_{\epsilon_{\mathbb{F}_q}}$;

\item we identify $G_{\mathbb{F}_q}$ and $\Gal(\mathbb{F}_q^{\text{sep}}(T)/\mathbb{F}_q(T))$ via the natural restriction map. In particular, we view $\Frob_q$ as an element of $\Gal(\mathbb{F}_q^{\text{sep}}(T)/\mathbb{F}_q(T))$ and we fix once and for all a lift to $G_{\mathbb{F}_q(T)}$, which we will also denote by $\Frob_q$, committing an abuse of notation that should nevertheless not generate confusion. Furthermore, we rigidify our choice of lift $\Frob_q$ by demanding that $\chi_v(\Frob_q) = 0$ for each $v$ in $\Omega_{\mathbb{F}_q(T)}^{\text{fin}}$; 

\item by abuse of notation we will denote the element $\text{proj}(\Frob_q)$ of $G_{\mathbb{F}_q(T)}(2)$ by $\Frob_q$ as well;

\item we denote by $\chi_{\Frob}\colon G_{\mathbb{F}_q(T)} \twoheadrightarrow \Z_2$ the canonical surjection characterized by the formula
$$
g \cdot \alpha = (\Frob_q)^{\chi_{\Frob}(g)} \cdot \alpha,
$$
for each $g$ in $G_{\mathbb{F}_q(T)}$ and $\alpha \in \cup_{n \geq 0} \mathbb{F}_{q^{2^n}} =: \mathbb{F}_{q^{2^{\infty}}}$;

\item for a place $v$ of $\Omega_{\mathbb{F}_q(T)}$, we denote by $I_v$ the largest subgroup of $G_{\mathbb{F}_q(T)_v}$ that acts trivially on $i_v(\mathbb{F}_q^{\text{sep}})$ and we call it the inertia subgroup at $v$;

\item for $q$ a prime power with $q \equiv 1 \bmod 8$, we denote by $q_0$ the unique element of $\Z_2$ with $q_0^2 = q$ and $q_0 \equiv 1 \bmod 4$;

\item for a function field $K$, a character $\chi$ in $\Gamma_{\mathbb{F}_2}(K)$ and a $G_K$-module $A$, we denote by $A(\chi)$ the quadratic twist of $A$ by $\chi$, namely the $G_K$-module given by the abelian group $A$ with $G_K$-action provided by the formula
$$
g \cdot_{\chi} a := (-1)^{\chi(g)} \cdot g \cdot a,
$$
for all $g \in G_K$ and $a \in A$;

\item for an integer $n$, a function field or a local field $K$ and a $G_K$-module $M$, we denote by $M(n)$ the $G_K$-module given by the $n$-th Tate twist;

\item for a function field or a local field $K$ and a finite (discrete) $G_K$-module $M$, we denote by $M^\ast := \Hom_{\text{ab.gr.}}(M, \Q/\Z(1))$ its Tate dual;

\item for a function field $K$ and a continuous $1$-cochain $\phi: G_K \rightarrow \mathbb{F}_2$, we define a new $1$-cochain $\frac{\phi}{4}: G_K \rightarrow \Q_2/\Z_2$ by
$$
\frac{\phi}{4}(g) :=
\begin{cases}
1/4 &\text{if } \phi(g) = 1, \\
0 &\text{if } \phi(g) = 0;
\end{cases}
$$

\item for a product of finite sets $X := X_1 \times \dots \times X_n$, we define for each $i \in [n]$ the map $\mathrm{pr}_i\colon X \rightarrow X_i$, which is projection on the $i$th coordinate;

\item for a vector $\mathbf{a} := (a_i)_{1 \leq i \leq n} \in \{1, 2\}^n$, we define 
$$
\mathrm{sg}(\mathbf{a}) = (-1)^{|\{1 \leq i \leq n : a_i = 2\}|}.
$$
\end{itemize}

\subsection{Basic facts}
Let us recall the structure of the local pro-$2$ completions and how we view them as subgroups of $G_{\mathbb{F}_q(T)}(2)$. 

\begin{proposition} 
\label{Structure of tame inertia}
Let $v$ be a place in $\Omega_{\mathbb{F}_q(T)}$. The image of $I_v$ in $G_{\mathbb{F}_q(T)_v}(2)$ gives an isomorphism between $I_v(2)$ and its image. This image is a free pro-$2$ group of rank $1$ (isomorphic to $\mathbb{Z}_2$), sitting as a normal subgroup of $G_{\mathbb{F}_q(T)_v}(2)$, which by abuse of notation we still denote by $I_v(2)$. Furthermore, this yields a split exact sequence given by
$$
1 \to I_v(2) \to G_{\mathbb{F}_q(T)_v}(2) \to \Gal(\mathbb{F}_q(T)_v \cdot i_v(\mathbb{F}_{q^{2^{\infty}}}(T))/\mathbb{F}_q(T)_v) \to 1,
$$
where the last non-trivial term is also a free pro-$2$ group of rank $1$, freely generated by the element of that Galois group which acts as $q_v$-powering on $i_v(\mathbb{F}_{q^{2^{\infty}}})$. The action of any lift of such an element on $I_v(2)$ is multiplication by $q_v$.  
\end{proposition}

\begin{proof}
This is a well-known result of Iwasawa \cite[Theorem 7.5.3]{NSW}.
\end{proof}

Proposition \ref{Structure of tame inertia} says that the maximal pro-$2$ quotient of the local Galois group is the quotient of the free pro-$2$ group on two symbols $\sigma$ and $F$ by the relation 
$$
F\sigma F^{-1} = \sigma^{q_v},
$$
where $\sigma$ maps to a free generator for the image of inertia, while $F$ maps to a lift of the local Frobenius operator. We aim to fix, once and for all, global choices, sitting inside $G_{\mathbb{F}_q(T)}(2)$, of such images of $\sigma$ and $F$: these will play an important role in constructing normalized expansion maps, and also in evaluating the Artin pairing. Before doing so, we clarify how $G_{\mathbb{F}_q(T)_v}(2)$ maps into $G_{\mathbb{F}_q(T)}(2)$ in the following proposition. 

\begin{proposition} 
\label{pGv}
The natural homomorphism from $G_{{\mathbb{F}_q(T)}_v}(2)$ to $\textup{proj} \circ i_v^\ast(G_{{\mathbb{F}_q(T)}_v})$ is an isomorphism of profinite groups.     
\end{proposition}

\begin{proof}
Surjectivity is by definition, so it suffices to prove injectivity. To this end, we claim that the field extension $M/\mathbb{F}_q(T)_v$ corresponding to $\text{ker}(G_{{\mathbb{F}_q(T)}_v} \to G_{{\mathbb{F}_q(T)}_v}(2))$ equals $i_v(L) \cdot \mathbb{F}_q(T)_v$ for some Galois pro-$2$ extension $L$ of $\mathbb{F}_q(T)$. Observe that $M$ can be explicitly described as the extension coming from adding all of the splitting fields of the Eisenstein polynomials $x^{2^n} - \pi$, as $n$ runs through the positive integers, and for any fixed uniformizer $\pi$. In particular, if $\pi$ is in $\mathbb{F}_q(T)$, then this extension comes entirely from a pro-$2$ extension of $\mathbb{F}_q(T)$. This establishes the claim. 

Having shown the claim, let us derive injectivity. Suppose we have an element $g$ of $G_{{\mathbb{F}_q(T)}_v}(2)$ that is in the kernel. Then $g$ must act in particular trivially on $i_v(L)$ by definition. But then it acts certainly trivially on $i_v(L) \cdot \mathbb{F}_q(T)_v = M$. By definition of $M$, this means that $g$ is the trivial element of the group $G_{{\mathbb{F}_q(T)}_v}(2)$, giving the desired injectivity. 
\end{proof}

We fix once and for all a topological generator $\sigma_v$ of $\text{proj} \circ i_v^\ast(I_v)$. This is possible since $\text{proj} \circ i_v^\ast$ induces an isomorphism between $I_v(2)$ and $(\text{proj} \circ i_v^\ast)(I_v)$, in view of Proposition \ref{pGv}, and hence the latter is a free pro-$2$ group on one generator, in view of Proposition \ref{Structure of tame inertia}. 

We fix once and for all an element $\Frob_v$ of $\text{proj} \circ i_v^\ast(G_{{\mathbb{F}_q(T)}_v})$ with the property that $\chi_{\Frob}(\Frob_v) = \deg(v)$. By Proposition \ref{Structure of tame inertia}, this element generates 
$$
\frac{\text{proj}\circ i_v^\ast(G_{{\mathbb{F}_q(T)}_v})}{\text{proj} \circ i_v^\ast(I_v)}.
$$
With these notations set, the unique homomorphism from the free pro-$2$ group on two generators $\sigma$ and $F$, given by
$$
\sigma \mapsto \sigma_v, \quad F \mapsto \Frob_v,
$$
yields a surjective homomorphism to the group $\text{proj} \circ i_v^\ast(G_{{\mathbb{F}_q(T)}_v})$. The kernel of this homomorphism is precisely the smallest closed normal subgroup containing $F \sigma F^{-1} \sigma^{-q_v}$.

Finally, the following piece of notation will be useful in computing Artin symbols on top of fields changing with quadratic characters $\chi$ in $\Gamma_{\mathbb{F}_2}(\mathbb{F}_q(T))$. For such a character $\chi$ and for a place $v$ in $\Omega_{\mathbb{F}_q(T)}$ lying inside $\text{Ram}(\chi)$, we define
$$
\Frob_v(\chi),
$$ 
to be the unique element $F_1$ in $\{\Frob_v, \Frob_v\sigma_v\}$ satisfying the equation
$
\chi(F_1) = 0.
$
Observe that precisely one of these two elements satisfies this equation due to $\chi(\sigma_v) = 1$. 

We observe that the normalized choice of lift $\Frob_q$ made in Section \ref{ssNot} can be done by picking $\Frob_q := \Frob_{\infty}(\chi_T)$. We fix this choice for the remainder of the paper.

We conclude this subsection by collecting the above choices into a set of topological generators for $G_{\mathbb{F}_q(T)}(2)$. We put
$$
\mathfrak{G} := \{\sigma_v : v \in \Omega_{\mathbb{F}_q(T)}^{\text{fin}}\} \cup \{\Frob_q\}.
$$
For a pro-$2$ group $\mathcal{G}$, we say that a basis $\mathcal{B}$ of $H^1(\mathcal{G},\mathbb{F}_2)$ and a subset $X$ of $\mathcal{G}$ are \emph{dual to each other}, if there is a bijection $f: \mathcal{B} \rightarrow X$ such that $\chi(x) = 1$ if and only if $x = f(\chi)$. As we will recall in the proof of the following proposition, this implies that $X$ is a minimal set of topological generators for $\mathcal{G}$. 

\begin{proposition} 
\label{Topological generators}
The set $\{\chi_v : v \in \Omega_{\mathbb{F}_q(T)}^{\textup{fin}}\} \cup \{\chi_{\epsilon_{\mathbb{F}_q}}\}$ is a basis for $\Gamma_{\mathbb{F}_2}(\mathbb{F}_q(T))$ and the set $\mathfrak{G}$ is a minimal set of topological generators for $G_{\mathbb{F}_q(T)}(2)$. Furthermore, these sets are dual to each other.
\end{proposition}

\begin{proof}
It is clear that $\{\chi_v : v \in \Omega_{\mathbb{F}_q(T)}^{\textup{fin}}\} \cup \{\chi_{\epsilon_{\mathbb{F}_q}}\}$ is a basis for $\Gamma_{\mathbb{F}_2}(\mathbb{F}_q(T))$ and it is also clear that the sets $\{\chi_v : v \in \Omega_{\mathbb{F}_q(T)}^{\textup{fin}}\} \cup \{\chi_{\epsilon_{\mathbb{F}_q}}\}$ and $\mathfrak{G}$ are dual to each other.

In order to show that $\mathfrak{G}$ is a minimal set of topological generators for $G_{\mathbb{F}_q(T)}(2)$, we recall that a subset $S$ generates a finite group $G$ if and only if the image of $S$ generates modulo the Frattini $G/\Phi(G)$ (and that $\Phi(G) = G^2[G, G] = G^2$ for $2$-groups $G$). This implies that $\mathfrak{G}$ is a minimal set of topological generators for $G_{\mathbb{F}_q(T)}(2)$ if and only if its image is a minimal set of topological generators for $\Gal(M/\mathbb{F}_q(T))$, where $M$ is the maximal multiquadratic extension of $\mathbb{F}_q(T)$. But this last condition is clearly implied by $\mathfrak{G}$ being dual to $\{\chi_v : v \in \Omega_{\mathbb{F}_q(T)}^{\textup{fin}}\} \cup \{\chi_{\epsilon_{\mathbb{F}_q}}\}$, which we have already verified to be true at the start of the proof.
\end{proof}

\subsection{The twisting module}
For the remainder of this section, we assume that $q \equiv 1 \bmod 8$. The assignments $\Frob_q \mapsto q_0$ and $\Frob_q \mapsto -q_0$ extend uniquely to continuous group homomorphisms
$$
G_{\mathbb{F}_q} \to \Z_2^{\ast},
$$
turning $\Z_2$ into a $G_{\mathbb{F}_q}$-module in two different ways. Recall that we identified $G_{\mathbb{F}_q}$ with $\Gal(\mathbb{F}_q^{\text{sep}}(T)/\mathbb{F}_q(T))$. Hence, by inflation, we have endowed $\Z_2$ with two different $G_{\mathbb{F}_q(T)}$-module structures. We denote them respectively by
$$
T_1, \quad T_{-1}.
$$
For $i \in \{1, -1\}$, we consider also the $G_{\mathbb{F}_q(T)}$-modules
$$
V_i := T_i \otimes_{\Z_2} \Q_2, \quad N_i := \frac{V_i}{T_i}.
$$
Denote the respective quadratic twists by $T_i(\chi)$, $V_i(\chi)$ and $N_i(\chi)$.

\begin{remark} 
\label{remark: -1 module is twisting by epsilon}
We have the equalities of $G_{\mathbb{F}_q(T)}$-modules 
$$
T_{-1} = T_1(\chi_{\epsilon_{\mathbb{F}_q}}), \quad V_{-1} = V_1(\chi_{\epsilon_{\mathbb{F}_q}}), \quad N_{-1} = N_1(\chi_{\epsilon_{\mathbb{F}_q}}).
$$
More generally, for $\chi$ in $\Gamma_{\mathbb{F}_2}(\mathbb{F}_q(T))$, we have the equalities
$$
T_{-1}(\chi) = T_1(\chi + \chi_{\epsilon_{\mathbb{F}_q}}), \quad V_{-1}(\chi) = V_1(\chi + \chi_{\epsilon_{\mathbb{F}_q}}), \quad N_{-1}(\chi) = N_1(\chi + \chi_{\epsilon_{\mathbb{F}_q}}).
$$
\end{remark}

Observe that the action of $G_{\mathbb{F}_q(T)}$ on $T_i(\chi), V_i(\chi), N_i(\chi)$ factors through $G_{\mathbb{F}_q(T)}(2)$ for each $\chi \in \Gamma_{\mathbb{F}_2}(\mathbb{F}_q(T))$ and $i \in \{1, -1\}$. For $\chi \in \Gamma_{\mathbb{F}_2}(\mathbb{F}_q(T))$ and $i \in \{1, -1\}$, we denote by
$$
B_i(\chi) := \text{Cocy}_{\text{unr.}}(G_{\mathbb{F}_q(T)}, N_i(\chi)),
$$
where the right hand side consists of elements $\psi \in \text{Cocy}(G_{\mathbb{F}_q(T)},N_i(\chi))$ satisfying
\begin{equation}
\label{eBidef}
\psi(\sigma_v) = 0 \quad \quad \text{ for all } v \in \Omega_{\mathbb{F}_q(T)} \setminus \text{Ram}(\chi).
\end{equation}
Observe that this evaluation on an element of $G_{\mathbb{F}_q(T)}(2)$ makes sense, as 
$$
\text{Cocy}(G_{\mathbb{F}_q(T)},N_i(\chi))=\text{Cocy}(G_{\mathbb{F}_q(T)}(2),N_i(\chi)),
$$
in virtue of the fact that $N_i(\chi)$ is (inflated from) a discrete $2$-primary torsion $G_{\mathbb{F}_q(T)}(2)$-module. 

Identifying the groups $\mathbb{F}_2$ and $\{\pm 1\} \subseteq \operatorname{GL}_1(\mathbb C)$ we can view each element $\chi$ of $\Gamma_{\mathbb{F}_2}(\mathbb{F}_q(T))$ as a Galois representation, and thus associate to it an Artin $L$-function
\[
L(s, \chi) = \prod_{v \in \Omega_{\mathbb F_q(T)} \setminus \operatorname{Ram}(\chi)} \ \frac{1}{1 - \chi(\operatorname{Frob}_v)q_v^{-s}}. 
\]

We now explain the relevance of the modules $B_i(\chi)$ in the context of Chowla's non-vanishing conjecture. The first step is given by the following proposition. 

\begin{proposition} 
\label{prop: field-theory of non-vanishing}
Let $q \equiv 1 \bmod 8$. Let $\chi$ be an element of $\Gamma_{\mathbb{F}_2}(\mathbb{F}_q(T))$ whose restriction to $G_{\mathbb{F}_q^{\textup{sep}}(T)}$ is non-trivial. Then the Artin $L$-function of $\chi$ vanishes at the central point, namely
$$
L\left(\tfrac{1}{2},\chi \right) = 0,
$$
only if there exists an everywhere unramified $\Z_2$-extension $L/\mathbb{F}_q^{\textup{sep}}(T)(\chi)$ that is Galois over $\mathbb{F}_q(T)$ such that $\Frob_q$ acts on $\Gal(L/\mathbb{F}_q^{\textup{sep}}(T)(\chi))$ by multiplication by $q_0$ or by $-q_0$. 
\end{proposition}

\begin{proof}
We fix an embedding of the ring $\mathbb Z_2$ into $\mathbb C$ to be used (silently) throughout the proof.

The Artin $L$-function of $\chi$, viewed as a polynomial in the variable $u = q^{-s}$, is the numerator of the zeta function of the smooth projective geometrically connected (hyperelliptic) curve $C$ over $\mathbb F_q$ whose function field is $\mathbb F_q(T)(\chi)$. This means, according to the Grothendieck--Lefschetz trace formula for the prime $2 \nmid q$, that the $L$-function of $\chi$ equals the determinant of the action of $1 - u \operatorname{Frob}_q^{-1}$ on the free $\mathbb Z_2$-module $H^1_{\textup{\'et}}(C \times_{\mathbb F_q} \mathbb{F}_q^{\textup{sep}}, \mathbb Z_2)$, where $\operatorname{Frob}_q$ is the arithmetic Frobenius - the map induced on \'etale cohomology by the morphism of schemes 
\[
\mathrm{id}_C \times (x \mapsto x^q)\colon C \times_{\mathbb F_q} \mathbb{F}_q^{\textup{sep}} \to C \times_{\mathbb F_q} \mathbb{F}_q^{\textup{sep}}.
\]
Vanishing of the $L$-function of $\chi$ at $s = 1/2$ thus implies that (at least) one of $u = q_0^{-1}$ and $u= -q_0^{-1}$ is a zero of the aforementioned determinant of $1 - u \operatorname{Frob}_q^{-1}$ on $H^1_{\textup{\'et}}(C \times_{\mathbb F_q} \mathbb{F}_q^{\textup{sep}}, \mathbb Z_2)$.
We can therefore find $v \in H^1_{\textup{\'et}}(C \times_{\mathbb F_q} \mathbb{F}_q^{\textup{sep}}, \mathbb Z_2)$ with $\operatorname{Frob}_q v = \pm q_0^{-1} v$.

We have a natural $\operatorname{Frob_q}$-equivariant isomorphism of $\mathbb Z_2$-modules
\begin{equation} 
\label{FreeModIsomCanFrobZ2}
\begin{split}
H^1_{\textup{\'et}}(C \times_{\mathbb F_q} \mathbb{F}_q^{\textup{sep}}, \mathbb Z_2) &\cong \Hom_{\text{top.gr.}}(\pi_1^{\textup{\'et}}(C \times_{\mathbb F_q} \mathbb{F}_q^{\textup{sep}})^{\text{ab}}, \mathbb Z_2) \\
&= \Hom_{\text{top.gr.}}(\pi_1^{\textup{\'et}}(C \times_{\mathbb F_q} \mathbb{F}_q^{\textup{sep}})^{\text{ab}}(2), \mathbb Z_2).
\end{split}
\end{equation}
The short exact sequence of profinite groups
\[
1 \to \pi_1^{\textup{\'et}}(C \times_{\mathbb F_q} \mathbb{F}_q^{\textup{sep}}) \to \pi_1^{\textup{\'et}}(C) \to \pi_1^{\textup{\'et}}(\operatorname{Spec} \mathbb F_q) \to 1
\]
induces a $\mathbb Z_2$-linear action by conjugation of $\pi_1^{\textup{\'et}}(\operatorname{Spec} \mathbb F_q) = \operatorname{Gal}(\mathbb{F}_q^{\textup{sep}}/\mathbb F_q)$ on the (free) $\mathbb Z_2$-module 
$
\pi_1^{\textup{\'et}}(C \times_{\mathbb F_q} \mathbb{F}_q^{\textup{sep}})^{\textup{ab}}(2)
$
and thus also on its dual over $\mathbb Z_2$ appearing as the last term in equation \eqref{FreeModIsomCanFrobZ2}. In particular, the element $(x \mapsto x^q) \in \pi_1^{\textup{\'et}}(\operatorname{Spec} \mathbb F_q)$ acts $\mathbb Z_2$-linearly on the latter $\mathbb Z_2$-module, and its action on 
$
\pi_1^{\textup{\'et}}(C \times_{\mathbb F_q} \mathbb{F}_q^{\textup{sep}})^{\textup{ab}}(2)
$
is dual to the action of $\operatorname{Frob}_q$ on $H^1_{\textup{\'et}}(C \times_{\mathbb F_q} \mathbb{F}_q^{\textup{sep}}, \mathbb Z_2)$.

This duality tells us that the existence of $v \in H^1_{\textup{\'et}}(C \times_{\mathbb F_q} \mathbb{F}_q^{\textup{sep}}, \mathbb Z_2)$ with $\operatorname{Frob}_q v = \pm q_0^{-1} v$ as above is equivalent to the existence of a quotient of 
$
\pi_1^{\textup{\'et}}(C \times_{\mathbb F_q} \mathbb{F}_q^{\textup{sep}})^{\textup{ab}}(2)$
isomorphic to $\mathbb Z_2$ on which the element $(x \mapsto x^q) \in \pi_1^{\textup{\'et}}(\operatorname{Spec} \mathbb F_q)$ acts via multiplication by $\pm q_0$. Such a quotient corresponds to an everywhere unramified $\mathbb Z_2$-extension $L$ of $\mathbb{F}_q^{\textup{sep}}(T)(\chi)$ that is Galois over $\mathbb F_q(T)(\chi)$ with $\Frob_q$ acting on $\Gal(L/\mathbb{F}_q^{\textup{sep}}(T)(\chi))$ via multiplication by $\pm q_0$. The field $L$ is Galois even over $\mathbb F_q(T)$ because the Galois group of $\mathbb F_q(T)(\chi)$ over $\mathbb F_q(T)$ acts by negation on $\pi_1^{\textup{\'et}}(C \times_{\mathbb F_q} \mathbb{F}_q^{\textup{sep}})^{\textup{ab}}$. 
\end{proof}

\subsection{Artin pairings} 
\label{Section: Artin pairings}
We begin with the following fact, which will be crucial for the definition of Artin pairings. 

\begin{proposition} 
\label{Prop: Inertia and frob commute} 
Let $\chi\in\Gamma_{\mathbb F_2}(\mathbb F_q(T))$, let $v \in \textup{Ram}(\chi)$ and let $i \in \{1, -1\}$. Let $\iota: \mathbb{F}_q(T)^{\textup{sep}} \xhookrightarrow{} \mathbb{F}_q(T)_v^{\textup{sep}}$ be an embedding and let $\iota^\ast: G_{\mathbb{F}_q(T)_v} \rightarrow G_{\mathbb{F}_q(T)}$ be the resulting injection on Galois groups. Let $F$ be any lift of Frobenius to $G_{\mathbb{F}_q(T)_v}$ and let $\sigma$ be any lift of a topological generator of tame inertia to $G_{\mathbb{F}_q(T)_v}$. Then we have for each cocycle $\psi \in \mathrm{Cocy}(G_{\mathbb{F}_q(T)_v}, N_i(\chi))$
$$
\psi(F \sigma) = \psi(\sigma F).
$$
\end{proposition}

\begin{proof}
The condition in the proposition is equivalent to the homomorphism
$$
f := (\psi, \chi, \chi_\Frob)\colon G_{\mathbb{F}_q(T)_v} \rightarrow N_i(\chi) \rtimes (\FF_2 \times \Z_2)
$$
satisfying the condition $f(F \sigma) = f(\sigma F)$. Since $v \in \textup{Ram}(\chi)$, a direct computation shows that $f(\sigma)$ is always an element of order $2$. By Proposition \ref{Structure of tame inertia}, we have 
$$
f(F \sigma F^{-1} \sigma^{-q_v}) = 0. 
$$
Upon combining this with $f(\sigma^2) = 0$ and using that $f$ is a homomorphism, the proposition follows.
\end{proof}

Expanding the identity of Proposition \ref{Prop: Inertia and frob commute} with the definition of $1$-cocycle, we obtain
$$
((i \cdot q_0)^{\deg(v)} - 1) \cdot \psi(\sigma_v) = -2 \cdot \psi(\Frob_v(\chi)).
$$
Hence we have
$$
\Pi_v(\psi) := \frac{(i \cdot q_0)^{\deg(v)} - 1}{2} \cdot \psi(\sigma_v) + \psi(\Frob_v(\chi)) \in N_i(\chi)[2] = \mathbb{F}_2.
$$
This defines a map
$$
\Pi := (\Pi_v)_{v \in \text{Ram}(\chi)}\colon B_i(\chi) \to \mathbb{F}_2^{\text{Ram}(\chi)}.
$$
To keep the notation light, we have omitted the dependence on $i \in \{1, -1\}$ and $\chi$ in the definitions of $\Pi$ and $\Pi_v$.

Observe that given a cocycle $\psi$ valued in $N_i(\chi)[2^n]$, the obstruction to the existence of a cocycle $\tilde{\psi}$ valued in $N_i(\chi)[2^{n + 1}]$ with $2 \cdot \tilde{\psi} = \psi$ is a class $\theta(\psi) \in H^2(G_{\mathbb{F}_q(T)}, \mathbb{F}_2)$, arising from the embedding problem given by the central extension below
$$
0 \to \mathbb{F}_2 \to N_i(\chi)[2^{n + 1}] \rtimes (\mathbb{F}_2 \times \mathbb{Z}_2) \to N_i(\chi)[2^n] \rtimes (\mathbb{F}_2 \times \mathbb{Z}_2) \to 0. 
$$

\begin{proposition} 
\label{prop: detecting locally}
Let $\chi$ be in $\Gamma_{\mathbb{F}_2}(\mathbb{F}_q(T))$. Let $n$ be a nonnegative integer, let $i \in \{1,-1\}$ and let $v \in \textup{Ram}(\chi)$. Let $\psi$ be a cocycle valued in $N_i(\chi)[2^n]$. Let $\iota: \mathbb{F}_q(T)^{\textup{sep}} \xhookrightarrow{} \mathbb{F}_q(T)_v^{\textup{sep}}$ be an embedding and let $\iota^\ast: G_{\mathbb{F}_q(T)_v} \rightarrow G_{\mathbb{F}_q(T)}$ be the resulting injection on Galois groups. Let $F$ be any lift of Frobenius to $G_{\mathbb{F}_q(T)_v}$ satisfying $\chi(F) = 0$ and let $\sigma$ be any lift of a topological generator of tame inertia to $G_{\mathbb{F}_q(T)_v}$. Then we have
$$
\frac{(i \cdot q_0)^{\deg(v)} - 1}{2} \cdot \psi(\iota^\ast(\sigma)) + \psi(\iota^\ast(F)) = \inv_v(\theta(\psi)).
$$
In particular, we have $\Pi_v(\psi) = \inv_v(\theta(\psi))$.
\end{proposition}

\begin{proof}
We fix lifts $g_1$ and $g_2$ of respectively $\psi(\iota^\ast(\sigma))$ and $\psi(\iota^\ast(F))$ to $N_i(\chi)[2^{n + 1}] \rtimes (\mathbb{F}_2 \times \Z_2/2^N)$, where $N$ is a sufficiently large integer such that the action of $G_{\mathbb{F}_q(T)_v}$ factors through $\mathbb{F}_2 \times \Z_2/2^N$. In view of Proposition \ref{Prop: Inertia and frob commute}, the images of $g_1$ and $g_2$ in $N_i(\chi)[2^n] \rtimes (\mathbb{F}_2 \times \Z_2/2^N)$ must commute. Thus $g_1g_2$ and $g_2g_1$ differ by an element of the central subgroup $2^n \cdot N_i(\chi)[2^{n + 1}] \rtimes \{0\}$ of order $2$. Let $z$ be the unique non-trivial element of this group.

\begin{claim}
We have that $g_1g_2 = g_2g_1 \cdot z^{\frac{(i \cdot q_0)^{\deg(v)} - 1}{2} \cdot \psi(\iota^\ast(\sigma)) + \psi(\iota^\ast(F))}$.    
\end{claim}

\begin{proof}[Proof of Claim] 
Write $g_1 := (\gamma_1, 1, 0), g_2 := (\gamma_2, 0, \deg(v)) \in N_i(\chi)[2^{n + 1}] \rtimes (\mathbb{F}_2 \times \Z_2/2^N)$ in the semi-direct product coordinates. We have that
$$
g_1g_2 = (\gamma_1 - \gamma_2, 1, \deg(v)), \quad g_2g_1 = ((i \cdot q_0)^{\deg(v)} \cdot \gamma_1 + \gamma_2, 1, \deg(v)).
$$
Therefore $z$ appears with trivial power if and only if
$$
((i \cdot q_0)^{\deg(v)} - 1) \cdot \gamma_1 = -2 \cdot \gamma_2.
$$
Note that $2 \cdot \gamma_1$ viewed as element of $N_i(\chi)[2^n]$, through the doubling map $N_i(\chi)[2^{n + 1}] \to N_i(\chi)[2^n]$, equals precisely $\psi(\iota^\ast(\sigma))$, while $2 \cdot \gamma_2 = \psi(\iota^\ast(F))$. Hence the above equality holds if and only if
$$
0 = ((i \cdot q_0)^{\deg(v)} - 1) \cdot \gamma_1 + \psi(\iota^\ast(F)).
$$
All in all, we deduce that $\frac{(i \cdot q_0)^{\deg(v)} - 1}{2} \cdot \psi(\iota^\ast(\sigma)) + \psi(\iota^\ast(F)) = 0$ if and only if $g_1g_2 = g_2g_1$ and $\frac{(i \cdot q_0)^{\deg(v)} - 1}{2} \cdot \psi(\iota^\ast(\sigma)) + \psi(\iota^\ast(F)) = 1$ if and only if $g_1g_2 = g_2g_1z$, which is the desired conclusion. 
\end{proof}

We now complete our argument. If $\frac{(i \cdot q_0)^{\deg(v)} - 1}{2} \cdot \psi(\iota^\ast(\sigma)) + \psi(\iota^\ast(F)) = 0$, then $\sigma \mapsto g_1, F \mapsto g_2$ yields the desired group-theoretic local lift. Conversely, suppose that a group-theoretic lift $\sigma \mapsto \tilde{g}_1, F \mapsto \tilde{g}_2$ exists. Then we can write $\tilde{g}_1 = g_1z_1, \tilde{g}_2 = g_2z_2$, for elements $z_1$ and $z_2$ in the center. It follows that 
$$
\tilde{g}_1\tilde{g}_2 = \tilde{g}_2\tilde{g}_1 \cdot z^{\frac{(i \cdot q_0)^{\deg(v)} - 1}{2} \cdot \psi(\iota^\ast(\sigma)) + \psi(\iota^\ast(F))}. 
$$
This shows that $\frac{(i \cdot q_0)^{\deg(v)} - 1}{2} \cdot \psi(\iota^\ast(\sigma)) + \psi(\iota^\ast(F)) = 0$. This proves the first part of the proposition, while the second follows trivially from the first. 
\end{proof}

\begin{proposition} 
\label{prop: Pi detects 2}
For each $\chi \in \Gamma_{\mathbb{F}_2}^{\textup{im}}(\mathbb{F}_q(T))$ and each $i \in \{1, -1\}$, we have that
$$
\ker(\Pi) = 2 \cdot B_i(\chi). 
$$
\end{proposition}

\begin{proof}
If $\psi\in2\cdot B_i(\chi)$, then the obstruction class $\theta(\psi)$ vanishes. Hence Proposition \ref{prop: detecting locally} gives $\Pi_v(\psi) = 0$ for every $v \in \operatorname{Ram}(\chi)$, and therefore $\Pi(\psi) = 0$.

Conversely, let $\psi$ be an element of $B_i(\chi)$ with $\Pi(\psi) = 0$. By Proposition \ref{prop: detecting locally}, this means that $\text{inv}_v(\theta(\psi)) = 0$ for each $v \in \text{Ram}(\chi)$. Furthermore, observe that, in virtue of the fact that $\psi(\sigma_v) = 0$ for all $v \not \in \text{Ram}(\chi)$, the class $\theta(\psi)$ comes by inflation from $H^2(G_{\mathbb{F}_{q_v}}, \mathbb{F}_2) = 0$ for each $v \not \in \text{Ram}(\chi)$. Therefore we deduce that $\inv_v(\theta(\psi)) = 0$ for all $v \in \Omega_{\mathbb{F}_q(T)}$ and thus, by local-to-global for $H^2(G_{\mathbb{F}_q(T)},\mathbb{F}_2)$, we deduce that $\theta(\psi) = 0$ in $H^2(G_{\mathbb{F}_q(T)}, \mathbb{F}_2)$. Tracing back to the embedding problem, this means that we can find a $1$-cocycle $\tilde{\psi}\colon G_{\mathbb{F}_q(T)} \to N_i(\chi)$ such that $2 \cdot \tilde{\psi} = \psi$. The field of definition $L(\tilde{\psi})/\mathbb{F}_q(T)$ ramifies at only finitely many primes. Hence the cocycle
$$
\tilde{\psi} + \sum_{v \in \Omega_{\mathbb{F}_q(T)}^{\text{fin}}} \mathbf{1}_{\tilde{\psi}(\sigma_v) \neq 0} \cdot \chi_v
$$
still doubles to $\psi$ and vanishes at all $\sigma_v$ with $v \in \Omega_{\mathbb{F}_q(T)} - \text{Ram}(\chi)$, thus lives inside $B_i(\chi)$.
\end{proof}

We define for each $i \in \{1, -1\}$, each nonnegative integer $n$ and each $\chi$ in $\Gamma_{\mathbb{F}_2}(\mathbb{F}_q(T))$, the $\mathbb{F}_2$-vector space 
$$
C_n^{(i)}(\chi) : =2^n \cdot B_i(\chi)[2^{n + 1}].
$$
We equip $\mathbb{F}_2^{\text{Ram}(\chi)}$ with the standard bilinear form given by the dot product, which we denote by $\langle v_1, v_2 \rangle$. 

We then define for each $i \in \{1, -1\}$ and for each nonnegative integer $n$
$$
D_n^{(i)}(\chi) := (\Pi(B_i(\chi)[2^n]))^{\perp}.
$$
Observe that, by construction, both $(C_n^{(i)}(\chi))_{n \geq 0}$ and $(D_n^{(i)}(\chi))_{n \geq 0}$ are non-ascending sequences of $\mathbb{F}_2$-vector spaces. We now define a fundamental piece of structure, namely a pairing that is able to detect the passage from each of these spaces to the subsequent one. 

\begin{proposition} 
\label{prop: elementary def of pairing}
Let $\chi$ be in $\Gamma_{\mathbb{F}_2}^{\textup{im}}(\mathbb{F}_q(T))$ and let $n$ be a nonnegative integer. Let $\chi_1 \in C_n^{(i)}(\chi)$ and $b \in D_n^{(i)}(\chi)$. The value of 
$
\langle \Pi(\psi), b \rangle
$
remains constant as $\psi$ varies among the elements of $B_i(\chi)$ such that $2^n \cdot \psi = \chi_1$. This defines a bilinear pairing
$$
\textup{Art}_{\chi, n + 1}^{(i)}\colon C_n^{(i)}(\chi) \times D_n^{(i)}(\chi) \to \mathbb{F}_2
$$
such that 
$$
\textup{Left-Ker}(\textup{Art}_{\chi, n + 1}^{(i)}) = C_{n + 1}^{(i)}(\chi)
$$
and
$$
\textup{Right-Ker}(\textup{Art}_{\chi, n + 1}^{(i)}) = D_{n + 1}^{(i)}(\chi).
$$
\end{proposition}

\begin{proof}
Observe that by definition $b$ is orthogonal to the image of $B_i(\chi)[2^n]$ under $\Pi$. Moreover, two lifts of $\chi_1$ must necessarily differ by an element of $B_i(\chi)[2^n]$. Therefore, invoking the bilinearity of $\langle -, - \rangle$, we deduce that $\langle \Pi(\psi), b \rangle$ is independent of the choice of $\psi \in B_i(\chi)[2^{n + 1}]$ satisfying $2^n \cdot \psi = \chi_1$. If we fix $\chi_1$ and an element $\psi \in B_i(\chi)[2^{n + 1}]$ satisfying $2^n \cdot \psi =\chi_1$, linearity in the second entry follows from bilinearity of the pairing $\langle -, - \rangle$, i.e.
$$
\langle \Pi(\psi), b_1 + b_2 \rangle = \langle \Pi(\psi), b_1 \rangle + \langle \Pi(\psi), b_2 \rangle.
$$
To prove that $\text{Art}_{\chi, n + 1}^{(i)}$ is also linear in the first entry, let us fix $\chi_1, \chi_2 \in C_n^{(i)}(\chi)$, $b \in D_n^{(i)}(\chi)$. We also fix lifts $\psi_1, \psi_2 \in B_i(\chi)[2^{n + 1}]$ satisfying $2^n \cdot \psi_1 = \chi_1$ and $2^n \cdot \psi_2 = \chi_2$. Observe that this implies that $2^n \cdot (\psi_1 + \psi_2) = \chi_1 + \chi_2$. Hence, invoking linearity of $\Pi$ and bilinearity of $\langle -, - \rangle$, we can compute 
\begin{align*}
\text{Art}_{\chi, n + 1}^{(i)}(\chi_1 + \chi_2,b) &= \langle \Pi(\psi_1 + \psi_2), b \rangle = \langle \Pi(\psi_1), b \rangle + \langle \Pi(\psi_2), b \rangle \\
&= \text{Art}_{\chi, n + 1}^{(i)}(\chi_1, b) + \text{Art}_{\chi, n + 1}^{(i)}(\chi_2, b).
\end{align*}
So far we have established that $\text{Art}_{\chi, n + 1}^{(i)}$ is independent of the choice of $\psi$ and that it defines a bilinear pairing
$$
\text{Art}_{\chi, n + 1}^{(i)}\colon C_n^{(i)}(\chi) \times D_n^{(i)}(\chi) \to \mathbb{F}_2,
$$
as claimed. We now want to show that the left and the right kernels are as claimed in the proposition. 

Let us begin with the right kernel. Observe that $b$ is in the right kernel if and only if $\langle \Pi(\psi), b \rangle = 0$ for each element $\psi \in B_i(\chi)[2^{n + 1}]$. This means that $b$ is in the right kernel if and only if it is in $(\Pi(B_i(\chi)[2^{n + 1}]))^{\perp} = D_{n + 1}^{(i)}(\chi)$, i.e.
$$
\text{Right-Ker}(\textup{Art}_{\chi, n + 1}^{(i)}) = D_{n + 1}^{(i)}(\chi)
$$
as claimed. 

Let us now establish the claim on the left kernel. Let $\chi_1$ be an element of $C_n^{(i)}(\chi)$. Fix $\psi$ in $B_i(\chi)[2^{n + 1}]$ such that $2^n \cdot \psi = \chi_1$. We have that $\chi_1$ is in $C_{n + 1}^{(i)}(\chi)$ if and only if there exists $\psi_0$ in $B_i(\chi)[2^{n}]$ such that
$$
\psi + \psi_0 \in 2 \cdot B_i(\chi) = \ker(\Pi),
$$
where the last equality is justified by Proposition \ref{prop: Pi detects 2}. So $\chi_1 \in C_{n + 1}^{(i)}(\chi)$ if and only if $\Pi(\psi) \in \Pi(B_i(\chi)[2^n])$. However, since $\langle -, - \rangle$ is a perfect pairing $\mathbb{F}_2^{\text{Ram}(\chi)} \times \mathbb{F}_2^{\text{Ram}(\chi)} \to \mathbb{F}_2$, $\Pi(\psi)$ is in the vector space $\Pi(B_i(\chi)[2^{n}])$ if and only if $\langle \Pi(\psi), - \rangle$ annihilates $\Pi(B_i(\chi)[2^{n}])^{\perp}=D_n^{(i)}(\chi)$. This, by definition, means precisely that $\chi_1$ is orthogonal to $D_n^{(i)}(\chi)$ under $\text{Art}_{\chi, n + 1}^{(i)}$. In other words, we have shown that $\chi_1$ is in $C_{n + 1}^{(i)}(\chi)$ if and only if it is in the left kernel of $\text{Art}_{\chi, n + 1}^{(i)}$.
\end{proof}

We have the following relation between the spaces as $i$ switches between $1$ and $-1$. 

\begin{remark}
\label{rSwitching}
Using Remark \ref{remark: -1 module is twisting by epsilon}, one sees that 
$$
B_{-1}(\chi) = B_1(\chi + \chi_{\epsilon_{\mathbb{F}_q}})
$$
and that for each nonnegative integer $n$ we have
$$
C_n^{(-1)}(\chi) = C_n^{(1)}(\chi + \chi_{\epsilon_{\mathbb{F}_q}}), \ \ D_n^{(-1)}(\chi) = D_n^{(1)}(\chi + \chi_{\epsilon_{\mathbb{F}_q}}), \ \ \textup{Art}_{\chi, n + 1}^{(-1)} = \textup{Art}_{\chi + \chi_{\epsilon_{\mathbb{F}_q}}, n + 1}^{(1)}.
$$
\end{remark}

For a character $\chi$ in $\Gamma_{\mathbb{F}_2}^{\text{im}}(\mathbb{F}_q(T))$ and for a place $v$ in $\text{Ram}(\chi)$, we denote by $e_v$ the standard basis vector supported only at $v$ in $\mathbb{F}_2^{\text{Ram}(\chi)}$. We put
$$
r_{\infty}(\chi) := \sum_{v \in \text{Ram}(\chi)} e_v.
$$
We have systematic elements in, respectively, all of the $C_n,D_n$, as the following fact illustrates. 

Let us define for each nonnegative integer $n$ the element of $B_i(\chi)$ given by
$$
\rho_n(g) := g \cdot \frac{1}{2^{n + 1}} - \frac{1}{2^{n + 1}},
$$
for each $g$ in $G_{\mathbb{F}_q(T)}$. A small computation shows that 
\begin{equation}
\label{eRho1}
\rho_1 = \chi + \frac{1-i}{2} \cdot \chi_{\epsilon_{\mathbb{F}_q}}.
\end{equation}

\begin{proposition} 
\label{prop: trivial elemts}
$(a)$ Let $\chi$ be an element of $\Gamma_{\mathbb{F}_2}^{\textup{im}}(\mathbb{F}_q(T))$. Then we have that
$$
r_{\infty}(\chi) \in D_n^{(i)}(\chi),
$$
for each $i \in \{1, -1\}$ and all nonnegative integers $n$. \\
$(b)$ Let $\chi$ be an element of $\Gamma_{\mathbb{F}_2}^{\textup{im}}(\mathbb{F}_q(T))$. Then we have that
$$
\chi \in C_n^{(1)}(\chi), \quad \chi + \chi_{\epsilon_{\mathbb{F}_q}} \in C_n^{(-1)}(\chi),
$$
for all nonnegative integers $n$.
\end{proposition}

\begin{proof}
Part $(b)$ follows immediately from the construction of $\rho_n$ combined with the fact that
$$
2^n \cdot \rho_{n + 1} = \rho_1 = \chi + \frac{(1 - i)}{2} \cdot \chi_{\epsilon_{\mathbb{F}_q}}
$$
for all nonnegative integers $n$, which also follows by definition of $\rho_n$ and equation \eqref{eRho1}. So we now focus on part $(a)$. 

We prove this by induction on $n$, where the base case $n = 0$ follows by definition. Assume the conclusion is true for $n$, and we now prove it for $n + 1$. Let $\psi$ be any element of $B_i(\chi)[2^{n + 1}]$ and $v$ be an element of $\text{Ram}(\chi)$. By Proposition \ref{prop: detecting locally}, we have that $\inv_v(\theta(\psi)) = 0$ if and only if $\Pi_v(\psi) = 0$. Observe that at all other places $w$ outside of $\text{Ram}(\chi)$, we have that $\inv_w(\theta(\psi)) = 0$. Applying Hilbert reciprocity to $\theta(\psi)$ yields 
$$
\text{Art}_{\chi, n + 1}^{(i)}(\chi_0, r_{\infty}(\chi)) = \sum_{v \in \text{Ram}(\chi)} \Pi_v(\psi) = 0,
$$
where $\chi_0 = 2^{n} \cdot \psi$. 

This shows that $r_{\infty}(\chi)$ pairs trivially with all of $C_n^{(i)}(\chi)$, hence $r_{\infty}(\chi) \in D_{n + 1}^{(i)}(\chi)$.
\end{proof}

We make the following simple observation.

\begin{proposition}
\label{prop: genus spaces}
Let $\chi$ be an element of $\Gamma_{\mathbb{F}_2}^{\textup{im}}(\mathbb{F}_q(T))$. Then the set
$$
\{\chi_v : v \in \textup{Ram}(\chi) \cap \Omega_{\mathbb{F}_q(T)}^{\textup{fin}}\} \cup \{\chi_{\epsilon_{\mathbb{F}_q}}\}
$$
is a basis of $C_0^{(i)}(\chi)$ for each $i \in \{1, -1\}$, and the tuple $(e_v)_{v \in \textup{Ram}(\chi)}$ is a basis of $D_0^{(i)}(\chi)$ for each $i \in \{1, -1\}$.
\end{proposition}

\begin{proof}
Since $N_i(\chi)[2] \cong \mathbb F_2$ with trivial Galois action, we have
\[
C_0^{(i)}(\chi) = B_i(\chi)[2] = \{\eta \in \Gamma_{\mathbb F_2}(\mathbb{F}_q(T)): \text{Ram}(\eta) \subseteq \text{Ram}(\chi)\}.
\]
Because $\infty\in\text{Ram}(\chi)$, Proposition \ref{Topological generators} shows that this
space has basis
\[
\{\chi_v : v \in \textup{Ram}(\chi) \cap \Omega_{\mathbb{F}_q(T)}^{\text{fin}}\} \cup \{\chi_{\epsilon_{\mathbb F_q}}\}.
\]
Moreover, $B_i(\chi)[1] = 0$, and hence
\[
D_0^{(i)}(\chi) = \Pi(B_i(\chi)[1])^\perp = \mathbb F_2^{\text{Ram}(\chi)},
\]
whose standard basis is $(e_v)_{v \in \text{Ram}(\chi)}$.
\end{proof}

Combining Proposition \ref{prop: genus spaces} with Proposition \ref{prop: elementary def of pairing}, we deduce that 
$$
\dim_{\mathbb{F}_2} \, C_n^{(i)}(\chi) = \dim_{\mathbb{F}_2} \, D_n^{(i)}(\chi)
$$
for all integers $n \geq 1$ and $i \in \{1, -1\}$. Therefore we see that $C_n^{(1)}(\chi) = \langle \chi \rangle$ (resp.~$C_{n}^{(-1)}(\chi) = \langle \chi + \chi_{\epsilon_{\mathbb{F}_q}} \rangle$) happens if and only if $D_n^{(1)}(\chi) = \langle r_{\infty}(\chi) \rangle$ (resp.~$D_n^{(-1)}(\chi) = \langle r_{\infty}(\chi) \rangle$).

We will need the following basic fact, which one may, in particular, use to describe the spaces $B_i(\chi)[4]$. 

\begin{proposition} 
\label{prop: description of N[4] cocycles}
Let $\chi, \chi' \in \Gamma_{\mathbb{F}_2}(\mathbb{F}_q(T))$ and let $i \in \{1, -1\}$. Let $\psi\colon G_{\mathbb{F}_q(T)} \to N_i(\chi)[4]$ be a continuous function such that $2 \cdot \psi =\chi'$. Then $\psi \in \textup{Cocy}(G_{\mathbb{F}_q(T)}, N_i(\chi)[4])$ if and only if there exists a continuous $1$-cochain $\phi\colon G_{\mathbb{F}_q(T)} \to \mathbb{F}_2$ satisfying 
$$
\mathrm{d}\phi(\sigma, \tau) = \left(\chi + \chi' + \frac{1 - i}{2} \cdot \chi_{\epsilon_{\mathbb{F}_q}}\right)(\sigma) \cdot \chi'(\tau)
$$
for all $\sigma, \tau \in G_{\mathbb{F}_q(T)}$ and $\psi = \phi + \frac{\chi'}{4}$.
\end{proposition}

\begin{proof}
The assumption that $2 \cdot \psi = \chi'$ translates into the existence of a continuous $1$-cochain $\phi \colon G_{\mathbb{F}_q(T)} \to \mathbb{F}_2$ such that $\psi = \phi + \frac{\chi'}{4}$. We now verify that the cocycle condition on $\psi$ precisely translates into the desired constraint on $\phi$. Working out the action of $G_{\mathbb{F}_q(T)}$ on $N_i(\chi)[4]$ gives that the cocycle equation can be rephrased as
$$
\psi(\sigma \tau) = (1 - 2)^{\chi(\sigma) + \frac{1 - i}{2} \cdot \chi_{\epsilon_{\mathbb{F}_q}}(\sigma)} \cdot \psi(\tau) + \psi(\sigma)
$$
for all $\sigma, \tau \in G_{\mathbb{F}_q(T)}$. Substituting $\psi = \phi + \frac{\chi'}{4}$ on both sides of the equation and doing carry arithmetic gives the relation
$$
\phi(\sigma\tau) = \phi(\sigma) + \phi(\tau) + \left(\chi(\sigma) + \frac{1 - i}{2} \cdot \chi_{\epsilon_{\mathbb{F}_q}}(\sigma)\right) \cdot \chi'(\tau) + \chi'(\sigma) \cdot \chi'(\tau),
$$
where the last term is the one coming from carry arithmetic modulo $4$. This is precisely the desired conclusion. 
\end{proof}

\subsection{Reinterpretation of non-vanishing}
We now reinterpret Proposition \ref{prop: field-theory of non-vanishing}.

\begin{proposition}
\label{pNonCrit}
Let $\chi$ be an element of $\Gamma_{\mathbb{F}_2}^{\textup{im}}(\mathbb{F}_q(T))$. Then $L\left(\frac{1}{2},\chi\right) \neq 0$ if there exists a nonnegative integer $n$ such that for all $i \in \{1, -1\}$ we have
$$
C_n^{(i)}(\chi) = \left \langle \chi + \frac{(1 - i)}{2} \cdot \chi_{\epsilon_{\mathbb{F}_q}} \right \rangle.
$$
\end{proposition}

\begin{proof}
To prove the contrapositive suppose that $L(\frac{1}{2}, \chi) = 0$. The character $\chi$ is non-trivial when restricted to $\mathbb{F}_q^{\text{sep}}(T)$ since it ramifies at the infinite place, so Proposition \ref{prop: field-theory of non-vanishing} provides us with an everywhere unramified $\Z_2$-extension $L/\mathbb{F}_q^{\text{sep}}(T)(\chi)$ that is Galois over $\mathbb{F}_q(T)$ with $\Frob_q$ acting on $\Gal(L/\mathbb{F}_q^{\text{sep}}(T)(\chi))$ by multiplication by $i \cdot q_0$ for some $i$ in $\{1, -1\}$. 

We next claim that we can descend $L$ to an everywhere unramified $\mathbb{Z}_2$-extension of
\[
E_2 := \mathbb{F}_{q^{2^\infty}}(T)(\chi).
\]
More precisely, we claim that there exists an everywhere unramified $\mathbb{Z}_2$-extension $L'/E_2$, which is Galois over $\mathbb{F}_q(T)$, such that conjugation by the fixed element $\Frob_q$ on $\Gal(L'/E_2)$ is multiplication by $i q_0$, and such that
$$
L = L' \cdot \mathbb{F}_q^{\text{sep}}(T).
$$
We now prove the claim. Put
$$
E := \mathbb{F}_q^{\text{sep}}(T)(\chi), \qquad \mathcal{G} := \Gal(L/\mathbb F_q(T)), \qquad H := \Gal(L/E) \cong \mathbb Z_2.
$$
Since $\chi$ is geometrically non-trivial, we have
$$
\mathcal{G}/H \cong \Gal(E/\mathbb F_q(T)) \cong \mathbb F_2 \times \widehat{\mathbb Z}.
$$
A lift of $\Frob_q$ acts on $H$ by multiplication by $i q_0$, while a lift of the non-trivial element of the first factor acts by multiplication by $-1$.

We use the following finite-group lemma.

\begin{lemma*}[Finite group lemma] Let $G$ be a finite group and let $H$ be a normal abelian subgroup of $G$. Suppose that $G/H$
is abelian and generated by $g_1H$ and $g_2H$. For $j \in \{1, 2\}$, let $\varphi_j$ be the automorphism of $H$ induced by conjugation by $g_j$. If both $\varphi_1 - \textup{id}$ and $\varphi_2 - \textup{id}$ are nilpotent, then $G$ is nilpotent.
\end{lemma*}

\begin{proof}[Proof of finite group lemma] Put
$$
N_j:=\varphi_j-\mathrm{id}.
$$
Since $[g_1, g_2] \in H$ and $H$ is abelian, the automorphisms $\varphi_1$ and $\varphi_2$ commute. Let $I$ be the ideal generated by $N_1, N_2$ in the commutative subring of $\mathrm{End}(H)$ generated by $\varphi_1, \varphi_2$. Since $N_1$ and $N_2$ are nilpotent, we have
$$
I^m \cdot H = 0
$$
for all sufficiently large $m$.

Every $\sigma\in G$ can be written as $\sigma = g_1^ag_2^bh$, with $h \in H$, and therefore
$$
\left.\operatorname{Ad}(\sigma)\right|_H - \mathrm{id} \in I.
$$
Since $[G, G]\subseteq H$, it follows inductively that
$$
\gamma_{m + 2}(G) \subseteq I^m \cdot H.
$$
Thus the lower central series of $G$ terminates, proving the lemma.
\end{proof}

We apply the lemma to every finite continuous quotient of $\mathcal{G}$. The two relevant endomorphisms of the image of $H$ are induced by multiplication by
$$
i q_0 - 1 \qquad \text{and} \qquad -2.
$$
Both belong to $2\mathbb Z_2$, and hence are nilpotent on every finite quotient of $H$. Therefore every finite quotient of $\mathcal{G}$ is nilpotent, so $\mathcal{G}$ is pro-nilpotent. Write
$$
\mathcal{G} = \prod_\ell \mathcal G_\ell
$$
for its decomposition into pro-$\ell$ Sylow subgroups, and put
$$
\mathcal G_{2'} := \prod_{\ell \neq 2} \mathcal G_\ell, \qquad L' := L^{\mathcal G_{2'}}.
$$
Since $H$ is pro-$2$, it is contained in $\mathcal G_2$, and the natural projection induces an exact sequence
$$
1 \longrightarrow H \longrightarrow \mathcal G_2 \longrightarrow \Gal(E_2/\mathbb F_q(T)) \longrightarrow 1.
$$
Consequently $L'/E_2$ is a $\mathbb Z_2$-extension, $L'$ is Galois over $\mathbb F_q(T)$, and
$$
L = L' \cdot E = L' \cdot \mathbb{F}_q^{\mathrm{sep}}(T).
$$
Moreover, $L'/E_2$ is everywhere unramified, since this may be checked after the constant field base change from $E_2$ to $E$. The odd Sylow factors commute with $H$, so the conjugation action of $\Frob_q$ on $H$ remains multiplication by $i q_0$. Replacing $L$ by $L'$, we obtain the claimed descent.

We proceed by making another claim.

\begin{claim} 
For any such extension $L$, the exact sequence of Galois groups
\begin{equation}
\label{eGaloisSplit}
0 \to \Gal(L/\mathbb{F}_{q^{2^{\infty}}}(T)(\chi)) \to \Gal(L/\mathbb{F}_q(T)) \xrightarrow{\pi} \Gal(\mathbb{F}_{q^{2^{\infty}}}(T)(\chi)/\mathbb{F}_q(T)) \to 0
\end{equation}
splits as a semi-direct product. 
\end{claim}

\begin{proof}[Proof of claim] Let $v := \infty$. Since the extension $L/\mathbb{F}_{q^{2^\infty}}(T)(\chi)$ is unramified, we have that the image of $\sigma_v$ in $\Gal(L/\mathbb{F}_q(T))$ must have order exactly $2$. Consider the natural quotient map $Q\colon G_{\mathbb{F}_q(T)} \rightarrow \Gal(L/\mathbb{F}_q(T))$. In view of Proposition \ref{Structure of tame inertia}, we see that the map $Q \circ i_v^\ast$ factors through
$$
\Gal(\mathbb{F}_q(T)_v \cdot i_v(E_2)/\mathbb{F}_q(T)_v) \cong \mathbb{F}_2 \times \Z_2,
$$
where the isomorphism is in the category of profinite groups. Furthermore, since $v$ is of odd degree, we see that the projection $\pi$ from \eqref{eGaloisSplit} induces an isomorphism when restricted to the image of $Q \circ i_v^\ast$, giving the desired splitting.
\end{proof}

The claim produces a continuous surjective group homomorphism
$$
(\psi, \chi, \chi_\Frob) \colon G_{\mathbb{F}_q(T)} \to T_{i}(\chi) \rtimes (\mathbb{F}_2 \times \Z_2)
$$
for some $i \in \{1, -1\}$ such that the resulting extension of $\mathbb{F}_{q^{2^{\infty}}}(T)(\chi)$ is everywhere unramified. By construction, for each $v \in \Omega_{\mathbb{F}_q(T)} \setminus \text{Ram}(\chi)$, we must have that $\psi(\sigma_v) = 0$. Furthermore, since $(\psi, \chi, \chi_\Frob)$ is a surjection, it induces a surjection on the abelianization. Hence $\{(\psi)_{\text{mod} \ 2}, \chi_{\epsilon_{\mathbb{F}_q}},\chi\}$ forms an independent set. By defining
$$
\psi_n := \left(\frac{1}{2^n}\psi\right)_{\text{mod} \ T_{i}(\chi)} \in \text{Cocy}(G_{\mathbb{F}_q(T)}, N_i(\chi)[2^n]),
$$
we see that $\psi_n(\sigma_v) = 0$ for each $v \in \Omega_{\mathbb{F}_q(T)} \setminus \text{Ram}(\chi)$, so we conclude that $\psi_n \in B_i(\chi)[2^n]$ by the defining equation \eqref{eBidef} of $B_i(\chi)$ and that $\psi_1 \neq \chi + \frac{(1 - i)}{2} \cdot \chi_{\epsilon_{\mathbb{F}_q}}$. Furthermore, by construction, we have that $2^n \cdot \psi_{n + 1} = \psi_1$. We thus deduce in particular that $\psi_1$ is an element of $C_n^{(i)}(\chi)$ for each nonnegative integer $n$, and $\psi_1 \neq \chi+\frac{(1-i)}{2} \cdot \chi_{\epsilon_{\mathbb{F}_q}}$. We have thus shown that $C_n^{(i)}(\chi) \neq \left \langle \chi + \frac{1-i}{2} \cdot \chi_{\epsilon_{\mathbb{F}_q}} \right \rangle$ for every nonnegative integer $n$ as required.
\end{proof}
\section{The first Artin pairing} \label{Section: first Artin pairing}
Let $q$ be a prime power congruent to $1$ modulo $8$. Let $D := D(T) \in \mathbb{F}_q[T]$ be a squarefree, monic polynomial of odd degree. Let $\pi_1, \ldots, \pi_r \in \mathbb F_q[T]$ be the distinct irreducible factors of $D$. By Proposition \ref{prop: genus spaces}, the set
$$
\{\chi_{\pi_i} : 1 \leq i \leq r\} \cup \{\chi_{\epsilon_{\mathbb{F}_q}}\}
$$
is a basis of $C_0^{(i)}(\chi_D)$ for each $i \in \{1, -1\}$, while $\{e_{\pi_i} : 1 \leq i \leq r\} \cup \{e_{\infty}\}$ is a basis of $D_0^{(i)}(\chi_D)$ for each $i \in \{1, -1\}$. Hence representing $\textup{Art}_{D, 1}^{(i)} = \text{Art}_{\chi_D, 1}^{(i)}$ in these bases yields two matrices, which we shall next describe. 

We begin by defining an $r \times r$ matrix with coefficients in $\mathbb{F}_2$, which we will denote by
$$
\text{R\'edei}(\chi_D) := (R_{i, j}(\chi_D))_{1 \leq i, j \leq r}
$$ 
and is given by the following assignment. For $i \neq j$ elements of $[r]$, we place
$$
R_{i,j}(\chi_D):=\iota \left(\frac{\pi_i}{\pi_j} \right),
$$
where $\iota$ is the unique identification between $\{1,-1\}$ and $\mathbb{F}_2$ as abelian groups. 

Observe that since $q$ is, in particular, $1$ modulo $4$, the matrix $\text{R\'edei}(\chi_D)$ is symmetric. Indeed, applying Hilbert reciprocity to the class $\chi_{\pi_i} \cup \chi_{\pi_j}$ in $H^2(G_{\mathbb{F}_q(T)},\mathbb{F}_2)$, we see that the contribution at $\infty$ vanishes as it can be rewritten, locally, as a multiple of $\chi_T \cup \chi_T$, which vanishes whenever $-1$ is a perfect square in the base field. The only other two places contributing are $\pi_i$ and $\pi_j$, yielding precisely the desired symmetry. 

We complete the description of $\text{R\'edei}(\chi_D)$ by prescribing the diagonal in the unique way that forces the sum of all columns (equivalently, by symmetry, rows) to be equal to $0$. 

Let $R^{(1)}(\chi_D) \in \text{Mat}_{r + 1}(\mathbb{F}_2)$ be the matrix that agrees with $\text{R\'edei}(\chi_D)$ in the first $r \times r$-minor, and is extended as follows. The entries $(i, r + 1)$ with $i \leq r$ are all $0$. The entry $(r + 1, i)$ equals (the reduction modulo $2$ of) $\deg(\pi_i)$ for each $i \leq r$. We take the entry $(r + 1, r + 1)$ to be $1$. 

Let $R^{(-1)}(\chi_D) \in \text{Mat}_{r + 1}(\mathbb{F}_2)$ be the matrix that equals $\text{R\'edei}(\chi_D) + \text{Diag}(\deg(\pi_i))_{i = 1}^{r}$ in the first $r \times r$-minor, and is continued as we explain next. The entries $(i, r + 1)$ with $i \leq r$ are (the reductions modulo $2$ of) $\deg(\pi_i)$. Likewise, the entry $(r + 1, i)$ equals (the reduction modulo $2$ of) $\deg(\pi_i)$ for each $i \leq r$. We declare the entry $(r + 1, r + 1)$ to be $1$. 

\begin{proposition}
\label{pRedei}
For $i \in \{1,-1\}$ the pairing $\textup{Art}_{D, 1}^{(i)}$ is represented in the two bases $\{\chi_{\pi_i} : 1 \leq i \leq r\} \cup \{\chi_{\epsilon_{\mathbb{F}_q}}\}$ and $\{e_{\pi_i} : 1 \leq i \leq r\} \cup \{e_{\infty}\}$ by the matrix
$
R^{(i)}(\chi_D).
$

\end{proposition}

\begin{proof}
We recall that our notation for $\Pi$ does not keep track of $i$ in $\{1, -1\}$. Hence, below, whenever we compute $\Pi_v(-)$, there is an associated choice of $i$ that is not reflected in the notation.

We begin by showing that the off-diagonal entries of the matrix are as claimed. Let us first fix two distinct finite places $v, w \in \text{Ram}(\chi_D)$. Then since $\chi_v(\sigma_w) = 0$, for both values of $i$ in $\{1, -1\}$ we have
$$
\Pi_w(\chi_v) = \chi_v(\Frob_w(\chi_D)) = \chi_v(\Frob_w) = \iota \left( \frac{\pi_v}{\pi_w} \right).
$$
Let us next assume that $v$ is a finite place and $w = \infty$. Let us consider first the case $i = 1$. Since $q_0 \equiv 1 \bmod 4$ in this case, we have
$$
\Pi_\infty(\chi_v) = \chi_v(\Frob_\infty(\chi_D)). 
$$
Observe that $\chi_v$ is locally at $\infty$ in the span of $\chi_D$. It follows that $\chi_v(\Frob_{\infty}(\chi_D)) = 0$, as claimed. 

Let us now consider the case $i = -1$, still with $v$ a finite place and $w=\infty$. In this case we have that
$$
\Pi_{\infty}(\chi_v) = \chi_v(\sigma_{\infty}) + \chi_v(\Frob_{\infty}(\chi_D)). 
$$
Arguing as above, we have that $\chi_v(\Frob_{\infty}(\chi_D)) = 0$. However, $\chi_v(\sigma_{\infty}) \equiv \deg(v) \bmod 2$. Hence we have that $\Pi_{\infty}(\chi_v) \equiv \deg(v) \bmod 2$, as claimed. 

Let us next consider the case, where we pair $\chi_{\epsilon_{\mathbb{F}_q}}$ with any place $v$ in $\text{Ram}(\chi_D)$. Observe that $\chi_{\epsilon_{\mathbb{F}_q}}(\sigma_v)=0$ for all places of $\Omega_{\mathbb{F}_q(T)}$. Hence we have that
$$
\Pi_v(\chi_{\epsilon_{\mathbb{F}_q}}) = \chi_{\epsilon_{\mathbb{F}_q}}(\Frob_{v}(\chi_D)) \equiv \deg(v) \bmod 2. 
$$
It only remains to compute the diagonal terms $\Pi_v(\chi_v)$ for $v$ finite. Thanks to Proposition \ref{prop: trivial elemts}, $r_{\infty}(\chi_D)$ is in $D_1^{(i)}(\chi_D)$. Thus we derive precisely the claimed relation that the sum of all columns is the zero vector, which forces the diagonal to be precisely as described in the two cases $i \in \{1, -1\}$. 
\end{proof}

\begin{remark} 
\label{rmk: identifying spaces}
For each $D$, there is a natural $\mathbb F_2$-linear map $f_D$ from $\FF_2^{r + 1}$ to $\Gamma_{\mathbb{F}_2}(\mathbb{F}_q(T))$. Concretely, the map $f_D$ sends a vector $(a_1, \dots, a_{r + 1}) \in \FF_2^{r + 1}$ to
$$
a_1 \chi_{\pi_1} + \dots + a_r \chi_{\pi_r} + a_{r + 1} \chi_{\epsilon_{\FF_q}}.
$$
By restriction, $f_D$ is also a natural map from the left kernel of $R^{(1)}(\chi_D)$ to $\Gamma_{\mathbb{F}_2}(\mathbb{F}_q(T))$. Observe that an element of the left kernel of $R^{(1)}(\chi_D)$ must have trivial last coordinate (i.e.~satisfies $a_{r + 1} = 0$) by the shape of the matrix. This means that all elements of $C_1^{(1)}(\chi_{D})$ have a monic representative. 

Hence, keeping in mind Proposition \ref{prop: trivial elemts}, we have for each $n \geq 0$ a natural splitting
$$
C_n^{(1)}(\chi_D) = \mathbb{F}_2 \cdot \chi_D \oplus C_n^{(1)}(\chi_D)^\circ,
$$
where $C_n^{(1)}(\chi_D)^\circ$ consists of the elements of $C_n^{(1)}(\chi_D)$ of even degree.

Arguing in a similar way for the right kernel gives 
for each $n \geq 0$ a decomposition 
$$
D_n^{(1)}(\chi_D) = \mathbb{F}_2 \cdot r_{\infty}(\chi_D) \oplus D_n^{(1)}(\chi_D)^\circ,
$$
where the second summand consists of elements supported only in finite places (their degree is necessarily even if $n$ is positive). 

Likewise looking at the shape of $R^{(-1)}(\chi_D)$, and keeping in mind Proposition \ref{prop: trivial elemts}, we see that elements of the left kernel must be either non-monic and of odd degree or monic and of even degree. So there, too, we have a decomposition 
$$
C_n^{(-1)}(\chi_D) = \mathbb{F}_2 \cdot (\chi_D + \chi_{\epsilon_{\mathbb{F}_q}}) \oplus C_n^{(-1)}(\chi_D)^\circ,
$$
where the second summand consists of the elements of even degree (which are automatically monic if $n$ is positive). 

The same applies to the right kernel, yielding a decomposition
$$
D_n^{(-1)}(\chi_D) = \mathbb{F}_2 \cdot r_{\infty}(\chi_D) \oplus D_n^{(-1)}(\chi_D)^\circ,
$$
where the second summand consists of the elements supported only in finite places (and necessarily of even degree if $n$ is positive). 
\end{remark}

We see that, thanks to Remark \ref{rmk: identifying spaces}, we have a natural identification of $\FF_2$-vector spaces
$$
\Sym(\chi_D) \colon C_0^{(i)}(\chi_D)^\circ \to D_0^{(i)}(\chi_D)^\circ,
$$
which is uniquely determined by the following properties:
\begin{itemize}
\item if $\chi_P$ is trivial at infinity (i.e.~$P$ is monic and of even degree), then $\chi_P$ is sent to $\sum_{v \in \text{Ram}(\chi_P)} e_v$, and
\item $\chi_{\epsilon_{\mathbb{F}_q}}$ is sent to $r_\infty(\chi_D) - e_\infty$.
\end{itemize}

A fundamental piece of extra structure on our pairings is the symmetry provided by the following theorem.

\begin{theorem} 
\label{thm: symmetry}
The map $\textup{Sym}(\chi_D)$ maps $C_n^{(i)}(\chi_D)^\circ$ isomorphically onto $D_n^{(i)}(\chi_D)^\circ$ and the induced pairing
$$
\textup{Art}_{D, n + 1}^{(i)} \circ (\textup{id} \times \textup{Sym}(\chi_D)) \colon C_n^{(i)}(\chi_D)^\circ \times C_n^{(i)}(\chi_D)^\circ \to \mathbb{F}_2
$$
is symmetric, for each $i \in \{1, -1\}$ and every nonnegative integer $n$. 
\end{theorem}

The proof of Theorem \ref{thm: symmetry} will require several preliminary results. Recall that the multiplication pairing $\Q_2 \times \Q_2 \to \Q_2$ induces, for every nonnegative integer $n$, a perfect symmetric pairing on 
$$
\frac{\frac{1}{2^n} \cdot \Z_2}{\Z_2} \times \frac{\frac{1}{2^n} \cdot \Z_2}{\Z_2} \to \frac{\Q_2}{\Z_2},
$$
given by
$$
(a, b) \mapsto \frac{1}{2^n} \cdot \left((2^n \cdot a) \cdot (2^n \cdot b) \right)
$$
for representatives $a, b \in \Q_2$. One can verify that this induced pairing does not depend on the choice of the representatives and is perfect and symmetric. 

\begin{proposition} 
\label{prop: main symmetry source}
For $\chi \in \Gamma_{\mathbb{F}_2}(\mathbb{F}_q(T))$ and $i \in \{1,-1\}$ the bilinear pairing $\mathbb{Z}_2 \times \mathbb{Z}_2 \to \mathbb{Z}_2$, given by multiplication, induces a $G_{\mathbb{F}_q(T)}$-equivariant perfect pairing
$$
T_i(\chi) \times T_i(\chi) \to \mathbb{Z}_2(1).
$$
Identifying $N_i(\chi)[2^n]$ and $T_i(\chi)/2^n$ through the multiplication by $2^n$ on $V_i(\chi)$, we have that this naturally induces a $G_{\mathbb{F}_q(T)}$-equivariant perfect pairing
$$
N_{i}(\chi)[2^n] \times N_i(\chi)[2^n] \to \frac{2^{-n} \mathbb Z_2}{\mathbb{Z}_2}(1).
$$
\end{proposition}

\begin{proof}
Clearly, the multiplication pairing on $\Z_2$ is a perfect bilinear pairing since $\Z_2$ is an integral domain. We now check Galois equivariance of the induced pairing
$$
T_i(\chi) \times T_i(\chi) \to \Z_2(1).
$$
Let $g \in G_{\mathbb{F}_q(T)}$ and let us recall that $g$ acts on $\Z_2(1)$ by multiplication by $q^{\chi_{\Frob}(g)}$. On the other hand $g$ acts on $T_i(\chi)$ by multiplication by $(i \cdot q_0)^{\chi_{\Frob}(g)} \cdot (-1)^{\chi(g)}$. Hence we have that
\begin{align*}
((i \cdot q_0)^{\chi_{\Frob}(g)} \cdot (-1)^{\chi(g)} \cdot a, (i \cdot q_0)^{\chi_{\Frob}(g)} \cdot (-1)^{\chi(g)} \cdot b) &\mapsto (i \cdot q_0)^{2 \cdot \chi_{\Frob}(g)} \cdot (-1)^{2 \cdot \chi(g)} \cdot a \cdot b \\
&= q^{\chi_{\Frob}(g)} \cdot (a \cdot b),
\end{align*}
which establishes the desired Galois equivariance. 

The induced pairing on $V_i(\chi) \times V_i(\chi) \to \Q_2(1)$ is also a perfect symmetric pairing, and one can check that the procedure to get a pairing on 
$$
N_i(\chi)[2^n] \times N_i(\chi)[2^n] \to \frac{\Q_2}{\Z_2}(1),
$$
as explained right above the proposition, is also Galois equivariant, since it is given by scaling both coordinates by $2^n$, which commutes with the action of Galois. 
\end{proof}

Thanks to Proposition \ref{prop: main symmetry source}, we have an identification $\lambda \colon N_i(\chi)[2^n] \to (N_i(\chi)[2^n])^\ast$ of $G_{\mathbb{F}_q(T)}$-modules. 

Let $n \geq 1$. We will start by giving, for each $v \in \text{Ram}(\chi)$, an explicit description of 
$$
H^1(G_{\mathbb{F}_q(T)_v}, N_i(\chi)[2^n]).
$$
We apply inflation-restriction with respect to the inertia subgroup $I_v$ at $v$. Observe that $N_i(\chi)^{I_v} = N_i(\chi)[2]$ and $H^2(G_{\mathbb{F}_{q_v}}, N_i(\chi)[2]) = 0$. Hence the inflation-restriction sequence yields
\begin{equation}
\label{eInfRes}
0 \to H^1(G_{\mathbb{F}_{q_v}}, N_i(\chi)[2]) \to H^1(G_{\mathbb{F}_q(T)_v}, N_i(\chi)[2^n]) \to H^1(I_v, N_i(\chi)[2^n])^{G_{\mathbb{F}_{q_v}}} \to 0.
\end{equation}
We introduce the shorthand $\epsilon_v := \epsilon_{\mathbb{F}_{q_v}}$. With this notation, the image of the first map equals precisely $\langle \chi_{\epsilon_v} \rangle$. But $H^1(I_v, N_i(\chi)[2^n])^{G_{\mathbb{F}_{q_v}}}$ is a group of order $2$. Notice now that the restriction of $\rho_n$ to $G_{\mathbb{F}_q(T)_v}$ gives a splitting of the exact sequence, so that we can decompose
$$
H^1(G_{\mathbb{F}_q(T)_v}, N_i(\chi)[2^n]) = \mathbb{F}_2 \cdot \chi_{\epsilon_v} \oplus \mathbb{F}_2 \cdot \rho_n.
$$
We now have two key computations. Note that evaluation at $\sigma_v$ on $H^1(I_v, N_i(\chi)[2^n])$ gives an element in $N_i(\chi)[2^n]/N_i(\chi)[2^{n - 1}] \cong \mathbb{F}_2$. We will write this element as $\mathrm{Ev}(\psi, \sigma_v) \in \mathbb{F}_2$.

\begin{proposition} 
\label{prop: computing coordinates locally}
Let $n \in \Z_{\geq 1}$. Let $v \in \textup{Ram}(\chi)$. We have for all $\psi \in H^1(G_{\mathbb{F}_q(T)_v}, N_i(\chi)[2^n])$
$$
\psi = \Pi_v(\psi) \cdot \chi_{\epsilon_v} + \mathrm{Ev}(\psi, \sigma_v) \cdot \rho_n.
$$
\end{proposition}

\begin{proof}
Since $\rho_n = 2 \cdot \rho_{n + 1}$, it follows that $\Pi_v(\rho_n) = 0$ in view of Proposition \ref{prop: detecting locally}. At the same time, it follows from the definition of $\Pi_v$ that
$$
\Pi_v(\chi_{\epsilon_v}) = \chi_{\epsilon_v}(\Frob_v(\chi)) = 1.
$$
This shows that if $\psi = \lambda_1 \cdot \chi_{\epsilon_v} + \lambda_2 \cdot \rho_n$, then $\Pi_v(\psi) = \lambda_1$.

To compute $\lambda_2$, we return to the exact sequence \eqref{eInfRes}. Note that $\chi_{\epsilon_v}$ restricts trivially to $I_v$, while the image of $\rho_n$ in $H^1(I_v,N_i(\chi)[2^n])$ is exactly $\mathrm{Ev}(\rho_n, \sigma_v) = 1$. Thus, we have
$$
\lambda_2 = \mathrm{Ev}(\psi, \sigma_v). 
$$
This is precisely the desired conclusion. 
\end{proof}

Finally, we have a cup product pairing
$$
\langle -, - \rangle_{v, n} \colon H^1(G_{\mathbb{F}_q(T)_v}, N_i(\chi)[2^n]) \times H^1(G_{\mathbb{F}_q(T)_v}, N_i(\chi)[2^n]) \to \frac{\mathbb{Q}_2}{\mathbb{Z}_2}
$$
through the identification $\lambda$. Since $H^1(G_{\mathbb{F}_q(T)_v}, N_i(\chi)[2^n])$ is a vector space over $\mathbb{F}_2$ by Proposition \ref{prop: computing coordinates locally}, we have that $\langle -, - \rangle_{v, n}$ lands in the $2$-torsion subgroup $\frac{\frac{1}{2} \cdot \mathbb{Z}_2}{\mathbb{Z}_2} \cong \mathbb{F}_2$. We now compute this pairing. 

\begin{proposition} 
\label{prop: computing local cups}
Let $n \in \Z_{\geq 1}$. Let $v$ be a place in $\textup{Ram}(\chi)$. We have that $\langle \chi_{\epsilon_v}, \chi_{\epsilon_v} \rangle_{v, n} = \langle \rho_n, \rho_n \rangle_{v, n} = 0$ and $\langle \chi_{\epsilon_v}, \rho_n \rangle_{v, n} = 1$.
\end{proposition}

\begin{proof}
The class $\chi_{\epsilon_v} \cup \chi_{\epsilon_v}$ comes by inflation from $H^2(G_{\mathbb{F}_{q_v}}, \frac{\Q_2}{\Z_2}(1)) = 0$, giving the first vanishing. To see that $\langle \chi_{\epsilon_v}, \rho_n \rangle_{v, n} = 1$, write $\iota \colon N_i(\chi)[2] \rightarrow N_i(\chi)[2^n]$ for the inclusion and consider the diagram
$$
\begin{tikzcd}[column sep=0.1in]
H^1(G_{\mathbb{F}_q(T)_v}, N_i(\chi)[2]) \ar[d, "\iota"] & \times & H^1(G_{\mathbb{F}_q(T)_v}, N_i(\chi)[2]) & \to & \Q_2/\Z_2 \ar[d, "="] \\
H^1(G_{\mathbb{F}_q(T)_v}, N_i(\chi)[2^n]) & \times & H^1(G_{\mathbb{F}_q(T)_v}, N_i(\chi)[2^n]) \ar[u, "\iota^\ast"] & \to & \Q_2/\Z_2
\end{tikzcd}
$$
which gives $\langle \iota(x), y \rangle_{v, n} = \langle x, \iota^\ast(y) \rangle_{v, 1}$ by a direct cochain computation. We apply this with $x = \chi_{\epsilon_v}$ and $y = \rho_n$ to deduce that
$$
\langle \chi_{\epsilon_v}, \rho_n \rangle_{v, n} = \left \langle \chi_{\epsilon_v}, \chi + \frac{1 - i}{2} \cdot \deg(v) \cdot \chi_{\epsilon_v} \right \rangle_{v, 1} = \left \langle \chi_{\epsilon_v}, \chi \right \rangle_{v, 1} = 1.
$$
We now proceed to show that $\langle \rho_n, \rho_n \rangle_{v, n} = 0$. We are going to apply \cite[Corollary 1.4.6]{NSW} with $C := \frac{\Q_2}{\Z_2}(1)$ and the exact sequence
$$
0 \to N_i(\chi)[2^n] \to N_i(\chi)[2^{2n}] \to N_i(\chi)[2^n] \to 0.
$$
Through our identification $\lambda$, this yields the diagram
$$
\begin{tikzcd}[column sep=0.1in]
H^1(G_{\mathbb{F}_q(T)_v}, N_i(\chi)[2^n]) \ar[d, "\delta"] & \times & H^1(G_{\mathbb{F}_q(T)_v}, N_i(\chi)[2^n]) & \to & \mathrm{Br}(\mathbb{F}_q(T)_v) \ar[d, "="] \\
H^2(G_{\mathbb{F}_q(T)_v}, N_i(\chi)[2^n]) & \times & H^0(G_{\mathbb{F}_q(T)_v}, N_i(\chi)[2^n]) \ar[u, "\delta"] & \to & \mathrm{Br}(\mathbb{F}_q(T)_v).
\end{tikzcd}
$$
Observe that $\rho_n = \delta(1/2)$. From the diagram, we thus obtain 
$$
\langle \rho_n, \rho_n \rangle_{v, n} = \langle \rho_n, \delta(1/2) \rangle_{v, n} = \delta(\rho_n) \cup (1/2).
$$
We now show that $\delta(\rho_n) = 0$ as an element of $H^2(G_{\mathbb{F}_{q}(T)_v}, N_i(\chi)[2^n])$. But indeed, a direct computation shows that $2^n \cdot \rho_{2n} = \rho_n$, and thus the desired vanishing follows from the long exact sequence.
\end{proof}

\begin{proof}[Proof of Theorem \ref{thm: symmetry}] 
Theorem \ref{thm: symmetry} follows from Proposition \ref{pRedei} for $n = 0$. We henceforth assume that $n \geq 1$. We begin with the following claim. 

\begin{claim} 
Let $\chi_1, \chi_2$ be two elements of $C_n^{(i)}(\chi_D)^\circ$ such that $\Sym(\chi_D)(\chi_1)$ and $\Sym(\chi_D)(\chi_2)$ are both elements of $D_n^{(i)}(\chi_D)$. Then 
$$
\textup{Art}_{D, n + 1}^{(i)}(\chi_1, \Sym(\chi_D)(\chi_2)) = \textup{Art}_{D, n + 1}^{(i)}(\chi_2, \Sym(\chi_D)(\chi_1)).
$$
\end{claim}

\begin{proof}[Proof of Claim] By definition of $C_n^{(i)}(\chi_D)$, we can find cocycles $\psi_1, \psi_2 \in B_i(\chi_D)[2^{n + 1}]$ such that $2^n \cdot \psi_j = \chi_j$. Via $\lambda$ we can view $\psi_1 \cup \psi_2$ as an element of $H^2(G_{\mathbb{F}_q(T)}, \frac{\Q_2}{\Z_2}(1))$. Therefore Hilbert reciprocity gives that
$$
\sum_{v \in \Omega_{\mathbb{F}_q(T)}} \inv_v(\psi_1 \cup \psi_2) = 0.
$$
The invariant map vanishes outside of $\text{Ram}(\chi_D)$, as in that case the cocycles $\psi_1$ and $\psi_2$ are unramified thanks to the definition of $B_i(\chi_D)$. Thus $\psi_1 \cup \psi_2$ comes from $H^2(G_{\mathbb{F}_{q_v}}, \frac{\Q_2}{\Z_2}(1)) = 0$ by inflation. Hence we conclude that 
$$
\sum_{v \in \text{Ram}(\chi_D)} \inv_v(\psi_1 \cup \psi_2) = 0.
$$
Applying Proposition \ref{prop: computing coordinates locally} and then Proposition \ref{prop: computing local cups} (both with $n + 1$ in place of $n$), we deduce that
$$
\sum_{v \in \text{Ram}(\chi_1)} \Pi_v(\psi_2) + \sum_{w \in \text{Ram}(\chi_2)} \Pi_w(\psi_1) = 0.
$$
For $n \geq 1$, the first summand equals precisely $\text{Art}_{D, n + 1}^{(i)}(\chi_2, \Sym(\chi_D)(\chi_1))$, while the second equals precisely $\text{Art}_{D, n + 1}^{(i)}(\chi_1, \Sym(\chi_D)(\chi_2))$. Therefore we conclude that
$$
\text{Art}_{D, n + 1}^{(i)}(\chi_1, \Sym(\chi_D)(\chi_2)) = \text{Art}_{D, n + 1}^{(i)}(\chi_2, \Sym(\chi_D)(\chi_1)),
$$
as desired. 
\end{proof}

Having proved the claim above, we now prove Theorem \ref{thm: symmetry} by induction. Observe that $\Sym(\chi_D)$ induces an isomorphism between $C_0^{(i)}(\chi_D)^\circ$ and $D_0^{(i)}(\chi_D)^\circ$. However, if we know that $\Sym(\chi_D)$ maps $C_n^{(i)}(\chi_D)^\circ$ isomorphically to $D_n^{(i)}(\chi_D)^\circ$ for a given value of $n$, then we deduce from the claim that the induced pairing 
$$
\text{Art}_{D, n + 1}^{(i)} \circ (\text{id} \times \Sym(\chi_D))
$$
is symmetric and therefore the left kernel and right kernel coincide. The left kernel is by Proposition \ref{prop: elementary def of pairing} precisely $C_{n + 1}^{(i)}(\chi_D)^\circ$, while, still by Proposition \ref{prop: elementary def of pairing}, the right kernel equals $\Sym(\chi_D)^{-1}(D_{n + 1}^{(i)}(\chi_D)^\circ)$. We conclude that $\Sym(\chi_D)$ maps $C_{n + 1}^{(i)}(\chi_D)^\circ$ isomorphically to $D_{n + 1}^{(i)}(\chi_D)^\circ$, completing the induction.
\end{proof}

We conclude with the following rephrasing of Proposition \ref{pNonCrit}.

\begin{proposition} 
\label{pNonCrit2}
Let $\chi$ be a monic element of $\Gamma_{\mathbb{F}_2}^{\textup{im}}(\mathbb{F}_q(T))$. Then $L\left(\frac{1}{2},\chi\right) \neq 0$ if there exists a nonnegative integer $n$ such that for all $i \in \{1, -1\}$ we have 
$$
C_n^{(i)}(\chi)^\circ = \{0\}.
$$
\end{proposition}

\begin{proof}
Indeed, we have that 
$$
C_n^{(i)}(\chi) = \mathbb{F}_2 \cdot \left(\chi + \frac{(1 - i)}{2} \cdot \chi_{\epsilon_{\mathbb{F}_q}}\right) \oplus C_n^{(i)}(\chi)^\circ.
$$
Therefore $C_n^{(i)}(\chi) = \mathbb{F}_2 \cdot \left(\chi + \frac{(1 - i)}{2} \cdot \chi_{\epsilon_{\mathbb{F}_q}}\right)$ if and only if $C_n^{(i)}(\chi)^\circ = \{0\}$. Hence the desired conclusion follows from Proposition \ref{pNonCrit}.
\end{proof}
\section{Expansion maps} 
\label{section: expansion maps}
In this section we explain the relevant changes to \cite[Section 2.2]{KP} that allow us to speak about expansion maps and normalized expansion maps also over $\mathbb{F}_q(T)$. 

We have that \cite[Definition 2.12]{KP} can be adapted verbatim, by replacing the symbol $\Q$ by $\mathbb{F}_q(T)$ in all occurrences. In this way, for a linearly independent finite set $X \subseteq \Gamma_{\mathbb{F}_2}(\mathbb{F}_q(T))$, and an element $\chi_0$ of $X$ we now have the definition of an \emph{expansion map} $\psi$ with support $X$ and \emph{pointer} $\chi_0$. As explained right below \cite[Definition 2.12]{KP}, such data are the same thing as a collection of continuous $1$-cochains $(\phi_Y)_{Y \subseteq X-\{\chi_0\}}$ satisfying the recursive formulas \cite[Equation 2.3]{KP}. The data of the corresponding expansion map is then recovered through \cite[Equation 2.2]{KP}. 

Next, as summarized right before \cite[Lemma 2.13]{KP}, given $X, \chi_0$ as above, there exists \emph{at most one} expansion map $\psi(\mathfrak{G}) = (\phi_Y(\mathfrak{G}))_{Y \subseteq X - \{\chi_0\}}$ satisfying
$$
\phi_Y(\sigma) = 0 \quad \text{ for all } \sigma \in \mathfrak{G} \text{ and for all subsets } \emptyset \subset Y \subseteq X - \{\chi_0\}
$$
The reason is precisely the same as the one explained right before \cite[Lemma 2.13]{KP}: thanks to Proposition \ref{Topological generators}, the set $\mathfrak{G}$ is a set of topological generators for $G_{\mathbb{F}_q(T)}(2)$. Hence this requirement determines the value of $\psi(\mathfrak{G})$ on $\mathfrak{G}$ and therefore by continuity on the entire $G_{\mathbb{F}_q(T)}(2)$. 

We will follow the same notation used in \cite{KP}. So, if $X$ corresponds to a set of squarefree polynomials in $\frac{\mathbb{F}_q(T)^\ast}{\mathbb{F}_q(T)^{\ast 2}}$, with pointer $\chi_0$ corresponding to a squarefree polynomial $d$, then we will denote by
$$
\phi_{S; d}(\mathfrak{G})
$$
the $1$-cochains corresponding to the expansion map $\psi(\mathfrak{G})$. Here $S$ is a subset of squarefree polynomials corresponding to a subset of characters of $X - \{\chi_0\}$. So in the same fashion, for a nonnegative integer $s$, for a set $S := \{a_1, \ldots, a_s\}$ of cardinality $s$, consisting of squarefree polynomials, for a squarefree polynomial $a_{s + 1}$ such that the set $\{a_1, \ldots, a_{s + 1}\}$ is linearly independent in $\frac{\mathbb{F}_q(T)^\ast}{\mathbb{F}_q(T)^{\ast 2}}$, we have at most one cochain
$$
\phi_{\{a_1, \ldots, a_s\}; a_{s + 1}}(\mathfrak{G}),
$$
as above. We will call these $1$-cochains \emph{normalized expansion maps} as well. By convention, we will define $\phi_{\{a_1, \ldots, a_s\}; a_{s + 1}}(\mathfrak{G})$ to be the zero map if one of the $a_i$ is equal to $1$.

With these conventions, \cite[Lemma 2.13]{KP} holds verbatim in our setting, with the obvious substitution that $\Q$ gets replaced by $\mathbb{F}_q(T)$ and squarefree integers by squarefree polynomials. This also gives us the \emph{additivity equations} of \cite[Lemma 2.13]{KP} for expansion maps. The basic group-theoretic result \cite[Proposition 2.14]{KP} holds in our case as well.

So far the translation has been almost entirely literal. In fact, this part of the material is entirely formal, and could be developed over any field (replacing the parlance of squarefree numbers or polynomials simply with classes in $\frac{K^\ast}{K^{\ast 2}}$, and replacing $\mathfrak{G}$ with any minimal set of topological generators for $G_{K}(2)$). However the material spanning from \cite[Proposition 2.15]{KP} until the end of that section has arithmetic content, and in particular it partly relies on the presence of a wildly ramified prime (namely $2$) and on the notion of finite primes of the rationals. Since the analogy with function fields is not entirely perfect, we opted to rework this remaining material from scratch. We have two remaining goals: 
\begin{enumerate}
\item[(1)] provide control on the ramification behavior of normalized expansion maps,
\item[(2)] provide a criterion to construct inductively such normalized maps.
\end{enumerate}

\noindent Let us start by addressing point $(1)$. Given linearly independent $a_1, \ldots, a_n$ in $\frac{\mathbb{F}_q(T)^\ast}{\mathbb{F}_q(T)^{\ast 2}}$, we define for every place $v \in \Omega_{\mathbb{F}_q(T)}^{\textup{fin}}$ the place $\mathrm{Up}(v) \in \Omega_{\mathbb{F}_q(T)(\{\chi_{a_i} : i \in [n]\})}$ given by restricting $i_v$ to $\mathbb{F}_q(T)(\{\chi_{a_i} : i \in [n]\})$. We then define $\mathrm{AUp}(v)$ to be $g \cdot \mathrm{Up}(v)$, where we take $g \in \Gal(\mathbb{F}_q(T)(\{\chi_{a_i} : i \in [n]\})/\mathbb{F}_q(T))$ to be the unique element sending $\sqrt{a_i}$ to $-\sqrt{a_i}$ for all $i$.

\begin{proposition} 
\label{normalized expansion maps are unramified}
Let $n$ be a positive integer. Let $a_1, \ldots, a_{n + 1}$ be squarefree polynomials that are linearly independent in $\frac{\mathbb{F}_q(T)^{\ast}}{\mathbb{F}_q(T)^{\ast 2}}$ with $\textup{Ram}(\chi_{a_i}) \cap \textup{Ram}(\chi_{a_j}) \cap \Omega_{\mathbb{F}_q(T)}^{\textup{fin}} = \emptyset$ for each distinct $i, j \in [n + 1]$. Suppose that the normalized expansion map $\phi_{\{a_1, \ldots, a_n\}; a_{n + 1}}(\mathfrak{G})$ exists. 
\begin{enumerate}
\item[$(a)$] The extension $L(\phi_{\{a_1, \ldots, a_n\}; a_{n + 1}}(\mathfrak{G}))/\mathbb{F}_q(T)(\{\chi_{a_i} : i \in [n]\})$ is a multiquadratic extension containing $\mathbb{F}_q(T)(\chi_{a_{n + 1}})$. Furthermore, the extension 
$$
L(\phi_{\{a_1, \ldots, a_n\}; a_{n + 1}}(\mathfrak{G}))/\mathbb{F}_q(T)(\{\chi_{a_i} : i \in [n + 1]\})
$$
is unramified at all places of $\mathbb{F}_q(T)(\{\chi_{a_i} : i \in [n + 1]\})$ lying above $\Omega_{\mathbb{F}_q(T)}^{\textup{fin}}$. 

Moreover, let $g \in \Gal(\mathbb{F}_q(T)(\{\chi_{a_i} : i \in [n]\})/\mathbb{F}_q(T))$ and define $\mathrm{Supp}(g)$ to be the subset of $i \in [n]$ such that $g(\sqrt{a_i}) = -\sqrt{a_i}$. Then
\begin{equation}
\label{eConjugateExpansion}
\phi_{\{a_1, \ldots, a_n\}; a_{n + 1}}(\mathfrak{G})^g = \sum_{[n] - \textup{Supp}(g) \subseteq T \subseteq [n]} \phi_{\{a_i : i \in T\}; a_{n + 1}}(\mathfrak{G}).
\end{equation}
Finally, for each finite place $v$ in $\textup{Ram}(\chi_{a_{n + 1}})$, the quadratic character obtained by restricting $\phi_{\{a_1, \ldots, a_n\}; a_{n + 1}}(\mathfrak{G})$ to $G_{\mathbb{F}_q(T)(\{\chi_{a_i} : i \in [n]\})}$ ramifies at $\mathrm{AUp}(v)$ and at no other places above $v$.
\item[$(b)$] Assume now also that $a_i$ is monic of even degree for each $i \in [n + 1]$. Then 
$$
\{\res_{\infty}(\phi_{\{a_i : i \in U\}; a_{n + 1}}(\mathfrak{G})) : U \subseteq [n]\} \subseteq \Gamma_{\mathbb{F}_2}(\mathbb{F}_q(T)_{\infty}),
$$
and moreover the characters $\res_{\infty}(\phi_{\{a_i : i \in U\}; a_{n + 1}}(\mathfrak{G}))$ are contained in the $1$-dimensional space $\mathbb{F}_2 \cdot \res_{\infty}(\chi_T)$.
\end{enumerate}
\end{proposition}

\begin{proof}
For $(a)$, it can be shown easily using the recursive equation for expansion maps that the field of definition $L(\phi_{\{a_1, \ldots, a_n\}; a_{n + 1}}(\mathfrak{G}))$ coincides with the field of definition of the homomorphism $L(\psi_{\{a_1, \ldots, a_n\}; a_{n + 1}}(\mathfrak{G}))$ corresponding to the normalized expansion map.

The Galois group $\Gal(L(\phi_{\{a_1, \ldots, a_n\}; a_{n + 1}}(\mathfrak{G}))/\mathbb{F}_q(T)(\{\chi_{a_i} : i \in [n]\}))$ therefore gets mapped through the expansion map precisely on $\mathbb{F}_2[\mathbb{F}_2^n] \rtimes \{0\}$, showing, thanks to Galois theory, that the extension is a multiquadratic extension of degree $2^{2^n}$. The character $\chi_{a_{n + 1}}$ is precisely given by $\phi_{\emptyset; a_{n + 1}}(\mathfrak{G})$. Thanks to definition of normalized expansion maps, we see that for each finite place $v$ in $\Omega_{\mathbb{F}_q(T)}$ outside of $\cup_{i \in [n + 1]} \text{Ram}(\chi_{a_i})$, we must have that $\sigma_v$ is sent to the element of $\mathbb{F}_2[\mathbb{F}_2^n] \rtimes \mathbb{F}_2^n$ having all coordinates equal to $0$, which is nothing else than the identity element of this group, showing that the extension is unramified at these places. 

For all finite places in $\Omega_{\mathbb{F}_q(T)}$ contained in $\cup_{i \in [n + 1]} \, \text{Ram}(\chi_{a_i})$, we see that $\sigma_v$ is sent to an involution. It follows that the ramification index of $v$ inside $L(\phi_{\{a_1, \ldots, a_n\}; a_{n + 1}}(\mathfrak{G}))/\mathbb{F}_q(T)$ equals $2$. However, this is also true for $\mathbb{F}_q(T)(\{\chi_{a_i} : i \in [n + 1]\})/\mathbb{F}_q(T)$. It follows that 
$$
L(\phi_{\{a_1, \ldots, a_n\}; a_{n + 1}}(\mathfrak{G}))/\mathbb{F}_q(T)(\{\chi_{a_i} : i \in [n + 1]\})
$$
is unramified above all places lying above a finite place in $\Omega_{\mathbb{F}_q(T)}$. 

Next, we note that equation \eqref{eConjugateExpansion} follows from the recursive formula for expansion maps \cite[Equation 2.3]{KP}. Finally, let $v \in \text{Ram}(\chi_{a_{n+1}})$. Now recall that $i_v$ induces a prime, denoted $\mathrm{Up}(v)$, in $\mathbb{F}_q(T)(\{\chi_{a_i} : i \in [n]\})$ that lies above $v$. Then $\res_{G_{\mathbb{F}_q(T)(\{\chi_{a_i} : i \in [n]\})}} \phi_{\{a_1, \ldots, a_n\}; a_{n + 1}}(\mathfrak{G})$ ramifies at $g \cdot \mathrm{Up}(v)$ if and only if the character 
$$
\phi_{\{a_1, \ldots, a_n\}; a_{n + 1}}(\mathfrak{G})^g = \sum_{[n] - \text{Supp}(g) \subseteq T \subseteq [n]} \phi_{\{a_i; i \in T\}; a_{n + 1}}(\mathfrak{G})
$$
ramifies at $\mathrm{Up}(v)$. This latter condition can now directly be checked by plugging in $\sigma_v$ and using that the expansion map is normalized; we see that the total sum yields $1$ if and only if $\text{Supp}(g) = [n]$. This proves part $(a)$. 

We now prove part $(b)$. The fact that all expansion maps are characters locally at $\infty$ follows at once from the expression for their differential combined with the fact that all of the $a_i$, for $i$ in $[n]$, are monic polynomials of even degree, so that $\chi_{a_i}$ is a trivial character locally at $\infty$. Since $a_{n + 1}$ is also monic and hence vanishes on $\Frob_{\infty}(\chi_T)$ and since normalized expansion maps vanish on $\Frob_{\infty}(\chi_T)$, it follows that any conjugate of a normalized expansion map (which can be re-expressed as a linear combination of normalized expansion maps and $\chi_{a_{n + 1}}$) vanishes on $\Frob_{\infty}(\chi_T)$. The only such non-trivial local character in $\Gamma_{\mathbb{F}_2}(\mathbb{F}_q(T)_{\infty})$ is precisely $\chi_T$, yielding the desired conclusion.  
\end{proof}

We now deal with point $(2)$. We have the following. 

\begin{proposition} 
\label{inductively construct expansion maps}
Let $n$ be a nonnegative integer. Let $a_1, \ldots, a_{n + 1}$ be squarefree polynomials that are linearly independent in $\frac{\mathbb{F}_q(T)^{\ast}}{\mathbb{F}_q(T)^{\ast 2}}$ with $\textup{Ram}(\chi_{a_i}) \cap \textup{Ram}(\chi_{a_j}) \cap \Omega_{\mathbb{F}_q(T)}^{\textup{fin}} = \emptyset$ for each distinct $i, j \in [n + 1]$. Suppose that for each distinct $i, j \in [n + 1]$ and for each $v \in \textup{Ram}(\chi_{a_i}) \cap \Omega_{\mathbb{F}_q(T)}^{\textup{fin}}$, we have that
$$
\chi_{a_j} \circ i_v^\ast = 0.
$$
Suppose also that $a_i$ is monic of even degree for each $i \in [n]$.

Then $\phi_{\{a_1,\ldots, a_n\}; a_{n + 1}}(\mathfrak{G})$ exists if and only if $\phi_{\{a_i : i \neq j\}; a_{n + 1}}(\mathfrak{G})$ exists for all $j \in [n]$ and every place $v \in \textup{Ram}(\chi_{a_j}) \cap \Omega_{\mathbb{F}_q(T)}^{\textup{fin}}$ splits completely in $L(\phi_{\{a_i : i \neq j\}; a_{n + 1}}(\mathfrak{G}))$.
\end{proposition}

\begin{proof}
Suppose first that $\phi_{\{a_i : i \neq j\}; a_{n + 1}}(\mathfrak{G})$ exists for all $j \in [n]$ and that every place $v \in \textup{Ram}(\chi_{a_j}) \cap \Omega_{\mathbb{F}_q(T)}^{\textup{fin}}$ splits completely in $L(\phi_{\{a_i : i \neq j\}; a_{n + 1}}(\mathfrak{G}))$. Let us consider the defining cocycle of the expansion map 
$$
\theta := \sum_{\emptyset \neq T \subseteq [n]} \chi_{\{a_i : i \in T\}}(\sigma) \cdot \phi_{\{a_h : h \in [n] \setminus T\}; a_{n + 1}}(\mathfrak{G})(\tau).
$$
We now show that $\theta$ is locally trivial everywhere, which implies vanishing inside $H^2(G_{\mathbb{F}_q(T)},\mathbb{F}_2)$ by local-to-global. Firstly, $\theta$ vanishes at $\infty$, since all of the $a_i$ are monic and of even degree. Similarly, $\theta$ vanishes for places $v \in \text{Ram}(\chi_{a_{n + 1}})$ thanks to our assumption that $\chi_{a_i}$, with $i \in [n]$, is locally trivial at all such places $v$. It remains to check the places $v \in \text{Ram}(\chi_{a_j})$ for some $j \in [n]$. Using one more time our assumption that all of the other $\chi_{a_i}$, for $i \in [n] \setminus \{j\}$ are locally trivial at such a $v$, we see that the only term surviving in the expression for the cocycle above, when restricted to $G_{\mathbb{F}_q(T)_v}$, is
$$
\chi_{a_j}(\sigma) \cdot \phi_{\{a_h: h \in [n] \setminus \{j\}\}; a_{n + 1}}(\mathfrak{G})(\tau).
$$
On the other hand, $\phi_{\{a_h : h \in [n] \setminus \{j\}\}; a_{n + 1}}(\mathfrak{G})$ vanishes at $\sigma_v$ by definition. Moreover, since $L(\phi_{\{a_h : h \in [n] \setminus \{j\}\}; a_{n + 1}}(\mathfrak{G}))/\mathbb{F}_q(T)$ splits completely at $v$ by assumption, we conclude that this expression is literally the zero function at any such $v$. It remains to examine $\theta$ at places $v$ which are finite and out of $\cup_{i \in [n + 1]} \text{Ram}(\chi_{a_i})$. Since the field of definition of all these expansion maps and characters will be unramified at such a place, we see that in that case $\theta$ comes from the inflation from the absolute Galois group of the residue field, and therefore vanishes in cohomology. 

All in all, we conclude that $\theta$ is zero inside $H^2(G_{\mathbb{F}_q(T)},\mathbb{F}_2)$. Let $\phi \in C^1(G_{\mathbb{F}_q(T)}, \mathbb{F}_2)$ with $\mathrm{d}\phi = \theta$. We see that $\phi$ vanishes at all but finitely many elements of $\mathfrak{G}$. Moreover, the expression
$$
\phi':=\phi + \phi(\Frob_q) \cdot \chi_{\epsilon_{\mathbb{F}_q}} + \sum_{v \in \Omega_{\mathbb{F}_q(T)}^{\text{fin}}} \phi(\sigma_v) \cdot \chi_v
$$
is well-defined, satisfies $\mathrm{d}\phi' = \theta$ and vanishes at $\mathfrak{G}$. Hence $\phi_{\{a_1, \ldots, a_n\}; a_{n + 1}}(\mathfrak{G})$ exists. 

For the other direction, observe that if $\phi_{\{a_1, \ldots, a_n\}; a_{n + 1}}(\mathfrak{G})$ exists, then by definition all of its lower level expansion maps must exist. Furthermore, by construction $\theta$ above must vanish in cohomology. Examining its behavior at the irreducible polynomials dividing $a_1, \ldots, a_n$, we see precisely the required splitting demand. This ends the proof. 
\end{proof}

We also need the following ``additivity'' result.

\begin{proposition}
\label{pExpansionAdditive}
Let $n \in \Z_{\geq 0}$. For $h \in \{1, 2, 3\}$ let 
$$
(a_{1, h}, \ldots, a_{n + 1, h}) \in \left(\frac{\mathbb{F}_q(T)^\ast}{\mathbb{F}_q(T)^{\ast 2}}\right)^{[n + 1]}
$$ 
be a vector whose underlying set has cardinality $n + 1$ and is linearly independent for each $h$. Suppose furthermore that there exists one entry $n_0 \in [n + 1]$ such that the vectors are identical outside of that entry, while at $n_0$ we have $a_{n_0, 1} a_{n_0, 2} a_{n_0, 3} = 1$.

Suppose furthermore that $\phi_{\{a_{i, h} : i \in [n]\}; a_{n + 1, h}}(\mathfrak{G})$ exists for $h \in \{1, 2\}$.

Then $\phi_{\{a_{i, 3} : i \in [n]\}; a_{n + 1, 3}}(\mathfrak{G})$ also exists and furthermore
$$
\sum_{h = 1}^3 \phi_{\{a_{i, h} : i \in [n]\}; a_{n + 1, h}}(\mathfrak{G}) = 0. 
$$
\end{proposition}

\begin{proof}
This follows by the argument proving \cite[Lemma 2.13]{KP}, which applies verbatim over $\mathbb{F}_q(T)$. 
\end{proof}
\section{\texorpdfstring{The theory of $\hat{\phi}$ cochains}{The theory of hat phi cochains}} \label{Section: phi hat}
Let $q$ be a prime power congruent to $1$ modulo $8$. The purpose of this section is to describe a family of cochains over $\mathbb{F}_q(T)$ that, when restricted to suitable subfields, provide trivializations of cup products of expansion maps. 

We begin with some general group-theoretic constructions. Let $A$ and $B$ be two disjoint, finite sets. We introduce the following groups
$$
G_A := \mathbb{F}_2[\mathbb{F}_2^A] \rtimes \mathbb{F}_2^A, \quad \quad G_B := \mathbb{F}_2[\mathbb{F}_2^B] \rtimes \mathbb{F}_2^B.
$$
Here $(\mathbb{F}_2^A)$ acts on $\mathbb{F}_2[\mathbb{F}_2^A]$ through the left regular action, and similarly for $B$.

We have a natural isomorphism
$$
G_A \times G_B \cong (\mathbb{F}_2[\mathbb{F}_2^A] \times \mathbb{F}_2[\mathbb{F}_2^B]) \rtimes (\mathbb{F}_2^A \times \mathbb{F}_2^B),
$$
where $\mathbb{F}_2^A$ is acting trivially on $\mathbb{F}_2[\mathbb{F}_2^B]$ and $\mathbb{F}_2^B$ is acting trivially on $\mathbb{F}_2[\mathbb{F}_2^A]$. The coordinate projections on respectively $\mathbb{F}_2[\mathbb{F}_2^A]$ and $\mathbb{F}_2[\mathbb{F}_2^B]$ give two elements
$$
\psi_A \in Z^1(G_A \times G_B, \mathbb{F}_2[\mathbb{F}_2^A]), \quad \quad \psi_B \in Z^1(G_A \times G_B, \mathbb{F}_2[\mathbb{F}_2^B]).
$$
Write $e_i$ and $e_j$ for respectively the standard basis vectors of $\mathbb{F}_2^A$ and $\mathbb{F}_2^B$. We fix the monomial basis $\{t_V : V \subseteq A\}$ and $\{t_W : W \subseteq B\}$ of respectively $\mathbb{F}_2[\mathbb{F}_2^A]$ and $\mathbb{F}_2[\mathbb{F}_2^B]$, where 
$$
t_V := \prod_{i \in V} (1 + e_i), \quad \quad t_W := \prod_{j \in W} (1 + e_j).
$$
Using the shorthand $t_i$ for $t_{\{i\}}$ we can then naturally view
$$
\mathbb{F}_2[\mathbb{F}_2^A] \cong \frac{\mathbb{F}_2[\{t_i : i \in A\}]}{\langle \{t_i^2 : i \in A\} \rangle}, \quad \quad \mathbb{F}_2[\mathbb{F}_2^B] \cong \frac{\mathbb{F}_2[\{t_j : j \in B\}]}{\langle \{t_j^2 : j \in B\} \rangle}
$$
as rings of truncated polynomials, where $\langle \cdot \rangle$ denotes the smallest \emph{ideal} generated by the given elements. We have a natural multiplication pairing, yielding an identification
$$
\mathbb{F}_2[\mathbb{F}_2^A] \otimes_{\mathbb{F}_2} \mathbb{F}_2[\mathbb{F}_2^B] \cong_{\text{ring}} \mathbb{F}_2[\mathbb{F}_2^{A \cup B}],
$$
given on monomials by the assignment $t_V \otimes t_W \mapsto t_{V \cup W}$. This identification is $\mathbb{F}_2^A\times \mathbb{F}_2^B$-equivariant, where the action on the right-hand side is the left regular action under the natural identification of groups
$$
\mathbb{F}_2^A \times \mathbb{F}_2^B \simeq \mathbb{F}_2^{A\cup B}.
$$
Through this bilinear equivariant pairing we can cup $\psi_A$ and $\psi_B$ in order to get a natural $2$-cocycle
$$
\theta_{(A, B)} := \psi_A \cup \psi_B \in Z^2(G_A \times G_B, \mathbb{F}_2[\mathbb{F}_2^{A \cup B}]). 
$$
Recall from Section \ref{section: expansion maps} that $\psi_A$ and $\psi_B$ are given in the monomial coordinates as
$$
\psi_A(g) = \sum_{V \subseteq A}\phi_V(g) \cdot t_V, \quad \quad \psi_B(g) = \sum_{W \subseteq B}\phi_{W}(g) \cdot t_W,
$$
where the functions $\phi_V$ and $\phi_W$ satisfy the usual recurrence of an expansion map, namely
$$
\mathrm{d}\phi_V(g_1,g_2) = \sum_{\emptyset \neq V' \subseteq V} \chi_{V'}(g_1) \phi_{V - V'}(g_2), \quad \mathrm{d}\phi_W(g_1,g_2) = \sum_{\emptyset \neq W' \subseteq W} \chi_{W'}(g_1) \phi_{W - W'}(g_2).
$$
Unraveling the definition of the cup product we get that
\begin{align*}
\theta_{(A, B)}(g_1, g_2) &= (\psi_A \cup \psi_B)(g_1, g_2) \\
&= \left(\sum_{V \subseteq A} \phi_V(g_1) \cdot t_V\right) \cdot \left(\sum_{W_1 \subseteq B} \chi_{W_1}(g_1) \cdot t_{W_1} \cdot \sum_{W_2 \subseteq B} \phi_{W_2}(g_2) \cdot t_{W_2}\right) \\
&= \sum_{V \subseteq A, W \subseteq B} \sum_{W_1 \sqcup W_2 = W} \phi_V(g_1) \cdot \chi_{W_1}(g_1) \cdot \phi_{W_2}(g_2) \cdot t_{V \cup W} \\
&= \sum_{V \subseteq A \cup B} \phi_{V \cap A}(g_1) \cdot \sum_{U \subseteq V \cap B} \chi_U(g_1) \cdot \phi_{V \cap B - U}(g_2) \cdot t_V,
\end{align*}
where the last equation is obtained by a change of variables. This cocycle corresponds to a group extension
$$
1 \to \mathbb{F}_2[\mathbb{F}_2^{A \cup B}] \to \hat{G}_{(A, B)} \to G_A \times G_B \to 1,
$$
where $\hat{G}_{(A, B)}$ is the group given by $(\mathbb{F}_2[\mathbb{F}_2^{A \cup B}] \times G_A \times G_B, \ast_{\theta_{(A, B)}})$, with the usual general rule of a group constructed by a cocycle on the product set, through the formula
$$
(h_1, g_1) \ast_{\theta_{(A, B)}}(h_2, g_2) = (h_1 + g_1 \cdot h_2 + \theta_{(A, B)}(g_1, g_2), g_1g_2). 
$$
Having defined $\hat{G}_{(A, B)}$, we next unravel the data of a homomorphism into $\hat{G}_{(A, B)}$, still using monomial coordinates. 

Observe that subsets $V$ of $A \cup B$ are in bijection with pairs of subsets $(V_1, V_2)$ of, respectively, $A$ and $B$. We will denote the projection functions on $t_{(V_1, V_2)}$ as $\hat{\phi}_{(V_1, V_2)}$. One has for these $\mathbb{F}_2$-valued cochains, the following recursive equation
\begin{multline*}
\mathrm{d} \hat{\phi}_{(V_1, V_2)}(g_1, g_2) = \left(\sum_{\substack{(\emptyset, \emptyset) \neq (U_1, U_2) \\ U_1 \subseteq V_1, U_2 \subseteq V_2}} \chi_{U_1 \cup U_2}(g_1) \cdot \hat{\phi}_{(V_1 - U_1, V_2 - U_2)}(g_2) \right) + \\
\phi_{V_1}(g_1) \cdot \left(\sum_{U \subseteq V_2} \chi_{U}(g_1) \phi_{V_2 - U}(g_2) \right).
\end{multline*}

\subsection{\texorpdfstring{Basic group theory of $\hat{G}_{(A, B)}$}{Basic group theory of G(A, B)}}
Let $G$ be a finite $2$-group. We denote by $\Phi(G)$ the Frattini subgroup.

\begin{proposition} 
\label{prop: iso mod frattini}
Let $A$ and $B$ be two disjoint finite sets. Then the projection map $\hat{G}_{(A, B)} \to G_A \times G_B$ induces an isomorphism
$$
\frac{\hat{G}_{(A, B)}}{\Phi(\hat{G}_{(A, B)})} \to \frac{G_A \times G_B}{\Phi(G_A \times G_B)} \cong_{\textup{ab.gr.}} \mathbb{F}_2^{|A| + |B| + 2}. 
$$
Furthermore, we have that $[\hat{G}_{(A, B)}, \hat{G}_{(A, B)}] = \Phi(\hat{G}_{(A, B)})$. 
\end{proposition}

\begin{proof}
For subsets $U \subseteq A$ and $V \subseteq B$, call $g_U$ and $g_V$ the elements of $\hat{G}_{(A, B)}$ given by extending respectively $t_U$ and $t_V$ to $\hat{G}_{(A, B)}$ by taking all other coordinates to be $0$. Then it is an immediate computation that 
$$
[g_U, g_V] = t_{(U, V)} \in \mathbb{F}_2[\mathbb{F}_2^{A \cup B}],
$$
viewed as an element of $\hat{G}_{(A, B)}$ by also setting all other coordinates to be zero. In this way we see that $[\hat{G}_{(A, B)}, \hat{G}_{(A, B)}]$ contains 
$$
\ker(\hat{G}_{(A, B)} \to G_A \times G_B).
$$
This shows that the two groups have the same abelianization, which immediately gives the desired conclusion since the abelianizations of $G_A$ and $G_B$ are given by $\frac{\mathbb{F}_2[\mathbb{F}_2^A]}{I_{\mathbb{F}_2^A} \cdot \mathbb{F}_2[\mathbb{F}_2^A]} \times \mathbb{F}_2^A$ and $\frac{\mathbb{F}_2[\mathbb{F}_2^B]}{I_{\mathbb{F}_2^B} \cdot \mathbb{F}_2[\mathbb{F}_2^B]} \times \mathbb{F}_2^B$. We thus infer the claimed equality between Frattini and commutator subgroups.
\end{proof}

Our main goal will be to gain sufficient control on how performing nested commutators works in the group $\hat{G}_{(A, B)}$. We begin with the following basic observation, providing a partial splitting for the cocycle $\theta_{(A, B)}$. To state it, we remark that $\hat{G}_{(A, B)}$ may also be viewed set-theoretically as the product
$$
\mathbb{F}_2[\mathbb{F}_2^{A \cup B}] \times \left(\mathbb{F}_2[\mathbb{F}_2^A] \times \mathbb{F}_2[\mathbb{F}_2^B] \right) \times \left(\mathbb{F}_2^A \times \mathbb{F}_2^B \right).
$$

\begin{proposition} 
\label{prop: first two lifts}
The cocycle $\theta_{(A, B)}$ becomes the zero function when restricted to $\{0\} \times \mathbb{F}_2[\mathbb{F}_2^A] \times \{0\} \times \mathbb{F}_2^A \times \{0\}$, and similarly when restricted to $\{0\} \times \{0\} \times \mathbb{F}_2[\mathbb{F}_2^B] \times \{0\} \times \mathbb{F}_2^B$. More precisely, the subsets 
$$
H_A := \{0\} \times \mathbb{F}_2[\mathbb{F}_2^A] \times \{0\} \times \mathbb{F}_2^A \times \{0\}, \quad \quad H_B := \{0\} \times \{0\} \times \mathbb{F}_2[\mathbb{F}_2^B] \times \{0\} \times \mathbb{F}_2^B
$$
are actually subgroups of $\hat{G}_{(A, B)}$ projecting isomorphically to $G_A$ and $G_B$ respectively. 
\end{proposition}

\begin{proof}
A direct verification with the cocycle defining the group shows that when evaluated in these subsets it vanishes (even termwise), yielding immediately the desired conclusion from the way the group law is defined. 
\end{proof}

Our next proposition provides a third interesting subgroup of $\hat{G}_{(A, B)}$. 

\begin{proposition} 
\label{prop: third lift}
The subset $H_{(A, B)}$ of elements of $\hat{G}_{(A, B)}$ having $0$-coordinate in $\mathbb{F}_2[\mathbb{F}_2^A]$ and $\mathbb{F}_2[\mathbb{F}_2^B]$ is a subgroup and the natural bijection with $G_{A \cup B}$ is a group isomorphism. 
\end{proposition}

\begin{proof}
From the group law of $\hat{G}_{(A, B)}$ defined by the cocycle $\theta_{(A, B)}$, we see that this is indeed a subgroup. Furthermore, we see that when we restrict the cochains $\hat{\phi}_{(V_1, V_2)}$ to $H_{(A, B)}$ they satisfy precisely the recursive equation of an expansion map given that the second piece of the recursion
$$
\phi_{V_1}(g_1) \cdot \left(\sum_{U \subseteq V_2} \chi_U(g_1) \phi_{V_2 - U}(g_2) \right)
$$
consistently vanishes on $H_{(A,B)}$ by definition. This gives the desired isomorphism. 
\end{proof}

The last two propositions above yield already a vast amount of control for computing nested commutators. Our next proposition will be the final ingredient. 

\begin{proposition} 
\label{prop: mutually centralizing}
The subgroup $E_A := \{0\} \times \{0\} \times \{0\} \times \mathbb{F}_2^A \times \{0\}$ centralizes $H_B$ and the subgroup
$E_B := \{0\} \times \{0\} \times \{0\} \times \{0\} \times \mathbb{F}_2^B$ centralizes $H_A$.
\end{proposition}

\begin{proof}
Let $g_1$ be in $E_A$ and $g_2$ be in $H_B$. Observe that for all $V \subseteq A$, we have that $\phi_V(g_1) = 0 = \phi_V(g_2)$. Therefore it follows from the formula for $\theta_{(A, B)}$ that $\theta_{(A, B)}(g_1, g_2) = 0 = \theta_{(A, B)}(g_2, g_1)$. This gives that $E_A$ centralizes $H_B$. Swapping the roles of $A$ and $B$, the same argument shows that $E_B$ centralizes $H_A$. 
\end{proof}

We next define a convenient minimal set of generators for $\hat{G}_{(A, B)}$. For $i \in A$ (and likewise for $j \in B$), we let $\sigma_i$ be the unique element of $\FF_2^A$ that projects to $1$ precisely on the $i$th coordinate. We next define $\sigma_{\text{pt}_A} := 1 \in \mathbb{F}_2[\mathbb{F}_2^A]$ (in the monomial coordinates, this is exactly $t_\emptyset$) and similarly for $\sigma_{\text{pt}_B}$. We denote by 
$$
X := \{\sigma_i : i \in A\} \cup \{\sigma_j : j \in B\} \cup \{\sigma_{\text{pt}_A}, \sigma_{\text{pt}_B}\}.
$$
Thanks to Proposition \ref{prop: iso mod frattini} we see that $X$ is a minimal set of generators for $\hat{G}_{(A, B)}$. We now compute a general nested commutator evaluated on this minimal set of generators. 

\begin{proposition} 
\label{prop: computing nested commutators}
$(a)$ Let $s \geq 2$ be an integer. Let $(x_1, \ldots, x_s) \in X^s$. Then 
$$
[x_1, [x_2, [\cdots, [x_{s - 1}, x_s] \cdots]]]
$$
does not vanish if and only if the following conditions simultaneously hold:
\begin{itemize}
    \item The sequence $(x_1, \ldots, x_s)$ has no repetitions.
    \item We have $\sigma_{\textup{pt}_A} \in \{x_{s - 1}, x_s\}$ and the largest index $h$ with $x_h \in \{\sigma_j : j \in B\} \cup \{\sigma_{\textup{pt}_B}\}$ satisfies $x_h = \sigma_{\textup{pt}_B}$ (this is to be interpreted as vacuously true in case no such index exists), or we have $\sigma_{\textup{pt}_B} \in \{x_{s - 1}, x_s\}$ and the largest index $h$ with $x_h \in \{\sigma_i : i \in A\} \cup \{\sigma_{\textup{pt}_A}\}$ satisfies $x_h = \sigma_{\textup{pt}_A}$.
\end{itemize}
$(b)$ In case the above two conditions hold, we have that 
$$
[x_1, [x_2, [\cdots, [x_{s - 1}, x_s] \cdots]]] =: c
$$
can be computed by the following rules: 
\begin{itemize}
    \item If $\sigma_{\textup{pt}_A} \in \{x_{s - 1}, x_s\}$ and there is no index $h$ with $x_h = \sigma_{\textup{pt}_B}$, then
    $$
    c = (0, t_R, 0, 0, 0),
    $$
    where $R$ is the subset of $i$ in $A$ such that $x_j = \sigma_i$ for a necessarily unique $j$ in $[s]$.
    \item If $\sigma_{\textup{pt}_B} \in \{x_{s - 1}, x_s\}$ and there is no index $h$ with $x_h = \sigma_{\textup{pt}_A}$, then
    $$
    c = (0, 0, t_R, 0, 0),
    $$
    where $R$ is the subset of $i$ in $B$ such that $x_j = \sigma_i$ for a necessarily unique $j$ in $[s]$.
    \item If $\sigma_{\textup{pt}_A}$ and $\sigma_{\textup{pt}_B}$ both occur, then
    $$
    c = (t_{(R_1, R_2)}, 0, 0, 0, 0),
    $$
    where $R_1 \subseteq A$ and $R_2 \subseteq B$ are such that $R_1 \cup R_2$ contains an index $i$ if and only if $\sigma_i$ occurs in the sequence $(x_1, \ldots, x_s)$.
\end{itemize}
\end{proposition}

\begin{proof}
By Proposition \ref{prop: first two lifts}, the commutator vanishes if $x_{s - 1}, x_s \in E_A$. Similarly, the commutator vanishes if $x_{s - 1}, x_s \in E_B$. Also, by Proposition \ref{prop: mutually centralizing}, the commutator vanishes if $|E_A \cap \{x_{s - 1}, x_s\}| = |E_B \cap \{x_{s - 1}, x_s\}| = 1$ or if $\{x_{s - 1}, x_s\}$ consists of $\sigma_{\text{pt}_A}$ and an element of $E_B$, or $\sigma_{\text{pt}_B}$ and an element of $E_A$. Finally, the commutator clearly vanishes if $x_{s - 1} = x_s$.

Hence, if the commutator does not vanish, the only possibilities left for $\{x_{s - 1}, x_s\}$ are that:
\begin{itemize}
    \item we take $\sigma_{\text{pt}_A}$ and an element of $E_A$,
    \item we take $\sigma_{\text{pt}_B}$ and an element of $E_B$,
    \item we take $\{\sigma_{\text{pt}_A}, \sigma_{\text{pt}_B}\}$.
\end{itemize}
We begin by addressing the first case; so suppose $\sigma_{\text{pt}_A} \in \{x_{s - 1}, x_s\}$ and the other element is in $E_A$. Therefore, by Proposition \ref{prop: first two lifts}, we must be in the lift $H_A$. Now, running through the variables in the commutator from right to left, we see by Proposition \ref{prop: first two lifts} that we will remain in $H_A$ for as long as we use variables in $H_A$. This means that if the first index, from right to left, that is not in $H_A$ happens not to be $\sigma_{\text{pt}_B}$, then the commutator vanishes in view of Proposition \ref{prop: mutually centralizing}. We conclude that the first index that we use, from right to left, that is not in $H_A$, must be $\sigma_{\text{pt}_B}$. 

Let $R \subseteq A$ be the set of indices $i$ such that $\sigma_i$ occurs to the right of $\sigma_{\text{pt}_B}$. Then we have that up until $\sigma_{\text{pt}_B}$ the element of $G_A$ that we get is simply the monomial $\prod_{i \in R} t_i$, which forces in particular no repetitions up to this point. When we take the commutator with $\sigma_{\text{pt}_B}$, we get the element
$$
(t_R, 0, 0, 0, 0).
$$
From now on, we can simply use any remaining set of variables other than $\sigma_{\text{pt}_A},\sigma_{\text{pt}_B}$: indeed elements in $\mathbb{F}_2[\mathbb{F}_2^{A \cup B}] \times \{0\} \times \{0\} \times \{0\} \times \{0\}$ commute with $\sigma_{\text{pt}_A},\sigma_{\text{pt}_B}$. From Proposition \ref{prop: third lift} we get precisely the element $(t_{R \cup R'}, 0, 0, 0, 0)$, where $R'$ is the set of remaining variables that we use. 

The second case is entirely symmetrical and we now deal with the third case. Observe that $[\sigma_{\text{pt}_A}, \sigma_{\text{pt}_B}] \in \mathbb{F}_2[\mathbb{F}_2^{A \cup B}] \times \{0\} \times \{0\} \times \{0\} \times \{0\}$. More precisely, we have
$$
(t_{\emptyset}, 0, 0, 0, 0) = [\sigma_{\text{pt}_A}, \sigma_{\text{pt}_B}].
$$
From this point onward, we can no longer use either $\sigma_{\text{pt}_A}$ or $\sigma_{\text{pt}_B}$, as observed already above, and then we can directly appeal to Proposition \ref{prop: third lift} and see that our only constraint is to use pairwise distinct variables. The resulting commutator is then
$$
(t_R, 0, 0, 0, 0),
$$
where $R \subseteq A \cup B$ is the total set of non-pointer variables used, precisely as claimed. 
\end{proof}

An immediate corollary is as follows. 

\begin{proposition} 
The group $\hat{G}_{(A, B)}$ has nilpotency class equal to $|A| + |B| + 2$.
\end{proposition}

\begin{proof}
Proposition \ref{prop: computing nested commutators} implies that in the group $\hat{G}_{(A, B)}$ there is a commutator of length $|A| + |B| + 2$ that does not vanish (e.g.~picking $x_{|A| + |B| + 1} = \sigma_{\text{pt}_A}$, $x_{|A| + |B| + 2} = \sigma_{\text{pt}_B}$ and picking any other choice without repetitions for the remaining variables), while all the commutators of length exceeding $|A| + |B| + 2$ vanish since there is a repetition by the pigeonhole principle. This means that the nilpotency class is precisely $|A| + |B| + 2$.
\end{proof}

\subsection{\texorpdfstring{$\hat{\phi}$-maps for a profinite group}{Hat phi maps for a profinite group}} \label{section: hat phi for profinite groups}
Let now $\mathcal{G}$ be any profinite group. Suppose that $A$ and $B$ are two disjoint finite sets of continuous $\mathbb{F}_2$-valued characters of $\mathcal{G}$ such that $A \cup B$ is a linearly independent set. Let $\chi_1$ and $\chi_2$ be two further characters that together with $A \cup B$ form a linearly independent set. Let $\psi_A$ and $\psi_B$ be two expansion maps for $\mathcal{G}$ with pointers respectively $\chi_1$ and $\chi_2$. 

We say that $\psi_{(A, B)}$ is a $\hat{\phi}$-expansion map pointed at $(\psi_A, \psi_B)$ if $\psi_{(A, B)}$ is a continuous group homomorphism 
$$
\psi_{(A, B)}\colon \mathcal{G} \to \hat{G}_{(A, B)}
$$
that composed with the natural projection $\hat{G}_{(A, B)} \twoheadrightarrow G_A \times G_B$ coincides with the pair $(\psi_A, \psi_B)$. Observe that to give such a function from $\mathcal{G}$ to $\hat{G}_{(A, B)}$ is tantamount to giving a system of continuous $\mathbb{F}_2$-valued cochains 
$$
\{\hat{\phi}_{(V_1, V_2)} : V_1 \subseteq A, V_2 \subseteq B\}
$$
giving the monomial coordinates of $\psi_{(A, B)}$ in $\mathbb{F}_2[\mathbb{F}_2^{A \cup B}]$. 

\begin{proposition} 
\label{prop: recursive equation for hat phi}
A function $\psi := \left( (\hat{\phi}_{(V_1, V_2)})_{(V_1 \subseteq A, V_2 \subseteq B)}, \psi_A, \psi_B \right)$ induces a surjective continuous group homomorphism
$$
\psi\colon \mathcal{G} \to \hat{G}_{(A, B)}
$$
if and only if each of the $\mathbb{F}_2$-valued cochains $\hat{\phi}_{(V_1, V_2)}$ is continuous and they satisfy the recursive equation
\begin{multline}
\label{eHatRecursion}
\mathrm{d} \hat{\phi}_{(V_1, V_2)}(g_1, g_2) = \left(\sum_{\substack{(\emptyset, \emptyset) \neq (U_1, U_2) \\ U_1 \subseteq V_1, U_2 \subseteq V_2}} \chi_{U_1 \cup U_2}(g_1) \cdot \hat{\phi}_{(V_1-U_1,V_2-U_2)}(g_2) \right) + \\
\phi_{V_1}(g_1) \cdot \left(\sum_{U \subseteq V_2}\chi_{U}(g_1)\phi_{V_2-U}(g_2) \right)
\end{multline}
for all $g_1, g_2 \in \mathcal{G}$. 
\end{proposition}

\begin{proof}
Clearly, $\psi$ is continuous if and only if all of its coordinates $\hat{\phi}_{(V_1,V_2)}$ are continuous functions, hence settling the continuity part of the if and only if.

Denote by $\psi'$ the projection of $\psi$ on $\mathbb{F}_2[\mathbb{F}_2^{A \cup B}] \times \{0\} \times \{0\} \times \{0\} \times \{0\}$. Then, by definition of the group law for $\hat{G}_{(A, B)}$, we have that $\psi = (\psi' ,\psi_A, \psi_B)$ is a group homomorphism if and only if
$$
\mathrm{d}\psi' = \psi_A \cup \psi_B.
$$
Computing explicitly the cup product and the differential in the monomial coordinates, as executed at the beginning of this section, yields precisely the conclusion that $\psi$ is a homomorphism if and only if the $\hat{\phi}_{(V_1,V_2)}$ satisfy that recursive system of equations. 

Finally, as soon as $\psi$ is a homomorphism, it will be automatically surjective. Indeed, Proposition \ref{prop: iso mod frattini} tells us that $\text{Im}(\psi)$ will automatically generate modulo the Frattini (in view of our running assumption that $A \cup B \cup \{\chi_1,\chi_2\}$ is a linearly independent set over $\mathbb{F}_2$), and therefore it will generate, since $\hat{G}_{(A,B)}$ is a $2$-group. 
\end{proof}

\begin{remark} 
\label{remark: group of def of an exp map}
The group of definition of an expansion map $\psi_A$ equals the group of definition of its $\mathbb{F}_2$-valued top cochain $\phi_A(\psi_A)$. This justifies a slight abuse of notation, one that we will systematically adopt in later sections, by specifying only the top cochains of the two underlying expansion maps.
\end{remark}

Our second proposition highlights the role of $\hat{\phi}$-expansion maps in trivializing cup products of usual expansion maps. 

\begin{proposition} 
\label{prop: d hat phi equals cup of phi's}
Let $\left( (\hat{\phi}_{(V_1, V_2)})_{(V_1 \subseteq A, V_2 \subseteq B)}, \psi_A, \psi_B \right)$ be a $\hat{\phi}$-expansion map. Let 
$$
\mathcal{H} := \ker\left( \mathcal{G} \xrightarrow{\bigoplus_{\chi \in A \cup B} \chi} \FF_2^{A \cup B} \right).
$$
Then for all $h_1, h_2 \in \mathcal{H}$ we have
$$
\mathrm{d}\hat{\phi}_{(A, B)}(h_1, h_2) = \phi_A(h_1) \cdot \phi_B(h_2).
$$
\end{proposition}

\begin{proof}
We consider the recursive equation \eqref{eHatRecursion} of Proposition \ref{prop: recursive equation for hat phi} applied to $V_1 = A$ and $V_2 = B$. Each summand in the first sum has a term of the shape $\chi_{U_1 \cup U_2}(h_1)$; this term is zero when restricted to $\mathcal{H}$ since $U_1$ and $U_2$ are not both empty. Among the terms in the second summand, we see that they all contain $\chi_U(h_1)$, which vanishes whenever $U$ is not empty. We conclude that the only remaining term corresponds to $U = \emptyset$, which contributes exactly $\phi_A(h_1) \cdot \phi_B(h_2)$.
\end{proof}

We have the following basic fact. 

\begin{proposition} 
\label{prop: at most one normalized map}
Let $\mathcal{G}$ be a pro-$2$ group. Let $\mathfrak{G}_0$ be a minimal set of topological generators for $\mathcal{G}$. Let $A,B, \{\chi_1,\chi_2\}$ be sets of continuous $\mathbb{F}_2$-valued characters of $\mathcal{G}$ forming altogether a linearly independent set. Then there exists at most one pair of expansion maps $\psi_A$ and $\psi_B$ with support set respectively $A \cup \{\chi_1\}, B \cup \{\chi_2\}$ and such that 
$$
\phi_U(\psi_A)(\mathfrak{G}_0) = \phi_V(\psi_B)(\mathfrak{G}_0) = \{0\},
$$
for all non-empty subsets $U \subseteq A$ and all non-empty subsets $V \subseteq B$. 

Furthermore, there exists at most one $\hat{\phi}$-expansion map $\psi_{(A, B)}$ pointed at $(\psi_A, \psi_B)$ such that
$$
\hat{\phi}_{(U, V)}(\psi_{(A, B)})(\mathfrak{G}_0) = \{0\}
$$
for all $U \subseteq A$ and all $V \subseteq B$.
\end{proposition}

\begin{proof}
A continuous homomorphism from $\mathcal{G}$ to a finite discrete group is entirely determined by its values on a set of topological generators. Since we have prescribed $\phi_U(\psi_A)(\mathfrak{G}_0) = \phi_V(\psi_B)(\mathfrak{G}_0) = \{0\}$ for all non-empty subsets $U \subseteq A$ and all non-empty subsets $V \subseteq B$, this determines the image of $\mathfrak{G}_0$ inside $G_A$ and $G_B$. Since we have also prescribed the values of $\hat{\phi}_{(U, V)}(\psi_{(A, B)})$ on $\mathfrak{G}_0$ for all $U \subseteq A$ and all $V \subseteq B$, we obtain the proposition.
\end{proof}

We will say that a $\hat{\phi}$-expansion map satisfying the properties in Proposition \ref{prop: at most one normalized map} is $\mathfrak{G}_0$-normalized. 

\subsection{\texorpdfstring{$\hat{\phi}$-maps for $\mathbb{F}_q(T)$}{Hat phi maps for function fields}}
Let $q \equiv 1 \bmod 8$ be a prime power. Let $A$ and $B$ be two finite disjoint subsets of $\Gamma_{\mathbb{F}_2}(\mathbb{F}_q(T))$ such that $A \cup B$ is a linearly independent set. Let $\chi_1$ and $\chi_2$ be two characters in $\Gamma_{\mathbb{F}_2}(\mathbb{F}_q(T))$ such that $A \cup B \cup \{\chi_1, \chi_2\}$ is a linearly independent set of size $|A| + |B| + 2$.

Let $\psi_A$ and $\psi_B$ be two expansion maps for $G_{\mathbb{F}_q(T)}$ with pointers respectively $\chi_1$ and $\chi_2$. By taking $\mathcal{G} := G_{\mathbb{F}_q(T)}(2)$ in the context of Subsection \ref{section: hat phi for profinite groups}, we can talk about $\hat{\phi}$-expansion maps pointed at $(\psi_A, \psi_B)$, which we will denote by $\psi_{(A, B)}$.

Our next goal is to define a \emph{normalized} $\hat{\phi}$-expansion map. To that end, let $s, r \in \Z_{\geq 0}$ and let $S := \{a_1, \ldots, a_s\}$ and $R := \{b_1, \ldots, b_r\}$ be two subsets of $\frac{\mathbb{F}_q(T)^\ast}{\mathbb{F}_q(T)^{\ast 2}}$, let $a_{s + 1}, b_{r + 1}$ be elements of $\frac{\mathbb{F}_q(T)^\ast}{\mathbb{F}_q(T)^{\ast 2}}$ such that $S \cup R \cup \{a_{s + 1}\} \cup \{b_{r+1}\}$ is linearly independent of size $s + r + 2$. Suppose that both expansion maps $\phi_{S; a_{s + 1}}(\mathfrak{G})$ and $\phi_{R; b_{r + 1}}(\mathfrak{G})$ exist. Then we say that 
$$
\hat{\phi}_{S, R; a_{s + 1}, b_{r + 1}}(\mathfrak{G})
$$
\emph{exists} if there is a $\hat{\phi}$-expansion map pointed at $(\phi_{S; a_{s + 1}}(\mathfrak{G}), \phi_{R; b_{r + 1}}(\mathfrak{G}))$, which is $\mathfrak{G}$-normalized. In view of Proposition \ref{prop: at most one normalized map}, there exists at most one such normalized set of cochains
$$
(\hat{\phi}_{\{a_i : i \in U\}, \{b_j : j \in V\}; a_{s + 1}, b_{r + 1}}(\mathfrak{G}))_{U \subseteq [s], V \subseteq [r]}.
$$
We will denote by
$$
\hat{\psi}_{\{a_i : i \in [s]\}, \{b_j : j \in [r]\}; a_{s + 1}, b_{r + 1}}(\mathfrak{G}) \colon \mathcal{G} \to \hat{G}_{(\{a_i : i \in [s]\},\{b_j : j \in [r]\})}
$$
the underlying group homomorphism.

In what follows, the statement that a $\hat{\phi}$-normalized map exists will also include the existence of the two underlying normalized $\phi$-maps. In the next proposition we will show that normalized $\hat{\phi}$-expansion maps behave additively.

\begin{proposition} 
\label{prop: hat phi are additive}
Let $s, r \in \Z_{\geq 0}$. For $h \in \{1, 2, 3\}$ let 
$$
((a_{1, h}, \ldots, a_{s + 1, h}), (b_{1, h}, \ldots, b_{r + 1, h})) \in \left(\frac{\mathbb{F}_q(T)^\ast}{\mathbb{F}_q(T)^{\ast 2}}\right)^{[s + 1]}\times \left(\frac{\mathbb{F}_q(T)^\ast}{\mathbb{F}_q(T)^{\ast 2}}\right)^{[r + 1]}
$$ 
be a vector whose underlying set has size $r + s +2$ and is linearly independent for each $h$. Suppose furthermore that there exists $n_0$ in $[r + 1]$ or in $[s + 1]$, such that the vectors are identical outside of that entry, while at $n_0$ the three corresponding characters sum to $0$. 

Suppose furthermore that $\hat{\phi}_{\{a_{i, h} : i \in [s]\}, \{b_{j, h} : j \in [r]\}; a_{s + 1, h}, b_{r + 1, h}}(\mathfrak{G})$ exists for $h \in \{1, 2\}$.

Then $\hat{\phi}_{\{a_{i, 3} : i \in [s]\}, \{b_{j, 3} : j \in [r]\}; a_{s + 1, 3}, b_{r + 1, 3}}(\mathfrak{G})$ also exists and furthermore
$$
\sum_{h = 1}^3 \hat{\phi}_{\{a_{i, h} : i \in [s]\}, \{b_{j, h} : j \in [r]\}; a_{s + 1, h}, b_{r + 1, h}}(\mathfrak{G}) = 0. 
$$
\end{proposition}

\begin{proof}
This follows from Proposition \ref{pExpansionAdditive} and its proof strategy.
\end{proof}

We now give some control on the ramification of fields of definition of normalized $\hat{\phi}$-expansion maps. 

\begin{proposition} 
\label{prop: field of definition of hat phi}
Let $s, r \in \Z_{\geq 0}$. Let 
$$
((a_1, \ldots, a_{s + 1}),(b_1, \ldots, b_{r + 1})) \in \left(\frac{\mathbb{F}_q(T)^\ast}{\mathbb{F}_q(T)^{\ast 2}}\right)^{[s + 1]} \times \left(\frac{\mathbb{F}_q(T)^\ast}{\mathbb{F}_q(T)^{\ast 2}}\right)^{[r + 1]}
$$ 
be a vector consisting of classes of squarefree monic polynomials of positive degree and pairwise coprime. Set $S := \{a_1, \ldots, a_s\}$ and $R := \{b_1, \dots, b_r\}$. Suppose that $\hat{\phi}_{S, R; a_{s + 1}, b_{r + 1}}(\mathfrak{G})$ exists. Then the following holds.
\begin{enumerate}
\item[$(a)$] The field extension $L(\hat{\psi}_{S, R; a_{s + 1}, b_{r + 1}}(\mathfrak{G}))/\mathbb{F}_q(T)(\{\chi_{a_i} : i \in [s + 1]\} \cup \{\chi_{b_j} : j \in [r + 1]\})$ is a Galois extension of nilpotency class at most $2$, realizing precisely the $2^{r + s}$ classes given by
$$
\phi_{\{a_i : i \in U\}; a_{s + 1}}(\mathfrak{G}) \cup \phi_{\{b_j : j \in V\}; b_{r + 1}}(\mathfrak{G})
$$
with trivializing cochain provided exactly by
$$
\hat{\phi}_{\{a_i : i \in U\}, \{b_j : j \in V\}; a_{s + 1}, b_{r + 1}}(\mathfrak{G}).
$$
\item[$(b)$] The field extension $L(\hat{\psi}_{S, R; a_{s + 1}, b_{r + 1}}(\mathfrak{G}))/\mathbb{F}_q(T)(\{\chi_{a_i} : i \in [s + 1]\} \cup \{\chi_{b_j} : j \in [r + 1]\})$ is unramified above all places of $\mathbb{F}_q(T)(\{\chi_{a_i} : i \in [s + 1]\} \cup \{\chi_{b_j} : j \in [r + 1]\})$ lying above $\Omega_{\mathbb{F}_q(T)}^{\textup{fin}}$. 
\item[$(c)$] Suppose further that $a_i$ and $b_j$ are of even degree for each $i \in [s + 1]$ and each $j \in [r + 1]$. Also assume that $L(\phi_{\{a_i : i \in [s]\}; a_{s + 1}}(\mathfrak{G}))/\mathbb{F}_q(T)$ splits completely at $\infty$. Then, for each $U \subseteq [s], V \subseteq [r]$, we have that $\hat{\phi}_{\{a_i : i \in U\}, \{b_j : j \in V\}; a_{s + 1}, b_{r + 1}}(\mathfrak{G})$ restricts to an element of $\Gamma_{\mathbb{F}_2}(\mathbb{F}_q(T)_{\infty})$ and lands in the subspace $\mathbb{F}_2 \cdot \chi_T$. 
\end{enumerate}
\end{proposition}

\begin{proof}
We start with $(a)$. Set $K := \mathbb{F}_q(T)(\{\chi_{a_i} : i \in [s + 1]\} \cup \{\chi_{b_j} : j \in [r + 1]\})$. By the recursive equation for expansion maps, $\phi_{\{a_i : i \in U\}; a_{s + 1}}(\mathfrak{G})$ and $\phi_{\{b_j : j \in V\}; b_{r + 1}}(\mathfrak{G})$ restrict to elements of $\Gamma_{\mathbb{F}_2}(K)$. We have by Proposition \ref{prop: d hat phi equals cup of phi's}
$$
\mathrm{d}\hat{\phi}_{U, V; a_{s + 1}, b_{r + 1}}(\mathfrak{G})(\sigma, \tau) = \phi_{\{a_i : i \in U\}; a_{s + 1}}(\mathfrak{G})(\sigma) \cdot \phi_{\{b_j : j \in V\}; b_{r + 1}}(\mathfrak{G})(\tau)
$$
for all $\sigma, \tau \in G_K$, where the right hand side is the cup product restricted to $G_K$. It follows that $L(\hat{\psi}_{S, R; a_{s + 1}, b_{r + 1}}(\mathfrak{G}))/\mathbb{F}_q(T)(\{\chi_{a_i} : i \in [s + 1]\} \cup \{\chi_{b_j} : j \in [r + 1]\})$ is a $2$-step nilpotent extension obtained by trivializing the cup product of quadratic characters, giving part $(a)$ of the statement. 

We now establish part $(b)$. We begin with any prime $v$ of $\Omega_{\mathbb{F}_q(T)}^{\text{fin}}$ such that
$$
v \not \in \text{Ram}(\{\chi_{a_i} : i \in [s + 1]\} \cup \{\chi_{b_j} : j \in [r + 1]\}).
$$
By definition of normalized $\hat{\phi}$-maps, $\sigma_v$ is sent to the vector of $\hat{G}_{(\{a_i : i \in [s]\},\{b_j : j \in [r]\})}$ having all entries equal to $0$. This is the identity element for this group, which shows that the extension $L(\hat{\psi}_{S, R; a_{s + 1}, b_{r + 1}}(\mathfrak{G}))/\mathbb{F}_q(T)$ is unramified at all primes above $v$.

We now address the case of a finite prime $w$ such that
$$
w \in \text{Ram}(\{\chi_{a_i} : i \in [s + 1]\} \cup \{\chi_{b_j} : j \in [r + 1]\}).
$$
By construction, $\sigma_w$ is sent to an element of $E_{\{a_i : i \in [s]\}} \cup E_{\{b_j : j \in [r]\}}$ or to $\sigma_{\text{pt}_A},\sigma_{\text{pt}_B}$. These are all involutions, so that we have that the extension $L(\hat{\psi}_{S, R; a_{s + 1}, b_{r + 1}}(\mathfrak{G}))/\mathbb{F}_q(T)$ has ramification index $2$ at the place $w$. Since $\mathbb{F}_q(T)(\{\chi_{a_i} : i \in [s + 1]\} \cup \{\chi_{b_j} : j \in [r + 1]\})/\mathbb{F}_q(T)$ also ramifies at $w$ with order $2$, it follows that 
$$
L(\hat{\psi}_{S, R; a_{s + 1}, b_{r + 1}}(\mathfrak{G}))/\mathbb{F}_q(T)(\{\chi_{a_i} : i \in [s + 1]\} \cup \{\chi_{b_j} : j \in [r + 1]\})
$$
is unramified at all places above $w$. All in all, we have shown that the extension 
$$
L(\hat{\psi}_{S, R; a_{s + 1}, b_{r + 1}}(\mathfrak{G}))/\mathbb{F}_q(T)(\{\chi_{a_i} : i \in [s + 1]\} \cup \{\chi_{b_j} : j \in [r + 1]\})
$$
is unramified at all places of $\mathbb{F}_q(T)(\{\chi_{a_i} : i \in [s + 1]\} \cup \{\chi_{b_j} : j \in [r + 1]\})$ lying above a prime of $\Omega_{\mathbb{F}_q(T)}^{\text{fin}}$.

We now prove part $(c)$. We claim that the $2$-cochain $\mathrm{d}\hat{\phi}_{\{a_i : i \in U\}, \{b_j : j \in V\}; a_{s + 1}, b_{r + 1}}(\mathfrak{G})(\sigma, \tau)$ restricts to the zero function locally at $\infty$. In fact, we will inspect the various terms occurring in the definition, and we check they vanish one by one. 

This is clear for all the terms having $\chi_{a_i}$ or $\chi_{b_j}$, as $a_i$ and $b_j$ are monic of even degree by assumption. The term $\phi_{\{a_i : i \in U\}; a_{s + 1}}(\mathfrak{G})(\sigma) \cdot \phi_{\{b_j : j \in V\}; b_{r + 1}}(\mathfrak{G})(\tau)$ vanishes, because the assumptions imply that $\phi_{\{a_i : i \in U\}; a_{s + 1}}(\mathfrak{G})(\sigma)$ is equal to zero locally at $\infty$: indeed, we know that $\infty$ splits completely in $L(\phi_{\{a_i : i \in U\}; a_{s + 1}}(\mathfrak{G}))/\mathbb{F}_q(T)$ and we also know that $\phi_{\{a_i : i \in U\}; a_{s + 1}}(\mathfrak{G})$ vanishes at the identity. So this shows that $\hat{\phi}_{\{a_i : i \in U\}, \{b_j : j \in V\}; a_{s + 1}, b_{r + 1}}(\mathfrak{G})$ is a character locally at $\infty$. Furthermore, since it is a normalized $\hat{\phi}$-map, we know that it vanishes at $\Frob_{\infty}(\chi_T)$. The elements of $\Gamma_{\mathbb{F}_2}(\mathbb{F}_q(T)_{\infty})$ vanishing at $\Frob_{\infty}(\chi_T)$ are precisely the subspace $\mathbb{F}_2 \cdot \chi_T$. 
\end{proof}

We next give explicit control of the value of normalized $\hat{\phi}$-maps on nested commutators evaluated in $\mathfrak{G}$. 

\begin{proposition} 
\label{prop: nested commutators in inertia}
Let $s, r \in \Z_{\geq 0}$. Let 
$$
((a_1, \ldots, a_{s + 1}),(b_1, \ldots, b_{r + 1})) \in \left(\frac{\mathbb{F}_q(T)^\ast}{\mathbb{F}_q(T)^{\ast 2}}\right)^{[s + 1]}\times \left(\frac{\mathbb{F}_q(T)^\ast}{\mathbb{F}_q(T)^{\ast 2}}\right)^{[r + 1]}
$$ 
be a vector consisting of classes of squarefree monic polynomials of positive degree and pairwise coprime. Suppose that $\hat{\phi}_{\{a_i : i \in [s]\}, \{b_j : j \in [r]\}; a_{s + 1}, b_{r + 1}}(\mathfrak{G})$ exists.

Pick for each $i \in [s + 1]$ a finite prime $v_i$ of $\mathbb{F}_q(T)$ such that $v_i(a_i)$ is odd, and likewise for each $j \in [r + 1]$ pick a finite prime $w_j$ of $\mathbb{F}_q(T)$ such that $w_j(b_j)$ is odd. Set 
$$
X := \{\sigma_{v_i} : i \in [s]\} \cup \{\sigma_{w_j} : j \in [r]\} \cup \{\sigma_{v_{s + 1}}, \sigma_{w_{r + 1}}\}.
$$
Let now $(x_1, \ldots, x_{s + r + 2}) \in X^{s + r + 2}$. Let $U \subseteq [s]$ and $V \subseteq [r]$ be two subsets. Then
$$
\hat{\phi}_{\{a_i : i \in U\}, \{b_j : j \in V\}; a_{s + 1}, b_{r + 1}}(\mathfrak{G})([x_1, [x_2, [\cdots, [x_{s + r + 1}, x_{s + r + 2}] \cdots]]]) = 1
$$
if and only if the following conditions simultaneously hold:
\begin{itemize}
    \item $U = [s]$ and $V = [r]$.
    \item Each element of $X$ appears precisely once in the list $(x_1, \ldots, x_{s + r + 2})$.
    \item Either $\sigma_{v_{s + 1}}$ is in $\{x_{s + r + 1},x_{s + r + 2}\}$ and for the largest index $i$ satisfying the condition $x_i \in \{\sigma_{w_1}, \ldots, \sigma_{w_{r + 1}}\}$ one has that $x_i = \sigma_{w_{r + 1}}$ or $\sigma_{w_{r + 1}}$ is in $\{x_{s + r + 1}, x_{s + r + 2}\}$ and for the largest index $j$ such that $x_j \in \{\sigma_{v_1}, \ldots, \sigma_{v_{s + 1}}\}$ one has that $x_j = \sigma_{v_{s + 1}}$.
\end{itemize}
\end{proposition}
\begin{proof}
This is a straightforward translation of Proposition \ref{prop: computing nested commutators}.    
\end{proof}

Finally we are able to prove a linear independence result between $\hat{\phi}$-expansion maps and usual $\phi$-maps, which will be crucial for us later. By repeatedly invoking Proposition \ref{prop: nested commutators in inertia}, it would be possible to state a much more precise description of the linear dependencies between $\hat{\phi}$ and $\phi$-maps. However, we shall restrict merely to the one that suffices for our purposes. 

\begin{proposition}
\label{prop: linear independence of hat phi and phi}
Let $s \geq 3$ be an integer. Let $(a_1, \ldots, a_{s + 1}) \in \left(\frac{\mathbb{F}_q(T)^\ast}{\mathbb{F}_q(T)^{\ast 2}}\right)^{[s + 1]}$ consist of classes of squarefree monic polynomials of positive degree and pairwise coprime. Suppose that for each $i \in [s]$ and each pair of disjoint strict subsets $S_1, S_2 \subseteq [s] - \{i\}$, the map 
$$
\hat{\phi}_{\{a_j : j \in S_1\}, \{a_h : h \in S_2\}; a_i, a_{s + 1}}(\mathfrak{G})
$$ 
exists. Assume furthermore that, for every $i \in [s]$, the expansion map
$$
\phi_{\{a_h : h \in [s] - \{i\}\}; a_i}(\mathfrak{G})
$$
exists. Moreover, assume that $\phi_{\{a_j : j \in [s]\}; a_{s + 1}}(\mathfrak{G})$ exists. Set for each $i \in [s]$
$$
\hat{\phi}_i := \phi_{\{a_j : j \in [s]\}; a_{s + 1}}(\mathfrak{G}) + \sum_{\substack{S_1 \sqcup S_2 = [s] - \{i\} \\ S_1, S_2 \subset [s] - \{i\}}} \hat{\phi}_{S_1, S_2; a_i, a_{s + 1}}(\mathfrak{G}), \quad \mathcal{H} := G_{\prod_{i \in [s]} L(\phi_{\{a_h : h \in [s] - \{i\}\}; a_i}(\mathfrak{G}))}.
$$
Set $\mathcal{Q} := \Gal(\FF_q(T)(\{\chi_{a_j} : j \in [s]\})/\FF_q(T))$. Then $\res_{\mathcal{H}} \, \phi_{\{a_j : j \in [s]\}; a_{s + 1}}(\mathfrak{G})$ and $\res_{\mathcal{H}} \, \hat{\phi}_i$ (with $i \in [s]$) are all quadratic characters and satisfy the following independence
$$
\res_{\mathcal{H}} \, \sum_{g \in \mathcal{Q} - \{1\}} \phi_{\{a_j : j \in [s]\}; a_{s + 1}}(\mathfrak{G})^g \not \in \langle \res_{\mathcal{H}} \, \hat{\phi}_1, \dots, \res_{\mathcal{H}} \, \hat{\phi}_s \rangle.
$$
\end{proposition}

\begin{proof}
By the definition of $\mathcal{H}$ and Remark \ref{remark: group of def of an exp map}, every character $\chi_{a_j}$ with $j \in [s]$, and every coordinate
$$
\phi_{\{a_h : h \in S\}; a_i}(\mathfrak{G}), \qquad S \subseteq [s] - \{i\},
$$
vanishes on $\mathcal{H}$. The recursive equation for ordinary expansion maps therefore shows that
$$
\mathrm{d}(\res_{\mathcal{H}} \, \phi_{\{a_j : j \in [s]\}; a_{s + 1}}(\mathfrak{G})) = 0,
$$
and similarly for all of its conjugates. Likewise, in the recursive equation defining $\hat{\phi}_i$, every term vanishes on $\mathcal{H}$, hence $\res_{\mathcal{H}} \, \hat{\phi}_i$ is also a quadratic character.

Pick for each $i \in [s + 1]$ a finite prime $v_i$ of $\mathbb{F}_q(T)$ such that $v_i(a_i)$ is odd. Set 
$$
X := \{\sigma_{v_i} : i \in [s + 1]\}.
$$
For each $T \in \text{Sym}([s])$, we denote by
$$
\text{nested}(T) := [\sigma_{v_{T(1)}}, [\sigma_{v_{T(2)}}, [\cdots, [\sigma_{v_{T(s)}}, \sigma_{v_{s + 1}}] \cdots]]].
$$
We claim that for each $i \in [s]$
$$
\hat{\phi}_i(\text{nested}(T)) = \mathbf{1}_{i \in \{T(1), T(s)\}}.
$$
Indeed, observe that using Proposition \ref{prop: nested commutators in inertia}, we have that
$$
\hat{\phi}_i(\text{nested}(T)) = \phi_{\{a_j : j \in [s]\}; a_{s + 1}}(\mathfrak{G})(\text{nested}(T)) + \sum_{\substack{U_1 \sqcup U_2 = [s] - \{i\} \\ U_1, U_2 \neq \emptyset}} \prod_{u \in U_1} \mathbf{1}_{T^{-1}(u) < T^{-1}(i)}.
$$
This equality follows at once from Proposition \ref{prop: computing nested commutators} and from the standard fact that the expansion map does not vanish on any of these commutators (which can be actually also viewed as a special case of Proposition \ref{prop: computing nested commutators}).

We see that an even number of subsets $U_1$ are detected precisely when $T^{-1}(i) \in \{1, s\}$, which is precisely the claim. 

Now observe that for each $T \in \text{Sym}([s])$, we have that $\text{nested}(T)$ is in $\mathcal{H}$. Suppose now, for the sake of contradiction, that there is a linear dependency of the shape
$$
\res_{\mathcal{H}} \, \sum_{g \in \mathcal{Q} - \{1\}} \phi_{\{a_j : j \in [s]\}; a_{s + 1}}(\mathfrak{G})^g = \res_{\mathcal{H}} \, \sum_{i = 1}^s \lambda_i \cdot \hat{\phi}_i.
$$
Evaluating at $\text{nested}(T)$, we have that
$$
1 = \lambda_{T(1)} + \lambda_{T(s)}
$$
for all $T \in \text{Sym}([s])$. Recalling that, in particular, $s \geq 2$, and since, as $T$ varies, the ordered pair $(T(1),T(s))$ ranges over all pairs of distinct elements of $[s]$, this means that
$$
1 = \lambda_i + \lambda_j,
$$
for all $i \neq j$ in $[s]$. Now, recalling that $s \geq 3$, we have the following equalities in $\mathbb{F}_2$
$$
1 = 1 + 1 + 1 = (\lambda_1 + \lambda_2) + (\lambda_2 + \lambda_3) + (\lambda_1 + \lambda_3) = 0,
$$
which is a contradiction. 
\end{proof}

In our main application, Proposition \ref{prop: linear independence of hat phi and phi} will be responsible for the survival of bilinear terms in all of our character sums. The cutoff $s \geq 3$ stems from the fact that we only use nested commutators of length $s + 1$ in the proof of Proposition \ref{prop: linear independence of hat phi and phi}. It is for this group-theoretic reason that a careful study of the distribution of the spaces $C_4^{(i)}(\chi)$ seems more delicate.

\begin{proposition} 
\label{prop: criterion of existence for hat phi}
Let $s, r \in \Z_{\geq 0}$ be nonnegative integers. Let 
$$
((a_1, \ldots, a_{s + 1}),(b_1, \ldots, b_{r + 1})) \in \left(\frac{\mathbb{F}_q(T)^\ast}{\mathbb{F}_q(T)^{\ast 2}}\right)^{[s + 1]} \times \left(\frac{\mathbb{F}_q(T)^\ast}{\mathbb{F}_q(T)^{\ast 2}}\right)^{[r + 1]},
$$ 
be pairwise coprime positive degree squarefree monic polynomials in $\mathbb{F}_q[T]$. We make the following assumptions:
\begin{itemize}
\item All of the characters $\chi_{a_1}, \ldots, \chi_{a_s}$ and $\chi_{b_1}, \ldots, \chi_{b_r}$ are locally trivial at all places in $\textup{Ram}(\{\chi_{a_1}, \ldots, \chi_{a_{s + 1}}, \chi_{b_1}, \ldots, \chi_{b_{r + 1}}\})$ where they do not ramify.
\item For all subsets $U_1 \subseteq [s], U_2 \subseteq [r]$ such that $|U_1| + |U_2| < s + r$, we have that the map
$$
\hat{\phi}_{\{a_i : i \in U_1\}, \{b_j : j \in U_2\}; a_{s + 1}, b_{r + 1}}(\mathfrak{G})
$$
exists. 
\item We assume that both of the maps $\phi_{\{a_1, \ldots, a_s\}; a_{s + 1}}(\mathfrak{G})$ and $\phi_{\{b_1, \ldots, b_r\}; b_{r + 1}}(\mathfrak{G})$ exist.
\item $L(\phi_{\{a_1, \ldots, a_s\}; a_{s + 1}}(\mathfrak{G}))$ splits completely at all places in $\cup_{j = 1}^{r + 1} \, \textup{Ram}(\chi_{b_j})$. 
\item $L(\phi_{\{b_1, \ldots, b_r\}; b_{r + 1}}(\mathfrak{G}))$ splits completely at all places in $\cup_{i = 1}^{s + 1} \, \textup{Ram}(\chi_{a_i})$. 
\end{itemize}
Then the map $\hat{\phi}_{\{a_h : h \in [s]\}, \{b_j : j \in [r]\}; a_{s + 1}, b_{r + 1}}(\mathfrak{G})$ exists if and only if 
\begin{itemize}
\item If $j$ is in $[r]$ and $v$ is in $\textup{Ram}(\chi_{b_j}) \cap \Omega_{\mathbb{F}_q(T)}^{\textup{fin}}$, then 
$$
\hat{\phi}_{\{a_h : h \in [s]\}, \{b_i : i \in [r] - \{j\}\}; a_{s + 1}, b_{r + 1}}(\mathfrak{G})(\Frob_v) = 0.
$$
\item If $i$ is in $[s]$ and $v$ is in $\textup{Ram}(\chi_{a_i}) \cap \Omega_{\mathbb{F}_q(T)}^{\textup{fin}}$, then 
$$
\hat{\phi}_{\{a_h : h \in [s]-\{i\}\}, \{b_j : j \in [r]\}; a_{s + 1}, b_{r + 1}}(\mathfrak{G})(\Frob_v) = 0.
$$
\end{itemize}
\end{proposition}

\begin{proof}
The proof proceeds exactly in the same way as the proof of Proposition \ref{inductively construct expansion maps} with the only difference that the cocycle $\theta$ appearing in that proof needs to be replaced with the cocycle
\begin{multline*}
\hat{\theta}(g_1, g_2) := \sum_{\substack{(\emptyset, \emptyset) \neq (U_1, U_2) \\ U_1 \subseteq \{a_1, \ldots, a_s\}, U_2 \subseteq \{b_1, \ldots, b_r\}}} \chi_{U_1 \cup U_2}(g_1) \cdot \hat{\phi}_{\{a_1, \ldots, a_s\} - U_1, \{b_1, \ldots, b_r\} - U_2; a_{s + 1}, b_{r + 1}}(\mathfrak{G})(g_2) \\
+ \phi_{\{a_1, \ldots, a_s\}; a_{s + 1}}(\mathfrak{G})(g_1) \cdot \left(\sum_{U \subseteq \{b_1, \ldots, b_r\}} \chi_{U}(g_1) \phi_{\{b_1, \ldots, b_r\} - U;b_{r + 1}}(\mathfrak{G})(g_2) \right).
\end{multline*}
We now inspect $\hat{\theta}$ locally at a place $v$ in $\Omega_{\mathbb{F}_q(T)}^{\text{fin}}$. First, consider the case where $v$ is not in $\text{Ram}(\{\chi_{a_1}, \ldots, \chi_{a_{s + 1}}, \chi_{b_1}, \ldots, \chi_{b_{r + 1}}\})$. Observe that since both the expansion maps and the $\hat{\phi}$-expansion maps appearing in $\hat{\theta}$ are normalized, we have by Proposition \ref{normalized expansion maps are unramified} and Proposition \ref{prop: field of definition of hat phi} that they factor through $G_{\mathbb{F}_q(T)_v}/I_v$. Hence $\hat{\theta}$ is the inflation of a class in $H^2(G_{\mathbb{F}_{q_v}}, \mathbb{F}_2) = 0$. 

Suppose now that $v \in \text{Ram}(\{\chi_{a_1}, \ldots, \chi_{a_s}, \chi_{b_1}, \ldots, \chi_{b_r}\})$. Let us first consider the case where $v \mid a_i$ for $i \in [s]$. Then, thanks to the first assumption, $\hat{\theta}$ locally at $v$ becomes
$$
\chi_{a_i}(g_1) \cdot \hat{\phi}_{\{a_1, \ldots, a_s\} - \{a_i\}, \{b_1, \ldots, b_r\}; a_{s + 1}, b_{r + 1}}(\mathfrak{G})(g_2)+\phi_{\{a_1, \ldots, a_s\}; a_{s + 1}}(\mathfrak{G})(g_1) \cdot \phi_{\{b_1, \ldots, b_r\}; b_{r + 1}}(\mathfrak{G})(g_2).
$$
By the last assumption, $L(\phi_{\{b_1,\ldots,b_r\};b_{r+1}}(\mathfrak{G}))$ splits completely at $v$, so
$\phi_{\{b_1,\ldots,b_r\};b_{r+1}}(\mathfrak{G})$ restricts trivially to $G_{\mathbb{F}_q(T)_v}$. Hence only the first term remains and it vanishes if and only if $\hat{\phi}_{\{a_1, \ldots, a_s\} - \{a_i\}, \{b_1, \ldots, b_r\}; a_{s + 1}, b_{r + 1}}(\mathfrak{G})$ restricts to the trivial character (since the $1$-cochain $\hat{\phi}_{\{a_1, \ldots, a_s\} - \{a_i\}, \{b_1, \ldots, b_r\}; a_{s + 1}, b_{r + 1}}(\mathfrak{G})$ restricts to an unramified character locally at $v$). The case $v \mid b_j$ is entirely symmetric. For the case where $v \mid a_{s+1}$ or $v \mid b_{r+1}$, we instead use the last two assumptions.

All in all, we see that $\hat{\theta}$ vanishes locally at all finite places if and only if the last two bullet points of the proposition hold. For the place at $\infty$ we use Hilbert reciprocity. Therefore, we find that $\hat{\theta}$ is locally trivial everywhere if and only if the last two bullet points hold. Once $\hat{\theta}$ vanishes in $H^2(G_{\mathbb{F}_q(T)},\mathbb{F}_2)$, we can find a normalized $\hat{\phi}$-expansion map by following the argument at the end of Proposition \ref{inductively construct expansion maps}. 
\end{proof}

We conclude with the following remark.

\begin{remark} 
\label{remark: degenerate phi-maps}
We declare a normalized expansion map as well as a normalized $\hat{\phi}$-expansion map to be trivial whenever one of the characters in the index set is trivial. We remark that this convention formally follows from the recursive equations defining these $1$-cochains if one drops the requirement that the underlying characters be linearly independent. In the same spirit, Proposition \ref{pExpansionAdditive} and Proposition \ref{prop: hat phi are additive} continue to hold in this slightly broader generality, where some of the entries are allowed to be trivial.
\end{remark}
\section{R\'edei symbols}\label{Section: redei}
In this section we define R\'edei symbols building on \cite{KS}, but using the generality that will be needed for our purposes: our symbols will be less general in terms of the underlying module. R\'edei symbols originate from the 1939 work \cite{Redei}: excellent modern treatments can be found in \cite{Ste, Stokvis}.

\subsection{Construction of the symbol}
Let $K/\mathbb{F}_q(T)$ be a finite separable extension. We begin by defining for which triples we shall be able to form the R\'edei symbol. For each $v$ in $\Omega_K$, we fix once and for all an element $\xi_v$ in 
$$
\Hom_{K\text{-alg}}(\mathbb{F}_q(T)^{\text{sep}}, K_v^{\text{sep}}).
$$
This gives us an embedding $\xi_v^\ast \colon G_{K_v} \to G_K$.

\begin{mydef} 
\label{def: Redei admissible}
Let $(\alpha, \beta, \gamma) \in (K^\ast/K^{\ast 2})^3$. Throughout this section, we set 
$$
S := \textup{Ram}(\chi_\alpha) \cup \textup{Ram}(\chi_\beta) \cup \textup{Ram}(\chi_\gamma) \cup \{v \in \Omega_K : v \mid \infty\}.
$$
We say that $(\alpha, \beta, \gamma)$ is \emph{R\'edei admissible} over $K$ if for all $v \in S$, at least one of $\chi_\alpha$, $\chi_\beta$ and $\chi_\gamma$ is trivial locally at $v$ and all three cup products $\chi_{\alpha} \cup \chi_{\beta}$, $\chi_{\alpha} \cup \chi_{\gamma}$ and $\chi_{\beta} \cup \chi_{\gamma}$ are trivial in $H^2(G_K, \FF_2)$. 
\end{mydef}

Henceforth we implicitly view $\FF_2$ in the unique way as a subgroup of $K^\ast$. Let $(\alpha, \beta, \gamma) \in (K^\ast/K^{\ast 2})^3$ be a R\'edei admissible triple. We let $G_{K, S}$ be the Galois group of the maximal extension of $K$ unramified outside $S$. By \cite[Proposition 8.3.17]{NSW}, we may pick $\epsilon \in C^2(G_{K, S}, \mathbb{F}_2)$ satisfying 
$$
\mathrm{d} \epsilon = \chi_\alpha \cup \chi_\beta \cup \chi_\gamma.
$$
In fact, for later purposes, we may even pick $\epsilon \in C^2(G_{K, S}(2), \mathbb{F}_2)$; this follows from \cite[Corollary 10.4.8]{NSW}, which continues to hold for function fields by the remark immediately after \cite[Theorem 10.4.2]{NSW}. For each $v \in S$ and $\xi \in \Hom_{K\text{-alg}}(\mathbb{F}_q(T)^{\text{sep}}, K_v^{\text{sep}})$, we have that $\epsilon \circ \xi^\ast \in Z^2(G_{K_v}, \mathbb{F}_2)$ by Definition \ref{def: Redei admissible}. As such, we can talk about $\inv_v(\epsilon \circ \xi^\ast)$. 

\begin{proposition} 
\label{prop: well-defined}
Let $(\alpha, \beta, \gamma)$ be in $(K^\ast/K^{\ast 2})^3$ be a triple that is R\'edei admissible over $K$. Then as $\epsilon$ varies over elements of $C^2(G_{K, S}, \mathbb{F}_2)$ satisfying $\mathrm{d} \epsilon = \chi_\alpha \cup \chi_\beta \cup \chi_\gamma$, the value of
$$
\sum_{v \in S} \inv_v(\epsilon \circ \xi_v^\ast)
$$
remains constant. 
\end{proposition}

\begin{proof}
Recall that for every $\theta \in H^2(G_K, \mathbb{F}_2)$ the equality 
$$
\inv_v(\theta) = \inv_v(\theta \circ \xi_v^\ast) 
$$
holds by definition of the invariant map. Therefore the proposition follows from Hilbert reciprocity (see e.g.~\cite[Proposition 8.3.11(iii)]{NSW}), which states that $\sum_{v \in S} \inv_v(\theta) = 0$ for every $\theta \in H^2(G_{K, S}, \mathbb{F}_2)$.
\end{proof}

By Proposition \ref{prop: well-defined}, we can now define the R\'edei symbol.

\begin{mydef}
Let $(\alpha, \beta, \gamma) \in (K^\ast/K^{\ast 2})^3$ be R\'edei admissible over $K$. Fix any choice of $\epsilon \in C^2(G_{K, S}, \mathbb{F}_2)$ satisfying $\mathrm{d} \epsilon = \chi_\alpha \cup \chi_\beta \cup \chi_\gamma$. Then we define
$$
[\alpha, \beta, \gamma]:= \sum_{v \in S} \inv_v(\epsilon \circ \xi_v^\ast).
$$
\end{mydef}

A priori the local contribution to the R\'edei symbol $\inv_v(\epsilon \circ \xi_v^\ast)$ might depend on our initial choice of $(\xi_v)_{v \in \Omega_K}$, however one can see that this is not the case by invoking \cite[Equation (3.2)]{MS}. As such, we shall henceforth suppress $\xi_v^\ast$ in our notation.

We now state the main properties of R\'edei symbols. 

\begin{theorem} 
\label{theorem: Redei reciprocity}
The R\'edei symbol has the following properties:
\begin{enumerate}
\item[$(1)$] Let $(\alpha, \beta, \gamma) \in (K^\ast/K^{\ast 2})^3$ be a triple that is R\'edei admissible over $K$. Then any permutation is also R\'edei admissible over $K$ and we have that
$$
[\alpha, \beta, \gamma] = [\beta, \alpha, \gamma] = [\alpha, \gamma, \beta].
$$
\item[$(2)$] Let $(\alpha_1, \beta, \gamma), (\alpha_2, \beta, \gamma)$ be triples that are R\'edei admissible over $K$. Define
\begin{align*}
S_1 &:= \textup{Ram}(\chi_{\alpha_1}) \cup \textup{Ram}(\chi_\beta) \cup \textup{Ram}(\chi_\gamma) \cup \{v \in \Omega_K : v \mid \infty\} \\
S_2 &:= \textup{Ram}(\chi_{\alpha_2}) \cup \textup{Ram}(\chi_\beta) \cup \textup{Ram}(\chi_\gamma) \cup \{v \in \Omega_K : v \mid \infty\} \\
S_3 &:= \textup{Ram}(\chi_{\alpha_1 \alpha_2}) \cup \textup{Ram}(\chi_\beta) \cup \textup{Ram}(\chi_\gamma) \cup \{v \in \Omega_K : v \mid \infty\}.
\end{align*}
Then $(\alpha_1 \alpha_2, \beta, \gamma)$ is also R\'edei admissible over $K$ and moreover we have that
$$
[\alpha_1 \alpha_2, \beta, \gamma]= [\alpha_1, \beta, \gamma] + [\alpha_2, \beta, \gamma]. 
$$
\end{enumerate}
\end{theorem}

\begin{proof}
We will first prove $(1)$. Let $(\alpha, \beta, \gamma) \in (K^\ast/K^{\ast 2})^3$ be a triple that is R\'edei admissible over $K$. It is clear that any permutation of $(\alpha, \beta, \gamma)$ is still R\'edei admissible over $K$. We claim that we have that
$$
[\alpha, \beta, \gamma] = [\beta, \alpha, \gamma];
$$
it is easy to adapt the proof to also show that $[\alpha, \beta, \gamma] = [\alpha, \gamma, \beta]$.

Take $\epsilon, \epsilon' \in C^2(G_{K, S}, \mathbb{F}_2)$ such that
$$
\mathrm{d} \epsilon = \chi_\alpha \cup \chi_\beta \cup \chi_\gamma, \quad \quad \mathrm{d} \epsilon' = \chi_\beta \cup \chi_\alpha \cup \chi_\gamma.
$$
Define $\rho \in C^1(G_{K, S}, \mathbb{F}_2)$ by $\rho(\sigma) := \chi_\alpha(\sigma) \cdot \chi_\beta(\sigma)$. A direct computation reveals the identity $\mathrm{d} \rho = -\chi_\alpha \cup \chi_\beta - \chi_\beta \cup \chi_\alpha$. Hence we have 
$$
\mathrm{d} \left(\epsilon + \epsilon' + \rho \cup \chi_\gamma\right) = 0.
$$
Therefore we conclude from~\cite[Proposition 8.3.11(iii)]{NSW} that
\begin{equation}
\label{eRR}
\sum_{v \in S} \inv_v \left(\epsilon + \epsilon' + \rho \cup \chi_\gamma\right) = 0.
\end{equation}
But observe that
$$
\inv_v \left(\epsilon + \epsilon' + \rho \cup \chi_\gamma\right) = \inv_v\left(\epsilon \right) + \inv_v\left(\epsilon' \right) + \inv_v\left(\rho \cup \chi_\gamma \right).
$$
For every $v \in S$, R\'edei admissibility implies that $\rho$ or $\chi_\gamma$ is zero when restricted to $G_{K_v}$. Combining this with equation \eqref{eRR} gives precisely part $(1)$.

As for part $(2)$, it is not difficult to show that $(\alpha_1 \alpha_2, \beta, \gamma)$ is R\'edei admissible over $K$. Take $\epsilon_i \in C^2(G_{K, S_i}, \mathbb{F}_2)$ such that
\begin{align*}
\mathrm{d} \epsilon_i &= \chi_{\alpha_i} \cup \chi_\beta \cup \chi_\gamma \quad \quad \text{ for } i \in \{1, 2\} \\ 
\mathrm{d} \epsilon_3 &= \chi_{\alpha_1 \alpha_2} \cup \chi_\beta \cup \chi_\gamma.
\end{align*}
Then, by definition of the R\'edei symbol, we need to show that
$$
\sum_{v \in S_1} \inv_v(\epsilon_1) + \sum_{v \in S_2} \inv_v(\epsilon_2) + \sum_{v \in S_3} \inv_v(\epsilon_3) = 0.
$$
We claim that for each $v \in (S_2 \cup S_3) - S_1$, we have that $\epsilon_1$ restricts to a $2$-cocycle locally at $v$ and satisfies $\inv_v(\epsilon_1) = 0$. Indeed, observe that $\mathrm{d} \epsilon_1 = 0$ at such a place by R\'edei admissibility, but $\epsilon_1$ is also unramified by definition. Arguing similarly for $\epsilon_2$ and $\epsilon_3$, we conclude that
$$
\sum_{v \in S_1} \inv_v(\epsilon_1) + \sum_{v \in S_2} \inv_v(\epsilon_2) + \sum_{v \in S_3} \inv_v(\epsilon_3) = \sum_{v \in S_1 \cup S_2 \cup S_3} \inv_v(\epsilon_1 + \epsilon_2 + \epsilon_3).
$$
Now $\epsilon_1 + \epsilon_2 + \epsilon_3 \in C^2(G_{K, S_1 \cup S_2 \cup S_3}, \mathbb{F}_2)$ is a $2$-cocycle. Hence the result follows from~\cite[Proposition 8.3.11(iii)]{NSW}.
\end{proof}

We will need to know how R\'edei symbols change under conjugation. Before that, let us recall some basic notation. Suppose that the function field $K/\mathbb{F}_q(T)$ is Galois. Let $g$ be an element of $\Gal(K/\mathbb{F}_q(T))$. 

\begin{proposition} 
\label{prop: conjugating Redei}
Assume that the function field $K/\mathbb{F}_q(T)$ is Galois. Let $g$ be an element of $\Gal(K/\mathbb{F}_q(T))$. Let $(\alpha, \beta, \gamma) \in (K^\ast/K^{\ast 2})^3$ be a triple that is R\'edei admissible over $K$. Then $(g(\alpha), g(\beta), g(\gamma))$ is R\'edei admissible with associated set equal to $g \cdot S$. Furthermore, we have that
$$
[\alpha, \beta, \gamma] = [g(\alpha), g(\beta), g(\gamma)].
$$
\end{proposition}

\begin{proof}
For each $g \in \Gal(K/\mathbb{F}_q(T))$, it is immediate that $(g(\alpha), g(\beta), g(\gamma))$ is R\'edei admissible. 

Let now $\epsilon \in C^2(G_{K, S}, \mathbb{F}_2)$ be such that
$$
\mathrm{d} \epsilon = \chi_\alpha \cup \chi_\beta \cup \chi_\gamma.
$$
Take a lift $\tilde{g}$ of $g$ to $G_{\mathbb{F}_q(T)}$. Then $\epsilon^{\tilde{g}}(\sigma, \tau) := \epsilon(\tilde{g}^{-1} \sigma \tilde{g}, \tilde{g}^{-1} \tau \tilde{g})$ satisfies 
$$
\mathrm{d} \epsilon^{\tilde{g}} = \chi_{g(\alpha)} \cup \chi_{g(\beta)} \cup \chi_{g(\gamma)}.
$$
Furthermore, $\epsilon^{\tilde{g}}$ visibly factors through $G_{K,gS}$. Therefore we can use $\epsilon^{\tilde{g}}$ to compute the R\'edei symbol of $(g(\alpha), g(\beta), g(\gamma))$. Then we have that
$$
\inv_v(\epsilon) = \inv_{g \cdot v}(\epsilon^{\tilde{g}}),
$$
which gives the desired conclusion. 
\end{proof}

\subsection{Trivializations}
The purpose of this subsection is to provide instances where we have particularly nice choices of the cochain $\epsilon$ and can obtain more explicit formulae for the R\'edei symbol. If $\rho \colon G_K \to \mathbb{F}_2$ is a $1$-cochain satisfying $\mathrm{d} \rho = \chi_\alpha \cup \chi_\beta$, then 
$$
\epsilon' := \rho \cup \chi_\gamma
$$
satisfies $\mathrm{d} \epsilon' = \chi_\alpha \cup \chi_\beta \cup \chi_\gamma$. In case $\rho$ factors through $G_{K,S}$, then $\epsilon'$ will as well and thus $\epsilon'$ is a valid trivialization of $\chi_\alpha \cup \chi_\beta \cup \chi_\gamma$. 

\begin{proposition} 
\label{prop: trivialization in general}
Let $(\alpha, \beta, \gamma) \in (K^\ast/K^{\ast 2})^3$ be a triple that is R\'edei admissible over $K$. Let $S'$ be a set of places of $K$ containing $S$. Suppose that there exists $\rho \in C^1(G_{K, S'}, \mathbb{F}_2)$ such that $\mathrm{d} \rho = \chi_\alpha \cup \chi_\beta$. Also suppose that for every $v \in S' - \textup{Ram}(\chi_\gamma)$, the character $\chi_\gamma$ is trivial locally at $v$. Then
$$
[\alpha, \beta, \gamma]= \sum_{v \in \textup{Ram}(\chi_\gamma)} \inv_v(\rho \cup \chi_\gamma).
$$
\end{proposition}

\begin{proof}
Let $\epsilon \in C^2(G_{K, S}, \mathbb{F}_2)$ be such that
$$
\mathrm{d} \epsilon = \chi_\alpha \cup \chi_\beta \cup \chi_\gamma.
$$
Also define $\epsilon' := \rho \cup \chi_\gamma$. The difference $\epsilon - \epsilon'$ is a $2$-cocycle in $Z^2(G_{K, S'}, \mathbb{F}_2)$. Moreover, both $\epsilon$ and $\epsilon'$ restrict to $2$-cocycles at all places $v \in S'$, and $\epsilon$ is unramified at all places in $S' - S$ so $\inv_v(\epsilon) = 0$ for such places. Hence we have
\begin{align*}
[\alpha, \beta, \gamma] &= \sum_{v \in S} \inv_v(\epsilon) = \sum_{v \in S'} \inv_v(\epsilon) = \sum_{v \in S'} \inv_v(\epsilon') \\
&= \sum_{v \in S'} \inv_v(\rho \cup \chi_\gamma) = \sum_{v \in \textup{Ram}(\chi_\gamma)} \inv_v(\rho \cup \chi_\gamma)
\end{align*}
by~\cite[Proposition 8.3.11(iii)]{NSW} and the triviality of $\gamma$ at all places in $S' - \textup{Ram}(\chi_\gamma)$.
\end{proof}

We shall also need a variant of the above trivialization lemma.

\begin{proposition} 
\label{prop: trivialization in general2}
Let $(\alpha, \beta, \gamma) \in (K^\ast/K^{\ast 2})^3$ be a triple that is R\'edei admissible over $K$. Suppose that there exists $\rho \in C^1(G_{K, S}, \mathbb{F}_2)$ such that $\mathrm{d} \rho = \chi_\alpha \cup \chi_\beta$. Also suppose that for every $v \in S - \textup{Ram}(\chi_\gamma)$, the character $\chi_\gamma$ is trivial locally at $v$ or $\rho$ restricts to an unramified quadratic character of $G_{K_v}$. Then
$$
[\alpha, \beta, \gamma] = \sum_{v \in \textup{Ram}(\chi_\gamma)} \inv_v(\rho \cup \chi_\gamma).
$$
\end{proposition}

\begin{proof}
We take $\epsilon = \rho \cup \chi_\gamma$ in order to compute the symbol. Then we have
$$
[\alpha, \beta, \gamma] = \sum_{v \in S} \inv_v(\rho \cup \chi_\gamma) = \sum_{v \in \textup{Ram}(\chi_\gamma)} \inv_v(\rho \cup \chi_\gamma).
$$
We now show that $\inv_v(\rho \cup \chi_\gamma) = 0$ for all $v \in S - \textup{Ram}(\chi_\gamma)$. If $\chi_\gamma$ is trivial locally, then this certainly holds. But if $\chi_\gamma$ is not trivial, then $\chi_\gamma$ is unramified (as $v \in S - \textup{Ram}(\chi_\gamma)$) and so is $\rho$ by assumption. Hence we obtain the cup product of two unramified characters, which vanishes.
\end{proof}

Before we state another result of the same type, we need a convenient formula for the invariant map.

\begin{lemma}
\label{lInvariantMap}
Let $v := q^n$ be an odd integer and let $M := \mathbb{F}_{q^n}((T))$ be a local field. Let $F$ be any lift of Frobenius to $G_M(2)$ and let $\sigma$ be a topological generator of the image of inertia in $G_M(2)$. Then the map $f: Z^2(G_M(2), \mu_2) \rightarrow \mu_2$ given by 
$$
c \mapsto c(F, \sigma) - c(\sigma^v, F) - \sum_{i = 1}^{v - 1} c(\sigma, \sigma^i)
$$
has $B^2(G_M(2), \mu_2)$ in its kernel, where we have used additive notation for $\mu_2$. Moreover, $f$ equals $\exp(2\pi i \cdot \mathrm{inv}_M)$.
\end{lemma}

\begin{proof}
For the last part, one just needs to check that $f$ is not identically zero, as in that case $f$ and $\exp(2\pi i \cdot \mathrm{inv}_M)$ both become non-zero homomorphisms $H^2(G_M(2), \mu_2) \rightarrow \mu_2$. Taking the cup product of a ramified cocycle in $H^1(G_M(2), \mu_2)$ with a non-trivial unramified cocycle in $H^1(G_M(2), \mu_2)$ shows that $f$ is indeed not zero, as desired.

Therefore it remains to verify that $f$ vanishes on $B^2(G_M(2), \mu_2)$. In this case, we see that
$$
c(g, h) = y(g) + y(h) - y(gh)
$$
for some $1$-cochain $y: G_M(2) \rightarrow \mu_2$. Plugging this in, we obtain that
\begin{multline}
\label{eSigmaFrob}
c(F, \sigma) - c(\sigma^v, F) - \sum_{i = 1}^{v - 1} c(\sigma, \sigma^i) = \\
y(\sigma) - y(\sigma^v) - y(F \sigma) + y(\sigma^v F) - \sum_{i = 1}^{v - 1} \left(y(\sigma) + y(\sigma^i) - y(\sigma^{i + 1})\right).
\end{multline}
By a well-known theorem of Iwasawa \cite[Theorem 7.5.3]{NSW} we have $F \sigma = \sigma^v F$, and hence $y(F \sigma) = y(\sigma^v F)$. Moreover, we may evaluate the sum by a telescoping trick and using that $v - 1 \equiv 0 \bmod 2$
$$
\sum_{i = 1}^{v - 1} \left(y(\sigma) + y(\sigma^i) - y(\sigma^{i + 1})\right) = (v - 1) y(\sigma) + \sum_{i = 1}^{v - 1} \left(y(\sigma^i) - y(\sigma^{i + 1})\right) = y(\sigma) - y(\sigma^v).
$$
Therefore the bottom expression in equation \eqref{eSigmaFrob} equals zero, as desired.
\end{proof}

\begin{proposition} 
\label{prop: trivialization in general3}
Let $(\alpha, \beta, \gamma) \in (K^\ast/K^{\ast 2})^3$ be a triple that is R\'edei admissible over $K$. Let $S \subseteq S'$. Suppose that there exists $\rho \in C^1(G_{K, S'}(2), \mathbb{F}_2)$ such that $\mathrm{d} \rho = \chi_\alpha \cup \chi_\beta$. Moreover, assume that the character $\chi_\gamma$ is trivial locally at all $v \in S - \textup{Ram}(\chi_\gamma)$. Then
$$
[\alpha, \beta, \gamma] = \sum_{v \in \textup{Ram}(\chi_\gamma)} \inv_v(\rho \cup \chi_\gamma)+ \sum_{v \in S' - S} \rho(\sigma_v) \cdot \chi_{\gamma}(\Frob_v),
$$ 
where $(\sigma_v, \Frob_v)_{v \in S' - S}$ are any choices of inertia and Frobenius elements for $G_{K_v}$. 
\end{proposition}

\begin{proof}
Fix $\epsilon \in C^2(G_{K, S}(2), \mathbb{F}_2)$ such that $\mathrm{d}\epsilon = \chi_\alpha \cup \chi_\beta \cup \chi_\gamma$. Observe that, without loss of generality, we may and will assume, for the rest of the proof, that $\epsilon(\text{id}, \text{id}) = 0$: indeed we can add to a given choice of $\epsilon_0$ the constant function on $G_{K, S}(2) \times G_{K, S}(2)$ given by $\epsilon_0(\text{id}, \text{id})$. 

The difference $\theta := \epsilon - \rho \cup \chi_\gamma$ is a $2$-cocycle in $Z^2(G_{K, S'}(2), \mathbb{F}_2)$. We apply ~\cite[Proposition 8.3.11(iii)]{NSW} and get 
$$
\sum_{v \in S'} \inv_v(\theta) = 0. 
$$
Unfolding the left hand side gives us the relation
$$
[\alpha, \beta, \gamma] = \sum_{v \in \textup{Ram}(\chi_\gamma)} \inv_v(\rho \cup \chi_\gamma) + \sum_{v \in S' - S} \inv_v(\epsilon - \rho \cup \chi_\gamma)
$$
because the terms $\inv_v(\rho \cup \chi_\gamma)$ for $v \in S - \text{Ram}(\chi_\gamma)$ vanish by hypothesis. Since $\epsilon$ is chosen with $\epsilon(\text{id}, \text{id}) = 0$, we see that the $2$-cocycle $\theta$ is normalized, i.e.~it satisfies $\theta(\text{id}, \text{id}) = 0$. It follows that $\theta(\text{id}, g) = \theta(g, \text{id}) = 0$ for each $g \in G_{K, S'}$. We next observe that for each $v \in S' - S$ and for each $g \in G_{K_v}$, the map $\theta(g, -)$ factors through $G_{K_v}/I_v$: this is obviously true for $\epsilon$ and it holds for $\rho \cup \chi_\gamma$ because $S' - S$ is disjoint from $\text{Ram}(\chi_{\gamma})$. Combining these observations we deduce that $\theta(-, \sigma) = 0$ for each $\sigma \in I_v$. Using this along with Lemma \ref{lInvariantMap}, we deduce that 
$$
\inv_v(\theta) = \rho(\sigma_v^{q_v}) \cdot \chi_\gamma(\Frob_v),
$$
for each topological generator $\sigma_v$ of tame inertia. Since $x \mapsto x^{q_v}$ is a bijection on the set of topological generators of tame inertia, we may simply rewrite this as
$$
\inv_v(\theta) = \rho(\sigma_v) \cdot \chi_\gamma(\Frob_v),
$$
as desired. 
\end{proof}
\section{Gridding} \label{Section: large sieve}
\subsection{Prime divisors}
Let $D \in \FF_q[T]$ be a monic, squarefree polynomial of degree $n$. Denote by $\pi_1, \dots, \pi_r$ its monic irreducible factors ordered such that $\deg(\pi_i) \leq \deg(\pi_j)$ for $i \leq j$. The factorization pattern of $D$ is by definition the tuple $(\deg(\pi_1), \dots, \deg(\pi_r))$.

\begin{mydef}
\label{dGrid}
Given an integer $n \geq 2$, we say that a factorization pattern $(n_1, \dots, n_r)$ (with $n = n_1 + \cdots + n_r$) is admissible if:
\begin{enumerate}
\item[(1)] we have 
$$
r \leq (\log n)^2,
$$
\item[(2)] if $j \neq k$ and $n_j = n_k$, then $n_j \leq n^{1/100}$,
\item[(3)] we have
$$
\# \{1 \leq i \leq r : n_i \leq n^{1/100}\} < \frac{\log n}{20},
$$
\item[(4)] we have
$$
\# \{1 \leq i \leq r: n_i \equiv 0 \bmod 2\} \geq 0.4 \log n.
$$
\end{enumerate}
\end{mydef}

\begin{mydef}
Let $n \in \Z_{\geq 2}$. Then a grid for $n$ is a product space $X_1 \times \dots \times X_r$ with the following properties:
\begin{itemize}
\item every set $X_i$ is a subset of monic irreducible polynomials of degree $n_i$ such that $n_1 \leq n_2 \leq \cdots \leq n_r$ and such that $n_1 + \cdots + n_r = n$,
\item if $n_i \leq n^{1/100}$, then $|X_i| = 1$, while if $n_i > n^{1/100}$, then $X_i$ consists of all irreducible polynomials of degree $n_i$,
\item we have $X_j = X_k$ if and only if $j = k$,
\item every element $x \in X_1 \times \dots \times X_r$ (viewed as a polynomial by multiplying out coordinates) has an admissible factorization pattern.
\end{itemize}
\end{mydef}

\begin{theorem}
\label{tGrid}
There is an absolute constant $c > 0$ such that the number of monic, squarefree polynomials of degree $n$ over $\mathbb F_q$ that have an inadmissible factorization pattern is bounded from above by $q^n/n^c$ for all sufficiently large $n$.
\end{theorem}

Set $N := \lfloor n^{1/100}\rfloor$, and denote by $I_q(d)$ the number of monic irreducible polynomials in $\mathbb F_q[T]$ of degree $d$. For a partition $\sum_{d \ge 1} d a_d = n$ of $n$, the number of monic squarefree degree $n$ polynomials in $\mathbb F_q[T]$ having the corresponding factorization pattern is
\[
\prod_{d \ge 1} \binom{I_q(d)}{a_d} \le \prod_{d \ge 1} \frac{I_q(d)^{a_d}}{a_d!}.
\]
Because $I_q(d) \le \frac{q^d}{d}$, the number of monic squarefree polynomials with this factorization pattern is at most
\[
\prod_{d \ge 1}\frac{I_q(d)^{a_d}}{a_d!}
\le \prod_{d \ge 1}\frac{(q^d/d)^{a_d}}{a_d!}
= q^n\prod_{d \ge 1}\frac{1}{d^{a_d}a_d!}.
\]

On the other hand, if $\sigma_n$ is a uniformly random permutation of $\{1,\dots,n\}$, then the probability that $\sigma_n$ has the corresponding cycle structure is
\[
\prod_{d \ge 1}\frac{1}{d^{a_d}a_d!}
\]
so it remains to sum it over the inadmissible partitions.

For $d \ge 1$, let $A_d$ be the number of $d$-cycles of $\sigma_n$, and put
\[
S := \sum_{1 \le d \le N} A_d,
\qquad
E := \sum_{d \equiv 0 \bmod 2} A_d, 
\qquad
K := \sum_{d \ge 1} A_d.
\]
The cycle type is inadmissible if at least one of the events
\begin{align*}
&E_1 := \{K > \log^2 n\}, \\
&E_2 := \{\textup{there exists } d > N \text{ with } A_d \ge 2\}, \\
&E_3 := \{S \geq (\log n)/20\}, \\
&E_4 := \{E < 0.4 \log n\}
\end{align*}
occurs, so the union bound for probabilities gives
\begin{equation}
\label{eq:badunion}
\mathbb P(\sigma_n\text{ is inadmissible}) \le \mathbb P(E_1) + \mathbb P(E_2) + \mathbb P(E_3) + \mathbb P(E_4).
\end{equation}

\begin{lemma}
\label{genK}
For every real $u>0$ we have
\begin{equation*}
\mathbb E(u^K) = \frac{u(u+1) \cdots (u+n-1)}{n!}.
\end{equation*}
\end{lemma}

\begin{proof}
Let $S_n$ be the collection of permutations of $\{1, \dots, n\}$, and put
\[
F_n(u) = \sum_{\sigma\in S_n} u^{K(\sigma)}.
\]
Given a permutation of $\{1, \dots, n-1\}$, one can construct a permutation of $\{1, \dots, n\}$ either by adding $(n)$ as a new 1-cycle (which increases the number of cycles by $1$), or by inserting $n$ into one of the $n-1$ slots inside the existing cycle decomposition (which leaves the number of cycles unchanged). Hence
\[
F_n(u) = (u + n - 1) F_{n - 1}(u),
\qquad F_1(u) = u.
\]
Induction gives
\[
F_n(u) = u (u + 1) \cdots (u + n - 1),
\]
and dividing by $n!$, the number of permutations of $\{1, \dots, n\}$, we get the result.
\end{proof}

By Markov's inequality and Lemma \ref{genK} with $u=2$ we have
\[
\mathbb P(E_1)
=\mathbb P\bigl(2^K>2^{\log^2 n}\bigr)
\le 2^{-\log^2 n} \mathbb E(2^K)
= 2^{-\log^2 n}\frac{2\cdot 3\cdots (n+1)}{n!}
= \frac{n + 1}{2^{\log^2 n}},
\]
which decays faster than any negative power of $n$. In particular,
\begin{equation}
\label{eq:E1}
\mathbb P(E_1)\ll n^{-100}.
\end{equation}

\begin{lemma} 
\label{factorialmoments}
For every sequence $(m_d)_{d \ge 1}$ of nonnegative integers such that $\sum_{d \ge 1} d m_d\le n$, we have
\begin{equation*}
\mathbb E\!\Bigl[\prod_{d \ge 1} (A_d)_{m_d}\Bigr]
=\prod_{d \ge 1} d^{-m_d},
\end{equation*}
where $(x)_m=x(x-1)\cdots (x-m+1)$. 
\end{lemma}

\begin{proof}
Set $M = \sum_{d \ge 1} d m_d$. The random variable $\prod_d (A_d)_{m_d}$ counts the number of ordered choices of $m_d$ distinct $d$-cycles of $\sigma_n$, for each $d$. We count the average of this quantity by double counting.

First count the total number of ordered families of disjoint cycles of the prescribed sizes on $\{1,\dots,n\}$. Choose an ordered list of $M$ distinct elements of $\{1,\dots,n\}$; there are $(n)_M=n!/(n-M)!$ choices. Break this list into consecutive blocks of lengths $d,d,\dots,d$ repeated $m_d$ times for each $d$. Each block of length $d$ defines a $d$-cycle, and each $d$-cycle is represented exactly $d$ times by cyclic rotation. Therefore the number of ordered families of disjoint cycles is
\[
\frac{(n)_M}{\prod_d d^{m_d}}.
\]
Once such a family is fixed, the remaining $n-M$ elements may be permuted arbitrarily, so exactly $(n-M)!$ permutations contain that family. Thus the total number of pairs
\[
(\sigma,\text{ ordered family of cycles contained in }\sigma)
\]
is
\[
\frac{(n)_M}{\prod_d d^{m_d}}\,(n-M)!
=\frac{n!}{\prod_d d^{m_d}}.
\]
Dividing by $n!$ gives the result.
\end{proof}

If $A_d \ge 2$, then $\binom{A_d}{2}\ge 1$, so
\[
\mathbb P(A_d \ge 2)\le \mathbb E\binom{A_d}{2}.
\]
Now if $2d\le n$, then applying Lemma \ref{factorialmoments} with $m_d=2$ and all other $m_j=0$ gives
\[
\mathbb E\binom{A_d}{2} = \frac{1}{2}\mathbb E\bigl((A_d)_2\bigr) = \frac{1}{2d^2}.
\]
Hence, since $A_d \ge 2$ is impossible for $d>n/2$,
\begin{equation} 
\label{eq:E2}
\mathbb P(E_2)
\le \sum_{N < d \le n/2} \mathbb P(A_d \ge 2)
\le \sum_{N < d \le n/2} \mathbb E\binom{A_d}{2}
\le \frac12 \sum_{d > N} \frac{1}{d^2}
\ll \frac{1}{N}
\ll n^{-1/100}.
\end{equation}
We have the identity
\begin{equation}
\label{eFallingRearrange}
\left(\sum_{d \le N} A_d\right)_k
= \sum_{\substack{m_1, \ldots, m_N \ge 0 \\ m_1 + \cdots + m_N = k}}
\frac{k!}{\prod_{d \le N} m_d!}
\prod_{d \le N} (A_d)_{m_d}.
\end{equation}
Set $k := \left\lceil (\log n)/20 \right\rceil$, so $kN \le n$ for sufficiently large $n$. Applying equation \eqref{eFallingRearrange} and Lemma \ref{factorialmoments}, we find
\[
\mathbb E\bigl((S)_k\bigr)
=\sum_{\substack{m_1, \ldots, m_N \ge 0 \\ m_1 + \cdots + m_N = k}}
\frac{k!}{\prod_{d \le N} m_d!}
\prod_{d \le N} d^{-m_d}
= \left(\sum_{d \le N} \frac1d\right)^k \leq (1 + \log N)^k.
\]
Also, if $S \ge k$ then $(S)_k \ge k!$, so Markov's inequality yields
\begin{equation} 
\label{eq:E3}
\mathbb P(E_3) 
\le \mathbb P(S \ge k)
\le \frac{\mathbb E((S)_k)}{k!}
\leq \frac{(1 + \log N)^k}{k!}
\le \left(\frac{e + e \log N}{k}\right)^k \ll n^{-c}
\end{equation}
for some small absolute constant $c > 0$. 

Next, we handle $E_4$. Denote by $Z_n(\sigma)$ the number of even cycles of $\sigma \in S_n$. Using the power series expansions
$$
\sum_{k \ge 1} \frac{u^k}{k} = - \log(1-u), \quad
\sum_{k \text{ even}} \frac{u^k}{k} = \frac{1}{2} \log\frac{1}{1-u^2}, \quad
\sum_{k \text{ odd}} \frac{u^k}{k} = - \log(1-u) - \frac{1}{2} \log\frac{1}{1-u^2},
$$
the exponential formula in combinatorics yields the identity
\[
\sum_{n \ge 0} \frac{u^n}{n!} \sum_{\sigma \in S_n} s^{Z_n(\sigma)}
= \exp\Bigg( \sum_{k \text{ odd}} \frac{1}{k} u^k + \sum_{k \text{ even}} \frac{s}{k} u^k \Bigg)
= \frac{(1 - u^2)^{-(s - 1)/2}}{1 - u} =: F(u, s).
\]
Hence we have $\mathbb{E}[s^{Z_n}] = [u^n] F(u, s)$, so we now evaluate $[u^n] F(u, s)$. Note that
$$
F(u, s) = \frac{\sum_{k = 0}^\infty \frac{(-1)^k \cdot \left(-(s - 1)/2\right)_k}{k!} \cdot u^{2k}}{1 - u} =: \frac{\sum_{k = 0}^\infty c_{k, s} u^{2k}}{1 - u}.
$$
We now fix a real number $0 < s < 1$. Set $\alpha := -(s - 1)/2 > 0$. Then $c_{k, s}$ is negative for all $k \geq 1$, and moreover
$$
c_{k, s} = \frac{-\alpha}{k} \prod_{i = 1}^{k - 1} \frac{i - \alpha}{i} = \frac{-\alpha}{k} \prod_{i = 1}^{k - 1} \left(1 - \frac{\alpha}{i}\right) \sim \frac{C(\alpha)}{k^{1 + \alpha}}
$$
for some constant $C(\alpha)$ as $k$ goes to infinity. Hence $\sum_{k = 0}^\infty c_{k, s}$ converges and thus equals $\lim_{u \rightarrow 1} (1 - u^2)^\alpha = 0$ by Abel's theorem. But $[u^n]F(u, s)$ is nothing more than
$$
\mathbb{E}[s^{Z_n}] = [u^n]F(u, s) = \sum_{k = 0}^{\lfloor n/2 \rfloor} c_{k, s} = \sum_{k = 0}^\infty c_{k, s} - \sum_{k > \lfloor n/2 \rfloor} c_{k, s} = - \sum_{k > \lfloor n/2 \rfloor} c_{k, s} = O_s(n^{-\alpha}).
$$
We now pick $s := 0.99$ and apply Markov's inequality to conclude that
\begin{align}
\mathbb{P}(E_4) 
&= \mathbb{P}(Z_n < 0.4 \log n) = \mathbb{P}(0.99^{Z_n} > 0.99^{0.4 \log n}) \le 0.99^{-0.4 \log n} \mathbb{E}[0.99^{Z_n}] \nonumber \\
&\ll 0.99^{-0.4 \log n} n^{-0.005} \leq n^{-0.0009}. \label{eq:E4}
\end{align}
We can now combine the various ingredients to prove Theorem \ref{tGrid}.

\begin{proof}[Proof of Theorem \ref{tGrid}]
From \eqref{eq:badunion}, \eqref{eq:E1}, \eqref{eq:E2}, \eqref{eq:E3} and \eqref{eq:E4}, we obtain
\[
\mathbb P(\sigma_n\text{ has an inadmissible cycle type}) \ll n^{-c},
\]
as desired.
\end{proof}

\subsection{The large sieve}
Our next lemma is a version of the large sieve sufficient for our purposes; it is likely that stronger results can be proved.

\begin{lemma}
\label{lLS}
There exists an absolute constant $C > 0$ such that the following statement holds. 

Let $K/\mathbb{F}_q(T)$ be a finite extension, let $n \in \Z_{\geq 1}$ and let $\chi_1, \dots, \chi_n \in \Hom(G_K, \{\pm 1\})$. Define
$$
D := \{1 \leq i, j \leq n : \chi_i = \chi_j \textup{ or } K(\chi_i \chi_j)/\mathbb{F}_q(T) \textup{ is not geom.}\}.
$$
Write $G$ for the maximum genus among all the $K(\chi_i \chi_j)/\mathbb{F}_q(T)$. Let $d \geq 1$, and let $Q_1, \dots, Q_s$ be the complete list of places of $K$ of degree $d$. Assume that each $Q_j$ is unramified over $\mathbb{F}_q(T)$ and unramified in each $K(\chi_i)$. Then we have for all complex numbers $\alpha_1, \dots, \alpha_n \in \mathbb{C}$ of magnitude at most $1$ and all complex numbers $\beta_1, \dots, \beta_s \in \mathbb{C}$ of magnitude at most $1$
$$
\left| \sum_{i = 1}^n \sum_{k = 1}^s \alpha_i \beta_k \chi_i(\Frob_{Q_k}) \right| \leq s |D|^{1/2} + C s^{1/2} n (G + [K : \mathbb{F}_q(T)])^{1/2}  q^{d/4}.
$$
\end{lemma}

\begin{proof}
We apply the triangle inequality and Cauchy--Schwarz to obtain the bounds
\begin{align*}
\left| \sum_{i = 1}^n \sum_{k = 1}^s \alpha_i \beta_k \chi_i(\Frob_{Q_k}) \right| 
&\leq \sum_{k = 1}^s \left| \sum_{i = 1}^n \alpha_i \chi_i(\Frob_{Q_k}) \right| \\
&\leq s^{1/2} \left( \sum_{1 \leq i_1, i_2 \leq n} \alpha_{i_1} \overline{\alpha_{i_2}} \sum_{k = 1}^s \chi_{i_1} \chi_{i_2}(\Frob_{Q_k}) \right)^{1/2}.
\end{align*}
The terms $(i_1, i_2) \in D$ lead to a contribution of size bounded by $s |D|^{1/2}$. For the terms $(i_1, i_2) \not \in D$, we apply for each inner sum the Chebotarev density theorem \cite[Proposition 7.4.8]{FJ} twice to the quadratic extension $K(\chi_i \chi_j)/K$; the first application takes the resulting conjugacy class to be the identity element, while the second application uses the unique non-trivial element of $\Gal(K(\chi_i \chi_j)/K)$.
\end{proof}

\subsection{Estimates for the 4-rank}
We start by parametrizing $W_0^{(h)}$ over a grid $X$. For this subsection, assume that $q \equiv 1 \bmod 8$, assume that the total degree $n$ is odd and fix $h \in \{1, -1\}$ throughout.

\begin{mydef}
\label{dGenusPar}
Let $X = X_1 \times \dots \times X_r$ be a grid with factorization pattern $(n_1, \dots, n_r)$. Take $W_0^{(h)}$ to be the subspace of $w = (w_1, \dots, w_r) \in \mathbb{F}_2^r$ such that
$$
\sum_{i = 1}^r w_i n_i \equiv 0 \bmod 2.
$$
For each $x \in X$ and each $h \in \{1, -1\}$ we define the $\mathbb F_2$-linear injection $f_x \colon W_0^{(h)} \xhookrightarrow{} C_0^{(h)}(\chi_x)^\circ$ given by sending $w = (w_1, \dots, w_r) \in W_0^{(h)}$ to 
$$
f_x(w) = \sum_{i = 1}^r w_i \chi_{\mathrm{pr}_i(x)}.
$$
Note that $f_x$ is the restriction of the map $f_D$ from Remark \ref{rmk: identifying spaces} to $W_0^{(h)}$. Denote by $i_{\textup{med}}$ the smallest number $i$ such that $n_i > n^{1/100}$.
\end{mydef}

Our next theorem gives us sufficient control over $C_1^{(h)}(\chi_x)^\circ$ to prove Theorem \ref{tChowla}. With a tad more work, it is possible to obtain the distribution of $\dim_{\FF_2} C_1^{(h)}(\chi_x)^\circ$. Since such a result would not improve the quality of the savings in our main theorem, we have opted for the following simpler result instead.

\begin{theorem}
\label{t4Rank}
There exists an absolute constant $C > 0$ such that the following holds. Let $X = X_1 \times \dots \times X_r$ be a grid with factorization pattern $(n_1, \dots, n_r)$ such that $n := n_1 + \dots + n_r$ is odd. Then $f_x(W_0^{(h)})$ contains $C_1^{(h)}(\chi_x)^\circ$.

Moreover, let $\mathbf{w} \in W_0^{(h)}$ be non-zero. 
\begin{enumerate}
\item[(a)] Suppose that either $\mathrm{pr}_i(\mathbf{w}) = 0$ for all $i \geq i_{\textup{med}}$ or $\mathrm{pr}_i(\mathbf{w}) = 1$ for all $i \geq i_{\textup{med}}$. Then we have
$$
|\{x \in X : f_x(\mathbf{w}) \in C_1^{(h)}(\chi_x)^\circ\}| \leq \frac{C |X|}{2^{0.35 \log n}}.
$$
\item[(b)] Suppose that for all $k \in \FF_2$ there exists $i \geq i_{\textup{med}}$ such that $\mathrm{pr}_i(\mathbf{w}) = k$. Then we have
$$
|\{x \in X : f_x(\mathbf{w}) \in C_1^{(h)}(\chi_x)^\circ\}| \leq \frac{C |X|}{2^r}.
$$
\end{enumerate}
\end{theorem}

\begin{proof}
For the first part, we note that any element of $C_0^{(h)}(\chi_x)^\circ$ is either of the shape $f_x(\mathbf{w})$ or $f_x(\mathbf{w}) + \chi_{\epsilon_{\mathbb{F}_q}}$. But the latter type of elements can not be in $C_1^{(h)}(\chi_x)^\circ$ by Remark \ref{rmk: identifying spaces}.

First suppose that $\mathrm{pr}_i(\mathbf{w}) = 0$ for all $i \geq i_{\textup{med}}$. Then $f_x(\mathbf{w})$ is unramified at $\mathrm{pr}_i(x)$ for all $i \geq i_{\textup{med}}$. Hence the condition $f_x(\mathbf{w}) \in C_1^{(h)}(\chi_x)^\circ$ implies that 
$$
f_x(\mathbf{w})(\Frob_{\mathrm{pr}_i(x)}) = 0
$$
for all $i \geq i_{\textup{med}}$. Since there are at least $0.35 \log n$ indices $i \geq i_{\textup{med}}$ by combining Definition \ref{dGrid}(3) and Definition \ref{dGrid}(4), part $(a)$ is a consequence of the Chebotarev density theorem \cite[Proposition 7.4.8]{FJ} and the bound on $r$ from Definition \ref{dGrid}(1), still under the assumption that $\mathrm{pr}_i(\mathbf{w}) = 0$ for all $i \geq i_{\textup{med}}$. 

In case $\mathrm{pr}_i(\mathbf{w}) = 1$ for all $i \geq i_{\textup{med}}$, we add $\chi_x$ for $h = 1$ and $\chi_x + \chi_{\epsilon_{\mathbb{F}_q}}$ for $h = -1$ to $f_x(\mathbf{w})$; this is allowed by Proposition \ref{prop: trivial elemts}. In this way we obtain a non-zero character that is unramified at $\mathrm{pr}_i(x)$ for all $i \geq i_{\textup{med}}$. Since $n$ is odd, it follows from the definition of $W_0^{(h)}$ that the resulting character is ramified at some finite place (and thus non-zero). Now we argue exactly as before and obtain again the bound $C |X|/2^{0.35 \log n}$. This ends the proof of $(a)$.
 
As for part $(b)$, write $\mathbf{w} = (w_1, \dots, w_r)$. We use Proposition \ref{pRedei} to deduce that $f_x(\mathbf{w}) \in C_1^{(h)}(\chi_x)^\circ$ if and only if the following conditions simultaneously hold:
\begin{itemize}
\item we have $\left( \frac{\prod_{j = 1}^r \mathrm{pr}_j(x)^{w_j}}{\mathrm{pr}_i(x)} \right) = 1$ for all $i$ such that $w_i = 0$,
\item we have $\left( \frac{\epsilon_{\FF_q}^{\frac{1 - h}{2}} \prod_{j = 1}^r \mathrm{pr}_j(x)^{w_j + 1}}{\mathrm{pr}_i(x)} \right) = 1$ for all $i$ such that $w_i = 1$.
\end{itemize}
Thus we have the following exact formula for $|\{x \in X : f_x(\mathbf{w}) \in C_1^{(h)}(\chi_x)^\circ\}|$
$$
\frac{1}{2^r} \sum_{x \in X} \prod_{i: w_i = 0} \left(1 + \left(\frac{\prod_{j = 1}^r \mathrm{pr}_j(x)^{w_j}}{\mathrm{pr}_i(x)}\right)\right) \prod_{i: w_i = 1} \left(1 + \left(\frac{\epsilon_{\FF_q}^{\frac{1 - h}{2}} \prod_{j = 1}^r \mathrm{pr}_j(x)^{w_j + 1}}{\mathrm{pr}_i(x)}\right)\right).
$$
We rewrite this as
$$
\frac{1}{2^r} \sum_{x \in X} \sum_{\mathbf{a} \in \FF_2^{\{i: w_i = 0\}}} \sum_{\mathbf{b} \in \FF_2^{\{i: w_i = 1\}}} \left(\frac{\prod_{j = 1}^r \mathrm{pr}_j(x)^{w_j}}{\prod_{i : w_i = 0, \mathrm{pr}_i(\mathbf{a}) = 1} \mathrm{pr}_i(x)}\right) \left(\frac{\epsilon_{\FF_q}^{\frac{1 - h}{2}} \prod_{j = 1}^r \mathrm{pr}_j(x)^{w_j + 1}}{\prod_{i : w_i = 1, \mathrm{pr}_i(\mathbf{b}) = 1} \mathrm{pr}_i(x)}\right).
$$
We swap summations and further rewrite this as
\begin{equation}
\label{e4RankSum}
\frac{1}{2^r} \sum_{\substack{\mathbf{a} \in \FF_2^{\{i : w_i = 0\}} \\ \mathbf{b} \in \FF_2^{\{i : w_i = 1\}}}} \sum_{x \in X} \prod_{i \neq j} \left(\frac{\mathrm{pr}_j(x)}{\mathrm{pr}_i(x)}\right)^{w_j (w_i + 1) \mathrm{pr}_i(\mathbf{a}) + (w_j + 1) w_i \mathrm{pr}_i(\mathbf{b})} \prod_{i : w_i = 1, \mathrm{pr}_i(\mathbf{b}) = 1} \left(\frac{\epsilon_{\FF_q}^{\frac{1 - h}{2}}}{\mathrm{pr}_i(x)}\right),
\end{equation}
where we use the convention that $\mathrm{pr}_i(\mathbf{a}) = 0$ for the $i$ with $w_i = 1$ and that $\mathrm{pr}_i(\mathbf{b}) = 0$ for the $i$ with $w_i = 0$. Using quadratic reciprocity to group the term $(i, j)$ with $(j, i)$ for all distinct $1 \leq i, j \leq r$, we see that the exponent of $(\mathrm{pr}_j(x)/\mathrm{pr}_i(x))$ is
$$
e(i, j) := w_j (w_i + 1) \mathrm{pr}_i(\mathbf{a}) + (w_j + 1) w_i \mathrm{pr}_i(\mathbf{b}) + w_i (w_j + 1) \mathrm{pr}_j(\mathbf{a}) + (w_i + 1) w_j \mathrm{pr}_j(\mathbf{b}).
$$
We now select an index $i_0 \geq i_{\text{med}}$ such that $w_{i_0} = 0$ and we also select an index $i_1 \geq i_{\text{med}}$ such that $w_{i_1} = 1$. Such indices exist by hypothesis. Now if there were to exist some $j \neq i_0$ such that
\begin{equation}
\label{eCons1}
w_j \mathrm{pr}_{i_0}(\mathbf{a}) + w_j \mathrm{pr}_j(\mathbf{b}) = e(i_0, j) = 1,  
\end{equation}
then we apply the Chebotarev density theorem to show that equation \eqref{e4RankSum} is negligible. If instead there were to exist some $j \neq i_1$ such that
\begin{equation}
\label{eCons2}
(w_j + 1) \mathrm{pr}_{i_1}(\mathbf{b}) + (w_j + 1) \mathrm{pr}_j(\mathbf{a}) = e(i_1, j) = 1,
\end{equation}
then the sum is again negligible. Now keeping in mind that $\mathbf{a}$ is indexed by $\{i : w_i = 0\}$ and that $\mathbf{b}$ is indexed by $\{i : w_i = 1\}$, the negation of equation \eqref{eCons2} implies that $\mathrm{pr}_j(\mathbf{a}) = \mathrm{pr}_k(\mathbf{a})$ for all $j, k$ inside the index set of $\mathbf{a}$ and the negation of equation \eqref{eCons1} implies that $\mathrm{pr}_j(\mathbf{b}) = \mathrm{pr}_k(\mathbf{b})$ for all $j, k$ inside the index set of $\mathbf{b}$. This leaves at most four options for the pair $(\mathbf{a}, \mathbf{b})$, and bounding trivially gives a contribution of size $4|X|/2^r$.
\end{proof}

\section{Reflection principles} 
\label{Section: reflection principles}
Let $q \equiv 1 \bmod 8$ be a prime power. We begin by introducing the combinatorial arrangements of squarefree polynomials that will be pivotal in our arguments. 

\begin{mydef}
Let $s$ be a positive integer. We call a set of the form
$$
\left(\prod_{i = 1}^s \{\pi_i(1),\pi_i(2)\}\right) \times \{a\} \times \{d\}
$$
a \emph{cube} in case:
\begin{itemize}
     \item $\pi_i(j)$ is an irreducible monic even degree element of $\mathbb{F}_q[T]$ for each $i \in [s]$ and $j \in \{1, 2\}$, 
     \item $a$ is a monic even degree squarefree polynomial in $\mathbb{F}_q[T]$, 
     \item $d$ is a monic odd degree squarefree polynomial in $\mathbb{F}_q[T]$,
     \item the set $\{\pi_i(j) : i \in [s], j \in \{1, 2\}\} \cup \{a, d\}$ has cardinality $2s + 2$ and its elements are pairwise coprime.
\end{itemize}
\end{mydef}

Observe that, upon taking the product of the entries, we can identify each point of a cube $C$ with a squarefree odd degree monic polynomial, so that $C$ can be viewed as a set consisting of $2^s$ distinct squarefree odd degree monic polynomials. We will often make this identification. We call $s$ the \emph{dimension} of the cube. 

Define $\mathrm{sg}(x) := (-1)^{|\{1 \leq i \leq s : \mathrm{pr}_i(x) = \pi_i(2)\}|}$. In the rest of the paper, we will denote by $x_0$ the vertex of $C$ given by
$$
x_0 := \left(\prod_{i = 1}^s \pi_i(1) \right) \cdot a \cdot d.
$$
For each $j \in [s]$, we will denote by $C_j$ the sub-cube of dimension $s - 1$ of $C$ consisting of the collection of $2^{s - 1}$ elements of $C$ whose $j$th coordinate equals $\pi_j(2)$. More generally, for each subset $T$ of $[s]$, we denote by $C_T$ the sub-cube of dimension $s - |T|$ given by the collection of $2^{s - |T|}$ elements of $C$ whose $j$th coordinate projection equals $\pi_j(2)$ for each $j \in T$. 

Given a cube $C$ of dimension $s$, we introduce the field 
$$
F(C) := \mathbb{F}_q(T)(\{\sqrt{\pi_i(1)\pi_i(2)} : i \in [s]\}).
$$
We will see that, under favorable conditions, one is able to control sums of Artin pairings along $C$ with a Legendre symbol in $F(C)$. A key observation, that is largely responsible for such results, is that the modules $N_h(x)$ are all isomorphic as $G_{F(C)}$-modules. 

Given $x \in C$ and $h \in \{1, -1\}$, we say that we have an $h$-\emph{raw cocycle} for $x$, in case we have a collection
$$
(\psi_i(x))_{1 \leq i \leq m},
$$
for some positive integer $m$ such that $\psi_i(x) \in B_h(\chi_{x})[2^i]$ and such that $2 \cdot \psi_{i + 1}(x) = \psi_i(x)$ for each $1 \leq i < m$. We will often conflate the raw cocycle $(\psi_i(x))_{1 \leq i \leq m}$ with its top cocycle $\psi_m(x)$ since the rest of the cocycles can be visibly reconstructed from this one. Observe that for each $h$-raw cocycle $\psi_m(x)$ for $x$, we have that $\psi_1(x)$ is naturally an element of $\Gamma_{\mathbb{F}_2}(\mathbb{F}_q(T))$. For a $1$-cochain $\phi\colon G_{\mathbb{F}_q(T)} \to N_h(x_0)$, we denote by $\mathrm{d}_{x_0}(\phi)$ the operator
$$
\mathrm{d}_{x_0}(\phi)(\sigma, \tau) = \sigma \cdot_{x_0} \phi(\tau) + \phi(\sigma) - \phi(\sigma \tau).
$$
The coming cocycle computation in Theorem \ref{tCocyComp} will be a key ingredient in our arguments and is based on earlier work of Smith \cite{Smi1}. To ease notation, we define for each subset $T \subseteq [s]$ the $1$-cochain $G_{\mathbb{F}_q(T)} \rightarrow \mathbb{F}_2$ by $\chi_T(\sigma) := \chi_{\{\pi_i(1) \pi_i(2) : i \in T\}}(\sigma)$.

\begin{theorem}
\label{tCocyComp}
Let $C$ be a cube of dimension $s$. Let $j \in \Z_{>0}$. Let $h \in \{1, -1\}$. For each $x \in C$, let $\psi_j(x)$ be an $h$-raw cocycle for $x$. Then
\begin{multline*}
\mathrm{d}_{x_0}\left(\sum_{x \in C} \mathrm{sg}(x) \psi_j(x)\right)(\sigma, \tau) = \\
\sum_{\emptyset \neq T \subseteq [s]} (h \cdot q_0)^{\chi_{\Frob}(\sigma)} \cdot (-1)^{|T| + 1 + \chi_{x_0}(\sigma)} \cdot \mathbf{1}_{\chi_T(\sigma) = 1} \cdot 
\left( \sum_{x \in C_T} \mathrm{sg}(x) \psi_{j - |T|}(x)(\tau) \right)
\end{multline*}
with the convention that $\psi_j(x) := 0$ for $j \leq 0$.
\end{theorem}

\begin{proof}
For every $\sigma, \tau \in G_{\FF_q(T)}$ and every $x \in C$, we compute
\begin{align*}
\mathrm{d}_{x_0}(\psi_j(x))(\sigma, \tau) &= \sigma \cdot_{x_0} \psi_j(x)(\tau) + \psi_j(x)(\sigma) - \psi_j(x)(\sigma \tau) \\
&= (hq_0)^{\chi_\Frob(\sigma)} \left((-1)^{\chi_{x_0}(\sigma)} - (-1)^{\chi_x(\sigma)}\right) \psi_j(x)(\tau) \\
&= (hq_0)^{\chi_\Frob(\sigma)} \cdot (-1)^{\chi_{x_0}(\sigma)} \left(1 - (-1)^{\chi_{x x_0}(\sigma)}\right) \psi_j(x)(\tau).
\end{align*}
We identify the image of $\sigma$ inside $\Gal(F(C)/\FF_q(T))$ with the largest $T \subseteq [s]$ such that $\chi_T(\sigma) = 1$. We denote this subset by $T_\sigma$. Similarly, we introduce, for each $x \in C$, the subset $T_x \subseteq [s]$ characterized by the property that $i \in T_x$ if and only if $\mathrm{pr}_i(x) = \pi_i(2)$. We now rewrite
$$
1 - (-1)^{\chi_{x x_0}(\sigma)} = 1 - (-1)^{|T_\sigma \cap T_x|} = 1 - (1 - 2)^{|T_\sigma \cap T_x|} = \sum_{\emptyset \neq U \subseteq T_x \cap T_\sigma} (-1)^{|U| + 1} 2^{|U|}.
$$
Using this for each $x \in C$, we conclude that $\mathrm{d}_{x_0}\left(\sum_{x \in C} \mathrm{sg}(x) \psi_j(x)\right)(\sigma, \tau)$ equals
\begin{align*}
&= \sum_{x \in C} \mathrm{sg}(x) \cdot (hq_0)^{\chi_\Frob(\sigma)} \cdot (-1)^{\chi_{x_0}(\sigma)} \left(1 - (-1)^{\chi_{x x_0}(\sigma)}\right) \psi_j(x)(\tau) \\
&= \sum_{x \in C} \mathrm{sg}(x) \cdot (hq_0)^{\chi_\Frob(\sigma)} \cdot (-1)^{\chi_{x_0}(\sigma)} \sum_{\emptyset \neq U \subseteq T_x \cap T_\sigma} (-1)^{|U| + 1} \psi_{j - |U|}(x)(\tau).
\end{align*}
We now exchange the order of summation. Noting that the condition $U \subseteq T_x \cap T_\sigma$ is equivalent to $x \in C_U$ and $\chi_U(\sigma) = 1$ ends the proof of the theorem.
\end{proof}

We now specialize further our definition of cubes, endowing them with extra structure. 

\begin{mydef} 
\label{def: acceptable cube}
Let $C$ be a cube of dimension $s$. We say that $C$ is \emph{pre-acceptable} in case the following demands hold:
\begin{enumerate}
    \item[$(A1)$] for each $i \in [s]$ we have that $\pi_i(1) \pi_i(2)$ is locally a square at each place where 
    $$
    \left(\prod_{j \in [s] -\{i\}} \pi_j(1)\pi_j(2)\right) \cdot a \cdot d
    $$
    has odd valuation,
    \item[$(A2)$] the map $\phi_{\{\pi_1(1)\pi_1(2), \ldots, \pi_s(1)\pi_s(2)\}; a}(\mathfrak{G})$ exists,
    \item[$(A3)$] for each $i \in [s]$, the field $L(\phi_{\{\pi_1(1)\pi_1(2), \ldots, \pi_s(1)\pi_s(2)\} - \{\pi_i(1) \pi_i(2)\}; a}(\mathfrak{G}))/\mathbb{F}_q(T)$ has residue field degree $1$ at each place $v \in \Omega_{\mathbb{F}_q(T)}$ such that $v(ad)$ is odd. Moreover, the extension 
    $$
    L(\phi_{\{\pi_1(1)\pi_1(2), \ldots, \pi_s(1)\pi_s(2)\}; a}(\mathfrak{G}))/\mathbb{F}_q(T)
    $$ 
    is unramified at infinity (and therefore trivial at infinity by definition of normalized expansion maps).
\end{enumerate}
We say that a cube $C$ is acceptable if $C$ is pre-acceptable and if moreover the field extension $L(\phi_{\{\pi_1(1)\pi_1(2), \ldots, \pi_s(1)\pi_s(2)\}; a}(\mathfrak{G}))/\mathbb{F}_q(T)$ has residue field degree $1$ at each place $v \in \Omega_{\mathbb{F}_q(T)}$ such that $v(ad)$ is odd. 

If $C$ is either acceptable or pre-acceptable, then we say that $C$ is \emph{profitable} if the map $\phi_{\{\pi_j(1)\pi_j(2) : j \in [s] - \{i\}\}; \pi_i(1)\pi_i(2)}(\mathfrak{G})$ exists for each $i \in [s]$ and moreover all the places $v$ in $\Omega_{\mathbb{F}_q(T)}$ such that $v(ad)$ is odd split completely in $L(\phi_{\{\pi_j(1)\pi_j(2) : j \in [s] - \{i\}\}; \pi_i(1)\pi_i(2)}(\mathfrak{G}))/\FF_q(T)$. 
\end{mydef}

We now introduce our third piece of extra structure. Define $r$ to be the total number of prime divisors of any element of $C$; by ordering the prime divisors of $a$ and $d$ in any fixed way we have natural projection maps $\mathrm{pr}_i$ for each $i \in [r]$. Given $\mathbf{t} \in N[2^{s + 1}]^{[r + 1]}$, such that the coordinates of $\mathbf{t}$ in $[s]$ are all in $N[2^s]$, a profitable cube $C$ and a subset $T \subseteq [s]$, we define
$$
\phi_T(C; \mathbf{t}) := \sum_{j \in T} \mathrm{qpr}_j(\mathbf{t}) \cdot \phi_{\{\pi_i(1) \pi_i(2) : i \in T - \{j\}\}; \pi_j(1) \pi_j(2)}(\mathfrak{G}),
$$
where $\mathrm{qpr}_j$ is the projection $\mathrm{pr}_j$ on the $j$th coordinate followed by the unique group epimorphism $N[2^{s}] \twoheadrightarrow \mathbb{F}_2$.

\begin{mydef} 
\label{desirable cube}
Let $h \in \{1, -1\}$, let $s$ be a positive integer and let $\mathbf{t} \in N[2^{s + 1}]^{[r + 1]}$. We say that a pre-acceptable/acceptable, profitable cube $C$ is \emph{of type} $\mathbf{t}$ if 
\begin{enumerate}
\item[$(T)$] for each $x \in C - \{x_0\}$ we have an $h$-raw cocycle $\psi_{s + 2}(x)$ with $\psi_1(x) = \chi_a$ and such that $\psi_{s + 1}(x)(\sigma_{\mathrm{pr}_i(x)}) = \mathrm{pr}_i(\mathbf{t})$ for all $i \in [r]$ and $\psi_{s + 1}(x)(\Frob_q) = \mathrm{pr}_{r + 1}(\mathbf{t})$. In this case, we will also say that $\psi_{s + 1}(x)$ is of type $\mathbf{t}$.
\end{enumerate}
We say that a pre-acceptable/acceptable, profitable cube $C$ is \emph{desirable of type} $\mathbf{t}$ if it satisfies $(T)$ and moreover satisfies
\begin{enumerate}
\item[$(D)$] for each $j \in [s]$ we have that
\begin{multline*}
\sum_{x \in C_j} \mathrm{sg}(x) \psi_{s + 1}(x)(\Frob_{\pi_j(1)}) = \\
\phi_{\{\pi_i(1)\pi_i(2) : i \in [s]\}; a}(\mathfrak{G})(\Frob_{\pi_j(1)}(\chi_{x_0})) + \phi_{[s]}(C; \mathbf{t})(\Frob_{\pi_j(1)}(\chi_{x_0})).
\end{multline*}
\end{enumerate}
\end{mydef}

Our definition of types is inspired by the notion of ``strict class'' from \cite[Section 8]{Smi1}.

For a cube $C$, we denote by $S(C)$ the subset of $v \in \Omega_{F(C)}$ where $v(ad)$ is odd. We make two important notational remarks that will help the reader navigate through the arguments in this section. We write $G(C) := \Gal(F(C)/\mathbb{F}_q(T))$. 

\begin{remark} 
\label{remark: phi g}
Let $C$ be desirable of type $\mathbf{t}$. Recall from Section \ref{section: expansion maps} that the continuous $1$-cochain
$$
\phi := \phi_{\{\pi_i(1) \pi_i(2) : i \in [s]\}; a}(\mathfrak{G}) + \phi_{[s]}(C; \mathbf{t})
$$
restricts to a quadratic character of $G_{F(C)}$, that is, an element of $\Gamma_{\mathbb{F}_2}(F(C))$. Recall that we have a natural action of $G(C)$ on $\Gamma_{\mathbb{F}_2}(F(C))$, where an element $g \in G(C)$ acts on a character $\rho \in \Gamma_{\mathbb{F}_2}(F(C))$ by sending it to the character
$$
\rho^g(\tau) := \rho(g^{-1} \tau g),
$$
where on the right hand side $g$ is viewed as an element of $G_{\mathbb{F}_q(T)}$ that projects to $g$: it is immediate to check that, since $\rho$ is a character, the choice of the lift for $g$ is immaterial yielding indeed a well-defined action of $G(C)$ on $\Gamma_{\mathbb{F}_2}(F(C))$.

Hence later, when we write $\phi^g$, we mean that we are restricting $\phi$ to $G_{F(C)}$, obtaining in this way a quadratic character of $G_{F(C)}$ in view of Section \ref{section: expansion maps}, and then we are acting upon $\phi$ with $g$ in the general way that we have just explained.  
\end{remark}

\begin{remark} 
In Section \ref{Section: Artin pairings}, our notation for the maps $\Pi_v$ and $\Pi$ did not keep track of the dependency on the character $\chi$, as that was fixed in that section. However, $\chi$ will vary in this section and hence it will be important to keep track of the dependence on $\chi$. For this reason, we will denote these maps by $\Pi_v(\chi)$ and $\Pi(\chi)$.
\end{remark}

\begin{proposition} 
\label{prop: sums of psi trivializes alpha alpha bar}
Let $h \in \{1, -1\}$. Let $s \in \Z_{\geq 2}$, let $\mathbf{t} \in N[2^{s + 1}]^{[r + 1]}$ and let 
$$
C := \left(\prod_{i = 1}^s \{\pi_i(1), \pi_i(2)\}\right) \times \{a\} \times \{d\}
$$
be an acceptable, profitable cube of type $\mathbf{t}$. Set 
$$
S := \{\pi_1(1)\pi_1(2), \ldots, \pi_s(1)\pi_s(2)\}. 
$$
Then the following conclusions hold:
\begin{enumerate}
\item[(i)] The map
\begin{equation}
\label{ex0}
\psi_{s + 1}(x_0) := \phi_{S; a}(\mathfrak{G}) + \phi_{[s]}(C; \mathbf{t}) - \sum_{x \in C - \{x_0\}} \mathrm{sg}(x) \psi_{s + 1}(x)
\end{equation}
defines an $h$-raw cocycle for $x_0$ of type $\mathbf{t}$. Furthermore, for each $i \in [s]$
\begin{multline}
\label{eGreatPi}
\Pi_{\pi_i(1)}(\chi_{x_0})(\psi_{s + 1}(x_0)) = \sum_{x \in C_i} \mathrm{sg}(x) \psi_{s + 1}(x)(\Frob_{\pi_i(1)}) + \\
\phi_{S; a}(\mathfrak{G})(\Frob_{\pi_i(1)}(\chi_{x_0})) + \phi_{[s]}(C; \mathbf{t})(\Frob_{\pi_i(1)}(\chi_{x_0})).
\end{multline}
\item[(ii)] If $C$ also satisfies $(D)$, then the element $\psi_{s + 1}(x_0)$ from part $(i)$ lies in $2 \cdot B_h(\chi_{x_0})$.
\item[(iii)] Keep assuming that $C$ satisfies $(D)$. Let $\psi_{s + 2}(x_0)$ be any choice of a cocycle guaranteed by part $(ii)$. Then the $1$-cochain
$$
\tilde{\phi} := \frac{\phi_{S; a}(\mathfrak{G}) + \phi_{[s]}(C; \mathbf{t})}{4} - \sum_{x \in C} \mathrm{sg}(x) \psi_{s + 2}(x)
$$
takes values in $N_h[2]$. Furthermore, the restriction of $\tilde{\phi}$ to $G_{F(C)}$ is an element of $C^1(G_{F(C), S(C)}, N_h[2])$ satisfying
\begin{align*}
\mathrm{d}\tilde{\phi} &= \left(\frac{(1 - h)}{2} \cdot \chi_{\epsilon_{\mathbb{F}_q}} + \chi_{\pi_1(1) \cdots \pi_s(1) d} + \sum_{g \in G(C) - \{1\}} (\phi_{S; a}(\mathfrak{G}) + \phi_{[s]}(C; \mathbf{t}))^g\right) \\
& \quad \cup (\phi_{S; a}(\mathfrak{G}) + \phi_{[s]}(C; \mathbf{t})).
\end{align*}
\end{enumerate}
\end{proposition}

\begin{proof}
To lighten up the notation, we introduce for each $T \subseteq [s]$ 
$$
\phi_T^{\text{m}} := \phi_{\{\pi_i(1)\pi_i(2) : i \in T\}; a}(\mathfrak{G}) + \phi_T(C; \mathbf{t}).
$$
We will first prove $(i)$. To this end, let us first verify that $\psi_{s + 1}(x_0)(\sigma_v) = 0$ for all $v$ such that $v(x_0)$ is even. Observe that such a $v$ must necessarily be a finite place. If $v$ does not divide $a \cdot d \cdot \prod_{j \in [s]} \pi_j(1)\pi_j(2)$, this is true for each of the $\psi_{s + 1}(x)$ for $x \in C - \{x_0\}$ by definition of $B_h(\chi_x)$ and it also holds for $\phi_{S; a}(\mathfrak{G})$, since $v$ is finite. Thus it remains to check that $\psi_{s + 1}(x_0)(\sigma_v) = 0$ for $v \in \{\pi_1(2), \dots, \pi_s(2)\}$, which follows from equation \eqref{ex0} and $(T)$. Finally, we also have $\psi_{s + 1}(x_0)(\sigma_{\mathrm{pr}_i(x_0)}) = \mathrm{pr}_i(\mathbf{t})$ for all $i \in [r]$ and $\psi_{s + 1}(x_0)(\Frob_q) = \mathrm{pr}_{r + 1}(\mathbf{t})$ as a consequence of equation \eqref{ex0} and $(T)$, so $\psi_{s + 1}(x_0)$ is of type $\mathbf{t}$.

We now verify that $\psi_{s + 1}(x_0) \in \text{Cocy}(G_{\mathbb{F}_q(T)}, N_h(\chi_{x_0}))$. To this end, let us take $\sigma, \tau \in G_{\mathbb{F}_q(T)}$. By Theorem \ref{tCocyComp} and \cite[Equation 2.3]{KP}, we have that $\mathrm{d}_{x_0}(\psi_{s + 1}(x_0))(\sigma, \tau)$ equals
$$
\sum_{\emptyset \neq T \subseteq [s]} (h \cdot q_0)^{\chi_{\Frob}(\sigma)} \cdot (-1)^{\chi_{x_0}(\sigma)} \cdot \chi_T(\sigma) \cdot \left( \phi_{[s] - T}^{\text{m}}(\tau) - \sum_{x \in C_T} \mathrm{sg}(x) \psi_{s + 1-|T|}(x)(\tau) \right).
$$
We claim that this equals zero. Indeed, we will show by downward induction that
$$
\phi_{[s] - T}^{\text{m}}(\tau) - \sum_{x \in C_T} \mathrm{sg}(x) \psi_{s + 1 - |T|}(x)(\tau) = 0
$$
for all $\emptyset \subset T \subseteq [s]$. For $T = [s]$, this follows from the equality $\psi_1(x) = \chi_a$. Define $x_T := (\prod_{j \in T} \pi_j(2)) \cdot (\prod_{j \in [s] - T} \pi_j(1)) \cdot a \cdot d$. For $\emptyset \subset T \subset [s]$, it then follows from Theorem \ref{tCocyComp} and induction that
$$
\mathrm{d}_{x_T}\left(\phi_{[s] - T}^{\text{m}} - \sum_{x \in C_T} \mathrm{sg}(x) \psi_{s + 1 - |T|}(x)\right) = 0,
$$
and hence $\phi_{[s] - T}^{\text{m}}(\tau) - \sum_{x \in C_T} \mathrm{sg}(x) \psi_{s + 1 - |T|}(x)(\tau)$ is a quadratic character. By using property $(T)$ from Definition \ref{desirable cube}, we see that this quadratic vanishes on every element of $\mathfrak{G}$, and hence must be zero, completing the induction.

Finally, we observe that
$$
\psi_1(x_0) = 2^s \cdot \left(\phi_{[s]}^{\text{m}} - \sum_{x \in C - \{x_0\}} \mathrm{sg}(x) \psi_{s + 1}(x)\right) = -(2^s-1) \cdot \chi_a = \chi_a,
$$
where in the second and third equality we have used that $s$ is a positive integer. We now want to show the final conclusion of part $(i)$. To this end, we first make a general observation that will also be repeatedly used in the rest of the proof. We compute for each place $v \in \text{Ram}(\chi_{x_0})$
\begin{align}
\Pi_v(\chi_{x_0})(\psi_{s + 1}(x_0)) &= \frac{(h \cdot q_0)^{\deg(v)} - 1}{2} \cdot \psi_{s + 1}(x_0)(\sigma_v) + \psi_{s + 1}(x_0)(\Frob_v(\chi_{x_0})) \nonumber \\
&= \frac{(h \cdot q_0)^{\deg(v)} - 1}{2} \cdot \left(\phi_{[s]}^{\text{m}}(\sigma_v) - \sum_{x \in C - \{x_0\}} \mathrm{sg}(x) \psi_{s + 1}(x)(\sigma_v)\right) \nonumber \\
&\quad + \left(\phi_{[s]}^{\text{m}} - \sum_{x \in C - \{x_0\}} \mathrm{sg}(x) \psi_{s + 1}(x) \right)(\Frob_v(\chi_{x_0})) \nonumber \\ 
&= \frac{(h \cdot q_0)^{\deg(v)} - 1}{2} \cdot \phi_{[s]}^{\text{m}}(\sigma_v) + \phi_{[s]}^{\text{m}}(\Frob_v(\chi_{x_0})) \nonumber \\
&\quad - \sum_{x \in C - \{x_0\}} \mathrm{sg}(x) \frac{(h \cdot q_0)^{\deg(v)} - 1}{2} \cdot \psi_{s + 1}(x)(\sigma_v) \nonumber \\
&\quad - \sum_{x \in C - \{x_0\}} \mathrm{sg}(x) \psi_{s + 1}(x)(\Frob_v(\chi_{x_0})). \label{ePix0}
\end{align}

Let now $i \in [s]$. Observe that $\Frob_{\pi_i(1)}$ is an element of $G_{F(C_i)}$ thanks to condition $(A1)$ of Definition \ref{def: acceptable cube}. We see that the first term is zero by the fact that $\pi_i(1)$ has even degree and $\phi_{[s]}^{\text{m}}$ is $2$-torsion. Furthermore, for each $x \in C - C_{i} - \{x_0\}$, the identity $\Frob_{\pi_i(1)}(\chi_x) = \Frob_{\pi_i(1)}(\chi_{x_0})$ shows that the corresponding contribution equals
$$
\Pi_{\pi_i(1)}(\chi_x)(\psi_{s + 1}(x)) = 0,
$$
where the equality follows from the existence of $\psi_{s + 2}(x)$ and Proposition \ref{prop: Pi detects 2}. For each $x \in C_i$ we have that $\psi_{s + 1}(x)(\sigma_{\pi_i(1)}) = 0$ in view of the definition of $B_h(\chi_x)$. This gives that
\begin{align*}
\Pi_{\pi_i(1)}(\chi_{x_0})(\psi_{s + 1}(x_0)) &= -\sum_{x \in C_i} \mathrm{sg}(x) \psi_{s + 1}(x)(\Frob_{\pi_i(1)}(\chi_{x_0})) + \phi_{[s]}^{\text{m}}(\Frob_{\pi_i(1)}(\chi_{x_0})) \\
&= -\sum_{x \in C_i} \mathrm{sg}(x) \psi_{s + 1}(x)(\Frob_{\pi_i(1)}) + \phi_{[s]}^{\text{m}}(\Frob_{\pi_i(1)}(\chi_{x_0})),
\end{align*}
where the last equality follows immediately from the fact that $\pi_i(1)$ is unramified in the compositum $\prod_{x \in C_i} L(\psi_{s + 1}(x))$. Since $\Pi_{\pi_i(1)}(\chi_{x_0})(\psi_{s + 1}(x_0))$ and $\phi_{S; a}(\mathfrak{G})(\Frob_{\pi_i(1)}(\chi_{x_0}))$ are in $\mathbb{F}_2$, we may omit the minus sign and we have thus established part $(i)$. 

We now verify part $(ii)$. For that purpose, in view of Proposition \ref{prop: Pi detects 2}, we need to check that for each place $v$ of $\text{Ram}(\chi_{x_0})$ 
$$
\Pi_v(\chi_{x_0})(\psi_{s + 1}(x_0)) = 0.
$$
Let us now distinguish two cases. 

\paragraph{Case 1.} Consider first the case that $v$ is in $\text{Ram}(\chi_{ad})$ (which always includes $v = \infty$). We start by analyzing the subcase where $v$ is finite. Observe that by definition of normalized expansion maps, the expansion map $\phi_{[s]}^{\text{m}}$ vanishes at $\sigma_v$. Furthermore, $\phi_{[s]}^{\text{m}}$ also vanishes at $\Frob_v(\chi_{x_0})$ by acceptability and profitability, as we now explain. 

Indeed, observe that $\phi_{[s]}^{\text{m}}$ restricts to a character of $G_{\mathbb{F}_q(T)_v}$ in view of $(A1)$ of Definition \ref{def: acceptable cube} and the defining equations for governing expansions as explained in Section \ref{section: expansion maps}. By definition of normalized expansion maps, this character vanishes on $\sigma_v$ (for finite $v$) and is hence unramified. Therefore we have that 
$$
\phi_{[s]}^{\text{m}}(\Frob_v(\chi_{x_0})) = \phi_{[s]}^{\text{m}}(\Frob_v).
$$
The right hand side needs to be $0$, since if it were non-zero, then the place $|\cdot|_v \circ i_v$ of $F(C)$ would be inert in the quadratic extension of $F(C)$ cut out by $\phi_{[s]}^{\text{m}}$, and hence the residue field degree of
$$
L(\phi_{[s]}^{\text{m}})/\mathbb{F}_q(T)
$$
at $v$ would be bigger than $1$, contradicting the acceptable condition of Definition \ref{def: acceptable cube} and profitability. This shows that the contribution from $\phi_{[s]}^{\text{m}}$ to equation \eqref{ePix0} vanishes. 

Next, we observe that $\Frob_v(\chi_x) = \Frob_v(\chi_{x_0})$, since $\chi_x$ and $\chi_{x_0}$ restrict to the same character locally at $v$ by $(A1)$ of Definition \ref{def: acceptable cube}. Combining Proposition \ref{prop: Pi detects 2} with the equation $\psi_{s + 1}(x) = 2 \cdot \psi_{s + 2}(x)$ for $x \in C - \{x_0\}$, we deduce that
\begin{equation}
\label{eLiftFinite}
\Pi_v(\chi_{x_0})(\psi_{s + 1}(x)) = \Pi_v(\chi_x)(\psi_{s + 1}(x)) = 0
\end{equation}
for all $x \in C - \{x_0\}$. Altogether we have shown that $\Pi_v(\chi_{x_0})(\psi_{s + 1}(x_0)) = 0$ for each finite place $v$ dividing $ad$. 

We now explain the necessary modifications for the infinite place $v = \infty$. Equation \eqref{eLiftFinite} remains valid without changes. The final part of $(A3)$ and profitability enforce vanishing of $\phi_{[s]}^{\text{m}}$ locally at infinity.

\paragraph{Case 2.} We are now left with the case that $v$ is in $\text{Ram}(\chi_{x_0})$ but not in $\text{Ram}(\chi_{ad})$. This means that $v := \pi_j(1)$ for some $j \in [s]$. Then the desired vanishing follows immediately from equation \eqref{eGreatPi} and part $(D)$ of Definition \ref{desirable cube}, ending the proof of part $(ii)$.

We now proceed with the proof of part $(iii)$. We will first show that $2 \cdot \tilde{\phi} = 0$. But indeed, we have
$$
2 \cdot \tilde{\phi} = \phi_{[s]}^{\text{m}} - \sum_{x \in C} \mathrm{sg}(x) \psi_{s + 1}(x) = 0,
$$
where the last equality follows directly from the definition of $\psi_{s + 1}(x_0)$ made in part $(i)$. 

We next verify that $\tilde{\phi}$ restricted to $G_{F(C)}$ factorizes through $G_{F(C), S(C)}$. To this end, we claim that $L(\tilde{\phi})/\mathbb{F}_q(T)$ has ramification order $2$ at the $2s$ places $\{\pi_j(1), \pi_j(2) : j \in [s]\}$ and is unramified for all places $v$ outside of $\text{Ram}(\chi_{ad}) \cup \{\pi_j(1), \pi_j(2) : j \in [s]\}$. Let us first explain why the claim gives the desired factorization. Indeed, $F(C)/\mathbb{F}_q(T)$ is ramified at the $2s$ places $\{\pi_j(1), \pi_j(2) : j \in [s]\}$, hence the claim shows that $L(\tilde{\phi})/F(C)$ is unramified outside $S(C)$. Since the restriction of $\tilde{\phi}$ to $G_{F(C)}$ clearly factorizes through $\Gal(L(\tilde{\phi})/F(C))$, this shows that the claim is sufficient to give the desired factorization.

We now prove the claim. Observe that $L(\tilde{\phi})$ is contained in
$$ 
L(\phi_{[s]}^{\text{m}}) \cdot \prod_{x \in C} L(\psi_{s + 2}(x)).
$$
The elements $\sigma_{\pi_i(1)}$ and $\sigma_{\pi_i(2)}$ project to an involution of $\Gal(L(\phi_{[s]}^{\text{m}})/\mathbb{F}_q(T))$ for each $i \in [s]$, while $\sigma_v$ projects to the trivial element for every other finite place $v$ not in $\{\pi_j(1), \pi_j(2) : j \in [s]\} \cup \text{Ram}(\chi_{ad})$. Moreover, observe that for each $x \in C$ and each $v \in \text{Ram}(\chi_x)$, we have that $\sigma_v$ is an involution in $\Gal(L(\psi_{s + 2}(x))/\mathbb{F}_q(T))$ and is the identity element otherwise. This establishes the claim. 

To complete the proof of $(iii)$, we only need to compute $\mathrm{d} \tilde{\phi}$ when restricted to $G_{F(C)}$, which we do next. Observe that the modules $N_h(\chi_{x})$, as $x$ varies in $C$, become all isomorphic as $G_{F(C)}$-modules, since $\chi_x$ all restrict to the same common character in $\Gamma_{\mathbb{F}_2}(F(C))$. It follows that all the $\psi_{s + 2}(x)$ are elements of $\text{Cocy}(G_{F(C)}, N_h(\chi_{x_0}))$. Hence we have that
\begin{equation}
\label{eTPsiDef}
\tilde{\psi} := \sum_{x \in C} \mathrm{sg}(x) \psi_{s + 2}(x) \in \text{Cocy}(G_{F(C)}, N_h(\chi_{x_0})[4])
\end{equation}
with $2 \cdot \tilde{\psi} = \phi_{[s]}^{\text{m}}$. Thus, invoking Proposition \ref{prop: description of N[4] cocycles}, we write
\begin{equation}
\label{eTPsiFormula}
\tilde{\psi} = \frac{\phi_{[s]}^{\text{m}}}{4} + \tilde{\phi},
\end{equation}
with $\tilde{\phi}$ an $N_h[2]$-valued cochain satisfying
$$
\mathrm{d} \tilde{\phi}(\sigma,\tau) = \left(\frac{1 - h}{2} \cdot \chi_{\epsilon_{\mathbb{F}_q}}(\sigma) + \chi_{x_0}(\sigma) + \phi_{[s]}^{\text{m}}(\sigma) \right) \cdot \phi_{[s]}^{\text{m}}(\tau).
$$
Using the relation $\chi_a = \sum_{g \in G(C)} {(\phi_{[s]}^{\text{m}})}^g$ (which follows for example from equation \eqref{eConjugateExpansion}), we obtain the desired conclusion. 
\end{proof}

We are now ready to state and prove our first reflection principle. In this section, we will abuse notation by writing
$$
\textup{Art}_{D, n + 1}^{(i)} \circ (\textup{id} \times \textup{Sym}(\chi_D)) \colon C_n^{(i)}(\chi_D)^\circ \times C_n^{(i)}(\chi_D)^\circ \to \mathbb{F}_2
$$
simply as $\textup{Art}_{D, n + 1}^{(i)}$, leaving the application of $\textup{Sym}(\chi_D)$ implicit. We refer the reader to Remark \ref{remark: phi g} for clarifying the notation in the next theorem statement.

\begin{theorem} 
\label{theorem: reflection principle}
Let $h \in \{1, -1\}$. Let $s \in \Z_{\geq 2}$ and let $\mathbf{t} \in N[2^{s + 1}]^{[r + 1]}$. Let  
$$
C := \left(\prod_{i = 1}^s \{\pi_i(1), \pi_i(2)\}\right) \times \{a\} \times \{d\}
$$
be an acceptable, profitable, desirable cube of type $\mathbf{t}$. Set $\phi := \phi_{\{\pi_i(1) \pi_i(2) : i \in [s]\}; a}(\mathfrak{G}) + \phi_{[s]}(C; \mathbf{t})$. Then the following conclusions hold:
\begin{enumerate}
\item[1.] The triple $\left(\frac{(1 - h)}{2} \cdot \chi_{\epsilon_{\mathbb{F}_q}} + \chi_{\pi_1(1) \cdots \pi_s(1) d} + \sum_{g \in G(C) - \{1\}} \phi^g, \phi, \phi \right)$ is R\'edei admissible over $F(C)$. 
\item[2.] There holds
$$
\sum_{x \in C} \mathrm{Art}_{x, s + 2}^{(h)}(\chi_a, \chi_a) = \left[\frac{(1 - h)}{2} \cdot \chi_{\epsilon_{\mathbb{F}_q}} + \chi_{\pi_1(1) \cdots \pi_s(1)d} + \sum_{g \in G(C) - \{1\}} \phi^g, \phi, \phi \right]_{F(C)}.
$$
\end{enumerate}
\end{theorem}

\begin{proof}
We now prove part $1$ and show that the triple 
$$
\left(\frac{(1 - h)}{2} \cdot \chi_{\epsilon_{\mathbb{F}_q}} + \chi_{\pi_1(1) \cdots \pi_s(1) d} + \sum_{g \in G(C) - \{1\}} \phi^g, \phi, \phi \right)
$$ 
is R\'edei admissible over $F(C)$. Note that the cup product 
$$
\left(\frac{(1 - h)}{2} \cdot \chi_{\epsilon_{\mathbb{F}_q}} + \chi_{\pi_1(1) \cdots \pi_s(1) d} + \sum_{g \in G(C) - \{1\}} \phi^g \right) \cup \phi
$$
is trivial in $H^2(G_{F(C)}, \mathbb{F}_2)$, since we have proved in Proposition \ref{prop: sums of psi trivializes alpha alpha bar} that it equals $\mathrm{d} \tilde{\phi}$. Moreover, we know that the ramification loci of the above two characters are disjoint (which, combined with the triviality in $H^2(G_{F(C)}, \mathbb{F}_2)$, gives the desired vanishing at all places except for the places above infinity). Since $\phi$ is also locally trivial at infinity by part $(A3)$ of Definition \ref{def: acceptable cube} and profitability, this gives R\'edei admissibility, ending the proof of part $1$.

We now prove part $2$. By definition, we have
\begin{align}
\sum_{x \in C} \mathrm{Art}_{x, s + 2}^{(h)}(\chi_a, \chi_a) &= \sum_{x \in C} \sum_{v \in \text{Ram}(\chi_a)} \Pi_v(\chi_x)(\psi_{s + 2}(x)) \nonumber \\ 
&= \sum_{v \in \text{Ram}(\chi_a)} \sum_{x \in C} \Pi_v(\chi_x)(\psi_{s + 2}(x)). \label{eArtRewrite1}
\end{align}
Observe that, when restricted to $v \in \text{Ram}(\chi_a)$, the character $\chi_x$ does not depend on the choice of $x \in C$ by part $(A1)$ of Definition \ref{def: acceptable cube}. Let $v \in \text{Ram}(\chi_a)$ and let $w \in \text{Ram}(\phi|_{G_{F(C)}})$ be any place above $v$ (so $w = \mathrm{AUp}(v)$ by Proposition \ref{normalized expansion maps are unramified}). We claim that
\begin{equation}
\label{eArtRewrite2}
\sum_{x \in C} \Pi_v(\chi_x)(\psi_{s + 2}(x)) = \mathbf{1}_{\deg(w) \equiv 1 \bmod 2, q_0 \equiv 5 \bmod 8} + \tilde{\phi}\left(\Frob_w\left(\chi_{x_0} + \frac{1 - h}{2} \cdot \chi_{\epsilon_{\mathbb{F}_q}}\right)\right).
\end{equation}
To show the validity of equation \eqref{eArtRewrite2}, we fix an embedding, which we will denote by $i_w$, of $\mathbb{F}_q(T)^{\text{sep}}$ into $\mathbb{F}_q(T)_v^{\text{sep}}$ that induces on $F(C)$ precisely the place $w$. The map $i_w$ induces, by Krasner's lemma, an embedding
$$
i_w^\ast: G_{\mathbb{F}_q(T)_v} \to G_{\mathbb{F}_q(T)},
$$
where the image of $i_w^\ast$ will be a conjugate subgroup of the image of $i_v^\ast$ (the embedding we fixed at the beginning of the paper in Section \ref{ssNot}). Recall that $\Pi_v(\chi_x)(\psi_{s + 2}(x))$ computes precisely $\inv_v(\theta(\psi_{s + 2}(x)))$, viewed as an element of the $2$-torsion of $H^2(G_{\mathbb{F}_q(T)_v}, \overline{\mathbb{F}_q(T)_v}^\ast)$. The value of this invariant does not depend on the choice of the embedding $i_w$, given that the resulting subgroups are all conjugate. Therefore by Proposition \ref{prop: detecting locally}
$$
\Pi_v(\chi_x)(\psi_{s + 2}(x)) = \inv_v(\theta(\psi_{s + 2}(x))) = \frac{(h \cdot q_0)^{\deg(v)} - 1}{2} \cdot \psi_{s + 2}(x)(\sigma_w) + \psi_{s + 2}(x)(\Frob_w(\chi_x)).
$$
Here $\Frob_w(\chi)$ and $\sigma_w$ are taken to be the conjugates of $\Frob_v(\chi)$ and $\sigma_v$ under the automorphism $g$ in $G_{\mathbb{F}_q(T)}$ that intertwines $i_v$ and $i_w$. 

Summing all of the contributions and using that the restriction of $\chi_x$ to $i_w^\ast(G_{\mathbb{F}_q(T)_v})$ equals $\chi_{x_0}$ for all $x \in C$ (in view of part $(A1)$ of Definition \ref{def: acceptable cube}), we have that
\begin{align}
\label{eSumRewrite}
\sum_{x \in C} \Pi_v(\chi_x)(\psi_{s + 2}(x)) &= \sum_{x \in C} \mathrm{sg}(x) \left(\frac{(h \cdot q_0)^{\deg(v)} - 1}{2} \cdot \psi_{s + 2}(x)(\sigma_w) + \psi_{s + 2}(x)(\Frob_w(\chi_x)) \right) \nonumber \\
&= \frac{(h \cdot q_0)^{\deg(v)} - 1}{2} \cdot \tilde{\psi}(\sigma_w) + \tilde{\psi}(\Frob_w(\chi_{x_0})),
\end{align}
where we recall the definition of $\tilde{\psi}$ from \eqref{eTPsiDef}. Let us also recall from equation \eqref{eTPsiFormula} that 
\begin{equation}
\label{ePsiExpand}
\tilde{\psi} = \frac{\phi}{4} + \tilde{\phi}.
\end{equation}
We also record the observation 
\begin{equation}
\label{eFrobTrivial}
\phi\left(\Frob_w\left(\chi_{x_0} + \frac{1 - h}{2} \cdot \chi_{\epsilon_{\mathbb{F}_q}}\right)\right) = 0
\end{equation}
in view of acceptability and profitability; here we use that $\chi_a$ is locally in the span of $\chi_{x_0} + \frac{1 - h}{2} \cdot \chi_{\epsilon_{\mathbb{F}_q}}$. In order to prove \eqref{eArtRewrite2}, we now distinguish five cases.

\paragraph{Case 1: $v$ has even degree.} In this case the term $\frac{(h \cdot q_0)^{\deg(v)} - 1}{2}$ is divisible by $4$ and therefore the only remaining term in \eqref{eSumRewrite} is
$$
\sum_{x \in C} \Pi_v(\chi_x)(\psi_{s + 2}(x)) = \tilde{\psi}(\Frob_w(\chi_{x_0})).
$$
Combining \eqref{eFrobTrivial} with the observation that $\chi_{\epsilon_{\mathbb{F}_q}}$ is locally trivial at $w$, we deduce that $\phi(\Frob_w(\chi_{x_0})) = 0$. So from \eqref{ePsiExpand}, we get precisely $\tilde{\phi}(\Frob_w(\chi_{x_0}))$ as was claimed in \eqref{eArtRewrite2}.

\paragraph{Case 2: $v$ has odd degree, $q_0 \equiv 1 \bmod 8$ and $h = 1$.} In Case 2, the exact same logic as Case 1 applies.

\paragraph{Case 3: $v$ has odd degree, $q_0 \equiv 1 \bmod 8$ and $h = -1$.} Using \eqref{eSumRewrite} as before, we get in this case
$$
\sum_{x \in C} \Pi_v(\chi_x)(\psi_{s + 2}(x)) = -\tilde{\psi}(\sigma_w) + \tilde{\psi}(\Frob_w(\chi_{x_0})).
$$
Expanding \eqref{ePsiExpand} in coordinates we get
$$
\phi(\sigma_w) + \tilde{\phi}(\sigma_w) + \frac{\phi(\sigma_w)}{4} + \tilde{\phi}(\Frob_w(\chi_{x_0})) + \frac{\phi(\Frob_w(\chi_{x_0}))}{4}.
$$
Notice also that since $w$ is in $\text{Ram}(\phi|_{G_{F(C)}})$, we have established already that $\tilde{\phi}$ is a character locally at $w$ when we checked that $(\phi + \chi_{x_0} + \frac{1 - h}{2} \cdot \chi_{\epsilon_{\mathbb{F}_q}}, \phi, \phi)$ is R\'edei admissible. It follows from this that
$$
\tilde{\phi}(\sigma_w) + \tilde{\phi}(\Frob_w(\chi_{x_0})) = \tilde{\phi}\left(\Frob_w\left(\chi_{x_0} + \frac{1 - h}{2} \cdot \chi_{\epsilon_{\mathbb{F}_q}}\right)\right),
$$
where in this equality we used that $w$ is of odd degree and therefore $\chi_{\epsilon_{\mathbb{F}_q}}$ restricts locally at $w$ to $\chi_{\epsilon_{\mathbb{F}_{q_w}}}$. Using the above and carrying out the addition modulo $4$ we get in this case precisely
$$
\tilde{\phi}(\Frob_w(\chi_{x_0} + \chi_{\epsilon_{\mathbb{F}_q}})),
$$
since the term $\phi(\sigma_w) + \phi(\sigma_w) \cdot \phi(\Frob_w(\chi_{x_0})) = 0$ disappears due to equation \eqref{eFrobTrivial}. This matches the claim from equation \eqref{eArtRewrite2}.

\paragraph{Case 4: $v$ has odd degree, $q_0 \equiv 5 \bmod 8$ and $h = 1$.} Upon combining \eqref{eSumRewrite} and \eqref{ePsiExpand}, we are left with
$$
\sum_{x \in C} \Pi_v(\chi_x)(\psi_{s + 2}(x)) = 2 \cdot \tilde{\psi}(\sigma_w) + \tilde{\psi}(\Frob_w(\chi_{x_0})) = \phi(\sigma_w) + \tilde{\phi}(\Frob_w(\chi_{x_0})) + \frac{\phi(\Frob_w(\chi_{x_0}))}{4}.
$$
Equation \eqref{eFrobTrivial} becomes $\phi(\Frob_w(\chi_{x_0})) = 0$ in this case. Moreover, $\phi$ ramifies at $w$, and thus $\phi(\sigma_w) = 1$. Hence we get
$$
1 + \tilde{\phi}(\Frob_w(\chi_{x_0})),
$$
which matches equation \eqref{eArtRewrite2}.

\paragraph{Case 5: $v$ has odd degree, $q_0 \equiv 5 \bmod 8$ and $h = -1$.} In this case equation \eqref{eSumRewrite} becomes
$$
\sum_{x \in C} \Pi_v(\chi_x)(\psi_{s + 2}(x)) = \tilde{\psi}(\sigma_w) + \tilde{\psi}(\Frob_w(\chi_{x_0})).
$$
Expanding in coordinates with equation \eqref{ePsiExpand} and using \eqref{eFrobTrivial}, we get
\begin{multline*}
\tilde{\phi}(\sigma_w) + \frac{\phi(\sigma_w)}{4} + \tilde{\phi}(\Frob_w(\chi_{x_0})) + \frac{\phi(\Frob_w(\chi_{x_0}))}{4} \\
= \tilde{\phi}(\sigma_w) + \tilde{\phi}(\Frob_w(\chi_{x_0})) + \phi(\Frob_w(\chi_{x_0})) \cdot \phi(\sigma_w) 
= \tilde{\phi}(\sigma_w) + \tilde{\phi}(\Frob_w(\chi_{x_0})) + \phi(\sigma_w).
\end{multline*}
Now recalling once more that $\tilde{\phi}$ is a character locally at $w$ and that $w$ is of odd degree, we obtain
$$
\tilde{\phi}(\sigma_w) + \tilde{\phi}(\Frob_w(\chi_{x_0})) = \tilde{\phi}(\Frob_w(\chi_{x_0} + \chi_{\epsilon_{\mathbb{F}_q}})),
$$
precisely as explained above. Furthermore, once more, as $\phi$ ramifies at $w$, we have that $\phi(\sigma_w) = 1$. All in all, we are left with 
$$
1 + \tilde{\phi}(\Frob_w(\chi_{x_0} + \chi_{\epsilon_{\mathbb{F}_q}})),
$$
as desired. This concludes the proof of all cases, and thus the proof of \eqref{eArtRewrite2}. 

We now further claim
\begin{equation}
\label{eArtRewrite3}
\tilde{\phi}\left(\Frob_w\left(\chi_{x_0} + \frac{1 - h}{2} \cdot \chi_{\epsilon_{\mathbb{F}_q}}\right)\right) = \inv_w(\tilde{\phi} \cup \phi).
\end{equation}
Indeed locally at $w$, $\phi$ must coincide with $\chi_{x_0} + \frac{1 - h}{2} \cdot \chi_{\epsilon_{\mathbb{F}_q}}$ in virtue of acceptability and profitability. Recalling that the local Tate pairing on $H^1(G_{F(C)_w}, \mathbb{F}_2)$ is alternating, since $q \equiv 1 \bmod 4$. Thus we see that both sides of \eqref{eArtRewrite3} measure whether $\tilde{\phi}$ is in the span of $\chi_{x_0} + \frac{1 - h}{2} \cdot \chi_{\epsilon_{\mathbb{F}_q}}$.

Having shown the claim, we put equations \eqref{eArtRewrite1}, \eqref{eArtRewrite2} and \eqref{eArtRewrite3} together to conclude that
\begin{align*}
\sum_{x \in C} \mathrm{Art}_{x, s + 2}^{(h)}(\chi_a, \chi_a) &= \sum_{\substack{w \in \text{Ram}(\phi|_{G_{F(C)}})}} \left(\mathbf{1}_{\deg(w) \equiv 1 \bmod 2, q_0 \equiv 5 \bmod 8} + \inv_w(\tilde{\phi} \cup \phi)\right) \\
&= \sum_{\substack{w \in \text{Ram}(\phi|_{G_{F(C)}})}} \inv_w(\tilde{\phi} \cup \phi) \\
&= \left[\frac{(1 - h)}{2} \cdot \chi_{\epsilon_{\mathbb{F}_q}} + \chi_{\pi_1(1) \cdots \pi_s(1) d} + \sum_{g \in G(C) - \{1\}} \phi^g, \phi, \phi \right]
\end{align*}
by Proposition \ref{prop: trivialization in general}. This is the desired conclusion and ends the proof. 
\end{proof}

To explain our next step, we shall introduce some notation. Let $F$ be a global function field containing a primitive fourth root of unity that is Galois over $\mathbb{F}_q(T)$. Let $H_2(F)/F$ be the largest abelian exponent $2$ extension of $F$ that is unramified everywhere. Let $S_\emptyset$ be a $\Gal(F/\mathbb{F}_q(T))$-invariant set of places of $F$ whose Frobenius elements generate $\Gal(H_2(F)/F)$. Notice that, by Nakayama's lemma, this implies that they generate also $\Gal(H_4(F)/F)$, where $H_4(F)/F$ is the largest abelian exponent $4$ extension of $F$ that is unramified everywhere. 

Moreover, for each place $v \in \Omega_F - S_\emptyset$ we fix once and for all an element $\psi_{S_\emptyset, v}$ in $\Gamma_{\Z/4\Z}(F)$ totally ramified at $v$ and unramified at all places in $\Omega_F - S_\emptyset - \{v\}$. We further normalize this choice as follows. Fix a primitive fourth root of unity $\zeta \in F$. We demand that $\psi_{S_\emptyset, v}$ satisfies $\psi_{S_\emptyset, v}(\sigma) = 1$ for the topological generator $\sigma$ of tame inertia satisfying $\sigma(\sqrt[4]{\pi_v})/\sqrt[4]{\pi_v} = \iota_v(\zeta)$ for all uniformizers $\pi_v$. 

Note that $\psi_{S_\emptyset, v}$ exists thanks to the vanishing of $\Sha^1(G_{F, S_\emptyset}, \Z/4\Z) = 0$ and two applications of \cite[Theorem 8.7.9]{NSW} (with respectively $S_\emptyset$ and $S_\emptyset \cup \{v\}$). We put $\chi_{S_\emptyset, v} := 2 \cdot \psi_{S_\emptyset, v}$. For each character $\chi$ ramified only at $S_\emptyset$, we define $\xi_{S_\emptyset}(\chi)$ to be any element of $\Gamma_{\Z/4\Z}(F)$ ramified only at $S_\emptyset$ and satisfying $2 \cdot \xi_{S_\emptyset}(\chi) = \chi$. 
Our next step is to apply Theorem \ref{theorem: Redei reciprocity} to the output of Theorem \ref{theorem: reflection principle} and obtain
\begin{multline*}
\left[ \frac{(1 - h)}{2} \cdot \chi_{\epsilon_{\mathbb{F}_q}} + \chi_{\pi_1(1) \cdots \pi_s(1) d} + \sum_{g \in G - \{1\}} \phi^g, \phi, \phi \right] = \\
\left[ \phi , \phi, \frac{(1 - h)}{2} \cdot \chi_{\epsilon_{\mathbb{F}_q}} + \chi_{\pi_1(1) \cdots \pi_s(1) d} + \sum_{g \in G - \{1\}} \phi^g \right].
\end{multline*}
The right hand side is often called a quartic spin symbol in the literature. Our idea here is that if one is able to break $a = a_0rs$ for two suitable primes $r$ and $s$, then one can \emph{factorize} such a quartic symbol, in multiplicative notation, as 
$$
\lambda_r\cdot \lambda_s \cdot \left(\frac{\alpha_r}{\alpha_s}\right)_{F(C),2},
$$
that is, as a bilinear Legendre symbol expression over $F(C)$, where $\alpha_r, \alpha_s$ can be intuitively thought as primes of $F(C)$ lying above $r$ and $s$. This desire to split the quartic spin symbol in bilinear Legendre symbols is formalized in the following result. 

Now, for each $C := \{\pi_1(1), \pi_1(2)\} \times \cdots \times \{\pi_s(1), \pi_s(2)\}$, we fix once and for all a set $S_\emptyset$ for $F := F(C)$ as above. Finally, for each finite place $r$ in $\Omega_{\mathbb{F}_q(T)}^{\text{fin}}$ not lying below $S_\emptyset$, we put 
$$
\psi_{S_\emptyset, r} := \psi_{S_\emptyset, \text{AUp}(r)}, \quad \quad \chi_{S_\emptyset, r} := \chi_{S_\emptyset, \text{AUp}(r)}.
$$

\begin{theorem} 
\label{theorem: Applying quartic reciprocity}
Let $C := \{\pi_1(1), \pi_1(2)\} \times \cdots \times \{\pi_s(1), \pi_s(2)\}$ be a cube of $2s$ distinct monic even degree irreducible polynomials and let $\mathbf{t} \in \mathbb{F}_2^{[s]}$. Let $S_\emptyset \subseteq \Omega_{F(C)}$ be any $G(C)$-invariant set of finite places generating the Galois group $\Gal(H_2(F(C))/F(C))$. Let $a'$ be a squarefree monic polynomial coprime with $\pi_1(1), \pi_1(2), \dots, \pi_s(1), \pi_s(2)$ and $S_\emptyset$, and let $S(a')$ be the set of places of $F(C)$ above an irreducible polynomial dividing $a'$. Fix choices 
$$
(L_{i, v})_{v \in S_\emptyset \cup S(a')} \in \prod_{v \in S_\emptyset \cup S(a')} \Gamma_{\Z/4\Z}(F(C)_v)
$$
for each $i \in \{1, 2, 3\}$. Define
$$
\Delta_C(r) := \sum_{g \in G(C) - \{1\}} \psi_{S_\emptyset, r}^g(\Frob_{\mathrm{AUp}(r)}).
$$
Then, for each $\mathbf{a}, \mathbf{b} \in \mathbb{F}_2^{S_\emptyset}$, there is a function
$$
\lambda_{\mathbf{a}, \mathbf{b}, a', S_\emptyset} \colon \prod_{v \in S_\emptyset \cup S(a')} \Gamma_{\Z/4\Z}(F(C)_v)^3 \rightarrow \frac{1}{4}\Z/\Z
$$
such that for all $r_1, r_2, r_3$ satisfying the conditions
\begin{enumerate}
\item the $r_i$ are monic irreducible polynomials that are pairwise coprime, coprime to $a'$, coprime to $\pi_1(1) \pi_1(2), \dots, \pi_s(1) \pi_s(2)$ and do not divide any place in $S_\emptyset$,
\item defining $\phi := \phi_{C; a'r_1r_2r_3}(\mathfrak{G}) + \sum_{j \in [s]} \mathrm{qpr}_j(\mathbf{t}) \cdot \phi_{\{\pi_i(1) \pi_i(2) : i \in [s] - \{j\}\}; \pi_j(1) \pi_j(2)}(\mathfrak{G})$, the triple 
$$
\left( \sum_{g \in G(C) - \{1\}} \phi^g, \phi, \phi \right)_{F(C)} 
$$
is R\'edei admissible and $\phi$ is locally trivial at $\infty$,
\item the characters $\psi_{S_\emptyset, r_i}$ restrict respectively to $L_{i, v}$ locally at each $v \in S_\emptyset \cup S(a')$,
\item we have $\phi(\Frob_v) = \mathrm{pr}_v(\mathbf{a})$ and $\sum_{g \in G(C) - \{1\}} \phi^g(\Frob_v) = \mathrm{pr}_v(\mathbf{b})$ for all $v \in S_\emptyset$,
\end{enumerate}
we have 
\begin{multline*}
\left[ \sum_{g \in G(C) - \{1\}} \phi^g, \phi, \phi \right]_{F(C)} - \lambda_{\mathbf{a}, \mathbf{b}, a', S_\emptyset}((L_{i, v})_{i \in [3], v \in S_\emptyset \cup S(a')}) - \sum_{i = 1}^3 \Delta_C(r_i) = \\ \sum_{g \in G(C) - \{1\}} \sum_{1 \leq i < j \leq 3} \inv_{\mathrm{AUp}(r_j)}\left(\chi_{S_\emptyset, r_i}^g \cup \chi_{S_\emptyset, r_j}\right).
\end{multline*}
\end{theorem}

\begin{proof}
Our first step is to apply Theorem \ref{theorem: Redei reciprocity} and rewrite our R\'edei symbol as
$$
\left[ \sum_{g \in G(C) - \{1\}} \phi^g, \phi, \phi \right]_{F(C)} = \left[ \phi, \phi, \sum_{g \in G(C) - \{1\}} \phi^g \right]_{F(C)}.
$$
We introduce
$$
\chi := \phi + \sum_{v \mid a'r_1r_2r_3} \chi_{S_\emptyset, v},
$$
and
$$
\psi := \xi_{S_\emptyset}(\chi) + \sum_{v \mid a'r_1r_2r_3} \psi_{S_\emptyset, v}.
$$
We claim that $\xi_{S_\emptyset}(\chi)$ is entirely determined by $a'$, $(L_{i, v})_{i \in [3], v \in S_\emptyset}$ and $\mathbf{a}$. Indeed, the restriction of $\chi$ locally at $S_\emptyset$ is certainly determined by this data. This then uniquely pins down $\chi$ thanks to $\Sha^1(G_{F(C), S_\emptyset}, \FF_2) = 0$ and the fact that $\chi$ is unramified outside of $S_\emptyset$ by Proposition \ref{normalized expansion maps are unramified}$(a)$.

Observe that $2 \cdot \psi = \phi$. Invoking Proposition \ref{prop: description of N[4] cocycles}, we can write $\psi = \frac{\phi}{4} + \rho$, where $\rho$ is a $\mathbb{F}_2$-valued continuous cochain satisfying the equation
$$
\mathrm{d} \rho = \phi \cup \phi.
$$
From what we have just shown, we see that the restriction of $\rho$ to $G_{F(C)_v}$ (with $v \in S_\emptyset$) is entirely determined by $a'$, $(L_{i, v})_{i \in [3], v \in S_\emptyset}$ and $\mathbf{a}$. 

Define $T$ to be the set of places of $F(C)$ at which $\sum_{g \in G(C) - \{1\}} \phi^g$ ramifies. We now invoke Proposition \ref{prop: trivialization in general3} and obtain 
$$
\left[ \phi, \phi, \sum_{g \in G(C) - \{1\}} \phi^g \right]_{F(C)} = \sum_{v \in T} \inv_v \left(\rho \cup \sum_{g \in G(C) - \{1\}} \phi^g \right) + \sum_{g \in G(C) - \{1\}} \sum_{v \in S_\emptyset} \rho(\sigma_v) \cdot \phi^g(\Frob_v).
$$
By construction, we have that for each $v \in T$, the quartic character $\psi$ is unramified at $v$ and furthermore $\phi(\Frob_v) = 0$. Hence we can rewrite for each such place $v \in T$
$$
\inv_v \left( \rho \cup \sum_{g \in G(C) - \{1\}} \phi^g \right) = \rho(\Frob_v) = \psi(\Frob_v).
$$
All in all we can rewrite
$$
\sum_{v \in T} \inv_v \left( \rho \cup \sum_{g \in G(C) - \{1\}} \phi^g \right) = \sum_{v \in T} \psi(\Frob_v).
$$
We unfold the definition of $\psi$ and $T$ to rewrite the right hand side as
$$
\sum_{v \in T} \xi_{S_\emptyset}(\chi)(\Frob_v) + \sum_{v \in T} \sum_{w \mid a'r_1r_2r_3} \psi_{S_\emptyset, w}(\Frob_v).
$$
Fixing $w \mid a'r_1r_2r_3$, we have the identity
$$
\sum_{v \in T} \psi_{S_\emptyset, w}(\Frob_v) = \sum_{u \mid a'r_1r_2r_3} \sum_{g \in G(C) - \{1\}} \psi_{S_\emptyset, w}^g(\Frob_{\text{AUp}(u)}),
$$
where for $\psi_0 \in \Gamma_{\Z/4\Z}(F(C))$ and $g \in G(C)$, we have set $\psi_0^g(\tau) := \psi_0(\tilde{g}^{-1} \tau \tilde{g})$ for any lift $\tilde{g}$ of $g$ to $G_{\mathbb{F}_q(T)}$ (this is independent of the choice of lift $\tilde{g}$).

Let $g \in G(C) - \{1\}$. Since $q \equiv 1 \bmod 8$, we apply Hilbert reciprocity to the cocycle $\psi_{S_\emptyset, w}^g \cup \psi_{S_\emptyset, u}$ and derive that for all distinct $u, w \mid r_1r_2r_3$
$$
\psi_{S_\emptyset,w}^g(\Frob_{\text{AUp}(u)}) = D + \psi_{S_\emptyset, u}^{g}(\Frob_{\text{AUp}(w)}),
$$
where the constant $D$ depends merely on $(L_{i, v})_{i \in [3], v \in S_\emptyset}$. All in all, we get that up to a function of the form $\lambda_{\mathbf{a}, \mathbf{b}, a', S_\emptyset}$, we can rewrite
$$
\left[ \phi, \phi, \sum_{g \in G(C) - \{1\}} \phi^g \right]_{F(C)} - \lambda_{\mathbf{a}, \mathbf{b}, a', S_\emptyset}((L_{i, v})_{i \in [3], v \in S_\emptyset \cup S(a')}) - \sum_{1 \leq i \leq 3} \Delta_C(r_i)
$$
as
$$
\sum_{1 \leq i < j \leq 3} \sum_{g \in G(C) - \{1\}} 2 \cdot \psi_{S_\emptyset, \mathrm{AUp}(r_i)}^g(\Frob_{\text{AUp}(r_j)}) = \sum_{1 \leq i<j \leq 3} \sum_{g \in G(C) - \{1\}} \chi_{S_\emptyset, r_i}^g(\Frob_{\text{AUp}(r_j)}).
$$
Considering that the character $\chi_{S_\emptyset, r_i}^g$ is unramified at all places above $r_j$ for each $g \in G(C) - \{1\}$, and that $\chi_{S_\emptyset, r_j}$ ramifies at $\text{AUp}(r_j)$ by definition, we obtain that
$$
\chi_{S_\emptyset, r_i}^g(\Frob_{\text{AUp}(r_j)}) = \inv_{\text{AUp}(r_j)}(\chi_{S_\emptyset, r_i}^g \cup \chi_{S_\emptyset, r_j}),
$$
as desired.
\end{proof}

Our next result takes care of computing the symbols that detect when a cube satisfies $(D)$.

\begin{theorem} 
\label{theorem: profitability}
Let $h \in \{1, -1\}$. Let $s \in \Z_{\geq 2}$ and let $\mathbf{t} \in N[2^{s + 1}]^{[r + 1]}$. For each $k_1, k_2 \in \{1, 2\}$, let  
$$
C(k_1, k_2) := \left(\prod_{i = 1}^s \{\pi_i(1), \pi_i(2)\}\right) \times \{r_{k_1}\} \times \{s_{k_2}\} \times \{a'\} \times \{d\}
$$
be a pre-acceptable, profitable cube of type $\mathbf{t}$. Set 
$$
S := \{\pi_i(1) \pi_i(2) : i \in [s]\}, \quad \quad \rho_i := \phi_{C_i; \pi_i(1) \pi_i(2)}(\mathfrak{G}).
$$
Define $C(k_1, k_2)_i$ to be the sub-cube of $C(k_1, k_2)$ consisting of elements projecting to $\pi_i(2)$ and define $C_i$ the resulting projection on the coordinates in $[s] - \{i\}$ (which does not depend on $k_1$ or $k_2$). Then $\phi_{C_i; r_1r_2}(\mathfrak{G})$ and $\phi_{C_i; s_1s_2}(\mathfrak{G})$ exist and the triples
$$
\left(\sum_{g \in G(C_i) - \{1\}} \phi_{C_i; r_1r_2}(\mathfrak{G})^g, \phi_{C_i; s_1s_2}(\mathfrak{G}), \rho_i\right), \quad \left(\sum_{g \in G(C_i) - \{1\}} \phi_{C_i; s_1s_2}(\mathfrak{G})^g, \phi_{C_i; r_1r_2}(\mathfrak{G}), \rho_i\right)
$$
are R\'edei admissible. Moreover, setting $\phi(k_1, k_2)_i := \phi_{S; a' r_{k_1} s_{k_2}}(\mathfrak{G}) + \phi_{[s]}(C(k_1, k_2); \mathbf{t})$, we have
\begin{multline*}
\sum_{k_1, k_2 \in \{1, 2\}} \left(\sum_{x \in C(k_1, k_2)_i} \mathrm{sg}(x) \psi_{s + 1}(x)(\Frob_{\pi_i(1)}) + \phi(k_1, k_2)_i(\Frob_{\pi_i(1)}(\chi_{x_0}))\right) = \\
\left[\sum_{g \in G(C_i) - \{1\}} \phi_{C_i; r_1r_2}(\mathfrak{G})^g, \phi_{C_i; s_1s_2}(\mathfrak{G}), \rho_i\right] + \left[\sum_{g \in G(C_i) - \{1\}} \phi_{C_i; s_1s_2}(\mathfrak{G})^g, \phi_{C_i; r_1r_2}(\mathfrak{G}), \rho_i\right].
\end{multline*}
\end{theorem}

\begin{proof}
The character $\phi_{C_i; a'r_{k_1} s_{k_2}}(\mathfrak{G})$ is locally trivial at the places above $\pi_i(1)$ and $\pi_i(2)$, i.e.
\begin{equation}
\label{ePhiTriviali}
\res_w(\phi_{C_i; a'r_{k_1} s_{k_2}}(\mathfrak{G})) = 0 \quad \quad \text{ for } w \in \Omega_{F(C_i)} \text{ with } w \mid \pi_i(1) \pi_i(2)
\end{equation}
due to part $(A1)$ and $(A2)$ of Definition \ref{def: acceptable cube} combined with Proposition \ref{inductively construct expansion maps}: these contain all places where $\rho_i$ ramifies by Proposition \ref{normalized expansion maps are unramified}. Using $(A3)$, we also find that
$$
\res_w(\phi_{C_i; r_1 r_2}(\mathfrak{G})) = 0 \quad \quad \text{ for } w \in \Omega_{F(C_i)} \text{ with } w \mid s_1s_2
$$
and vice versa (with the roles of $r_1r_2$ and $s_1s_2$ reversed). Hence the first part of the theorem is a direct consequence of the fact that $C(k_1, k_2)$ is profitable.

It remains to prove the last part of the theorem. Define $x_0(k_1, k_2) = r_{k_1} s_{k_2} a' d \prod_{i = 1}^s \pi_i(1)$. We claim that
$$
\tilde{\psi}(k_1, k_2)_i := \sum_{x \in C(k_1, k_2)_i} \mathrm{sg}(x) \psi_{s + 1}(x)
$$
restricts to an element of $Z^1(G_{F(C_i)}, N_h(\chi_{x_0(k_1, k_2) \pi_i(1) \pi_i(2)})[4])$ with
\begin{equation}
\label{eDoubleTildePsi}
2 \cdot \tilde{\psi}(k_1, k_2)_i = \phi_{C_i; a'r_{k_1} s_{k_2}}(\mathfrak{G}) + \phi_{[s] - \{i\}}(C(k_1, k_2); \mathbf{t}).
\end{equation}
To see this, simply observe that all of the restricted characters $\res_{G_{F(C_i)}} \chi_x$ equal the character $\chi_{x_0(k_1, k_2) \pi_i(1) \pi_i(2)}$ for all $x \in C(k_1, k_2)_i$. Therefore the $G_{F(C_i)}$-modules given by $N_h(\chi_x)$, with $x$ in $C(k_1, k_2)_i$, are all isomorphic to $N_h(\chi_{x_0(k_1, k_2) \pi_i(1) \pi_i(2)})$. This shows that 
$$
\tilde{\psi}(k_1, k_2)_i \in Z^1(G_{F(C_i)}, N_h(\chi_{x_0(k_1, k_2) \pi_i(1) \pi_i(2)})). 
$$
On the other hand, arguing as in the proof of Proposition \ref{prop: sums of psi trivializes alpha alpha bar} (using property $(T)$ and Theorem \ref{tCocyComp}), we obtain
$$
2 \cdot \tilde{\psi}(k_1, k_2)_i = \phi_{C_i; a'r_{k_1} s_{k_2}}(\mathfrak{G}) + \phi_{[s] - \{i\}}(C(k_1, k_2); \mathbf{t}),
$$
thus proving the claim \eqref{eDoubleTildePsi}.

Invoking Proposition \ref{prop: description of N[4] cocycles}, we expand 
$$
\tilde{\psi}(k_1, k_2)_i = \gamma(k_1, k_2)_i + \frac{\phi_{C_i; a'r_{k_1} s_{k_2}}(\mathfrak{G}) + \phi_{[s] - \{i\}}(C(k_1, k_2); \mathbf{t})}{4},
$$
where $\gamma(k_1, k_2)_i$ is an $\mathbb{F}_2$-valued continuous $1$-cochain for $G_{F(C_i)}$ such that
\begin{multline}
\label{ePhiiTriv}
\hspace{-0.3cm} \mathrm{d} (\gamma(k_1, k_2)_i) = \bigg(\chi_{\pi_1(1) \cdots \pi_i(2) \cdots \pi_s(1) d \epsilon_{\mathbb{F}_q}^{\frac{1 - h}{2}}} + \hspace{-0.1cm} \sum_{g \in G(C_i) - \{1\}} (\phi_{C_i; a'r_{k_1} s_{k_2}}(\mathfrak{G})^g + \phi_{[s] - \{i\}}(C(k_1, k_2); \mathbf{t})^g) \bigg) \\
\cup \bigg( \phi_{C_i; a'r_{k_1} s_{k_2}}(\mathfrak{G}) + \phi_{[s] - \{i\}}(C(k_1, k_2); \mathbf{t}) \bigg).
\end{multline}
Remark \ref{remark: degenerate phi-maps} shows that
\begin{equation}
\label{ePhiCoset}
\sum_{k_1, k_2 \in \{1, 2\}} \phi(k_1, k_2)_i = 0.
\end{equation}
By applying equation \eqref{ePhiCoset} plugged in at $\Frob_{\pi_i(1)}(\chi_{x_0(k_1, k_2)}) = \Frob_{\pi_i(1)}(\chi_{x_0(1, 1)})$, we can now write
\begin{multline}
\sum_{k_1, k_2 \in \{1, 2\}} \left(\sum_{x \in C(k_1, k_2)_i} \mathrm{sg}(x) \psi_{s + 1}(x)(\Frob_{\pi_i(1)}) + \phi(k_1, k_2)_i(\Frob_{\pi_i(1)}(\chi_{x_0(k_1, k_2)}))\right) \\
= \sum_{k_1, k_2 \in \{1, 2\}} \tilde{\psi}(k_1, k_2)_i(\Frob_{\pi_i(1)}). \label{ew1I}
\end{multline}
Using equation \eqref{ePhiTriviali} and the profitable assumption, we deduce that
\begin{equation}
\label{ew2II}
\tilde{\psi}(k_1, k_2)_i(\Frob_{\pi_i(1)}) = \gamma(k_1, k_2)_i(\Frob_{\pi_i(1)}).
\end{equation}
Moreover, setting $x_0'(k_1, k_2) := x_0(k_1, k_2) \pi_i(1) \pi_i(2)$, we have that 
\begin{align}
0 &= \sum_{x \in C(k_1, k_2)_i} \mathrm{sg}(x) \cdot \bigg( \psi_{s + 1}(x)\big(\Frob_{\pi_i(2)}\big(\chi_{x_0'(k_1, k_2)}\big)\big) + \frac{(hq_0)^{\deg(\pi_i(2))} - 1}{2} \cdot \psi_{s + 1}(x)(\sigma_{\pi_i(2)}) \bigg) \nonumber \\ 
&= \gamma(k_1, k_2)_i\left(\Frob_{\pi_i(2)}\left(\chi_{x_0'(k_1, k_2)}\right)\right). \label{eLocalGamma}
\end{align}
The first equality follows at once from the existence of $\psi_{s + 2}(x)$ and the condition $(A1)$ of Definition \ref{def: acceptable cube}, which tells us that $\Frob_{\pi_i(2)}(\chi_x)$ stays constant as $x$ varies in $C(k_1, k_2)_i$. The second equality follows by combining the following facts: 
\begin{itemize}
    \item Firstly, $\deg(\pi_i(2))$ is even and therefore $\frac{(hq_0)^{\deg(\pi_i(2))} - 1}{2}$ is divisible by $4$. Recalling that $\tilde{\psi}(k_1, k_2)_i$ takes values in $N[4]$, we may now remove the $\sigma_{\pi_i(2)}$ term.
    \item Secondly, invoking equation \eqref{ePhiTriviali} and the profitable assumption yields the equality
$$
\frac{(\phi_{C_i; a'r_{k_1} s_{k_2}}(\mathfrak{G}) + \phi_{[s] - \{i\}}(C(k_1, k_2); \mathbf{t}))\left(\Frob_{\pi_i(2)}\left(\chi_{x_0'(k_1, k_2)}\right)\right)}{4} = 0.
$$
\end{itemize} 
Let now $w(k)$ be the unique place above $\pi_i(k)$ where $\rho_i$ ramifies, i.e.~$w(k) = \mathrm{AUp}(\pi_i(k))$ by Proposition \ref{normalized expansion maps are unramified}. By $(T)$, we conclude that $\sum_{k_1, k_2 \in \{1, 2\}} \gamma(k_1, k_2)_i$ is unramified at $\pi_i(2)$, and hence $\sum_{k_1, k_2 \in \{1, 2\}} \gamma(k_1, k_2)_i$ is the trivial character locally at $\pi_i(2)$ by equation \eqref{eLocalGamma}. Therefore we conclude that
\begin{equation}
\label{eVanishingw2}
\sum_{k_1, k_2 \in \{1, 2\}} \inv_{w(2)}(\gamma(k_1, k_2)_i \cup \rho_i) = 0.
\end{equation}
All in all, we deduce from equations \eqref{ew1I}, \eqref{ew2II} and \eqref{eVanishingw2}
\begin{multline}
\sum_{k_1, k_2 \in \{1, 2\}} \left(\sum_{x \in C(k_1, k_2)_i} \mathrm{sg}(x) \psi_{s + 1}(x)(\Frob_{\pi_i(1)}) + \phi(k_1, k_2)_i(\Frob_{\pi_i(1)}(\chi_{x_0(k_1, k_2)}))\right) \\
= \sum_{k_1, k_2 \in \{1, 2\}} \sum_{w \in \text{Ram}(\rho_i)} \inv_w(\gamma(k_1, k_2)_i \cup \rho_i), \label{eSumPsiToTilde}
\end{multline}
since $\rho_i$ ramifies only at $w(1)$ and $w(2)$. 

Observe that 
\begin{equation}
\label{eTrivFact1}
\sum_{k_1, k_2 \in \{1, 2\}} \gamma(k_1, k_2)_i \in C^1(G_{F(C_i), S}, \mathbb{F}_2)
\end{equation}
where $S$ is the set of places dividing $r_1r_2s_1s_2$ and $\infty$. Let $S'$ be the union of $S$ with $\mathrm{AUp}(\pi_i(1))$ and $\mathrm{AUp}(\pi_i(2))$.

We also have a general identity
$$
\sum_{k_1, k_2 \in \FF_2} (\chi_1 + k_1t_1 + k_2t_2) \cup (\chi_2 + k_1u_1 + k_2u_2) = t_1 \cup u_2 + t_2 \cup u_1
$$
valid for all $\chi_1, \chi_2, t_1, t_2, u_1, u_2 \in \Hom(G_{F(C_i)}, \mathbb{F}_2)$. Thus, summing the terms in \eqref{ePhiiTriv}, we get
\begin{multline}
\label{eTrivFact2}
\mathrm{d}\left(\sum_{k_1, k_2 \in \{1, 2\}} \gamma(k_1, k_2)_i\right) = \left(\sum_{g \in G(C_i) - \{1\}} \phi_{C_i; r_1r_2}(\mathfrak{G})^g\right) \cup \phi_{C_i; s_1s_2}(\mathfrak{G}) + \\
\left(\sum_{g \in G(C_i) - \{1\}} \phi_{C_i; s_1s_2}(\mathfrak{G})^g\right) \cup \phi_{C_i; r_1r_2}(\mathfrak{G}),
\end{multline}
where we used the general identity (which applies thanks to Proposition \ref{pExpansionAdditive}).

For each place $v$ of $S'$ not dividing $\pi_i(1) \pi_i(2)$ we have that $\rho_i$ is trivial locally at $v$, since $C$ is a profitable cube. For the places $v$ of $S'$ dividing $\pi_i(1) \pi_i(2)$, we have already proven that $\sum_{k_1, k_2 \in \{1, 2\}} \gamma(k_1, k_2)_i$ is an unramified quadratic character. Combining \eqref{eSumPsiToTilde}, \eqref{eTrivFact1} and \eqref{eTrivFact2} with Proposition \ref{prop: trivialization in general2} yields the theorem.
\end{proof}

Our next theorem plays for Theorem \ref{theorem: profitability} the same role that Theorem \ref{theorem: Applying quartic reciprocity} played for Theorem \ref{theorem: reflection principle}, allowing us to split the relevant R\'edei symbol in terms of a factorization of $a := a' \cdot r_1 \cdot r_2$. In this analogy, the role played by auxiliary $\Z/4\Z$-expansion maps to break bilinearly the relevant R\'edei symbol will be played by auxiliary $\hat{\phi}$-expansion maps. 

\begin{theorem} 
\label{theorem: splitting the Redei symbol}
Let $s \in \Z_{\geq 2}$. Let $r_1, r_2, s_1, s_2$ be irreducible polynomials. Let 
$$
C := \{\pi_1(1), \pi_1(2)\} \times \dots \times \{\pi_s(1), \pi_s(2)\}.
$$
Assume that the set $\{\pi_1(1), \pi_1(2), \dots, \pi_s(1), \pi_s(2), r_1, r_2, s_1, s_2\}$ consists of $2s + 4$ distinct monic even degree irreducible polynomials. Let $i$ be an element of $[s]$.  

Assume that the following conditions are simultaneously satisfied:
\begin{enumerate}
\item The map $\rho_i := \phi_{C_i; \pi_i(1)\pi_i(2)}(\mathfrak{G})$ exists. Moreover, the infinite place $\infty$ and $r_1, r_2, s_1, s_2$ split completely in $L(\rho_i)/\FF_q(T)$. The maps $\phi_{C; r_1r_2}(\mathfrak{G})$ and $\phi_{C; s_1s_2}(\mathfrak{G})$ exist and are locally trivial at $\infty$. 
\item For each disjoint proper subsets $S_1, S_2 \subseteq [s] - \{i\}$ we have that the map
$$
\hat{\phi}_{\{\pi_{j_1}(1)\pi_{j_1}(2) : j_1 \in S_1\}; \{\pi_{j_2}(1)\pi_{j_2}(2) : j_2 \in S_2\}; \pi_i(1)\pi_i(2), r_1r_2}(\mathfrak{G})
$$ 
exists.
\item The primes $\pi_i(1)$ and $\pi_i(2)$ split completely in $$L(\phi_{C_i; r_1r_2}(\mathfrak{G})) \cdot L(\phi_{C_i; s_1s_2}(\mathfrak{G}))/\mathbb{F}_q(T), 
$$
the primes $r_1,r_2$ split completely in $L(\phi_{C_i; s_1s_2}(\mathfrak{G}))/\FF_q(T)$ and the primes $s_1,s_2$ split completely in $L(\phi_{C_i; r_1r_2}(\mathfrak{G}))/\FF_q(T)$.
\end{enumerate}
Then the triples 
$$\left(\sum_{g \in G(C_i) - \{1\}} \phi_{C_i;r_1r_2}(\mathfrak{G})^g, \phi_{C_i;s_1s_2}(\mathfrak{G}), \rho_i \right), \quad \left(\sum_{g \in G(C_i) - \{1\}} \phi_{C_i;s_1s_2}(\mathfrak{G})^g, \phi_{C_i;r_1r_2}(\mathfrak{G}), \rho_i \right)
$$
are R\'edei admissible over $F(C_i)$ and moreover
\begin{multline*}
\big[\sum_{g \in G(C_i) - \{1\}} \hspace{-0.3cm} \phi_{C_i;r_1r_2}(\mathfrak{G})^g, \phi_{C_i;s_1s_2}(\mathfrak{G}), \rho_i\big]_{F(C_i)} + \big[\sum_{g \in G(C_i) - \{1\}} \hspace{-0.3cm} \phi_{C_i;s_1s_2}(\mathfrak{G})^g, \phi_{C_i;r_1r_2}(\mathfrak{G}), \rho_i\big]_{F(C_i)} \\
= \hat{\phi}_i(r_1r_2)(\Frob_{\mathrm{AUp}(s_1)}) + \hat{\phi}_i(r_1r_2)(\Frob_{\mathrm{AUp}(s_2)}),
\end{multline*}
where 
$$
\hat{\phi}_i(r_1r_2) := \phi_{C; r_1r_2}(\mathfrak{G}) + \sum_{\substack{S_1 \sqcup S_2 \subseteq [s] - \{i\} \\ S_1, S_2 \subset [s] - \{i\}}} \hat{\phi}_{S_1; S_2; \pi_i(1)\pi_i(2), r_1r_2}(\mathfrak{G}).
$$
\end{theorem}

\begin{proof}
We start by showing that the triples 
$$
\left(\sum_{g \in G(C_i) - \{1\}} \phi_{C_i;r_1r_2}(\mathfrak{G})^g, \phi_{C_i;s_1s_2}(\mathfrak{G}), \rho_i \right), \quad \left(\sum_{g \in G(C_i) - \{1\}} \phi_{C_i;s_1s_2}(\mathfrak{G})^g, \phi_{C_i;r_1r_2}(\mathfrak{G}), \rho_i \right)
$$
are R\'edei admissible over $F(C_i)$. To this end, we notice that for each $g \in G(C_i) - \{1\}$ and for each $\{x, y\}=\{\phi_{C_i; r_1r_2}(\mathfrak{G}), \phi_{C_i; s_1s_2}(\mathfrak{G})\}$ we have that the triple
\begin{equation}
\label{eRRxyrho}
(x^g, y, \rho_i)
\end{equation}
is R\'edei admissible, which is an immediate consequence of condition $1$ and $3$.

Having shown this, we observe that this implies that the two original triples 
$$
\left(\sum_{g \in G(C_i) - \{1\}} \phi_{C_i;r_1r_2}(\mathfrak{G})^g, \phi_{C_i;s_1s_2}(\mathfrak{G}), \rho_i \right), \quad \left(\sum_{g \in G(C_i) - \{1\}} \phi_{C_i;s_1s_2}(\mathfrak{G})^g, \phi_{C_i;r_1r_2}(\mathfrak{G}), \rho_i \right)
$$
are R\'edei admissible in view of Theorem \ref{theorem: Redei reciprocity}$(2)$.

Furthermore, by Theorem \ref{theorem: Redei reciprocity}$(2)$, it suffices to show the desired conclusion for the R\'edei symbol
$$
\sum_{g \in G(C_i) - \{1\}} \left[ \phi_{C_i; r_1r_2}(\mathfrak{G})^g, \phi_{C_i; s_1s_2}(\mathfrak{G}), \rho_i \right ]_{F(C_i)} + \left[\phi_{C_i; s_1s_2}(\mathfrak{G})^g, \phi_{C_i; r_1r_2}(\mathfrak{G}), \rho_i  \right ]_{F(C_i)}.
$$
Invoking Proposition \ref{prop: conjugating Redei} and Theorem \ref{theorem: Redei reciprocity}$(1)$, we can rewrite this as 
$$
\sum_{g \in G(C_i) - \{1\}} \left[ \phi_{C_i; r_1r_2}(\mathfrak{G})^g, \phi_{C_i; s_1s_2}(\mathfrak{G}), \rho_i \right ]_{F(C_i)} + \left[ \phi_{C_i; r_1r_2}(\mathfrak{G})^g, \phi_{C_i; s_1s_2}(\mathfrak{G}), \rho_i^g \right]_{F(C_i)}.
$$
Applying Theorem \ref{theorem: Redei reciprocity}$(2)$, we can rewrite this as 
$$
\sum_{g \in G(C_i) - \{1\}} \left[ \phi_{C_i; r_1r_2}(\mathfrak{G})^g, \phi_{C_i; s_1s_2}(\mathfrak{G}), \rho_i + \rho_i^g \right ]_{F(C_i)}.
$$
We now add and subtract on each summand the symbol $\left[ \phi_{C_i; r_1r_2}(\mathfrak{G}), \phi_{C_i; s_1s_2}(\mathfrak{G}), \rho_i + \rho_i^g \right]_{F(C_i)}$. Using that $2^{s - 1} - 1$ is odd, we use equation \eqref{eConjugateExpansion} (applied with every $g \in G(C_i)$)
$$
\sum_{g \in G(C_i) - \{1\}} (\rho_i + \rho_i^g) = \sum_{g \in G(C_i)} \rho_i^g = \chi_{\pi_i(1)\pi_i(2)}
$$
to rewrite the desired symbol as 
\begin{multline*}
\left[ \phi_{C_i; r_1r_2}(\mathfrak{G}), \phi_{C_i; s_1s_2}(\mathfrak{G}), \chi_{\pi_i(1)\pi_i(2)} \right]_{F(C_i)} + \\
\sum_{g \in G(C_i) - \{1\}} \left[ \phi_{C_i; r_1r_2}(\mathfrak{G}) + \phi_{C_i; r_1r_2}(\mathfrak{G})^g, \phi_{C_i; s_1s_2}(\mathfrak{G}), \rho_i + \rho_i^g \right ]_{F(C_i)}.
\end{multline*}
We now focus on the first summand. We apply both parts of Theorem \ref{theorem: Redei reciprocity} to rewrite it as
$$
\left[ \chi_{\pi_i(1)\pi_i(2)}, \phi_{C_i; r_1r_2}(\mathfrak{G}), \phi_{C_i; s_1s_2}(\mathfrak{G}) \right]_{F(C_i)}.
$$
Observe that the map $\phi_{C; r_1r_2}(\mathfrak{G})$ exists thanks to our assumption $1$. Once restricted to $G_{F(C_i)}$ this gives a $1$-cochain satisfying
$$
\mathrm{d} \phi_{C; r_1r_2}(\mathfrak{G})(\sigma, \tau) = \chi_{\pi_i(1)\pi_i(2)}(\sigma) \cdot \phi_{C_i; r_1r_2}(\mathfrak{G})(\tau).
$$
In view of assumption $1$, we are in position to apply Proposition \ref{prop: trivialization in general}.

Having established this, we see that Proposition \ref{prop: trivialization in general} yields 
$$
\left[ \chi_{\pi_i(1)\pi_i(2)}, \phi_{C_i; r_1r_2}(\mathfrak{G}), \phi_{C_i; s_1s_2}(\mathfrak{G}) \right]_{F(C_i)} = \phi_{C; r_1r_2}(\mathfrak{G})(\Frob_{s_1}) + \phi_{C; r_1r_2}(\mathfrak{G})(\Frob_{s_2}).
$$
We now focus on the second summand
$$
\sum_{g \in G(C_i) - \{1\}} \left[ \phi_{C_i; r_1r_2}(\mathfrak{G}) + \phi_{C_i; r_1r_2}(\mathfrak{G})^g, \phi_{C_i; s_1s_2}(\mathfrak{G}), \rho_i + \rho_i^g \right]_{F(C_i)}.
$$

Moreover, we observe that equation \eqref{eRRxyrho} readily implies that 
$$
\left( \phi_{C_{U_1 \cup \{i\}}, r_1r_2}(\mathfrak{G}), \phi_{C_i; s_1s_2}(\mathfrak{G}), \phi_{C_{U_2 \cup \{i\}}; \pi_i(1)\pi_i(2)}(\mathfrak{G}) \right)
$$
is R\'edei admissible over $F(C_i)$ for all subsets $U_1, U_2 \subseteq [s] - \{i\}$. For the rest of the proof, we denote by
$$
\phi_T(1) := \phi_{\{\pi_j(1)\pi_j(2) : j \in T\};\pi_i(1)\pi_i(2)}(\mathfrak{G}), \quad \quad \phi_T(2) := \phi_{\{\pi_j(1)\pi_j(2) : j \in T\}; r_1r_2}(\mathfrak{G})
$$
for each subset $T$ of $[s] - \{i\}$. Furthermore, equation \eqref{eConjugateExpansion} states
$$
\phi_{[s] - \{i\}}(k) + \phi_{[s] - \{i\}}^g(k) = \sum_{\emptyset \neq T \subseteq \text{Supp}(g)} \phi_{[s] - \{i\} - T}(k)
$$
for both values of $k \in \{1, 2\}$. Here $\text{Supp}(g)$ denotes the set of $j \in [s] - \{i\}$ such that $\chi_{\pi_j(1)\pi_j(2)}(g) \neq 0$. Therefore, invoking Theorem \ref{theorem: Redei reciprocity}, we have that 
\begin{multline*}
\sum_{g \in G(C_i) - \{1\}} \left[ \phi_{C_i; r_1r_2}(\mathfrak{G}) + \phi_{C_i; r_1r_2}(\mathfrak{G})^g, \phi_{C_i; s_1s_2}(\mathfrak{G}), \rho_i + \rho_i^g \right ]_{F(C_i)} = \\
\sum_{g \in G(C_i) - \{1\}} \sum_{\emptyset \neq T_1, T_2 \subseteq [s] - \{i\}} \mathbf{1}_{T_1, T_2 \subseteq \text{Supp}(g)} \left[\phi_{[s] - \{i\} - T_1}(1), \phi_{[s] - \{i\} - T_2}(2), \phi_{C_i; s_1s_2}(\mathfrak{G})\right]_{F(C_i)}.
\end{multline*}
Switching order of summation yields
$$
\sum_{\emptyset \neq T_1, T_2 \subseteq [s]-\{i\}} 2^{s - 1 - |T_1 \cup T_2|} \cdot \left[\phi_{[s] - \{i\} - T_1}(1), \phi_{[s] - \{i\} - T_2}(2), \phi_{C_i; s_1s_2}(\mathfrak{G})\right]_{F(C_i)}.
$$
Hence the only summands left are those where $T_1 \cup T_2 = [s] - \{i\}$. We relabel $U_k := [s] - \{i\} - T_k$. Thus we can rewrite the sum as
$$
\sum_{\substack{U_1 \sqcup U_2 \subseteq [s] - \{i\} \\ U_1 \neq [s] - \{i\}, U_2 \neq [s] - \{i\}}} \left[ \phi_{U_1}(1), \phi_{U_2}(2),\phi_{C_i; s_1s_2}(\mathfrak{G}) \right].
$$
We now invoke part $2$ of our assumptions combined with Proposition \ref{prop: field of definition of hat phi} and Proposition \ref{prop: trivialization in general} and rewrite the symbol as
\begin{multline*}
\sum_{\substack{U_1 \sqcup U_2 \subseteq [s] - \{i\} \\ U_1, U_2 \neq [s] - \{i\}}} \inv_{\text{AUp}(s_1)}(\hat{\phi}_{\{\pi_{j_1}(1)\pi_{j_1}(2) : j_1 \in U_1\}, \{\pi_{j_2}(1)\pi_{j_2}(2) : j_2 \in U_2\}; \pi_i(1)\pi_i(2), r_1r_2} \cup \phi_{C_i; s_1s_2}(\mathfrak{G})) \\
+ \sum_{\substack{U_1 \sqcup U_2 \subseteq [s] - \{i\} \\ U_1, U_2 \neq [s] - \{i\}}} \inv_{\text{AUp}(s_2)}(\hat{\phi}_{\{\pi_{j_1}(1)\pi_{j_1}(2) : j_1 \in U_1\}, \{\pi_{j_2}(1)\pi_{j_2}(2) : j_2 \in U_2\}; \pi_i(1)\pi_i(2), r_1r_2} \cup \phi_{C_i; s_1s_2}(\mathfrak{G})).
\end{multline*}
Invoking Proposition \ref{prop: field of definition of hat phi}, we see that the first character in the cup product is unramified at $\text{AUp}(s_1)$ and $\text{AUp}(s_2)$. This ends the proof.
\end{proof}

We will now provide some intuition on the role of the group-theoretic Proposition \ref{prop: linear independence of hat phi and phi}. In our next section we will combine Theorem \ref{theorem: profitability} with Theorem \ref{theorem: splitting the Redei symbol} to detect whether a pre-acceptable, profitable cube of fixed type satisfies $(D)$. We then need to show that the resulting symbol does not cancel the symbol coming from the Artin pairing (this symbol is obtained by combining Theorem \ref{theorem: reflection principle} with Theorem \ref{theorem: Applying quartic reciprocity}). Proposition \ref{prop: linear independence of hat phi and phi} will ensure that there is always a surviving bilinear term in the pair $(r_i, s_j)$; and hence we will be able to bring Lemma \ref{lLS} into play.

\section{Repeated Cauchy--Schwarz}
\label{sCS}
\subsection{Abstract theory}
In this section, we develop some abstract analytic theory. The results here are directly inspired by the repeated Cauchy--Schwarz arguments found, for example, in additive combinatorics and Weyl differencing. 

\begin{mydef}
Let $Y$ be a finite set. Write $\mathcal{C}(Y)$ for the set of equivalence relations on $Y$. 

Let $X_1, \dots, X_n$ be finite sets. We say that $f_1, \dots, f_n$ is a functional sequence of equivalence relations if 
$$
f_i \in \mathrm{Map}\left(\prod_{j < i} X_j^2 \times \prod_{j > i} X_j, \mathcal{C}(X_i)\right)
$$
for each $i \in [n]$. The complexity of $f_1, \dots, f_n$ is the largest number of equivalence classes among all equivalence relations that are in the image of some $f_i$.
\end{mydef}

\begin{mydef}
Let $X_1, \dots, X_n$ be finite sets and let $X = X_1 \times \cdots \times X_n$. Define 
$$
\mathrm{Cube}(X) := \prod_{i = 1}^n \mathrm{Map}(\{1, 2\}, X_i).
$$
We shall often view $\mathrm{Cube}(X)$ implicitly as a subset of $\mathrm{Map}(\{1, 2\}^n, X)$ in the natural way.

We say that $F = (F_1, \dots, F_n) \in \mathrm{Cube}(X)$ satisfies a functional sequence of equivalence relations $f_1, \dots, f_n$ if for all $1 \leq i \leq n$ and $\bar{x} = ((F_j(1), F_j(2)))_{1 \leq j \leq i - 1}$ and all $\bar{y} \in X_{i + 1} \times \cdots \times X_n$ satisfying $\mathrm{pr}_j(\bar{y}) \in \{F_j(1), F_j(2)\}$ for all $i + 1 \leq j \leq n$, we have
$$
F_i(1) \sim_{f_i(\bar{x}, \bar{y})} F_i(2).
$$
\end{mydef}

\begin{lemma}[Repeated Cauchy--Schwarz]
\label{lRCS}
Let $n \in \Z_{\geq 2}$. Let $X_1, \dots, X_n$ be finite sets and let $X = X_1 \times \cdots \times X_n$. Let $f_1, \dots, f_n$ be a functional sequence of complexity bounded by $B \geq 1$. Let $g: X \rightarrow \R$. Then we have
$$
\left| \sum_{x \in X} g(x) \right| \leq B |X|^{\frac{2^{n - 1} - 1}{2^{n - 1}}} \left( \sum_{\substack{F \in \mathrm{Cube}(X) \\ F \textup{ sat. } f_1, \dots, f_n}} \prod_{\mathbf{a} \in \{1, 2\}^n} g(F(\mathbf{a})) \right)^{1/2^n}.
$$
\end{lemma}

\begin{proof}
We proceed by applying Cauchy--Schwarz over every variable exactly once. More precisely, we start with the inequality
\begin{align*}
\left| \sum_{x \in X} g(x) \right| 
&= \left| \sum_{x = (x_1, \dots, x_n) \in X} g(x) \right| \leq \sum_{(x_2, \dots, x_n) \in X_2 \times \cdots \times X_n} \left| \sum_{x_1 \in X_1} g(x_1, \dots, x_n) \right| \\
&\leq \sum_{(x_2, \dots, x_n) \in X_2 \times \cdots \times X_n} \sum_{c \in f_1(x_2, \dots, x_n)} \left| \sum_{\substack{x_1 \in X_1 \\ x_1 \in c}} g(x_1, \dots, x_n) \right|.
\end{align*}
Hence, by the Cauchy--Schwarz inequality we have
$$
\left| \sum_{x \in X} g(x) \right| \leq \left( \sum_{\substack{(x_2, \dots, x_n) \in X_2 \times \cdots \times X_n \\ c \in f_1(x_2, \dots, x_n)}} 1 \right)^{1/2} \left(\sum_{\substack{(x_2, \dots, x_n) \in X_2 \times \cdots \times X_n \\ c \in f_1(x_2, \dots, x_n)}} \left| \sum_{\substack{x_1 \in X_1 \\ x_1 \in c}} g(x_1, \dots, x_n) \right|^2\right)^{1/2}.
$$
The first factor is at most $B^{1/2} |X_2|^{1/2} \cdots |X_n|^{1/2}$. The second factor can be rewritten as
$$
\left(\sum_{F_1: \{1, 2\} \rightarrow X_1} \sum_{\substack{(x_2, \dots, x_n) \in X_2 \times \cdots \times X_n \\ F_1(1) \sim_{f_1(x_2, \dots, x_n)} F_1(2)}} \prod_{i \in \{1, 2\}} g(F_1(i), x_2, \dots, x_n)\right)^{1/2}.
$$
Proceeding exactly as before, we can bound this second factor by
$$
\left(\sum_{F_1: \{1, 2\} \rightarrow X_1} \sum_{\substack{(x_3, \dots, x_n) \in X_3 \times \cdots \times X_n \\ c \in f_2((F_1(1), F_1(2)), x_3, \dots, x_n)}} \left| \sum_{\substack{x_2 \in X_2, \ x_2 \in c \\ F_1(1) \sim_{f_1(x_2, \dots, x_n)} F_1(2)}} \prod_{i \in \{1, 2\}} g(F_1(i), x_2, \dots, x_n) \right|\right)^{1/2}
$$
and we apply Cauchy--Schwarz to upper bound our second factor as
\begin{multline*}
B^{1/4} |X_1|^{1/2} |X_3|^{1/4} \cdots |X_n|^{1/4} \times \\
\left(\sum_{\substack{F_1: \{1, 2\} \rightarrow X_1 \\ F_2: \{1, 2\} \rightarrow X_2}} \sum_{\substack{(x_3, \dots, x_n) \in X_3 \times \cdots \times X_n \\ \forall j \in \{1, 2\} : F_1(1) \sim_{f_1(F_2(j), \dots, x_n)} F_1(2) \\ F_2(1) \sim_{f_2((F_1(1), F_1(2)), x_3, \dots, x_n)} F_2(2)}} \prod_{i, j \in \{1, 2\}} g(F_1(i), F_2(j), x_3, \dots, x_n) \right)^{1/4}.
\end{multline*}
Collecting exponents, our total upper bound so far is
\begin{multline*}
B^{3/4} |X_1|^{1/2} |X_2|^{1/2} |X_3|^{3/4} \cdots |X_n|^{3/4} \times \\
\left(\sum_{\substack{F_1: \{1, 2\} \rightarrow X_1 \\ F_2: \{1, 2\} \rightarrow X_2}} \sum_{\substack{(x_3, \dots, x_n) \in X_3 \times \cdots \times X_n \\ \forall j \in \{1, 2\} : F_1(1) \sim_{f_1(F_2(j), \dots, x_n)} F_1(2) \\ F_2(1) \sim_{f_2((F_1(1), F_1(2)), x_3, \dots, x_n)} F_2(2)}} \prod_{i, j \in \{1, 2\}} g(F_1(i), F_2(j), x_3, \dots, x_n) \right)^{1/4}.
\end{multline*}
Continuing in this way, the lemma follows.
\end{proof}

\begin{lemma}[Removing boundaries]
\label{lBoundaries}
Let $n \in \Z_{\geq 2}$. Let $X_1, \dots, X_n$ be finite sets and let $X = X_1 \times \cdots \times X_n$. For each $1 \leq i \leq n$, let $p_i \colon X \rightarrow \C$ be a function such that
$$
p_i(x_1, \dots, x_{i - 1}, x_i, x_{i + 1}, \dots, x_n) = p_i(x_1, \dots, x_{i - 1}, y_i, x_{i + 1}, \dots, x_n)
$$
for all $x_i, y_i \in X_i$ and all $x_j \in X_j$ (with $j \neq i$). Assume that $|p_i(x)| \leq 1$ for all $1 \leq i \leq n$ and all $x \in X$. Let $g: X \rightarrow \R$. Then we have
$$
\left| \sum_{x \in X} g(x) \prod_{i = 1}^n p_i(x) \right| \leq |X|^{\frac{2^{n - 1} - 1}{2^{n - 1}}} \left( \sum_{F \in \mathrm{Cube}(X)} \prod_{\mathbf{a} \in \{1, 2\}^n} g(F(\mathbf{a})) \right)^{1/2^n}.
$$
\end{lemma}

\begin{proof}
We proceed by repeated Cauchy--Schwarz in the same way as Lemma \ref{lRCS}. After applying this for the first $i$ sets $X_1, \dots, X_i$, we get the following upper bound
\begin{multline*}
\prod_{1 \leq j \leq i} |X_j|^{\frac{2^{i - 1} - 1}{2^{i - 1}}} \cdot \prod_{i < j \leq n} |X_j|^{\frac{2^i - 1}{2^i}} \cdot \\
\left(\sum_{\substack{F_1: \{1, 2\} \rightarrow X_1 \\ \vspace{0.1cm} \vdots \vspace{0.1cm} \\ F_i: \{1, 2\} \rightarrow X_i}} \sum_{\mathbf{x} = (x_{i + 1}, \dots, x_n)} \prod_{\mathbf{k} \in \{1, 2\}^i} g(F_1(k_1), \dots, F_i(k_i), \mathbf{x}) \prod_{i < j \leq n} p_j^{\mathbf{k}}(F_1(k_1), \dots, F_i(k_i), \mathbf{x}) \right)^{\frac{1}{2^i}}
\end{multline*}
from this procedure. Here $p_j^{\mathbf{k}}$ means that we have to apply complex conjugation exactly when $|\{1 \leq j \leq i : k_j = 2\}|$ is odd. Continuing this procedure until $i = n$ gives the stated bound.
\end{proof}

\subsection{Equidistribution of Artin pairings}
Our next definition pulls back the Artin pairing to a common overarching space.

\begin{mydef}
Let $X = X_1 \times \dots \times X_r$ be a grid with factorization pattern $(n_1, \dots, n_r)$. Recall the definition of $W_0^{(h)}$ and $f_x$ from Definition \ref{dGenusPar}.

By pulling back under $f_x$, we also have spaces 
$$
W_i^{(h)}(x) := f_x^{-1}(C_i^{(h)}(\chi_x)^\circ).
$$
Define for $\mathbf{v}, \mathbf{w} \in W_i^{(h)}(x)$ the pairing
$$
\mathrm{Art}_{x, i + 1}^{(h)}(\mathbf{v}, \mathbf{w}) = \mathrm{Art}_{x,i + 1}^{(h)}(f_x(\mathbf{v}), \mathrm{Sym}(\chi_x)(f_x(\mathbf{w}))).
$$
\end{mydef}

By Theorem \ref{thm: symmetry}, we have the important relation
$$
\mathrm{Art}_{x, i + 1}^{(h)}(\mathbf{v}, \mathbf{w}) = \mathrm{Art}_{x, i + 1}^{(h)}(\mathbf{w}, \mathbf{v}),
$$
which explains our focus on symmetric pairings.

Our next definition describes a generic element $\mathbf{v} \in W_0^{(h)}$ to which our methods apply. In the absence of such an element, we shall be able to estimate our sum in a more trivial manner.

\begin{mydef}
\label{dGeneric}
Let $m$ and $n$ be integers. We say that $\mathbf{v} \in W_0^{(h)}$ is $(m, n)$-generic if: 
\begin{itemize}
    \item there is a set $S_{\textup{var}} \subseteq [r]$ of cardinality $m - 2$ such that for all $i \in S_{\text{var}}$ the following properties hold: $i > \log(n)/20$, $n_i$ is even and $\mathrm{pr}_i(\mathbf{v}) = 0$, and
    \item there is a set $S_{\textup{spl}} \subseteq [r]$ of cardinality $3$ such that for all $i \in S_{\textup{spl}}$ the following properties hold: $i > \log(n)/20$, $n_i$ is even and $\mathrm{pr}_i(\mathbf{v}) = 1$.
\end{itemize} 
\end{mydef}

\noindent For each $x \in X$, define
$$
\mathcal{T}_m(x, \mathbf{v}) := \left\{\mathbf{t} : \textup{there exists } \psi \in B_h(\chi_x)[2^m], 2^{m - 1} \psi = f_x(\mathbf{v}), 2\psi \text{ is of type } \mathbf{t}\right\}.
$$
For each $\mathbf{t} \in \mathcal{T}_m(x, \mathbf{v})$, fix an $h$-raw cocycle $\psi_m(x; \mathbf{t})$ lifting $f_x(\mathbf{v})$ such that $\psi_{m - 1}(x; \mathbf{t})$ is of type $\mathbf{t}$. Define 
$$
g(x) = 
\begin{cases}
0 &\text{if } f_x(\mathbf{v}) \not \in C_{m - 1}^{(h)}(\chi_x)^\circ, \\
(-1)^{\mathrm{Art}_{x, m}^{(h)}(\mathbf{v}, \mathbf{v})} &\text{if } f_x(\mathbf{v}) \in C_{m - 1}^{(h)}(\chi_x)^\circ.
\end{cases}
$$

\begin{theorem}
\label{tReduction}
There exist absolute constants $c > 1$ and $c_1, c_2 > 0$ such that the following holds. 

Let $n \in \Z_{\geq 3}$ be odd, let $X = X_1 \times \cdots \times X_r$ be a grid for $n$, let $m \in \Z_{\geq 5}$, let $h \in \{1, -1\}$, let $R: W_0^{(h)} \times W_0^{(h)} \rightarrow \mathbb{F}_2$ be symmetric and let $\mathbf{v} \in W_0^{(h)}$ be $(m, n)$-generic. Then we have
$$
\left| \sum_{x \in X, \ \mathrm{Art}_{x, 1}^{(h)} = R} g(x) \right| \leq \frac{c_1 |X|}{\exp(n^{c_2}/c^m)}.
$$
\end{theorem}

\begin{proof}[Proof that Theorem \ref{tReduction} implies Theorem \ref{tChowla}.]
Let $c_1, c_2, c > 0$ be the real numbers guaranteed by Theorem \ref{tReduction}. Set $K := \min\left(\lfloor 0.01 \log(n) \rfloor, \lfloor \frac{c_2 \log(n)}{2 \log(c)} \rfloor\right)$. We claim the existence of $C_1, C_2 > 0$ such that
\begin{equation}
\label{eBigClaim}
\# \left\{x \in X : \, C_K^{(h)}(\chi_x)^\circ \neq \{0\}\right\} \leq \frac{C_1|X|}{n^{C_2}}
\end{equation}
for every $h \in \{1, -1\}$ and for every grid $X$ for $n$.

Once the claim is proven, then Theorem \ref{tChowla} holds. Indeed, by Proposition \ref{pNonCrit2} (and Remark \ref{rSwitching} to reduce to the monic case) we have
\begin{multline*}
\# \left\{f \in \FF_q[T] : \deg(f) = n, \ f \textup{ is squarefree}, \ L\left(\tfrac{1}{2}, \chi_f\right) = 0\right\} \leq \\
(q - 1) \sum_{h \in \{1, -1\}} \# \left\{f \in \FF_q[T] : \deg(f) = n, \ f \textup{ is monic and squarefree}, \ C_K^{(h)}(\chi_f)^\circ \neq \{0\}\right\}.
\end{multline*}
By Theorem \ref{tGrid}, the polynomials $f$ that do not lie in a grid $X$ for $n$ fit into the error term of Theorem \ref{tChowla}, and for the remaining $f$, we partition over all grids $X$ for $n$ and apply the bound in equation \eqref{eBigClaim}, thus establishing Theorem \ref{tChowla}.

Therefore it remains to prove equation \eqref{eBigClaim}. Since $n$ is odd, we record that
\[
\ker(R) = \ker\bigl(\mathrm{Art}^{(h)}_{x, 1}\bigr) = W_1^{(h)}(x)
\]
thanks to Remark \ref{rmk: identifying spaces}, which we will use repeatedly without further mention. We say that a bilinear pairing $R: W_0^{(h)} \times W_0^{(h)} \rightarrow \mathbb{F}_2$ is bad if:
\begin{enumerate}
\item[(1)] there exists a non-zero element $\mathbf{v} \in W_0^{(h)}$ such that $\mathbf{v}$ is not $(K, n)$-generic and such that $\mathbf{v}$ is in the left kernel of $R$,
\item[(2)] the kernel of $R$ has dimension larger than $K/10$.
\end{enumerate}
We will also say that $R$ is bad for reason $(1)$ respectively for reason $(2)$, taking the obvious meaning. 

Let us first estimate the number of $x \in X$ for which $\mathrm{Art}_{x, 1}^{(h)}$ is bad for reason $(1)$. Using Definition \ref{dGrid}(4) and using Hoeffding's inequality, it follows that the number of vectors $\mathbf{v} \in W_0^{(h)}$, which are not $(K, n)$-generic, is at most $2^r/n^{C_3}$ for some absolute constant $C_3 > 0$. By Theorem \ref{t4Rank}$(a)$ and Theorem \ref{t4Rank}$(b)$ and the bound $i_{\text{med}} - 1 \leq 0.05\log(n)$ from Definition \ref{dGrid}(3), we deduce that
\begin{align*}
\# \left\{x \in X : \mathrm{Art}_{x, 1}^{(h)} \text{ bad for reason (1)} \right\} &\leq \sum_{\mathbf{v} \neq \mathbf{0} \text{ not} (K, n)\text{-generic}} \# \left\{x \in X : f_x(\mathbf{v}) \in C_1^{(h)}(\chi_x)^\circ\right\} \\
&= O\left(\frac{2^{0.05 \log(n)} |X|}{2^{0.35 \log(n)}} + \frac{2^r}{n^{C_3}} \frac{|X|}{2^r}\right).
\end{align*}
Next, we rule out that $\mathrm{Art}_{x, 1}^{(h)}$ is bad for reason $(2)$. It follows from Theorem \ref{t4Rank}$(a)$ and Theorem \ref{t4Rank}$(b)$ that
$$
\sum_{\mathbf{v} \in W_0^{(h)}} \sum_{x \in X} \mathbf{1}_{f_x(\mathbf{v}) \in C_1^{(h)}(\chi_x)^\circ} = O(|X|).
$$
Hence we conclude that
$$
\# \left\{x \in X : \mathrm{Art}_{x, 1}^{(h)} \text{ bad for reason (2)}\right\} = O\left(\frac{|X|}{2^{K/10}}\right)
$$
by Markov's inequality. By the union bound, we therefore have
$$
\# \left\{x \in X : \mathrm{Art}_{x, 1}^{(h)} \text{ bad}\right\} = O\left(\frac{|X|}{n^{C_4}}\right)
$$
for some absolute $C_4 > 0$. Thus, in order to establish equation \eqref{eBigClaim}, it suffices to prove the existence of $C_1, C_2 > 0$ such that
\begin{equation}
\label{eBigClaim2}
\# \left\{x \in X : \, C_K^{(h)}(\chi_x)^\circ \neq \{0\}, \mathrm{Art}_{x, 1}^{(h)} \text{ not bad}\right\} \leq \frac{C_1|X|}{n^{C_2}}.
\end{equation}
We split over all pairings $R: W_0^{(h)} \times W_0^{(h)} \rightarrow \mathbb{F}_2$ that are not bad. We then split over all non-zero $\mathbf{v}$ in the kernel of $R$. This gives
\begin{multline*}
\# \left\{x \in X : \, C_K^{(h)}(\chi_x)^\circ \neq \{0\}, \mathrm{Art}_{x, 1}^{(h)} \text{ not bad}\right\} \leq \\
\sum_{R \text{ not bad}} \sum_{\mathbf{v} \in \ker(R) - \{0\}} \# \left\{x \in X : \, f_x(\mathbf{v}) \in C_K^{(h)}(\chi_x)^\circ, \mathrm{Art}_{x, 1}^{(h)} = R\right\}.
\end{multline*}
By construction of $K$ and using that $R$ is not bad for reason (1), it follows from Theorem \ref{tReduction} that for all $4 \leq k < K$ 
\begin{multline*}
\# \left\{x \in X : \, f_x(\mathbf{v}) \in C_{k + 1}^{(h)}(\chi_x)^\circ, \mathrm{Art}_{x, 1}^{(h)} = R\right\} \leq \\
\frac{1}{2} \# \left\{x \in X : f_x(\mathbf{v}) \in C_k^{(h)}(\chi_x)^\circ, \mathrm{Art}_{x, 1}^{(h)} = R\right\} + \frac{c_1 |X|}{\exp(n^{c_2}/c^K)}.
\end{multline*}
Since $\dim_{\FF_2} C_4^{(h)}(\chi_x)^\circ \leq \dim_{\FF_2} \ker(R) \leq K/10$ by recalling that $R$ is not bad for reason (2), equation \eqref{eBigClaim2} follows from a small computation.
\end{proof}

Henceforth we fix $n, X, m, R, \mathbf{v}, h$ as in the statement of Theorem \ref{tReduction}. By genericity, we may fix a subset $S_{\text{var}}$ and $S_{\text{spl}}$ as in Definition \ref{dGeneric}. Define $s := |S_{\text{var}}|$. We shall often implicitly biject $[s]$ with $S_{\text{var}}$ (implicitly fixing one such choice for the rest of this paper).

We also fix an odd integer $(\log n)^{100} < j < 2(\log n)^{100}$ such that it is not of the form $n_i$ for any $1 \leq i \leq r$. Such an integer $j$ always exists for $n$ sufficiently large by definition of a grid. Moreover, we define $S_{\emptyset, \text{below}}$ to be the set of monic irreducible polynomials of degree $j$.

Given a cube $F \in \mathrm{Cube}(\prod_{i \in S_{\text{var}}} X_i)$ and a subset $S \subseteq S_{\text{var}}$, we define $\mathrm{Cube}(S, F)$ to be the natural element in $\mathrm{Cube}(\prod_{i \in S} X_i)$. We now consider a certain set of properties that a cube may or may not satisfy.

\begin{mydef} 
Let $\mathbf{v} \in W_0^{(h)}$ be $(m, n)$-generic, let $F = (F_1, \dots, F_s) \in \mathrm{Cube}(\prod_{i \in S_{\text{var}}} X_i)$ and let $P \in \prod_{i \in [r] - S_{\text{var}}} X_i$. Throughout this definition, we shall identify a pair $(F, P)$ with the product space
$$
C := \{F_1(1), F_1(2)\} \times \cdots \times \{F_s(1), F_s(2)\} \times \{P\},
$$
and we shall view for each $\mathbf{a} \in \{1, 2\}^s$ the pair $(F(\mathbf{a}), P)$ as an element of $C$. Set $P = (p_{s + 1}, \ldots, p_r)$ and
\[
a := \prod_{\substack{s < j \leq r \\ \operatorname{pr}_j(\mathbf{v}) = 1}} p_j,
\qquad
d := \prod_{\substack{s < j \leq r \\ \operatorname{pr}_j(\mathbf{v}) = 0}} p_j.
\]
We say that $(F, P)$ is wonderful if
\begin{enumerate}
\item[$(W1)$] $(F, P)$ is acceptable, and
\item[$(W2)$] $(F, P)$ is profitable.
\end{enumerate}
\end{mydef}

\begin{mydef}
Given a symmetric pairing $R: W_0^{(h)} \times W_0^{(h)} \rightarrow \mathbb{F}_2$, we define $Y(R, \mathbf{v})$ to be the subset of $x \in X$ such that $\mathrm{Art}_{x, 1}^{(h)} = R$ and $\mathbf{v} \in W_{m - 1}^{(h)}(x)$. For a collection $\mathbf{T}$ of types, define $Y(R, \mathbf{T}, \mathbf{v})$ to be the subset of $x \in Y(R, \mathbf{v})$ such that $\mathbf{T} \subseteq \mathcal{T}_m(x, \mathbf{v})$. 
\end{mydef}

\begin{lemma}
\label{lFuncSeqEq}
There exists an absolute constant $C > 0$ such that the following holds.

Let $R, \mathbf{T}, \mathbf{v}$ be given. For each element $P \in \prod_{i \not \in S_{\textup{var}}} X_i$, there exists a functional sequence $f_1, \dots, f_s$ of equivalence relations on $\prod_{i \in S_{\textup{var}}} X_i$ of complexity at most $\exp(C \cdot r^2 \cdot 2^s)$ such that the following statement holds:

For all cubes $F = (F_1, \dots, F_s) \in \mathrm{Cube}(\prod_{i \in S_{\textup{var}}} X_i)$ such that $(F(\mathbf{a}), P) \in Y(R, \mathbf{T}, \mathbf{v})$ for all $\mathbf{a} \in \{1, 2\}^{S_{\textup{var}}}$, we have that $F$ satisfies $f_1, \dots, f_s$ if and only if $(F, P)$ is wonderful.
\end{lemma}

\begin{proof}
In order to lighten the notation, we relabel the indices in $[r]$ in such a way that $S_{\text{var}}$ is identified with $[s]$.

We will construct the elements $f_h$ of the desired functional sequence $f_1, \ldots, f_s$ of equivalence relations inductively. This will be done in four separate subsections. In Subsection I, we will make a list of formal demands on our sequence $f_h$. In Subsection II, we prove the base case of the induction. In Subsection III, we focus on the inductive step. In Subsection IV, we finish the proof.

We recall that
\[
P = (p_{s + 1}, \ldots, p_r)
\]
and that
\[
a := \prod_{\substack{s < j \leq r \\ \operatorname{pr}_j(\mathbf{v}) = 1}} p_j,
\qquad
d := \prod_{\substack{s < j \leq r \\ \operatorname{pr}_j(\mathbf{v}) = 0}} p_j.
\]
Then $f_x(\mathbf{v}) = \chi_a$ for every $x = (F(\mathbf{a}), P)$. Since $\mathbf{v} \in W^{(h)}_0$, the polynomial $a$ has even degree. Since every prime in $S_{\text{var}}$ has even degree and the total degree is odd, the polynomial $d$ has odd degree.

\subsubsection*{I: Definition of helpful sequences} 
Let $h \in [s]$. We say that $f_1, \ldots, f_h$ is \emph{helpful} if for all cubes $F \in \mathrm{Cube}\left(\prod_{i \in [h]} X_i\right)$ and all points $P' \in \prod_{h < j \leq r} X_j$ that project to $P$ outside of $[s]$ and satisfy $(F(\mathbf{a}), P') \in Y(R, \mathbf{T}, \mathbf{v})$ for all $\mathbf{a} \in \{1, 2\}^{[h]}$, we have that $(F, P')$ satisfies $f_1, \ldots, f_h$ if and only if the following demands are met:
\begin{enumerate}
    \item[$(C1)$] The map $\phi_{\{F_i(1)F_i(2) : i \in [h]\}; a}(\mathfrak{G})$ exists and $L(\phi_{\{F_i(1)F_i(2) : i \in [h]\}; a}(\mathfrak{G}))/\mathbb{F}_q(T)$ has residue field degree $1$ at all the primes in $P'$ and splits completely at $\infty$.
    \item[$(C2)$] For each $i \in [h]$, we have that the map $\phi_{\{F_j(1)F_j(2) : j \in [h] - \{i\}\}; F_i(1)F_i(2)}(\mathfrak{G})$ exists and $L(\phi_{\{F_j(1)F_j(2) : j \in [h] - \{i\}\}; F_i(1)F_i(2)}(\mathfrak{G}))/\mathbb{F}_q(T)$ splits completely at all primes in $P'$ and at $\infty$.
\end{enumerate}
We will now inductively construct $(f_1, \ldots, f_s)$ in such a way that they form a helpful sequence. 

\subsubsection*{II: Construction of $f_1$} 
The goal of this subsection is to define $f_1$. For each $P' \in \prod_{2 \leq j \leq r} X_j$, put
\[
Z(P') := \{\pi \in X_1 : (\pi, P') \in Y(R, \mathbf{T}, \mathbf{v})\}.
\]
We will define two equivalence relations $\sim_{P'}^{(1)}$ and $\sim_{P'}^{(2)}$ on $Z(P')$, where $\sim_{P'}^{(2)}$ is a refinement of $\sim_{P'}^{(1)}$. Once this is done, we define $f_1(P')$ to be the relation $\sim_{P'}^{(2)}$ on $Z(P')$, and we extend this to $X_1$ by declaring all remaining elements in $X_1 - Z(P')$ to be equivalent.

To start, for $\pi_1, \pi_2 \in Z(P')$, we declare
$$
\pi_1 \sim_{P'}^{(1)} \pi_2 \Longleftrightarrow \left(\frac{\pi_1\pi_2}{\text{pr}_i(P')}\right) = 1 \text{ for all } 2 \leq i \leq r.
$$
By multiplicativity of the Legendre symbol, $\sim_{P'}^{(1)}$ is an equivalence relation. 

Observe that for $(\pi_1, P'), (\pi_2, P') \in Y(R, \mathbf{T}, \mathbf{v})$, the equivalence
$$
\pi_1 \sim_{P'}^{(1)} \pi_2
$$
forces the map $\phi_{\pi_1\pi_2; a}(\mathfrak{G})$ to exist. Indeed, to see this, we apply Proposition \ref{inductively construct expansion maps}: the local conditions at the prime divisors of $a$ are guaranteed by $\sim_{P'}^{(1)}$, while the local conditions at $\pi_1,\pi_2$ are respectively guaranteed by $\chi_a$ being in the kernel of $R$.

Having shown this, let now $\pi_1, \pi_2 \in Z(P')$ be such that $\pi_1 \sim_{P'}^{(1)} \pi_2$. We set 
\[
\pi_1\sim_{P'}^{(2)}\pi_2
\Longleftrightarrow 
L\bigl(\phi_{\pi_1\pi_2;a}(\mathfrak{G})\bigr)/\mathbb F_q(T)
\text{ has residue field degree }1
\text{ at }\operatorname{pr}_i(P')
\text{ for }2 \leq i \leq r,
\]
\[
\text{and splits completely at infinity.}
\]
We remark that this is an equivalence relation in view of Proposition \ref{pExpansionAdditive}.

We now claim that choosing $f_1 := \sim_{P'}^{(2)}$ is helpful. Take $\pi_1, \pi_2 \in X_1$ with $(\pi_1, P'), (\pi_2, P') \in Y(R, \mathbf{T}, \mathbf{v})$. Write $F_1: \{1, 2\} \rightarrow X_1$ for the corresponding map.

If $\pi_1\sim_{P'}^{(2)}\pi_2$, then $(C1)$ holds by construction. For $h = 1$, the map occurring in $(C2)$ is
\[
\phi_{\emptyset; \pi_1\pi_2}(\mathfrak{G}) = \chi_{\pi_1\pi_2}.
\]
The relation $\pi_1 \sim_{P'}^{(1)} \pi_2$ says precisely that $\mathbb{F}_q(T)(\chi_{\pi_1\pi_2})/\mathbb{F}_q(T)$ splits completely at every prime in $P'$. It also splits completely at $\infty$ because $\pi_1\pi_2$ is monic of even degree. Thus $(C2)$ holds.

Conversely, suppose that $(C1)$ and $(C2)$ hold. Condition $(C2)$ says that every prime in $P'$ splits completely in $\mathbb{F}_q(T)(\chi_{\pi_1\pi_2})/\mathbb F_q(T)$. Therefore $\pi_1\pi_2$ is a square modulo every prime occurring in $P'$, and hence $\pi_1 \sim_{P'}^{(1)} \pi_2$. The residue field one conditions and the condition at $\infty$ in $(C1)$ then give $\pi_1 \sim_{P'}^{(2)} \pi_2$.

\subsubsection*{III: The inductive step} 
Recall that $s \in \Z_{\geq 3}$, so $s - 2 \geq 1$. With this in mind, let $h \in [s - 1]$ and suppose that we have constructed a helpful functional sequence of equivalence relations $(f_1, \ldots, f_h)$. We are now going to construct $f_{h + 1}$. 

Let $F_{\text{pre}} \in \text{Cube}(\prod_{i \in [h]} X_i)$ and let $P_{\text{post}}' \in \prod_{h + 2 \leq i \leq r} X_i$. Put
\[
Z(F_{\mathrm{pre}}, P'_{\mathrm{post}})
:=
\left\{
\pi\in X_{h+1}:
\begin{array}{l}
(F_{\mathrm{pre}}(\boldsymbol\epsilon),\pi,
 P'_{\mathrm{post}})
 \in Y(R, \mathbf{T}, \mathbf{v})
 \text{ for every }
 \boldsymbol\epsilon\in\{1,2\}^{[h]},\\[2mm]
(F_{\mathrm{pre}},\pi,P'_{\mathrm{post}})
 \text{ satisfies } f_1, \ldots, f_h
\end{array}
\right\}.
\]
We will construct an equivalence relation on $Z(F_{\mathrm{pre}}, P'_{\mathrm{post}})$. We extend this to an equivalence relation on $X_{h + 1}$ by declaring the remaining elements $X_{h + 1} - Z(F_{\mathrm{pre}}, P'_{\mathrm{post}})$ to be one additional equivalence class. Take
\[
\pi_1, \pi_2 \in Z(F_{\mathrm{pre}}, P'_{\mathrm{post}}).
\]
We shall now define a sequence of equivalence relations on $Z(F_{\mathrm{pre}}, P'_{\mathrm{post}})$, which we denote by $\sim^{(k)}_{(F_{\mathrm{pre}}, P'_{\mathrm{post}})}$, where the $(k + 1)$-th relation is going to be a refinement of the $k$-th relation. We begin explaining the first step of this construction. We define
$$
\sim_{(F_{\text{pre}}, P_{\text{post}}')}^{(1)}
$$
to be the equivalence relation given by demanding that $\pi_1\pi_2$ is a square locally at all of the primes in $F_{\text{pre}}$ and $P_{\text{post}}'$. This is, as explained in Subsection II, an equivalence relation. Moreover, this already forces the following maps to exist by Proposition \ref{inductively construct expansion maps}
\begin{enumerate}
    \item $\phi_{\pi_1\pi_2; a}(\mathfrak{G})$, and
    \item $\phi_{F_i(1)F_i(2); \pi_1\pi_2}(\mathfrak{G})$ for each $i \in [h]$, and
    \item $\phi_{\pi_1\pi_2; F_i(1)F_i(2)}(\mathfrak{G})$ for each $i \in [h]$.
\end{enumerate}
We shall now explain how to inductively iterate this construction to construct an equivalence relation $\sim^{(k_0)}_{(F_{\mathrm{pre}},P'_{\mathrm{post}})}$ for every integer $1 \leq k_0 \leq h + 2$. By the induction hypothesis on $\sim^{(k_0)}_{(F_{\mathrm{pre}},P'_{\mathrm{post}})}$, we have the following conditions:
\begin{enumerate}
    \item $\phi_{\{F_i(1)F_i(2) : i \in S\} \cup \{\pi_1\pi_2\}; a}(\mathfrak{G})$ exists for each subset $S \subseteq [h]$ with $|S| \leq k_0 - 1$,
    \item $\phi_{\{F_j(1)F_j(2) : j \in S\}; \pi_1\pi_2}(\mathfrak{G})$ exists for each subset $S \subseteq [h]$ with $|S| \leq k_0$,
    \item $\{\phi_{\{F_i(1)F_i(2) : i \in S - \{j\}\} \cup \{\pi_1\pi_2\}; F_j(1)F_j(2)}(\mathfrak{G}) : j \in S\}$ exists for each subset $S \subseteq [h]$ with $|S| \leq k_0$,
\end{enumerate}
together with the following list of splitting requirements
\begin{enumerate}
    \item $L(\phi_{\{F_i(1)F_i(2) : i \in S\} \cup \{\pi_1\pi_2\}; a}(\mathfrak{G}))/\mathbb{F}_q(T)$ has residue field degree $1$ at each place of $F_{\text{pre}}$ and $P_{\text{post}}'$ different from $\{F_i(1), F_i(2) : i \in S\}$ and splits completely at $\infty$, for each subset $S \subseteq [h]$ with $|S| \leq k_0 - 2$, and
    \item $L(\phi_{\{F_j(1)F_j(2) : j \in S\}; \pi_1\pi_2}(\mathfrak{G}))/\mathbb{F}_q(T)$ splits completely at all of the primes in $F_{\text{pre}}$ different from $\{F_i(1), F_i(2) : i \in S\}$ and all of $P_{\text{post}}'$ and $\infty$, for each subset $S \subseteq [h]$ with $|S| \leq k_0 - 1$, and
    \item each of the fields $\{L(\phi_{\{F_i(1)F_i(2) : i \in S - \{j\}\} \cup \{\pi_1\pi_2\}; F_j(1)F_j(2)}(\mathfrak{G})) : j \in S\}$ splits completely at all of the primes in $F_{\text{pre}}$ different from $\{F_i(1), F_i(2) : i \in S\}$ and all of $P_{\text{post}}'$ and $\infty$, for each subset $S \subseteq [h]$ with $|S| \leq k_0 - 1$.
\end{enumerate}
For $1 \leq k_0 \leq h + 1$, assuming $\pi_1 \sim^{(k_0)}_{(F_{\text{pre}}, P_{\text{post}}')} \pi_2$, we now declare $\pi_1 \sim_{(F_{\text{pre}}, P_{\text{post}}')}^{(k_0 + 1)} \pi_2$ if and only if:
\begin{enumerate}
\item For every subset $S \subseteq [h]$ of cardinality $k_0 - 1$, the extension
\[
L\bigl(
\phi_{\{F_i(1)F_i(2) : i\in S\} \cup \{\pi_1\pi_2\}; a}(\mathfrak{G})
\bigr)/\mathbb F_q(T)
\]
has residue field degree $1$ at every prime occurring in $F_{\mathrm{pre}}$ or $P'_{\mathrm{post}}$ other than $F_i(1), F_i(2)$ for $i \in S$, and splits completely at $\infty$.
\item For every subset $S \subseteq [h]$ of cardinality $k_0$, the extension
\[
L\bigl(
\phi_{\{F_i(1)F_i(2):i\in S\};\pi_1\pi_2}(\mathfrak{G})
\bigr)/\mathbb F_q(T)
\]
splits completely at every prime occurring in $F_{\mathrm{pre}}$ or $P'_{\mathrm{post}}$ other than $F_i(1),F_i(2)$ for $i \in S$, and also splits completely at $\infty$.
\item For every subset $S \subseteq [h]$ of cardinality $k_0$ and every $j \in S$, the extension
\[
L\bigl(\phi_{\{F_i(1)F_i(2) : i\in S - \{j\}\} \cup \{\pi_1\pi_2\}; F_j(1)F_j(2)}(\mathfrak{G})\bigr)/\mathbb F_q(T)
\]
splits completely at every prime occurring in $F_{\mathrm{pre}}$ or $P'_{\mathrm{post}}$ other than $F_i(1), F_i(2)$ for $i \in S$, and also splits completely at $\infty$.
\end{enumerate}
Thanks to Proposition \ref{pExpansionAdditive}, this defines an equivalence relation. And enforcing
$$
\pi_1 \sim_{(F_{\text{pre}},P_{\text{post}}')}^{(k_0 + 1)} \pi_2
$$
precisely guarantees the splitting conditions for these maps allowing to continue the induction. We stop after constructing $\sim^{(h + 2)}_{(F_{\mathrm{pre}}, P'_{\mathrm{post}})}$, and define
\[
f_{h + 1}(F_{\mathrm{pre}}, P'_{\mathrm{post}}) := \sim^{(h + 2)}_{(F_{\mathrm{pre}}, P'_{\mathrm{post}})}
\]
on $Z(F_{\mathrm{pre}}, P'_{\mathrm{post}})$, extending it to $X_{h + 1}$ by declaring all elements in $X_{h + 1} - Z(F_{\mathrm{pre}}, P'_{\mathrm{post}})$ to be equivalent.

When $k_0 = h + 2$, the imposed conditions give exactly $(C1)$ and $(C2)$. Conversely, if $(C1)$ and $(C2)$ hold, the necessity
direction of Proposition \ref{inductively construct expansion maps} gives
\[
\pi_1 \sim^{(h + 2)}_{(F_{\mathrm{pre}}, P'_{\mathrm{post}})} \pi_2.
\]
Thus $f_1, \ldots, f_{h + 1}$ is helpful.

\subsubsection*{IV: End of the proof}
Let $F = (F_1, \ldots, F_s) \in \operatorname{Cube}\left(\prod_{i \in [s]} X_i\right)$ and suppose that
\[
(F(\boldsymbol{\epsilon}), P) \in Y(R, \mathbf{T}, \mathbf{v}) \qquad \text{for every } \boldsymbol{\epsilon} \in \{1, 2\}^{[s]}.
\]
Put $u_i := F_i(1)F_i(2)$ for $i \in [s]$, and let $C$ be the resulting $s$-dimensional cube. Suppose first that $F$ satisfies $f_1, \ldots, f_s$. Condition $(A1)$ holds, since it follows from the first equivalence relation imposed on each $X_h$. The remaining acceptable conditions follow exactly from $(C1)$ (applied with $h = s$). Moreover, $C$ is profitable exactly by $(C2)$. Thus $(F, P)$ is wonderful.

We now prove the converse. But indeed, this follows from the necessity direction of Proposition \ref{inductively construct expansion maps}; the conditions to apply this proposition hold by $(A1)$.

It remains to bound the complexity. At the construction of $f_{h + 1}$, we invoke $|S| + 2$ expansion maps for every fixed subset $S \subseteq [h]$. Each map is evaluated at $O(r)$ primes. Hence the number of bits needed to define $f_{h + 1}$ is 
$$
O\left( r \sum_{S \subseteq [h]} (2 + |S|) \right) = O\left( r \cdot (2^{h + 1} + h 2^{h - 1}) \right) = O \left(r \cdot (h + 1) \cdot 2^h \right).
$$
Consequently, for some absolute constants $C_1, C_2 > 0$,
\[
\operatorname{complexity}(f_1, \ldots, f_s) \leq 2^{C_1rs2^s} \leq \exp(C_2 \cdot r^2 \cdot 2^s).
\]
This completes the proof.
\end{proof}

We have now all the ingredients to prove Theorem \ref{tReduction}.

\begin{proof}[Proof of Theorem \ref{tReduction}]
We fix $n$, $X$, $m$, $R$, $\mathbf{v}$ and $h$ as in the statement of Theorem \ref{tReduction}. Recall that, given this data, we have fixed sets $S_{\text{var}}$ and $S_{\text{spl}}$, and recall that $s := |S_{\text{var}}| = m - 2$. 

For $x \in X$ satisfying $\mathrm{Art}_{x, 1}^{(h)} = R$, the condition $x \in Y(R, \mathbf{v})$ is equivalent to $\mathcal{T}_m(x, \mathbf{v}) \neq \emptyset$. Observe that the cardinality of $\mathcal{T}_m(x, \mathbf{v})$ is $0$ or $2^\alpha$ for some $0 \leq \alpha \leq (r + 1) m$. Let $P$ be the unique polynomial of degree $1 + (r + 1) m$ that satisfies
$$
P(0) = 0, \quad P(2^\alpha) = 1 \text{ for all } 0 \leq \alpha \leq (r + 1) m,
$$
so $P(X) = 1 - \prod_{0 \leq \alpha \leq (r + 1) m} \left(1 - \frac{X}{2^\alpha}\right)$. For $x \in X$ satisfying $\mathrm{Art}_{x, 1}^{(h)} = R$, it follows that
$$
\mathbf{1}_{x \in Y(R, \mathbf{v})} = \mathbf{1}_{\mathcal{T}_m(x, \mathbf{v}) \neq \emptyset} = P(|\mathcal{T}_m(x, \mathbf{v})|).
$$
Since $P(0) = 0$, we can write $P(X) = \sum_{i = 1}^{1 + (r + 1)m} a_i X^i$ with $|a_i| \leq 1 + 2^{1 + (r + 1) m}$. Therefore we have
$$
P(|\mathcal{T}_m(x, \mathbf{v})|) = \sum_{i = 1}^{1 + (r + 1)m} a_i |\mathcal{T}_m(x, \mathbf{v})|^i = \sum_{i = 1}^{1 + (r + 1)m} a_i \sum_{\mathbf{t}_1, \dots, \mathbf{t}_i} \mathbf{1}_{\{\mathbf{t}_1, \dots, \mathbf{t}_i\} \subseteq \mathcal{T}_m(x, \mathbf{v})}.
$$
As our first move, we fix $1 \leq i \leq 1 + (r + 1) m$ and we fix a collection of types $\mathbf{T} = \{\mathbf{t}_1, \dots, \mathbf{t}_i\}$. Since $i \geq 1$, we note that $\mathbf{T}$ is non-empty, so in particular we may fix one of its elements $\tilde{\mathbf{t}}$. By the triangle inequality, we then have to estimate
$$
\left| \sum_{x \in Y(R, \mathbf{T}, \mathbf{v})} g(x) \right|.
$$
As our next move, we apply H\"older's inequality to bound this sum by 
\begin{equation}
\label{eOpeningHolder}
\left(\prod_{i \not \in S_{\text{var}}} |X_i|\right)^{\frac{2^s - 1}{2^s}} \times \left(\sum_{\substack{P \in \prod_{i \not \in S_{\text{var}}} X_i}} \left| \sum_{\substack{Q \in \prod_{i \in S_{\text{var}}} X_i \\ (Q, P) \in Y(R, \mathbf{T}, \mathbf{v})}} g(Q, P) \right|^{2^s}\right)^{1/2^s},
\end{equation}
where we have applied the trivial bound to estimate the total number of choices for $P$ by $\prod_{i \not \in S_{\text{var}}} |X_i|$. We will now aim to bound the second term in \eqref{eOpeningHolder}.

For fixed $R$, $\mathbf{T}$, $\mathbf{v}$ and $P$, let $f_1, \dots, f_s$ be the functional sequence of equivalence relations from Lemma \ref{lFuncSeqEq}. We apply Lemma \ref{lRCS} for each such fixed $R$, $\mathbf{T}$, $\mathbf{v}$ and $P$ (with the finite sets being $X_i$ with $i \in S_{\text{var}}$). This shows that the second term in \eqref{eOpeningHolder} is at most
\begin{multline*}
\exp\big(C \cdot r^2 \cdot 2^s\big) \prod_{i \in S_{\text{var}}} |X_i|^{\frac{2^{s - 1} - 1}{2^{s - 1}}} \times \\
\left( \sum_{\substack{P \in \prod_{i \not \in S_{\text{var}}} X_i}} \sum_{\substack{F \in \mathrm{Cube}(\prod_{i \in S_{\text{var}}} X_i), \ (F, P) \text{ wonderful} \\ \forall \, \mathbf{a} \in  \{1, 2\}^{S_{\text{var}}} \, : \, (F(\mathbf{a}), P) \in Y(R, \mathbf{T}, \mathbf{v})}} \prod_{\mathbf{a} \in \{1, 2\}^{S_{\text{var}}}} g(F(\mathbf{a}), P) \right)^{1/2^s}
\end{multline*}
for some absolute constant $C$. We will focus again on the second expression, which we write as $\Sigma_1^{1/2^s}$. Define $S_{\text{mi}} := [r] - S_{\text{var}} - S_{\text{spl}}$. 
By the triangle inequality, we have the estimate
$$
\Sigma_1 \leq \sum_{\substack{P \in \prod_{i \in S_{\text{mi}}} X_i \\ Q_2 = (Q_{i, 2})_i \in \prod_{i \in S_{\text{var}}} X_i}} \left| \sum_{\substack{Q_1 = (Q_{i, 1})_i \in \prod_{i \in S_{\text{var}}} X_i \\ \mathbf{r} \in \prod_{i \in S_{\text{spl}}} X_i, \ \ ((Q_1, Q_2), P \times \mathbf{r}) \text{ wonderful} \\ \forall \, \mathbf{a} \in \{1, 2\}^{S_{\text{var}}} \, : \, (\prod_{i \in S_{\text{var}}} Q_{i, \mathrm{pr}_i(\mathbf{a})}, P \times \mathbf{r}) \in Y(R, \mathbf{T}, \mathbf{v})}} \hspace{-1.5cm} \prod_{\mathbf{a} \in \{1, 2\}^{S_{\text{var}}}} g(F(\mathbf{a}), P \times \mathbf{r}) \right|,
$$
where we have identified a pair $(Q_1, Q_2)$ with an element of $\mathrm{Cube}(\prod_{i \in S_{\text{var}}} X_i)$ in the natural way. Henceforth we write $C$ for the cube corresponding to $(Q_1, Q_2)$.

We will work towards an application of Proposition \ref{prop: sums of psi trivializes alpha alpha bar} and Theorem \ref{theorem: reflection principle}, which will affect our sum in two different but vital ways. Put
\[
a := \prod_{\substack{i \not \in S_{\text{var}} \\ \text{pr}_i(\mathbf{v}) = 1}} \text{pr}_i(x), \qquad
d := \prod_{\substack{i \not \in S_{\text{var}} \\ \text{pr}_i(\mathbf{v}) = 0}} \text{pr}_i(x).
\]
where these expressions are independent of the choice of $x\in C$.

We first observe that, under the current summation conditions, every vertex of $C$, including $x_0$, belongs to $Y(R, \mathbf{T}, \mathbf{v})$. We will define $C$ to be $2$-wonderful if it is also desirable of type $\mathbf{t}$ for each $\mathbf{t} \in \mathbf{T}$. We claim that all wonderful $C$ are automatically $2$-wonderful. Firstly, condition $(T)$ holds because every vertex belongs to $Y(R, \mathbf{T}, \mathbf{v})$.

Secondly, let $\Psi_{s + 1}(x_0; \mathbf{t})$ be the $1$-cochain defined by the right-hand side of equation \eqref{ex0}. Proposition \ref{prop: sums of psi trivializes alpha alpha bar} part $(i)$ shows that $\Psi_{s + 1}(x_0; \mathbf{t})$ is an $h$-raw cocycle of type $\mathbf{t}$ for each $\mathbf{t} \in \mathbf{T}$. On the other hand, membership of $x_0$ in $Y(R, \mathbf{T}, \mathbf{v})$ supplies an $h$-raw cocycle $\psi_{s + 2}(x_0; \mathbf{t})$, and $2\psi_{s + 2}(x_0; \mathbf{t})$ is also of type $\mathbf{t}$. The two cocycles agree on the topological generators from Proposition \ref{Topological generators}, and therefore
\[
\Psi_{s + 1}(x_0; \mathbf{t}) = 2\psi_{s + 2}(x_0; \mathbf{t}).
\]
Proposition \ref{prop: Pi detects 2} and equation \eqref{eGreatPi} now give condition $(D)$. Thus every wonderful cube occurring in the present sum is $2$-wonderful.

Now suppose that $C$ is $2$-wonderful and that all its vertices other than $x_0$ belong to $Y(R, \mathbf{T}, \mathbf{v})$. Proposition \ref{prop: sums of psi trivializes alpha alpha bar} part $(i)$ and $(ii)$ then produce an $h$-raw cocycle $\psi_{s + 2}(x_0; \mathbf{t})$ such that $\psi_{s + 1}(x_0; \mathbf{t})$ is of type $\mathbf{t}$ for each $\mathbf{t} \in \mathbf{T}$. Hence the condition $x_0 \in Y(R, \mathbf{T}, \mathbf{v})$ may be replaced by
\[
\operatorname{Art}_{x_0,1}^{(h)} = R.
\]
But this equality follows from
\[
\operatorname{Art}_{x_0Q_{i,1}Q_{i,2},1}^{(h)} = R
\]
for any fixed $i \in S_{\mathrm{var}}$, together with condition $(A1)$. We may therefore remove the condition $x_0\in Y(R, \mathbf{T}, \mathbf{v})$ from the summation conditions.

Define
$$
a' := \prod_{\substack{i \in S_{\text{mi}} \\ \text{pr}_i(\mathbf{v}) = 1}} \text{pr}_i(x).
$$ 
We also pick indices $i_1, i_2, i_3$ enumerating $S_{\text{spl}}$, i.e.~$S_{\text{spl}} := \{i_1, i_2, i_3\}$. Defining $C^- := \{1, 2\}^{S_{\text{var}}} - \{(1, \dots, 1)\}$, the second part of Theorem \ref{theorem: reflection principle} (applied to the type $\tilde{\mathbf{t}}$ we fixed at the start of the proof) allows us to rewrite our sum as 
$$
\sum_{\substack{P \in \prod_{i \in S_{\text{mi}}} X_i \\ Q_2 = (Q_{i, 2})_i \in \prod_{i \in S_{\text{var}}} X_i}} \left| \sum_{\substack{Q_1 = (Q_{i, 1})_i \in \prod_{i \in S_{\text{var}}} X_i \\ \mathbf{r} := (r_i)_{i \in S_{\text{spl}}} \in \prod_{i \in S_{\text{spl}}} X_i \\ ((Q_1, Q_2), P \times \mathbf{r}) \text{ 2-wonderful} \\ \forall \, \mathbf{a} \in C^- \, : \, (\prod_{i \in S_{\text{var}}} Q_{i, \mathrm{pr}_i(\mathbf{a})}, P \times \mathbf{r}) \in Y(R, \mathbf{T}, \mathbf{v})}} (-1)^{\mathrm{Re}(P \times \mathbf{r}, C)} \right|,
$$
where we set $\phi := \phi_{C; a' r_{i_1}r_{i_2}r_{i_3}}(\mathfrak{G}) + \phi_{[s]}(C; \tilde{\mathbf{t}})$ and
\begin{align*}
&\mathrm{Re}(P \times \mathbf{r}, C) := \mathrm{Re}_1(P \times \mathbf{r}, C) + \mathrm{Re}_2(P \times \mathbf{r}, C) \\
&\mathrm{Re}_1(P \times \mathbf{r}, C) := \left[ \sum_{g \in G(C) - \{1\}} \phi^g, \phi, \phi \right]_{F(C)} \\
&\mathrm{Re}_2(P \times \mathbf{r}, C) := \left[ \frac{(1 - h)}{2} \cdot \chi_{\epsilon_{\mathbb{F}_q}} + \chi_{d \prod_{i \in S_{\text{var}}} Q_{i, 1}}, \phi, \phi \right]_{F(C)}.
\end{align*}
We remark that this is a valid splitting of the symbol $\mathrm{Re}(P \times \mathbf{r}, C)$ by Theorem \ref{theorem: Redei reciprocity}, where an explicit verification establishes that both symbols are R\'edei admissible.

We will ultimately exploit bilinear oscillation over $r_{i_j}$ and $r_{i_k}$ for two suitably chosen (distinct) elements $j, k \in \{1, 2, 3\}$. In order to work towards this result, we must understand how our symbols change over the $r_{i_j}$ and $r_{i_k}$ axis, for which we now initiate the necessary preparations.

Define $S_\emptyset$ to be the set of places of $F(C)$ lying above a place in $S_{\emptyset, \text{below}}$ that splits completely in $F(C)/\mathbb{F}_q(T)$. The subgroup $M$ of $\Gal(H_2(F(C))/F(C))$ generated by the Frobenius elements in (any Galois stable subset of) $S_\emptyset$ forms a $\mathbb{F}_2[\Gal(F(C)/\FF_q(T))]$-module. By Nakayama's lemma, $M = \Gal(H_2(F(C))/F(C))$ if and only if $M$ generates modulo the augmentation ideal of $\mathbb{F}_2[\Gal(F(C)/\FF_q(T))]$. Let $K$ be the largest multiquadratic extension of $F(C)$  unramified  at all places of $F(C)$ such that $K$ is Galois over $\FF_q(T)$ and $\Gal(K/F(C))$ is a central subgroup of $\Gal(K/\FF_q(T))$. We claim that
\begin{equation}
\label{eboundforcentral}
[K : F(C)] \leq 2^{\binom{s + 2}{2}}.
\end{equation}
Indeed, by the inflation-restriction sequence, we have 
$$
\langle \{\chi_{Q_{i , 1}} : i \in S_{\text{var}}\} \cup \{\chi_{\epsilon_{\mathbb{F}_q}}\} \rangle \rightarrow \Gal(K/F(C))^\vee \rightarrow H^2(\Gal(F(C)/\mathbb{F}_q(T)), \FF_2).
$$
It is also well-known that $\dim_{\FF_2} \, H^2(\Gal(F(C)/\mathbb{F}_q(T)), \FF_2) = \binom{s + 1}{2}$. Combining this with the above exact sequence, we obtain \eqref{eboundforcentral}.

Writing $G := \Gal(K/\FF_q(T))$ and writing $\pi: G_{\mathbb{F}_q(T)} \rightarrow G$ for the corresponding homomorphism, we conclude that $M = \Gal(H_2(F(C))/F(C))$ generates if and only if
$$
\langle \{\pi(\Frob_v): v \in S_{\emptyset, \text{below}}, \ v \text{ splits completely in } F(C)/\FF_q(T)\} \rangle = \Gal(K/F(C)).
$$
But indeed, this holds by the Chebotarev density theorem, using that the degree of any element in $S_{\emptyset, \text{below}}$ is odd by definition. 

Therefore we can take a Galois stable subset $T(C) \subseteq S_\emptyset$ such that the Frobenius elements of the places in $T(C)$ span $\Gal(H_2(F(C))/F(C))$ and such that
\begin{equation}
\label{eNakaBound}
|T(C)| \leq \binom{s + 2}{2} \cdot 2^s.
\end{equation}
For each $w \in T(C) \cup S(a')$, we fix a choice of $\Frob_w$ and $\sigma_w$ in $G_{F(C)_w}$. For each $k \in [3]$, we also fix an element $(L_{k, w})_{w \in T(C) \cup S(a')} \in (\Z/4\Z \times \Z/4\Z)^{T(C) \cup S(a')}$.  We partition, for each $k \in [3]$, the set $X_{i_k}$ in such a way that 
$$
(\psi_{T(C), r_{i_k}}(\Frob_w), \psi_{T(C), r_{i_k}}(\sigma_w)) = L_{k, w}
$$
for all $w \in T(C) \cup S(a')$. Taking the notation from Theorem \ref{theorem: Applying quartic reciprocity}, we also freeze the behavior of $\phi$ and $\sum_{g \in G(C) - \{1\}} \phi^g$ locally at $w \in T(C)$, i.e.~we fix vectors $\mathbf{a}, \mathbf{b} \in \mathbb{F}_2^{T(C)}$ and impose the summation conditions
$$
(\phi(\Frob_w))_{w \in T(C)} = \mathbf{a}, \quad \left(\sum_{g \in G(C) - \{1\}} \phi^g(\Frob_w)\right)_{w \in T(C)} = \mathbf{b}.
$$
We now apply Theorem \ref{theorem: Applying quartic reciprocity} with the image of $\tilde{\mathbf{t}}$ (whose coordinates inside $S_{\text{var}}$ must live inside $N[2^s]$ by the equation $2^s \psi_{s + 1}(x; \tilde{\mathbf{t}}) = \chi_a$) inside $\mathbb{F}_2^{S_{\text{var}}}$ (using the natural quotient map $N[2^s] \rightarrow \mathbb{F}_2$) to break $\mathrm{Re}_1(P \times \mathbf{r}, C)$ over its various pieces. 
After abbreviating 
$$
\mu := \exp\left(2\pi i \cdot \bigg(\lambda_{\mathbf{a}, \mathbf{b}, a', T(C)}((L_{k, w})_{k \in [3], w \in T(C) \cup S(a')}) + \sum_{k = 1}^3 \Delta_C(r_{i_k})\bigg)\right),
$$
this leads to the sum 
$$
\sum_{\substack{P \in \prod_{i \in S_{\text{mi}}} X_i \\ Q_2 = (Q_{i, 2})_i \in \prod_{i \in S_{\text{var}}} X_i}} \left| \sum_{\substack{Q_1 = (Q_{i, 1})_i \in \prod_{i \in S_{\text{var}}} X_i, \ \mathbf{r} \in \prod_{i \in S_{\text{spl}}} X_i \\ (L_{k, w})_{k \in [3], w \in T(C) \cup S(a')}, \ \mathbf{a}, \mathbf{b} \in \mathbb{F}_2^{T(C)} \\ ((Q_1, Q_2), P \times \mathbf{r}) \text{ 2-wonderful}, \ \forall \, \mathbf{a} \in C^- \, : \, (\prod_{i \in S_{\text{var}}} Q_{i, \mathrm{pr}_i(\mathbf{a})}, P \times \mathbf{r}) \in Y(R, \mathbf{T}, \mathbf{v}) \\ \res_w \psi_{T(C), r_{i_k}} = L_{k, w}, \ (\phi(\Frob_w))_{w \in T(C)} = \mathbf{a}, \ (\sum_{g \in G(C) - \{1\}} \phi^g(\Frob_w))_{w \in T(C)} = \mathbf{b}}} \hspace{-2cm} \mu \cdot (-1)^{B(P \times \mathbf{r}, C)} \right|,
$$
where
$$
B(P \times \mathbf{r}, C) = \mathrm{Re}_2(P \times \mathbf{r}, C) + \sum_{g \in G(C) - \{1\}} \sum_{1 \leq k_1 < k_2 \leq 3} \inv_{\mathrm{AUp}(r_{i_{k_2}})}\left(\chi_{T(C), r_{i_{k_1}}}^g \cup \chi_{T(C), r_{i_{k_2}}}\right).
$$
We will now unwrap the $2$-wonderful condition for a cube $C$ with all of its vertices in $Y(R, \mathbf{T}, \mathbf{v})$ (except for possibly $x_0$). Such cubes still satisfy property $(T)$ for each $\mathbf{t} \in \mathbf{T}$, and thus 
\begin{align*}
\mathbf{1}_{((Q_1, Q_2), P \times \mathbf{r}) \text{ 2-wonderful}} &= \mathbf{1}_{((Q_1, Q_2), P \times \mathbf{r}) \text{ wonderful}} \prod_{\mathbf{t} \in \mathbf{T}} \prod_{i \in S_{\text{var}}} \frac{1}{2}\left(1 + (-1)^{\mathrm{Re}(P \times \mathbf{r}, C, i, \mathbf{t})}\right) \\
&= \frac{\mathbf{1}_{((Q_1, Q_2), P \times \mathbf{r}) \text{ wonderful}}}{2^{s \cdot |\mathbf{T}|}} \sum_{\mathfrak{T} \subseteq \mathbf{T} \times S_{\text{var}}} (-1)^{\sum_{(\mathbf{t}, i) \in \mathfrak{T}} \mathrm{Re}(P \times \mathbf{r}, C, i, \mathbf{t})}.
\end{align*}
where
\begin{multline*}
\mathrm{Re}(P \times \mathbf{r}, C, i, \mathbf{t}) := \sum_{x \in C_i} \mathrm{sg}(x) \psi_{s + 1}(x; \mathbf{t})(\Frob_{Q_{i, 1}}) + \\
\phi_{C; a}(\mathfrak{G})(\Frob_{Q_{i, 1}}(\chi_{x_0})) + \phi_{S_{\text{var}}}(C; \mathbf{t})(\Frob_{Q_{i, 1}}(\chi_{x_0})).
\end{multline*}
Before we proceed, we similarly detect the conditions 
$$
(\phi(\Frob_w))_{w \in T(C)} = \mathbf{a}, \quad \left(\sum_{g \in G(C) - \{1\}} \phi^g(\Frob_w)\right)_{w \in T(C)} = \mathbf{b}
$$
respectively by
\begin{align*}
&\frac{1}{2^{|T(C)|}} \sum_{\mathcal{I}_1 \subseteq T(C)} (-1)^{\sum_{w_1 \in \mathcal{I}_1} \left(\phi(\Frob_{w_1}) - \mathrm{pr}_{w_1}(\mathbf{a})\right)}, \\
&\frac{1}{2^{|T(C)|}} \sum_{\mathcal{I}_2 \subseteq T(C)} (-1)^{\sum_{w_2 \in \mathcal{I}_2} \left(\sum_{g \in G(C) - \{1\}} \phi^g(\Frob_{w_2}) - \mathrm{pr}_{w_2}(\mathbf{b)}\right)}.
\end{align*}
Next, we detect the acceptable condition at the places dividing $a'r_{i_1}r_{i_2}r_{i_3}$ via the expression
\begin{align*}
A(P \times \mathbf{r}, C) &:= \prod_{i \in [r], \mathrm{pr}_i(\mathbf{v}) = 1} \frac{1}{2} \left(1 + (-1)^{\phi_{C; a'r_{i_1}r_{i_2}r_{i_3}}(\mathfrak{G})(\Frob_{\mathrm{pr}_i(P \times \mathbf{r})}(\chi_{a'r_{i_1}r_{i_2}r_{i_3}}))}\right) \\
&= \frac{\sum_{I \subseteq \{i \in [r] : \mathrm{pr}_i(\mathbf{v}) = 1\}} (-1)^{\sum_{i \in I} \phi_{C; a'r_{i_1}r_{i_2}r_{i_3}}(\mathfrak{G})(\Frob_{\mathrm{pr}_i(P \times \mathbf{r})}(\chi_{a'r_{i_1}r_{i_2}r_{i_3}}))}}{2^{|\{i \in [r] : \mathrm{pr}_i(\mathbf{v}) = 1\}|}}.
\end{align*}
Call the exponent in the latter expression $A_1(P \times \mathbf{r}, C, I)$. 

We say that $((Q_1, Q_2), P \times \mathbf{r})$ is $3$-wonderful if it satisfies $(W2)$, is pre-acceptable, and $\phi_{C; a' r_{i_1}r_{i_2}r_{i_3}}(\mathfrak{G})$ has residue field degree $1$ at all places $v$ with $v(d) = 1$. We apply the triangle inequality to move the sums over $Q_1$, $L_{k, w}$, $\mathbf{a}$, $\mathbf{b}$, $\mathfrak{T}$, $I$, $\mathcal{I}_1$, $\mathcal{I}_2$ outside, which fixes $\lambda_{\mathbf{a}, \mathbf{b}, a', T(C)}((L_{k, w})_{k \in [3], w \in T(C) \cup S(a')})$. Therefore, summarizing our work so far, we have arrived at the sum
$$
\sum_{\substack{P \in \prod_{i \in S_{\text{mi}}} X_i \\ Q_1 = (Q_{i, 1})_i \in \prod_{i \in S_{\text{var}}} X_i \\ Q_2 = (Q_{i, 2})_i \in \prod_{i \in S_{\text{var}}} X_i \\ (L_{k, w})_{k \in [3], w \in T(C) \cup S(a')} \\ \mathbf{a}, \mathbf{b} \in \mathbb{F}_2^{T(C)}, \mathfrak{T}, I, \mathcal{I}_1, \mathcal{I}_2}} \left| \sum_{\substack{\mathbf{r} \in \prod_{i \in S_{\text{spl}}} X_i \\
((Q_1, Q_2), P \times \mathbf{r}) \text{ 3-wonderful}, \  \res_w \psi_{T(C), r_{i_k}} = L_{k, w} \\ \forall \, \mathbf{a} \in C^- \, : \, (\prod_{i \in S_{\text{var}}} Q_{i, \mathrm{pr}_i(\mathbf{a})}, P \times \mathbf{r}) \in Y(R, \mathbf{T}, \mathbf{v})}} \hspace{-1cm} (-1)^{B(P \times \mathbf{r}, C, \mathfrak{T}, I, \mathcal{I}_1, \mathcal{I}_2)} \right|,
$$
where
\begin{align*}
B(P \times \mathbf{r}, C, \mathfrak{T}, I, \mathcal{I}_1, \mathcal{I}_2) &= \mathrm{Re}_2(P \times \mathbf{r}, C) + \sum_{k = 1}^3 \Delta_C(r_{i_k}) \\
&+ \sum_{g \in G(C) - \{1\}} \sum_{1 \leq k_1 < k_2 \leq 3} \inv_{\mathrm{AUp}(r_{i_{k_2}})}\left(\chi_{T(C), r_{i_{k_1}}}^g \cup \chi_{T(C), r_{i_{k_2}}}\right) \\
&+ \sum_{w_1 \in \mathcal{I}_1} \left(\phi(\Frob_{w_1}) - \mathrm{pr}_{w_1}(\mathbf{a})\right) \\
&+ \sum_{w_2 \in \mathcal{I}_2} \left(\sum_{g \in G(C) - \{1\}} \phi^g(\Frob_{w_2}) - \mathrm{pr}_{w_2}(\mathbf{b})\right) \\
&+ A_1(P \times \mathbf{r}, C, I) + \sum_{(\mathbf{t}, i) \in \mathfrak{T}} \mathrm{Re}(P \times \mathbf{r}, C, i, \mathbf{t}).
\end{align*}
Since $S_{\text{spl}}$ has cardinality exactly $3$, there exist two indices $i_1, i_2 \in S_{\text{spl}}$ (which we may call $i_1$ and $i_2$ after relabeling if needed) that are either both contained in $I$ or are both contained in $[r] - S_{\text{var}} - I$. We apply Cauchy--Schwarz, in the sense of Lemma \ref{lRCS}, over the sets $X_{i_1}$ and $X_{i_2}$ (so $C$, $P$ and $X_{i_3}$ are fixed). In the process, we use the functional sequence of equivalence relations $f_1'$, $f_2'$ that precisely imposes the following conditions on $r_1, r_1' \in X_{i_1}$ and $r_2, r_2' \in X_{i_2}$:
\begin{enumerate}
\item[$(F1)$] the maps $\phi_{C; r_1r_1'}(\mathfrak{G})$ and $\phi_{C; r_2r_2'}(\mathfrak{G})$ exist and are locally trivial at all places $v$ with $v(d) = 1$,
\item[$(F2)$] for each $i \in S_{\text{var}}$ and for each pair of disjoint proper subsets $S_1, S_2 \subseteq S_{\text{var}} - \{i\}$, the map $\hat{\phi}_{\{Q_{j_1, 1}Q_{j_1, 2} : j_1 \in S_1\}, \{Q_{j_2, 1}Q_{j_2, 2} : j_2 \in S_2\}; Q_{i, 1}Q_{i, 2}, r_1r_1'}(\mathfrak{G})$ exists,
\item[$(F3)$] $\chi_{T(C), r_1} + \chi_{T(C), r_1'} = \res_{G_{F(C)}} \phi_{C; r_1r_1'}(\mathfrak{G})$ and $\chi_{T(C), r_2} + \chi_{T(C), r_2'} = \res_{G_{F(C)}} \phi_{C; r_2r_2'}(\mathfrak{G})$.
\end{enumerate}
Call such a cube $C \times \{r_1, r_1'\} \times \{r_2, r_2'\}$ prepared if it moreover satisfies the condition that $r_1, r_1', r_2, r_2'$ split completely in $\prod_{i \in S_{\text{var}}} L\left(\phi_{\mathrm{Cube}(S_{\text{var}} - \{i\}, C); Q_{i, 1} Q_{i, 2}}(\mathfrak{G})\right)/\mathbb{F}_q(T)$; note that this condition is implied by $(W2)$.

We will briefly explain how to adapt the proof of Lemma \ref{lFuncSeqEq} to prove that the conditions $(F1)$, $(F2)$, $(F3)$ can be expressed as a functional sequence of equivalence relations with complexity bounded by $\exp(c_1 \cdot r^2 \cdot 4^s)$ for some absolute $c_1 > 0$. Condition $(F1)$ can be treated in the same way as the acceptable condition was treated in Lemma \ref{lFuncSeqEq}. Condition $(F2)$ can be treated similarly to how $(A2)$ was treated in Lemma \ref{lFuncSeqEq}, except that we substitute Proposition \ref{prop: hat phi are additive} and Proposition \ref{prop: criterion of existence for hat phi} for Proposition \ref{pExpansionAdditive} and Proposition \ref{inductively construct expansion maps}. Condition $(F3)$ is visibly of the right shape but it is less clear how to bound the number of equivalence classes: to this end, we note that $\phi_{C; r_1r_1'}(\mathfrak{G}) - \chi_{T(C), r_1} - \chi_{T(C), r_1'}$ is unramified outside of $T(C)$, hence we can test whether this expression is zero by evaluating at $\sigma_w$ and $\Frob_w$ for all $w \in T(C)$; we recall that we have a good bound for $T(C)$ thanks to \eqref{eNakaBound}.

Now, as we sum $B(P \times \mathbf{r}, C, \mathfrak{T}, I, \mathcal{I}_1, \mathcal{I}_2)$ over the four natural subcubes $C \times \{r_{i_1, \alpha}\} \times \{r_{i_2, \alpha'}\}$ of $C \times \{r_{i_1, 1}, r_{i_1, 2}\} \times \{r_{i_2, 1}, r_{i_2, 2}\}$, we see that $\mathrm{Re}_2(P \times \mathbf{r}, C)$ sums to zero, and so do $\sum_{w_1 \in \mathcal{I}_1} \left(\phi(\Frob_{w_1}) - \mathrm{pr}_{w_1}(\mathbf{a})\right)$ and $\sum_{w_2 \in \mathcal{I}_2} \left(\sum_{g \in G(C) - \{1\}} \phi^g(\Frob_{w_2}) - \mathrm{pr}_{w_2}(\mathbf{b})\right)$. 

Critically, the term $A_1(P \times \mathbf{r}, C, I)$ also disappears, when summed over four such cubes. Indeed, if $i_1$ and $i_2$ are both contained in $[r] - S_{\text{var}} - I$, then this is clear. If instead $i_1, i_2$ are both contained in $I$, then the term 
$$
\phi_{C; r_1r_1'}(\mathfrak{G})(\Frob_{r_2}) + \phi_{C; r_1r_1'}(\mathfrak{G})(\Frob_{r_2'}) + \phi_{C; r_2r_2'}(\mathfrak{G})(\Frob_{r_1}) + \phi_{C; r_2r_2'}(\mathfrak{G})(\Frob_{r_1'}) = 0
$$
vanishes by Hilbert reciprocity applied to the cocycle $\phi_{C; r_1r_1'}(\mathfrak{G}) \cup \phi_{C; r_2r_2'}(\mathfrak{G})$ combined with the $3$-wonderful condition, which allows us to replace $\Frob_{r_i}$ with the Frobenius symbol of any place of $F(C)$ above $r_i$, so in particular with $\Frob_{\mathrm{AUp}(r_i)}$.

Finally, we apply both Theorem \ref{theorem: profitability} as well as Theorem \ref{theorem: splitting the Redei symbol} to simplify the sum over $\mathrm{Re}(P \times \mathbf{r}, C, i, \mathbf{t})$. Here we note that the final result of Theorem \ref{theorem: profitability} does not depend on the type $\mathbf{t}$. Moving the sum over $X_{i_3}$ to the outside and combining these manipulations gives the following sum
\begin{equation}
\label{ePreFinalHolder}
\sum_{\substack{\tilde{P} \in \prod_{i \in S_{\text{mi}} \cup \{i_3\}} X_i \\ Q_1 = (Q_{i, 1})_i \in \prod_{i \in S_{\text{var}}} X_i \\ Q_2 = (Q_{i, 2})_i \in \prod_{i \in S_{\text{var}}} X_i \\ (L_{k, w})_{k \in [3], w \in T(C) \cup S(a')} \\ \mathbf{a}, \mathbf{b} \in \mathbb{F}_2^{T(C)}, \mathfrak{T}, I, \mathcal{I}_1, \mathcal{I}_2}} \left| \sum_{\substack{
(r_{i_1, 1}, r_{i_1, 2}) \in X_{i_1}^2, \ (r_{i_2, 1}, r_{i_2, 2}) \in X_{i_2}^2 \\ 
((Q_1, Q_2), \tilde{P} \times \{r_{i_1, k_1}\} \times \{r_{i_2, k_2}\}) \text{ 3-wonderful}, \ \res_w \psi_{T(C), r_{i_\nu, \alpha}} = L_{\nu, w} \\
C \times \tilde{P} \times \{r_{i_1, 1}, r_{i_1, 2}\} \times \{r_{i_2, 1}, r_{i_2, 2}\} \text{ prepared} \\
\forall \, (\mathbf{a}, \alpha, \alpha') \in C^- \times \{1, 2\}^2 \, : \, (\prod_{i \in S_{\text{var}}} Q_{i, \mathrm{pr}_i(\mathbf{a})}, \tilde{P} \times \{r_{i_1, \alpha}\} \times \{r_{i_2, \alpha'}\}) \in Y(R, \mathbf{T}, \mathbf{v})}} 
\hspace{-1cm} (-1)^\xi \right|,
\end{equation}
where for some functions $\lambda_1$ and $\lambda_2$ which are respectively independent of the pair $(r_{i_2, 1}, r_{i_2, 2})$ and $(r_{i_1, 1}, r_{i_1, 2})$
\begin{multline*}
\xi := \lambda_1(r_{i_1, 1}, r_{i_1, 2}) + \lambda_2(r_{i_2, 1}, r_{i_2, 2}) + \hspace{-0.3cm} \sum_{\substack{j \in S_{\text{var}} \\ |\{\mathbf{t} : (\mathbf{t}, j) \in \mathfrak{T}\}| \equiv 1 \bmod 2}} \sum_{\alpha' \in \{1, 2\}} \hat{\phi}_{C, j}(r_{i_1, 1} r_{i_1, 2})(\Frob_{\mathrm{AUp}(r_{i_2, \alpha'})}) \\
+ \sum_{g \in G(C) - \{1\}} \sum_{\alpha' \in \{1, 2\}} \inv_{\mathrm{AUp}(r_{i_2, \alpha'})} \left(\phi_{C; r_{i_1, 1} r_{i_1, 2}}(\mathfrak{G})^g \cup \phi_{C; r_{i_2, 1} r_{i_2, 2}}(\mathfrak{G})\right),
\end{multline*}
where $\hat{\phi}_{C, j}$ denotes the map $\hat{\phi}_j$ from Theorem \ref{theorem: splitting the Redei symbol}.

We will also apply the triangle inequality to move $r_{i_1, 2}$ and $r_{i_2, 2}$ to the outer sum (note that we can conveniently combine them with $\tilde{P}$ to get a new element $P' := \tilde{P} \times \{r_{i_1, 2}\} \times \{r_{i_2, 2}\} \in \prod_{i \not \in S_{\text{var}}} X_i$). By introducing coefficients (depending only on $Q_1$ and $L_{k, w}$), we may move the sum over $Q_1$ and $L_{k, w}$ to the inside again. Considering the resulting inner sum in \eqref{ePreFinalHolder}, we will now work towards an application of Lemma \ref{lBoundaries}. We will apply Lemma \ref{lBoundaries} for each fixed value of the outer sum over $P'$ and $Q_2$. Critically, the condition 
$$
\forall \, (\mathbf{a}, \alpha, \alpha') \in C^- \times \{1, 2\}^2 \, : \, \left(\prod_{i \in S_{\text{var}}} Q_{i, \mathrm{pr}_i(\mathbf{a})}, \tilde{P} \times \{r_{i_1, \alpha}\} \times \{r_{i_2, \alpha'}\}\right) \in Y(R, \mathbf{T}, \mathbf{v})
$$
factors, for each fixed choice of $(\mathbf{a}, \alpha, \alpha')$, through at least one of the coordinates of the tuple $(Q_1, r_{i_1, 1}, r_{i_2, 1})$ (here it is essential that we have removed the point $(1, \dots, 1)$ from $C^-$). Thus, we can rewrite this condition as a product of the functions $p_i$ in Lemma \ref{lBoundaries}, and we can do so similarly for the condition $(W2)$ and the condition that $C$ is pre-acceptable and the condition that $\phi_{C; a' r_{i_1}r_{i_2}r_{i_3}}(\mathfrak{G})$ has residue field degree $1$ at all places $v$ with $v(d) = 1$. We can also absorb the condition $\res_w \psi_{T(C), r_{i_k}} = L_{k, w}$ in this way. Having done this, it follows from Lemma \ref{lBoundaries} and additivity that the inner sum in equation \eqref{ePreFinalHolder} is, up to acceptable losses, at most
$$
\prod_{i \in S_{\text{var}} \cup \{i_1, i_2\}} |X_i|^{\frac{2^{s + 1} - 1}{2^{s + 1}}} \sum_{\substack{P' \in \prod_{i \not \in S_{\text{var}}} X_i \\ Q_2 = (Q_{i, 2})_i \in \prod_{i \in S_{\text{var}}} X_i}} \left(\sum_{\substack{Q_1 = (Q_{i, 1})_i, Q_1' = (Q'_{i, 1})_i \in \prod_{i \in S_{\text{var}}} X_i \\ r_{i_1, 1}, r_{i_1, 1}' \in X_{i_1}, \ \ r_{i_2, 1}, r_{i_2, 1}' \in X_{i_2} \\ \text{ all cubes prepared}}} (-1)^{\xi'}\right)^{1/2^{s + 2}},
$$
where
\begin{multline*}
\xi' = \sum_{\substack{j \in S_{\text{var}} \\ |\{\mathbf{t} : (\mathbf{t}, j) \in \mathfrak{T}\}| \equiv 1 \bmod 2}} \left(\hat{\phi}_{C', j}'(r_{i_1, 1} r_{i_1, 1}')(\Frob_{\mathrm{AUp}(r_{i_2, 1})}) + \hat{\phi}_{C', j}'(r_{i_1, 1} r_{i_1, 1}')(\Frob_{\mathrm{AUp}(r_{i_2, 1}')})\right) \\
+ \sum_{g \in G(C') - \{1\}} \inv_{\mathrm{AUp}(r_{i_2, 1})} \left(\phi_{C'; r_{i_1, 1} r_{i_1, 1}'}(\mathfrak{G})^g \cup \phi_{C'; r_{i_2, 1} r_{i_2, 1}'}(\mathfrak{G})\right) \\
+ \sum_{g \in G(C') - \{1\}} \inv_{\mathrm{AUp}(r_{i_2, 1}')} \left(\phi_{C'; r_{i_1, 1} r_{i_1, 1}'}(\mathfrak{G})^g \cup \phi_{C'; r_{i_2, 1} r_{i_2, 1}'}(\mathfrak{G})\right)
\end{multline*}
and where $C'$ is the cube corresponding to $Q_1$ and $Q_1'$ and where $\hat{\phi}_{C', j}'$ is the map from Proposition \ref{prop: linear independence of hat phi and phi}.

We now bring the sum over $r_{i_1, 1}$, $r_{i_1, 1}'$ and $r_{i_2, 1}$ to the inside (keeping the variable $r_{i_2, 1}'$ outside of the sum). To end the proof of Theorem \ref{tReduction}, we will explain how to apply Lemma \ref{lLS} to this inner sum.

Firstly, we remark that
$$
\inv_{\mathrm{AUp}(r_{i_2, 1})} \left(\phi_{C'; r_{i_1, 1} r_{i_1, 1}'}(\mathfrak{G})^g \cup \phi_{C'; r_{i_2, 1} r_{i_2, 1}'}(\mathfrak{G})\right) = \phi_{C'; r_{i_1, 1} r_{i_1, 1}'}(\mathfrak{G})^g(\Frob_{\mathrm{AUp}(r_{i_2, 1})}) 
$$
and similarly
$$
\inv_{\mathrm{AUp}(r_{i_2, 1}')} \left(\phi_{C'; r_{i_1, 1} r_{i_1, 1}'}(\mathfrak{G})^g \cup \phi_{C'; r_{i_2, 1} r_{i_2, 1}'}(\mathfrak{G})\right) = \phi_{C'; r_{i_1, 1} r_{i_1, 1}'}(\mathfrak{G})^g(\Frob_{\mathrm{AUp}(r_{i_2, 1}')}).
$$
Secondly, setting
$$
K := \prod_{i \in S_{\text{var}}} L\left(\phi_{\mathrm{Cube}(S_{\text{var}} - \{i\}, C'); Q_{i, 1} Q'_{i, 1}}(\mathfrak{G})\right),
$$
we remark that $\phi_{C'; r_{i_1, 1} r_{i_1, 1}'}(\mathfrak{G})$ and every $\hat{\phi}_{C', j}'(r_{i_1, 1} r_{i_1, 1}')$ restrict to quadratic characters of $G_K$.

We take $n$ to be the number of pairs $(r_{i_1, 1}, r_{i_1, 1}') \in X_{i_1}^2$ such that $\phi_{C'; r_{i_1, 1} r_{i_1, 1}'}(\mathfrak{G})$ exists and $\hat{\phi}_{C', j}'(r_{i_1, 1} r_{i_1, 1}')$ exists for each $j$. To each such pair $(r_{i_1, 1}, r_{i_1, 1}')$, we attach the quadratic character 
$$
(-1)^{\sum\limits_{j \in S_{\text{var}}} |\{\mathbf{t} : (\mathbf{t}, j) \in \mathfrak{T}\}| \cdot \hat{\phi}_{C', j}'(r_{i_1, 1} r_{i_1, 1}') + \sum_{g \in G(C') - \{1\}} \phi_{C'; r_{i_1, 1} r_{i_1, 1}'}(\mathfrak{G})^g}.
$$
Then we can use Proposition \ref{prop: linear independence of hat phi and phi} (and ramification considerations as in Proposition \ref{normalized expansion maps are unramified} and Proposition \ref{prop: field of definition of hat phi}) to bound $|D|$ with $D$ as defined in Lemma \ref{lLS}. Choosing appropriate coefficients $\alpha_i$ and $\beta_j$ in Lemma \ref{lLS}, Theorem \ref{tReduction} follows.
\end{proof}


\begin{thebibliography}{39}
\bibitem{AK}
J. C. Andrade and J. P. Keating.
The mean value of $L(\frac{1}{2}, \chi)$ in the hyperelliptic ensemble.
\emph{J. Number Theory} 132 (2012), no. 12, 2793--2816.

\bibitem{BK}
S. Bloch and K. Kato.
$L$-functions and Tamagawa numbers of motives
The Grothendieck Festschrift, Vol. I, 333--400.
Progr. Math., 86
\emph{Birkh\"auser Boston, Inc., Boston, MA,} 1990.

\bibitem{BF18}
H. M. Bui and A. Florea.
Zeros of quadratic Dirichlet $L$-functions in the hyperelliptic ensemble.
\emph{Trans. Amer. Math. Soc.} 370 (2018), 8013--8045.

\bibitem{CDD}
C. Castillo, A. de Faveri and A. Dunn.
Non-vanishing for quartic Hecke $L$-functions and ranks of elliptic curves.
\emph{arXiv preprint:}2604.01316.

\bibitem{Cha97}
N. Chavdarov.
The generic irreducibility of the numerator of the zeta function in a family of curves with large monodromy. 
\emph{Duke Math. J.} 87 (1997), 151--180.

\bibitem{Cho65}
S. D. Chowla.
The Riemann Hypothesis and Hilbert’s Tenth Problem, Vol.~4.
\emph{Gordon and Breach Science Publishers, New York,} 1965, xv+119 pp.

\bibitem{CDLL}
A. Comeau-Lapointe, C. David, M. Lal\'in and W. Li.
On the vanishing of twisted $L$-functions of elliptic curves over rational function fields.
\emph{Res. Number Theory} 8 (2022), no. 4, Paper No. 76, 28 pp.

\bibitem{DDDS}
C. David, A. de Faveri, A. Dunn and J. Stucky.
Non-vanishing for cubic Hecke $L$-functions.
\emph{arXiv preprint:}2410.03048.

\bibitem{DFL}
C. David, A. Florea and M. Lal\'in.
Non-vanishing for cubic $L$--functions.
\emph{Forum Math. Sigma} 9 (2021), Paper No. e69, 58 pp.

\bibitem{DFL2}
C. David, A. Florea and M. Lal\'in.
Nonvanishing of $L$--functions associated to fixed order characters over function fields.
\emph{arXiv preprint:}2506.07815.

\bibitem{EVW16}
J. Ellenberg, A. Venkatesh and C. Westerland. 
Homological stability for Hurwitz spaces and the Cohen-Lenstra conjecture over function fields.
\emph{Ann. of Math. (2)} 183 (2016), 729--786.

\bibitem{ELS19} 
J. Ellenberg, W. Li and M. Shusterman.
Nonvanishing of hyperelliptic zeta functions over finite fields. 
\emph{Algebra Number Theory} 14 (2020), 1895--1909.

\bibitem{Flach}
M. Flach.
A generalisation of the Cassels-Tate pairing.
\emph{J. Reine Angew. Math}. 412 (1990), 113--127.

\bibitem{Florea}
A. Florea. 
The fourth moment of quadratic Dirichlet $L$-functions over function fields. 
\emph{Geom. Funct. Anal.} 27 (2017), no. 3, 541--595.

\bibitem{FK}
\'E. Fouvry and J. Kl\"uners.
On the $4$-rank of class groups of quadratic number fields.
\emph{Invent. Math.} 167 (2007), no. 3, 455--513.

\bibitem{FJ}
M. D. Fried and M. Jarden.
Field arithmetic.
Fourth edition. Revised by Moshe Jarden.
Ergeb. Math. Grenzgeb. (3), 11.
\emph{Springer, Cham,} 2023. xxxi+827 pp.

\bibitem{GZ}
P. Gao and L. Zhao. 
Moments of quadratic Dirichlet $L$-functions over function fields.
\emph{Finite Fields Appl.} 85 (2023), Paper No. 102113, 32 pp.

\bibitem{HB}
D. R. Heath-Brown.
The size of Selmer groups for the congruent number problem.
\emph{Invent. Math.} 111 (1993), no. 1, 171--195.

\bibitem{HB2}
D. R. Heath-Brown.
The size of Selmer groups for the congruent number problem. II.
\emph{Invent. Math.} 118 (1994), no. 2, 331--370.

\bibitem{ILS00}
H. Iwaniec, W. Luo and P. Sarnak. 
Low lying zeros of families of $L$-functions.
\emph{Inst. Hautes \'{E}tudes Sci. Publ. Math.} 91 (2000), 55--131.

\bibitem{Jut81}
M. Jutila.
On the mean value of $L\left( \frac{1}{2}, \chi \right)$ for real characters.
\emph{Analysis} 1 (1981), 149--161.

\bibitem{KS99}
N. Katz and P. Sarnak. 
Random matrices, Frobenius eigenvalues, and monodromy.
Amer. Math. Soc. Colloq. Publ., 45.
\emph{American Mathematical Society, Providence,} RI, 1999, xii+419 pp.

\bibitem{KP}
P. Koymans and C. Pagano.
On Stevenhagen's conjecture.
\emph{Acta Math.}, to appear.

\bibitem{KS}
P. Koymans and A. Smith.
Sums of rational cubes and the 3-Selmer group.
\emph{arXiv preprint:}2405.09311.

\bibitem{LOP}
M. Lal\'in, K.-H. Lee, T. Oliver and A. Pozdnyakov.
Murmurations of quadratic and cubic characters over function fields.
\emph{arXiv preprint:}2608.01337.

\bibitem{Li18}
W. Li.
Vanishing of hyperelliptic $L$-functions at the central point.
\emph{J. Number Theory} 191 (2018), 85--103.

\bibitem{LT19}
M. Lipnowski and J. Tsimerman.
Cohen--Lenstra heuristics for \'{e}tale group schemes and symplectic pairings.
\emph{Compos. Math.} 155 (2019), 758--775.

%
%
%
%

\bibitem{MS}
A. Morgan and A. Smith.
Field change for the Cassels-Tate pairing and applications to class groups.
\emph{Res. Number Theory} 10 (2024), Paper No. 61, 46 pp.

\bibitem{NSW}
J. Neukirch, A. Schmidt and K. Wingberg.
Cohomology of Number Fields. Second. Vol. 323. 
Grundlehren der Mathematischen Wissenschaften. 
\emph{Springer-Verlag, Berlin,} 2008.

\bibitem{OS99}
A. E. \"{O}zl\"{u}k and C. Snyder.
On the distribution of the nontrivial zeros of quadratic $L$-functions close to the real axis.
\emph{Acta Arith.} 91 (1999), 209--228.

\bibitem{Redei}
L. R\'edei.
Ein neues zahlentheoretisches Symbol mit Anwendungen auf die Theorie der quadratischen Zahlk\"orper. I. 
\emph{J. Reine Angew. Math.} 180 (1939), 1--43.

\bibitem{Rud10} 
Z. Rudnick.
Traces of high powers of the Frobenius class in the hyperelliptic ensemble.
\emph{Acta Arith.} 143 (2010), 81--99.

\bibitem{Rum93}
R. Rumely.
Numerical computations concerning the ERH. 
\emph{Math. Comp.} 61 (1993), 415--440.

\bibitem{Smi3}
A. Smith.
The Birch and Swinnerton-Dyer conjecture implies Goldfeld's conjecture.
\emph{arXiv preprint:}2503.17619.

\bibitem{Smi1}
A. Smith.
The distribution of $\ell^\infty$-Selmer groups in degree $\ell$ twist families I.
\emph{J. Amer. Math. Soc.} 39 (2026), no. 1, 1--72.

\bibitem{Smi2}
A. Smith.
The distribution of $\ell^\infty$-Selmer groups in degree $\ell$ twist families II.
\emph{J. Amer. Math. Soc.} 39 (2026), no. 2, 453--514.

\bibitem{Sou00}
K. Soundararajan.
Nonvanishing of quadratic Dirichlet $L$-functions at $s = \frac{1}{2}$.
\emph{Ann. of Math. (2)} 152 (2000), 447--488.

\bibitem{Ste}
P. Stevenhagen.
Redei reciprocity, governing fields and negative Pell.
\emph{Math. Proc. Cambridge Philos. Soc.} 172 (2022), no. 3, 627--654.

\bibitem{Stokvis}
J. Stokvis.
Governing fields for hyperelliptic function fields.
\emph{arXiv preprint:}2505.20117.
\end{thebibliography}
\end{document}